\documentclass[
               a4paper]{amsart}

\numberwithin{equation}{section}
\usepackage{amssymb}
\usepackage{mathrsfs}
\usepackage{xspace}
\usepackage{marvosym} 
\usepackage[sans]{dsfont}        
\usepackage{upgreek}
\usepackage{yhmath}
\usepackage{verbatim}
\usepackage[colorlinks]{hyperref}
\hypersetup{
    colorlinks=true,       
    linkcolor=solarized-red,          
    citecolor=solarized-green
}
\usepackage[inline]{enumitem}
\setlist[enumerate,1]{label=\textup{(\alph*)}}

\usepackage{pagecolor}
\definecolor{solarized-base03}{HTML}{262626}
\definecolor{solarized-base02}{HTML}{303030}
\definecolor{solarized-base01}{HTML}{6A6A6A}
\definecolor{solarized-base00}{HTML}{777777}
\definecolor{solarized-base0}{HTML}{969696}
\definecolor{solarized-base1}{HTML}{9E9E9E}
\definecolor{solarized-base2}{HTML}{E8E8E8}
\definecolor{solarized-base3}{HTML}{F6F6F6}
\definecolor{solarized-yellow}{HTML}{B58900}
\definecolor{solarized-orange}{HTML}{CB4B16}
\definecolor{solarized-red}{HTML}{DC322F}
\definecolor{solarized-magenta}{HTML}{D33682}
\definecolor{solarized-violet}{HTML}{6C71C4}
\definecolor{solarized-blue}{HTML}{268BD2}
\definecolor{solarized-cyan}{HTML}{2AA198}
\definecolor{solarized-green}{HTML}{859900}

\newtheorem{thm}{Theorem}[section]
\newtheorem{lem}[thm]{Lemma}
\newtheorem{cor}[thm]{Corollary}
\newtheorem{sublem}[thm]{Sub-lemma}
\newtheorem{prop}[thm]{Proposition}

\newtheorem{defin}[thm]{Definition}

\newtheorem{rem}[thm]{Remark}
\newtheorem*{irem}{Remark}
\newtheorem{notation}[thm]{Notation}

\newcommand{\cA}{{\mathcal{A}}}
\newcommand{\cB}{{\mathcal{B}}}
\newcommand{\cC}{{\mathcal{C}}}
\newcommand{\cD}{{\mathcal{D}}}
\newcommand{\cE}{{\mathcal{E}}}

\newcommand{\cG}{{\mathcal{G}}}
\newcommand{\cH}{{\mathcal{H}}}
\newcommand{\cI}{{\mathcal{I}}}
\newcommand{\cJ}{{\mathcal{J}}}
\newcommand{\cK}{{\mathcal{K}}}
\newcommand{\cL}{{\mathcal{L}}}
\newcommand{\cN}{{\mathcal{N}}}
\newcommand{\cM}{{\mathcal{M}}}
\newcommand{\cO}{{\mathcal{O}}}
\newcommand{\cP}{{\mathcal{P}}}
\newcommand{\cQ}{{\mathcal{Q}}}
\newcommand{\cR}{{\mathcal{R}}}
\newcommand{\cS}{{\mathcal{S}}}
\newcommand{\cT}{{\mathcal{T}}}

\newcommand{\cZ}{{\mathcal{Z}}}

\newcommand{\bA}{{\mathbb{A}}}

\newcommand{\bC}{{\mathbb{C}}}
\newcommand{\bD}{{\mathbb{D}}}
\newcommand{\bE}{{\mathbb{E}}}
\newcommand{\bF}{{\mathbb{F}}}
\newcommand{\bG}{{\mathbb{G}}}

\newcommand{\bK}{{\mathbb{K}}}

\newcommand{\bN}{{\mathbb{N}}}
\newcommand{\bM}{{\mathbb{M}}}

\newcommand{\bP}{{\mathbb{P}}}

\newcommand{\bR}{{\mathbb{R}}}
\newcommand{\bT}{{\mathbb{T}}}
\newcommand{\bX}{{\mathbb{X}}}
\newcommand{\bZ}{{\mathbb{Z}}}

\newcommand{\fkC}{{\mathfrak{C}}}
\newcommand{\fkD}{{\mathfrak{D}}}
\newcommand{\fkE}{{\mathfrak{E}}}

\newcommand{\fkK}{{\mathfrak{K}}}

\newcommand{\fkL}{{\mathfrak{L}}}

\newcommand{\fkP}{{\mathfrak{P}}}
\newcommand{\fkR}{{\mathfrak{R}}}

\newcommand\erR[1]{\fkR_{#1}}
\newcommand\oerR[1]{\overline\fkR_{#1}}

\newcommand\Osc{\textup{Osc}\,}

\newcommand\ve{\varepsilon}
\newcommand\eps{\epsilon}
\newcommand\vf{\varphi}
\newcommand\deltacomplex{\delta_{{\scriptscriptstyle\bC}}}
\newcommand\deltac{\delta_{\textup{c}}}

\newcommand{\bbK}{{\fkK}}
\newcommand{\bbbK}{{\bar\fkK}}
\newcommand{\bbKs}{{\fkK}}
\newcommand{\BOmega}{{\boldsymbol \Omega}}
\newcommand{\bchi}{{\boldsymbol \chi}}
\newcommand{\Clip}{C_{\textrm{Lip}}}
\newcommand{\pathnet}{\mathscr{A}}
\newcommand\vestep{\varsigma}

\newcommand\vesteps{\varsigma_\star}
\newcommand\sstep{Q}
\newcommand\fixsstep{q^*}

\newcommand\zet{\kappa}

\newcommand\Id{{\mathds{1}}}

\newcommand{\paramL}{\varsigma}
\newcommand{\paramb}{\beta}
\newcommand{\efrac}[2]{#1/#2} 
\newcommand{\eefrac}[2]{^{\efrac{#1}{#2}}} 
\newcommand\bRp{{\bR_{\scriptscriptstyle\ge0}}} 
\newcommand{\avgtheta}[1]{\thetas_{#1}}
\newcommand{\bavgtheta}[1]{\bar \theta^*_{#1}}
\newcommand\thetas{\theta^*}

\newcommand\sigmas{\sigma^*}
\newcommand\thetasl{{\avgtheta\ell}}
\newcommand{\thetaslk}[1]{\bar\theta^*_{\ell, {#1}}}

\newcommand{\thetaslkk}[2]{\bar\theta^*_{\ell_{#1}, {#2}}}
\newcommand\thetaslz{{\avgtheta\ellz}}
\newcommand\thetaref{{\theta^*}}

\newcommand\lprecisione{{\vartheta}}
\newcommand\lseparation{{\Delta_*}}
\newcommand\stable{{s}}  
\newcommand\Var{\boldsymbol\upsigma} 
\newcommand{\greenkubo}{\hat{\boldsymbol\upsigma}}
\newcommand\bVar{\greenkubo^2}
\newcommand\ellz{{\ell_0}}
\newcommand\rage{{\bar r}}

\newcommand\sillyconst{(\ln|\sigma|)^2|\sigma|\ve}
\newcommand\derivL{\cD}

\newcommand\bAs{\bA^{\hskip-1.3pt\dag}}
\newcommand\timeI{J}
\newcommand\gammah{\gamma_h}
\newcommand\raggio{\bar R_T}
\newcommand\raggiod{\hat R_\delta}

\newcommand\Czero{C_1}
\newcommand\Cuno{C_2}
\newcommand\Ctre{C_6}
\newcommand\Cquattro{C_5}
\newcommand\Ccinque{C_3}
\newcommand\Csei{C_7}
\newcommand\Csette{C_4}

\newcommand\CJ[1]{%
  \ifnum0=#1\relax%
  C_{10}
  \else%
  \ifnum1=#1\relax%
  C_{11}\else%
  \ifnum2=#1\relax%
  C_{12}\else%
  \ifnum3=#1\relax%
  C_{13}\else%
  \ifnum4=#1\relax%
  C_{14}\else%
  \ifnum5=#1\relax%
  \tilde C_0\else%
  \ifnum6=#1\relax%
  C_{16}\else%
  \ifnum7=#1\relax%
  C_{17}\else%
  \Const%
  \fi\fi\fi\fi\fi\fi\fi\fi} 

\newcommand{\Constgs}{C_{**}}

\newcommand{\CJdue}{C_{12}}

\newcommand{\bXs}{\bX^*}
\newcommand{\bXss}{\bX^{**}}

\newcommand{\testfunc}{g}
\newcommand{\testfuncp}{\overline \vf}
\newcommand{\auxPot}{\Phi}
\newcommand{\auxPots}{\Phi^*}

\newcommand{\bzetgtsl}{\zet_{\gamma,\thetasl}}
\newcommand{\bzetgts}{\zet_{\gamma,\thetas}}
\newcommand{\bzetgt}{\zet_{\gamma,\theta}}
\newcommand{\hgamma}{\hat\gamma}
\newcommand{\vgamma}{\check\gamma}
\newcommand{\tgamma}{\underline\gamma}

\newcommand{\Const}{{C_\#}}
\newcommand{\ConstPow}[1]{{C_\#^{#1}}}
\newcommand{\Consti}{{C_\#\invr}}
\newcommand{\const}{{c_\#}}

\newcommand{\vei}{\ve^{-1}}
\newcommand{\veh}{\ve^{\efrac12}}
\newcommand{\veih}{\ve^{-\efrac12}}

\newcommand{\deh}{d}
\newcommand{\de}[2]{\frac{\deh{}#1}{\deh{}#2}}

\newcommand{\st}{\;:\; }

\newcommand{\invr}{^{-1}}
\newcommand{\pheq}{\phantom{=}}

\newcommand{\dist}{\textup{dist}}
\newcommand{\supp}{\textup{supp}\,}
\newcommand{\pint}[1]{{\lfloor#1\rfloor}}

\newcommand{\Lip}{\operatorname{Lip}}

\newcommand{\first}{{\text{\rm pre}}}
\newcommand{\logmgf}{\Lambda_{\ell,\ve}}
\newcommand{\restoPalla}{\cR_{\ell,\ve}}

\newcommand{\spc}[1]{c_#1}
\newcommand{\thetaArgument}{\vartheta} 
\renewcommand{\ln}{\log} 
\newcommand{\ie}{i.e.\ }

\newcommand{\resp}{resp.\ }

\definecolor{majentic}{rgb}{0.5,0.1,0}
\definecolor{removal}{rgb}{0.6,0.6,1}
\renewcommand{\wp}{\varrho} 

\definecolor{referee}{rgb}{0.8,0,0}

\newcommand\numToPrime[1]{%
  \ifnum0=#1\relax%
  \else%
  \ifnum1=#1\relax%
  '\else%
  \ifnum2=#1\relax%
  ''\else%
  \ifnum3p=#1\relax%
  '''\else%
  ^{(#1)}%
  \fi\fi\fi\fi} 
\newcommand\dtt[1]{D_1\numToPrime{#1}}

\newcommand{\mr}{\mathring}

\makeatletter

\newcommand\mergedsub[2]{#1\sc@sub{#2}}

\def\sc@sub#1{\def\sc@thesub{#1}\@ifnextchar_{\sc@mergesubs}{_{\sc@thesub}}}
\def\sc@mergesubs_#1{_{\sc@thesub#1}}

\newcommand\mergedsubC[2]{#1\sc@subC{#2}}
\newcommand\newsubcommandC[3]{\newcommand#1{\mergedsubC{#2}{#3}}}
\def\sc@subC#1{\def\sc@thesub{#1}\@ifnextchar_{\sc@mergesubsC}{_{\sc@thesub}}}
\def\sc@mergesubsC_#1{_{#1,\sc@thesub}}

\newcommand\mergedsubCR[2]{#1\sc@subCR{#2}}
\newcommand\newsubcommandCR[3]{\newcommand#1{\mergedsubCR{#2}{#3}}}
\def\sc@subCR#1{\def\sc@thesub{#1}\@ifnextchar_{\sc@mergesubsCR}{_{\sc@thesub}}}
\def\sc@mergesubsCR_#1{_{\sc@thesub,#1}}

\newcommand\mergedsup[2]{#1\sc@sup{#2}}
\newcommand\newsupcommand[3]{\newcommand#1{\mergedsup{#2}{#3}}}

\def\sc@sup#1{\def\sc@thesup{#1}\@ifnextchar^{\sc@mergesups}{^{\sc@thesup}}}
\def\sc@mergesups^#1{^{\sc@thesup #1}}

\newcommand\mergedsupC[2]{#1\sc@supC{#2}}

\def\sc@supC#1{\def\sc@thesup{#1}\@ifnextchar^{\sc@mergesupsC}{^{\sc@thesup}}}
\def\sc@mergesupsC^#1{^{#1,\sc@thesup}}

\newcommand\mergedpar[2]{\sc@parC{#1}{#2}}

\def\sc@parC#1#2{\def\sc@thebase{#1}\def\sc@thepar{#2}\@ifnextchar_{\sc@mergeparsub}{\@ifnextchar'{\sc@mergeparprime}{\@ifnextchar^{\sc@mergeparsup}{\sc@thebase(\sc@thepar)}}}}
\def\sc@mergeparsub_#1{\sc@thebase_{#1}(\sc@thepar)}
\def\sc@mergeparprime#1{\sc@thebase #1(\sc@thepar)}
\def\sc@mergeparsup^#1{\sc@thebase^{#1}(\sc@thepar)}
\makeatother

\usepackage{rotating}

\newcommand{\regu}[2]{\mergedsup{#1}{(#2)}}
\newcommand{\regPi}[1]{\regu{\Pi}{#1}}
\newcommand{\udPi}{\mathbin{\rotatebox[origin=c]{180}{$\Pi$}}}

\newsupcommand\ellc{\ell}{\textup{c}}

\newcommand{\ratef}{\mathscr I}
\newsubcommandCR{\ratefLin}{\ratef}{\textup{Lin}}
\newcommand{\diffeoAvg}{\Upsilon}

\newsubcommandCR{\Ifirst}{I}{\first}

\newcommand{\gfunctional}{\varphi}
\newcommand{\gfunctionalh}{\varphi_h}
\newcommand{\sigmah}{\sigma_h}
\newcommand{\Leb}{\textup{Leb}}
\newcommand{\mnu}{\nu} 
\newcommand{\enu}{\upsilon^+} 
\newcommand{\emu}{\upsilon^{\textrm{c}}} 
\newcommand{\fm}{{\upnu\hskip-.8pt}} 
\newcommand{\fmm}{{\upnu\hskip-.8pt}}
\newcommand\tell{{\tilde\ell}}

\newsubcommandCR{\kse}{h}{\textup{KS}}

\newcommand{\expa}{a}
\newcommand{\expb}{b}

\newcommand{\expo}[1]{\exp\!\left[#1\right]}  

\newcommand\shiftPar{\kappa}

\newcommand\pot{\Omega}

\newcommand\potFamily{\fkP}
\newcommand\potSequence{\mathbf{\Omega}}
\newcommand\potShift{\mathbf{s}}
\newcommand{\SF}[2]{(\{#1\},#2)}

\newcommand{\stdf}{\fkL}

\newcommand{\convexP}{\varrho}
\newcommand\symrv\dagger 
\newsubcommandCR{\etaRefs}{H^*}{\ell}

\newcommand\etaRefb{H}
\newcommand\etaExp{\Xi}
\newcommand\etaExpb{\etaExp}
\newcommand\etaExph{\widehat \Xi}
\newsubcommandCR\etaExps{\Xi^*}{\ell}
\newsubcommandCR\etaExpsk{\Xi^*}{\ell_k}
\newsubcommandCR\etaExpsrage{\Xi^*}{\ell_\rage}

\newcommand{\imFun}{\mathcal H}

\newcommand{\un}{\textup{u}}
\newcommand{\nt}{\textup{c}}
\newcommand{\cone}{\mathfrak{C}}
\newcommand{\coneu}{\cone^\un}
\newcommand{\conec}{\cone^\nt}
\newcommand{\Kcone}{\gamma}
\newcommand{\Kconeu}{\Kcone^\un}
\newcommand{\Kconec}{\Kcone^\nt}

\newcommand\gpoly{\gamma_{
\ve}}    
\newcommand\gavg{\bar\gamma
}     
\newcommand\gspace{C^0([0,T],\bR^d)} 
\newsubcommandC{\ppath}{\bP}{\ve}       
\newsubcommandC{\epath}{\bE}{\ve}       
\newcommand\gball[1]{B
\def\temp{#1}\ifx\temp\empty
{}
\else
  |_{#1} 
  \fi}

\newcommand\QLipC{Q_{C,*}}

\newcommand{\BV}{\textup{BV}}

\newcommand{\tO}{\cL}
\newcommand{\tOh}{\widehat\tO}
\newcommand{\tOn}{\check\tO}
\newcommand{\tOp}{\cP}
\newcommand{\tOq}{\cQ}
\newcommand{\tOqb}{\tilde\tOq}
\newcommand{\tOpb}{\tilde\tOp}
\newcommand{\tOb}{\tilde\tO}
\newcommand{\chib}{\tilde\chi}
\newcommand{\tOqh}{\widehat\tOq}

\newcommand{\harb}{h^*}
\newcommand{\marb}{m^*}

\newcommand\nc[1]{_{\cC^{#1}}}
\newcommand\ncd[2]{_{\cC^{#1}(#2)}} 
\newcommand\nl[1]{_{L^{#1}}}
\newcommand\nw[2]{_{W^{#1,#2}}}

\newcommand\nBV{_\BV}

\newcommand\ho{\hat\omega}
\newcommand\hov{\varpi}

\newcommand{\hatA}{\hat A} 
\newcommand\bOmega{\bar\Omega}
\newsubcommandCR{\Omg}{\Omega}{\auxPot}
\newsubcommandCR{\hOmg}{\widehat\Omega}{\auxPot}
\newsubcommandCR{\smg}{\sigma}{\auxPot} \newsubcommandCR{\sgm}{\sigma}{\auxPot} 
\newsubcommandCR{\fixOmega}{\Omega^*}{\auxPot}
\newsubcommandCR{\fixhOmega}{\widehat\Omega^*}{\auxPot}
\newsubcommandCR{\sbOm}{\tilde\Omega^*}{\auxPot}

\newcommand{\aseq}{{\bf a}}

\newcommand{\dualsigma}{b}

\newcommand\Omegap[1]{\Omega^{(#1)}}

\newcommand\Omegaa{\Omegap0}
\newcommand\Omegab{\Omegap1}

\newcommand\trho{\rho_*}
\newcommand\trhove{\tilde\rho}
\newcommand{\rrho}{\mathring\rho}

\newcommand\bchimin{\bVar_-}
\newcommand\J[1]{\cJ_{#1}}
\newcommand\JJ[1]{\cI_{#1}}

\newsubcommandCR{\oW}{W}{\ell}
\newsubcommandCR{\oWk}{W}{\ell_k}
\newsubcommandCR{\ooW}{\overline W}{\ell}
\newsubcommandCR{\ooWk}{\overline W}{\ell_k}
\newsubcommandCR{\fkW}{{\mathfrak{W}}}{\ell}
\newsubcommandCR{\ofkW}{{\overline{\mathfrak{W}}}}{\ell}

\renewcommand\Re{\textup{Re}}
\renewcommand\Im{\textup{Im}}

\newcommand{\Tmax}{T_\textup{max}}

\newcommand\thvs{_{\theta,\paramL\Omega}}

\newcommand\thvss{_{\theta,{\paramL_2}\Omega}}
\newcommand\thvz{_{\theta,\Omega}}
\newcommand{\largeta}{\varkappa}
\newsubcommandCR\favg{\bar f}{\ell}
\newsubcommandCR\favgp{\bar f'}{\ell}
\newcommand\Favg{\bar F_\ve}
\newcommand\Favgnosub{\bar F}
\makeatletter
\def\underbracket{%
  \@ifnextchar[{\@underbracket}{\@underbracket [0.5pt]}%
}
\def\@underbracket[#1]{%
  \@ifnextchar[{\@under@bracket[#1]}{\@under@bracket[#1][0.4em]}%
}
\def\@under@bracket[#1][#2]#3{
  \mathop{\vtop{\m@th \ialign {##\crcr $\hfil \displaystyle {#3}\hfil $%
        \crcr \noalign {\kern 3\p@ \nointerlineskip }\upbracketfill {#1}{#2}
        \crcr \noalign {\kern 3\p@ }}}}\limits}
\def\upbracketfill#1#2{$\m@th \setbox \z@ \hbox {$\braceld$}
  \edef\@bracketheight{\the\ht\z@}\bracketend{#1}{#2}
  \leaders \vrule \@height #1 \@depth \z@ \hfill
  \leaders \vrule \@height #1 \@depth \z@ \hfill \bracketend{#1}{#2}$}
\def\bracketend#1#2{\vrule height #2 width #1\relax}
\makeatother

\newcommand{\bcoord}[1]{(\bar{x}_{#1},\thetaslk{#1})}
\newcommand{\bcoordk}[2]{(\bar{x}_{#2},\thetaslkk{#1}{#2})}
\newcommand{\coord}[1]{(x_{#1},\theta_{#1})}
\newcommand{\BLambda}{\bar\Lambda}

\usepackage{color}
\definecolor{jacopo}{RGB}{10,92,10}
\definecolor{carlangelo}{RGB}{92,10,10}
\definecolor{reviewColor}{RGB}{255,172,255}

\newcommand\deviatio{\Delta} \newsubcommandCR\deviation{\Delta}{\ell}

\newsubcommandCR\deviationsl{\Delta^*}{\ell}\newsubcommandCR\bdeviationsl{\bar\Delta^*}{\ell}

\newcommand\bdeviationslk[2]{\bar\Delta^*_{#1,#2}}
\newcommand\deviatione{\Delta}
\newcommand\devD{D}

\newcommand{\Averages}{\bD}
\newcommand{\measures}{\mathcal{P}} 
\newcommand{\Rad}[1]{R^{#1}}

\newcommand{\intr}{\textup{int}\,} %
\begin{document}
\title[Limit Theorems]{Limit Theorems for Fast-slow \\ partially hyperbolic systems}
\author{Jacopo De Simoi}
\address{Jacopo De Simoi\\
  Department of Mathematics\\
  University of Toronto\\
  40 St George St. Toronto, ON, Canada M5S 2E4}
\email{{\tt jacopods@math.utoronto.ca}}
\urladdr{\href{http://www.math.utoronto.ca/jacopods}{http://www.math.utoronto.ca/jacopods}}
\author{Carlangelo Liverani}
\address{Carlangelo Liverani\\
  Dipartimento di Matematica\\
  II Universit\`{a} di Roma (Tor Vergata)\\
  Via della Ricerca Scientifica, 00133 Roma, Italy.}
\email{{\tt liverani@mat.uniroma2.it}}
\urladdr{\href{http://www.mat.uniroma2.it/~liverani}{http://www.mat.uniroma2.it/~liverani}}

\begin{abstract}
  We prove several limit theorems for a simple class of partially hyperbolic fast-slow systems.  We start
  with some well know results on averaging, then we give a substantial refinement of known
  large (and moderate) deviation results and conclude with a completely new result (a
  local limit theorem) on the distribution of the process determined by the fluctuations
  around the average.  The method of proof is based on a mixture of standard pairs and
  transfer operators that we expect to be applicable in a much wider generality.
\end{abstract}
 \thanks{ We thank the Fields
  Institute, Toronto were this project started a very long time ago.  Both authors have been
  partially supported by the ERC Grant MALADY (ERC AdG 246953).  JDS acknowledges partial
  NSERC support. Most of all we would like to thank Dmitry Dolgopyat for his many comments
  and suggestion; without his advice this paper would certainly not exist. Finally, we are deeply indebted to the anonymous referees that we thank for their amazing job.}
\maketitle


\section{Introduction }
In this paper we analyze various limit theorems for a class of partially hyperbolic
systems of the fast-slow type.  Such systems are very similar to the ones studied by
Dolgopyat in~\cite{DimaAveraging}: in such paper the fast variables are driven by an
hyperbolic diffeomorphism or flow (see also~\cite{Bakhtin03, Bakhtin04, Kifer09,
  Melbourne-Sutart, Gottwald-Melbourne} for related results), here we consider the case in
which they are driven by an expanding map.  Notwithstanding the fact that we are not aware
of an explicit treatment of the latter case, the difference is not so relevant to justify,
by itself, a paper devoted to it.  In fact, we chose to deal with one dimensional
expanding maps only to simplify the exposition.  The point here is that, on the one hand,
we propose a different approach and, more importantly, on the other hand, we show that by
such an approach it is possible to obtain much sharper results: a \emph{Moderate and Large
  Deviation} Theorem and a \emph{Local Limit} Theorem.  To the best of our knowledge, this
is the first time a rate function is computed with such a precision to yield moderate
deviations of the paths and a local limit type theorem is obtained for a deterministic
evolution converging to a diffusion process with non constant diffusion.  Admittedly, the
present is not the most general case one would like to deal with, it is just a primer.
However, it shows that local limit results are attainable with an appropriate
combination/refinement of present days techniques (see the discussion below on how general
our approach really is).

The importance of local limit theorems hardly needs to be emphasized but, for the
skeptical reader, it is nicely illustrated in~\cite{DeL2, DeLPV}.  Indeed, in such papers
the present large and moderate deviations and local limit results are used in a
fundamental way to obtain a precise understanding of the statistical properties (e.g.\
existence and properties of the SRB measure, decay of correlations, meta-stability
etc$\ldots$) for the same class of systems for a small, but fixed, rate between the speeds
of the slow and fast motions. This provides a class of partially hyperbolic systems for
which very precise {\em quantitative} statistical properties can be established. In
addition, contrary to other cases, our results apply to an open set of systems (in the
$\cC^4$ topology).

For partially hyperbolic fast-slow systems several results concerning limit laws have
already been obtained.  In~\cite{Bakhtin03, Kifer04} it is proven that the motion
converges in probability to the motion determined by the averaged equation (morally a
\emph{law of large numbers}).  In~\cite{DimaAveraging} there are important results on the
fluctuations around the average (at a given time).  In particular, both \emph{large
  deviations} and converges in law to a diffusion for the fluctuation field (morally a
\emph{central limit theorem}) are obtained.  In~\cite{Bakhtin03, Bakhtin04} one can find
very sharp results on normal fluctuation and moderate deviation at a given time. In
particular, in~\cite{Bakhtin04} Bakhtin provides Cramer asymptotics for the distribution
of the slow variable at a fixed time for a system with fast motion given by a mixing
hyperbolic attractor. Such Cramer asymptotics gives estimates for moderate deviation, at a
fixed time, sharper than the one obtained here, but they do not provide directly a rate
function in path space, they hold only under the assumption that the dynamics is $\cC^r$
for a very large $r$ (contrary to our $\cC^4$ assumption) and they are not sufficient to
establish a local central limit theorem.  In~\cite{Kifer09} more general large deviation
results (in path space) are obtained. In particular, a variational formula for the rate
function is established. Yet, Kifer's results are not precise enough to treat moderate
deviations. To obtain a rate function for moderate deviations it is necessary to compute
the exponential momenta with a precision considerably higher than the $o(1)$ achieved
in~\cite{Kifer09}.  Here we present independent proofs of the above facts (or, better, of
the aforementioned substantial refinements of the above facts) and, most importantly, we
make a further step forward by addressing the issue of the local central limit theorem, a
result out of the reach of all previous approaches.

The lesson learned from~\cite{DimaAveraging} is that the \emph{standard pair} technique is
the best suited to investigate these type of partially hyperbolic systems.\footnote{ In
  particular, as far as we know, it represents the most efficient way to ``condition''
  with respect to the past in a field (deterministic systems) where conditioning poses
  obvious conceptual problems.}  Nevertheless, in the uniformly hyperbolic case,
techniques based on the study of the spectrum of the {\em transfer operator} are usually
much more efficient.  It is then tempting to try to mix the two points of view as much as
possible.  This was partially done already in \cite{Bakhtin04} and is also one of the
goals of our work. To simplify matters, we carry it out it in the simplest possible
setting (one dimensional expanding maps).  Nevertheless, we like to remark that extending many of
the present results to hyperbolic maps or flows is just a technical, not a conceptual
problem.  Indeed, till the recent past the use of transfer operators was limited to the
expanding case (or could be applied only after coding the system via Markov partitions,
greatly reducing the effectiveness of the method).  Yet, recently, starting
with~\cite{BKL02} and reaching maturity with~\cite{GL06, BT07, GL08, Liverani04, Tsujii10,
  Faure-Tsujii, Demers-Liverani, Demers-Zhang, Faure, FS2011, GLP}, it has been clarified
how to fully exploit the power of transfer operators in the hyperbolic, partially
hyperbolic and piecewise smooth setting.  Accordingly, it is now totally reasonable to
expect that any proof developed in the expanding case can be extended to the hyperbolic
one, whereby making the following arguments of a much more general interest.\footnote{ The only exception being the ``Dolgopyat estimate" necessary to compute the error term in the local limit theorem which still poses a conceptual challenge in the general hyperbolic case, but see \cite{Tsujii16} for recent progresses.}

The structure of the paper is as follows: we first describe the class of systems we are
interest in, and state precisely the main results. Then we discuss in detail the standard
pair technology. This must be done with care as we will need {\em higher smoothness} as
well as {\em complex} standard pairs, which have not been previously considered. In the
following section we use the tools so far introduced to establish an averaging theorem. As
already explained this result is not new, but it serves the purpose of illustrating the
generals strategy to the reader and the proof contains several facts needed in the following
arguments.  Section~\ref{subsec:markov} is devoted to the precise computation of the logarithmic moment generation function. This allows, in section~\ref{sec:upper}, to establish the large and moderate deviations of our dynamics from the average. We compute with unprecedented
precision the rate function of the large deviation principle. We stop short of providing a
full large and moderate deviations theory only to keep the exposition simple and since it
is not needed for our later purposes. Nevertheless, we improve considerably on known
results. Finally in Section~\ref{sub:lclt} we build on the previous work and prove a local
limit result for our dynamics. The proof is a bit lengthy but it follows the usual
approach: compute the Fourier transform of the distribution. This computation is very similar to the one in section~\ref{subsec:markov} only now we want to compute the expectation of a complex exponential rather then a real one, also we aim at a better precision. Yet, the strategy is essentially the same: we divide the time interval in shorter blocks (this is done in Section~\ref{sec:one-prelim}), then estimate
carefully the contribution of each block (this is done in Sections~\ref{sec:one-large} and
\ref{sec:one-small}) and we conclude by combining together the contributions of the single
blocks (done in section~\ref{proof-not-so-trivial-1}). Some fundamental technical tools
needed to perform such computations are detailed in the appendices. Appendix
\ref{subsec:transfer} contains a manifold of results on transfer operators and their
perturbation theory. In fact, not only it collects, for the reader convenience, many
results scattered in the literature, but also provides some new results. In addition, it
contains a discussion of the genericity of various conditions used in the paper including
the, to us, unexpected results that for smooth maps aperiodicity and not being
cohomologous to a constant are equivalent. Appendix~\ref{sec:dolgo} provides a detailed discussion of transfer
operators associated to semiflows that, although essentially present in the literature,
was not in the form needed for our needs (in particular we need uniform results for a one
parameter family of systems). Finally, Appendix~\ref{app:tedious} contains some simple and uneventful, but a bit lengthy, computations needed in the text.

\subsubsection*{{\bf Notation}} Through the paper we will use $\Const$ and $\const$ to
designate an arbitrary positive constant, depending only on our dynamical system, whose
value can change form an occurrence to the next even in the same line.  We will use
$C_{a,b,\cdots}$ to designate arbitrary  constants that depend on the quantities $a,b,\cdots$
while constants with other decorations (e.g. numbers as subscript) stand for a fixed
specific value.

Also we write $\cO(X)$ to denote a number which is bounded by $\Const X$ for any $\ve < \ve_\#$, where $\ve_\#$ depends only on the dynamics (note that $X$ might not depend on $\ve$, so that the second requirement becomes empty).
While we will use $\cO_\cB(X)$, where $(\cB, \|\cdot\|_{\cB})$ is a Banach space, to denote an element of $\cB$ with norm bounded by $\Const \|X\|_\cB$, again for all $\ve\leq \ve_\#$. We will always assume $\ve$ to be so small that this condition is met for every instance of the expression
  $\cO(\cdot)$.

  Finally, for $a\in \bR$ we will use $\pint{a}$ to designate its integer part, that is
  the largest integer that is smaller or equal to $a$.

\newpage
{\small 
\tableofcontents
}
\newpage
\section{The system and the results}\label{sec:results}
For $\ve>0$ let us consider the map $F_\ve\in\cC^{4}(\bT^2,\bT^2)$ defined by
\begin{equation}\label{eq:map}
  F_\ve(x,\theta)=(f(x,\theta),\theta+\ve\omega(x,\theta)),
\end{equation}
where $\|\omega\|_{\cC^4}=1$.  We assume that the $f_\theta=f(\cdot,\theta)$ are uniformly expanding, i.e.:
\begin{equation}\label{e_strongExpansion}
  \inf_{(x,\theta)\in\bT^2}\partial_x f(x,\theta)\ge \lambda,
\end{equation}
for some $\lambda > 1$; indeed by considering a suitable iterate of $F_\ve$, we will
assume without loss of generality that $\lambda > 2$.

This fact is well known to imply that each $f(\cdot,\theta)$ has a unique invariant probability
measure that is absolutely continuous with respect to the Lebesgue measure.  We denote
this measure (often called the SRB measure) by $\mu_\theta$.  Also, we assume that, for
every $\theta\in\bT$, $\omega$ is not $f_\theta$-cohomologous to a constant function, i.e.
\begin{enumerate}[label=\textup{(A\arabic*)},,ref=(A\arabic*)]
  \item\label{eq:cobo} For any $\theta\in\bT$ there exist no measurable\footnote{ It is
    well known by the Liv\v{s}ic Theorems that if $g_\theta$ is measurable, it is actually
    as smooth as the map $f_\theta$ (see also the proof of Lemma~\ref{lem:coboundary}).} function $g_\theta:\bR\to\bR$ and constant $a_\theta\in\bR$ so that
  $\omega(x,\theta)=g_\theta( f(x,\theta))-g_\theta(x)+a_\theta$.
\end{enumerate}

Note that the latter equation can hold only if for any invariant probability measure $\mu$ of
$f_\theta$, $\mu(\omega(\cdot,\theta))=\mu(a_\theta)=a_\theta$.  In particular, if
$\omega(\cdot,\theta)$ has different averages along two different periodic orbits of
$f_\theta$, then~\ref{eq:cobo} is satisfied.  It is then fairly easy to check such a
condition.  In particular note that the assumption above holds generically (see
Section~\ref{subsec:generic} for a more complete discussion of these issues).

Given $(x,\theta)\in\bT^2$, let us define the trajectory
$(x_n,\theta_n)=F_\ve^n(x,\theta)$ for any $n\in\bN$.

Here we describe a sequence of increasingly sharper results on the behavior of the
dynamics for times of order $\ve\invr$.\footnote{ In some cases it is also possible to
  obtain information for times of the order $\ve^{-2}$ (see~\cite{DimaAveraging}).  Yet,
  as far as we currently see, not of the quantitative type we are interested in.} We start
with well known facts, but we provide complete proofs both for the reader's convenience
and because they are a necessary preliminary to tackle our main results.
\subsection{The  skew product case}\label{sec:trivial}
For the reader convenience, we first give a brief, impressionistic, discussion of  the case in which $\partial_\theta f=\partial_\theta\omega=0$. In this case the map $F_\ve$ is a {\em skew product}: $F_\ve(x,\theta)=(f(x),\theta+\ve\omega(x))$. This case is fairly well understood as it can be reduced to the study of the statistical properties of the map $f$, let us recall why.

We are interested in the evolution of the slow variable $\theta_n$. Clearly we must wait for a time at least $\vei$ in order for a change of order one to be possible. It is then natural to rescale the time and introduce the {\em macroscopic time} $t=\ve n$.  The idea is then to fix some arbitrary $T > 0$ and then define, for all $t\in[0,T]$,
\[
 \theta_\ve(t)=\theta_\pint{t\vei}+(t\vei-\pint{t\vei})[\theta_{\pint{t\vei}+1}-\theta_\pint{t\vei}] \mod 1.
\]
Note that $\theta_\ve\in\cC^0([0,T],\bT)$.  The point here is twofold: on the one hand it
is clear that we cannot expect, at first, to control the behavior of $\theta_n$ for
arbitrarily large $n$. Hence we fix a time horizon $T\vei$, $T$ being arbitrary but
independent on $\ve$. On the other hand, it is natural to introduce a continuous
interpolation of the evolution of $\theta_n$ since
$|\theta_{\pint{t\vei}}-\theta_{\pint{s\vei}}|\leq |t-s|\|\omega\|_\infty$ hence, once
rescaled, the trajectory is Lipschitz on $\ve\bZ$ and it is then naturally interpolated by
a Lipschitz function on $\bR$.  Since
\[
F_\ve^n(x,\theta)=\left(f^n(x), \theta+\ve\sum_{k=0}^{n-1}\omega(f^k(x))\right)
\]
it follows that
\[
\left|\theta_\ve(t)-\frac{t}{\pint{\vei t}}\sum_{k=0}^{\pint{\vei t}-1}\omega(f^k(x)))\right|\leq\Const \ve.
\]
By the Birkhoff Ergodic Theorem, the sum converges almost surely with respect to each
invariant measure. This raises the issue of which measures to consider. In general this is
an issue open to debate, however here we take the point of view that the fast variable $x$
is originally distributed according to a probability measure absolutely continuous with
respect to Lebesgue and with a smooth density. This means that we are interested in the so
called {\em physical measures}. It is then well known that the distribution of $x$ will
tend exponential fast to the unique absolutely continuous invariant measure of $f$, call
it $\mu$, hence, $x$ almost surely,\footnote{ Note that, by the uniform Lipschitz property
  of the $\theta_\ve$ it suffices to control the limit on countably many $t$ to control it
  for all $t$.}
\[
\bar\theta(t)=\lim_{\ve\to 0}\theta_\ve(t)=t\mu(\omega)=:t\bar\omega=:\bar\theta(t).
\]
That is, the limit satisfies the autonomous differential equation $\dot{\bar \theta}=\bar\omega$.

Next, one is interested in the deviations from such a limit. This leads us to the study of
large deviations for an ergodic average. Such a problem has been intensively studies
starting with \cite{young90} and the situation can be summarized as follows: consider the
initial condition $\theta=\theta_0$ and $x$ distribute as above, then
$\gamma_\ve(\cdot)=\theta_\ve(\cdot)-\theta_0$ can be considered a random variable in
$\cC^0_*([0,T],\bT)=\{\gamma\in \cC^0\;:\;\gamma(0)=0\}$. Let $\bP_\ve$ be its law, then
for a sufficiently regular set $Q\in \cC^0_*([0,T],\bT)$
\[
\bP_\ve(Q)\sim e^{-\vei\inf_{\gamma\in Q}\ratef(\gamma)}
\]
where the {\em rate function} $\ratef$ is defined as
\[
\begin{split}
\ratef(\gamma)&=\begin{cases} -\infty \quad&\textrm{ if } \gamma \textrm{ is not Lipschitz}\\
\int_0^T\cZ(\gamma'(t)) dt &\textrm{ otherwise}
\end{cases}
\\
\cZ(b)&=-\sup_{\mnu\in\cM_\theta(b)}\{ \kse(\mnu) - \mnu(\log f_\theta')\},
\end{split}
\]
$\cM(b)=\{\mnu\in\cM\st \mnu(\omega)=b\}$, $\cM$  denotes the set of (ergodic) $f$-invariant probability measures, and $\kse(\mnu)$ is the Kolmogorov-Sinai entropy of the measure $\mnu$.
The above formula is very suggestive: if the statistics of a point $x$ is described by an invariant measure $\nu$, then it will give rise to a trajectory $\theta_\ve$ with velocity $\nu(\omega)$; moreover points that start in a $e^{-\const T\vei}$ neighborhood will have essentially the same trajectory for a time $T\vei$, hence the probability of order  $e^{-\const T\vei}$ for such a trajectory. Unfortunately, the formula for $\cZ$ is not very handy to compute. However the connection between the pressure and the maximal eigenvalue of the Ruelle transfer operator \cite{Baladibook} allows to compute the rate function for smooth $\gamma$ in a small neighborhood of $t\bar\omega$ yielding
\[
\ratef(\gamma)\sim\frac 12\int_0^T \sigma^{-2}\left[\gamma'(s)-\bar
    \omega\right]^2 \deh{}s,
\]
where, setting $\hat\omega=\omega-\bar\omega$, $\sigma^2 = \mu\left(\hat \omega^2\right)+ 2 \sum_{m=1}^{\infty} \mu_\theta\left( \hat \omega\circ f^m \hat \omega\right)$ is the variance of $\hat\omega$.

The above formula suggests that typical deviations are of order $\sqrt \ve$. It is then natural to study the fluctuations $\zeta_\ve(t)=\ve^{-\frac 12}[\theta_\ve(t)-\bar\theta(t)]$. This corresponds to the Central Limit theorem and its refinements (Local CLT, Berry-Essen estimates etc.).
The CLT in this context states that
\[
\bE(\vf(\zeta_\ve(t))\sim \int\vf(x)\frac{e^{-x^2/2\sigma^2}}{\sigma^2\sqrt{2\pi}}.
\]
Of course, for the applications it is essential to know {\em quantitatively} what the $\sim$ in the previous equations really means, that is we need an explicit estimate of the error. This is the task of the present paper as is explained shortly in the general case in which $f,\omega$ depend on the slow variable.

The basic idea used to extend results like the above to the general case is that the fast variable goes to its equilibrium (i.e. the physical measure) at an exponential speed. Hence in a times interval of size $\ve^\alpha$ for some $\alpha\in (1,0)$, the slow variable is almost constant and so is the dynamics. Since the invariant measure changes smoothly with the dynamics ({\em linear response}), then the statistical properties of the fast variable are more or less the same in the considered interval and large deviation results and LCLT hold. One can then use the Markov properties of the dynamics (here expressed in the formalism of {\em standard families}) to extend the result to longer times.

Note however that, while carrying out the above program, one must keep track of the mistakes in the various approximations and this is rather taxing. Especially if one needs to obtain very precise results like the ones achieved here. In fact, to understand if such error terms could be efficiently controlled was one of the motivations of the present paper. Finally, we should remark that most of our results are new even in the trivial skew product case discussed here.

\subsection{The Law of Large Numbers}
If we take the formal average with respect to the SRB measure of~\eqref{eq:map} we obtain
the following first order differential equation
\begin{align}\label{eq:averageeq}
  \frac{\deh\bar\theta}{\deh t}&=\bar\omega(\bar\theta)&
  \bar\theta(0)&=\theta_0,
\end{align}
where $\bar\omega(\theta)=\mu_\theta(\omega(\cdot,\theta))$.  For future use, let us also define the
function $\ho(x,\theta)=\omega(x,\theta)-\bar\omega(\theta)$.
\begin{irem}
  Note that, since $F_\ve\in\cC^4$, we can apply~\cite[Theorem 8.1]{GL06} with the Banach
  spaces $\{\cC^i\}_{i=0}^s$, $s=3$ and obtain that $\bar\omega\in\cC^{3-\alpha}$, for any
  $\alpha>0$.
\end{irem}
Accordingly, the above equation has a unique solution, which we denote
by $\bar\theta(t,\theta_0)$.  This can be generalized:
let $d\in\bN$, $B = (B_1,\cdots,B_{d-1})\in\cC^2(\bT^2,\bR^{d-1})$, and fix $\zeta_0=0$; for any
$k\in\bN$ let us define
\begin{equation}\label{eq:map-extra}
  \zeta_{k+1}=\zeta_k+\ve B(x_k,\theta_k).
\end{equation}
This equation describes the evolution of a passive quantity and it is relevant in many situations (see e.g.~\cite{DeL2, DeLPV}).  Then $\zeta_k$ should be \emph{close} (in a sense that will be detailed shortly) to $\bar\zeta(\ve k, \theta_0)$, the unique solution of
\begin{align}\label{e_averageA}
  \frac {d\bar\zeta(t, \theta_0)}{dt} &=\bar B(\bar\theta(t,\theta_0)) &
  \bar\zeta(0, \theta_0)&=0,
\end{align}
where we introduced the averaged function $\bar B(\theta)=\mu_\theta(B(\cdot,\theta))$.
It is then convenient to introduce the variables $z=(\theta,\zeta)\in \bR^d$ (for
convenience we have lifted $\theta$ to its universal cover) and $A=(A_1,\cdots,A_d)\in \cC^3(\bT^2,\bR^d)$, with
$A_1(x,\theta)=\omega(x,\theta)$ and $A_{i+1}=B_i$ for $i\in\{1,\cdots,d-1\}$.  Then the
evolution of the variables $(x,z)$ is described by the map
\begin{equation}\label{eq:map-full}
  \bF_\ve(x,z)=(f(x,\theta), z+\ve A(x,\theta));
\end{equation}
again we set $(x_k,z_k)=\bF_\ve^k(x,z)$, for $k\in\bN$.  A first relevant fact is that the
above averaging approximation can be justified rigorously.  These type of results are well
known and go back, at least, to Anosov \cite{Anosov}.  Fix $T > 0$ and, for
$t\in[0,T]$, let
\begin{equation}\label{eq:z-path}
  z_\ve(t)=(\theta_\ve(t),\zeta_\ve(t))=z_\pint{t\vei}+(t\vei-\pint{t\vei})[z_{\pint{t\vei}+1}-z_\pint{t\vei}].
\end{equation}
Observe that in the above definition we scale $t$ in such a way that the
  slow variable moves of $O(1)$ for times $t$ of order one.  This in turns corresponds to
  study the $F_\ve^n$ for $n\sim t\vei$.  In fact, given $T > 0$, we will
  study the evolution of $F_\ve$ up to iterates of order $T\vei$. Then
$z_\ve\in C^0([0,T],\bR^d)$, and we can consider it as a random element of
$C^0([0,T],\bR^d)$, the randomness being determined by the distribution of the initial
condition.

\begin{thm}[Averaging]\label{t_averaging}\ %
  Let $\theta_0\in \bT^1$ and $x_0$ be distributed according to a smooth
  distribution $\mu$; then for all $T>0$:
  \begin{align*}
    \lim_{\ve\to 0}z_\ve=\bar z(\cdot, \theta_0)
  \end{align*}
  where $\bar z(\cdot,\theta_0) = (\bar\theta(\cdot,\theta_0),\bar\zeta(\cdot, \theta_0))$
  and the limit is in probability with respect to the measure $\mu$ and the uniform
  topology in $C^0([0,T],\bR^d)$.
\end{thm}
The proof is more or less standard.  We provide it in Section~\ref{sec:averaging} for
reader's convenience.  Indeed, our proof contains, in an elementary form, some of the
ideas that will be instrumental in the following.  The reader not very familiar with the
transfer operator or standard pairs technology is advised to read
Sections~\ref{s_stdpairs} and~\ref{sec:averaging} first.

\subsection{Large and Moderate Deviations}\label{ss_largeModerateDeviationIntro}
We will consider $d$ and $A\in\cC^3(\bT^2,\bR^d)$, $A_1(x,\theta)=\omega(x,\theta)$, to be
fixed throughout the paper and to be data associated to the dynamical systems; although
many quantities will depend on $A$, we do not add subscripts emphasize this dependence.
In particular constants indicated with $\Const$ or $\const$ may indeed depend on $A$.

We find convenient  to define $\gpoly$ to be the random element of $\gspace$ obtained by subtracting
  to $z_\ve$ its (random) initial condition $z_\ve(0)$, i.e.\ we let
  \begin{align}\label{e_gpolyDef}
    \gpoly(t) = z_\ve(t)-z_{\ve}(0);
  \end{align}
  similarly, we define
  \begin{align}\label{e_gavgDef}
    \gavg(t,\theta) = \bar z(t,\theta)-\bar z(0,\theta).
  \end{align}
The next natural question concerns the behavior of deviations from the average.  To this
end it is more convenient to consider the fundamental probability space to be the
classical Wiener space $\gspace$ endowed with the Borel $\sigma$-algebra and the
probability measure given by $\ppath_\mu = (\gpoly)_*\mu$, where $\mu$ is the distribution
of initial conditions on $\bT^2$; in other words $\ppath_\mu$ is the law of $\gpoly$ under $\mu$.

Note that the paths $\gpoly$ are all Lipschitz with Lipschitz constant bounded by
$\|A\|_{\cC^0}$.
To obtain stronger results we need some extra hypotheses:
\begin{enumerate}[label=\textup{(A\arabic*')},,ref=(A\arabic*')]
  \item\label{eq:coboA} for any $\theta\in\bT$ and $\sigma\in\bR^d$, the function
  $\langle\sigma,A(\cdot,\theta)\rangle$ is not $f_\theta$-cohomologous to a constant (in
  particular, this implies~\ref{eq:cobo}).
\end{enumerate}
Note that such condition is implied by the existence of $d+1$ periodic orbits for which
the differences of the averages of $A$ span $\bR^d$.  In particular,
condition~\ref{eq:coboA} is generic.

Given this assumption, we prove upper and lower bounds for the probability of large and
moderate deviations.  The result we are after is much sharper than the one contained
in~\cite{Kifer09}. It is of a more quantitative nature (in the spirit
of~\cite{DimaAveraging} where the rate function is only estimated near zero and in a much
rougher manner). In particular, we provide bounds on the rate function that
allow to treat both large and moderate deviations for all $\ve$ small enough (not just
asymptotically for $\ve\to 0$).  We refrain from developing a more complete
theory\footnote{ For example, we do not strive for optimal results (such as the equivalence of the lower and upper bounds for all possible events in all the regimes under discussion, or the best possible estimate of the error terms).}  because on the one hand
it would not change substantially the result, on the other hand it would increase the
length of an already long paper and, finally, since the results presented here already
more than suffices for our purposes (i.e.\ both for our later use and to pedagogically
illustrates some ideas used in the following).  In fact, the theorem that we state next
does not contain even the full force of what we prove in Section~\ref{sec:upper},
nevertheless its statement requires already quite a bit of preliminary notations.  We advise
the reader that wants a quick, but sub-optimal, idea of the type of results that can be
obtained to jump directly to the Corollaries~\ref{cor:large-dev2} and~\ref{cor:large-dev3}.

The first objects we need, as in any respectable large deviation theory, are \emph{rate
  functions}.  Their precise properties will be specified in detail in
Section~\ref{subsec:rate}; here we summarize some basic facts.  For any $\theta\in\bT$ we
define the set\footnote{ For any $A\in\cC^0(\bT,\bR^d)$ and measure $\mu$ on $\bT$ we
  define $\mu(A)$ to be the vector $(\mu(A_i))\in \bR^d$, where
  $\mu(A_i)=\int_\bT A_i(x) \mu(d x)$.}
\begin{align*}
  \Averages(\theta) = \{\mu(A(\cdot,\theta))\st\mu \text{ is a } f_\theta-\text{invariant
  probability}\}.
\end{align*}
In other words $\Averages(\theta)$ is the set of all possible averages of $A$ with respect
to $f_\theta$-invariant measures.  Observe that $\Averages(\theta)$ can be determined with
arbitrary precision by studying the periodic orbits of the dynamics (see
Lemma~\ref{lem:unconstrained} for details).
The set $\Averages(\theta)$ is a compact convex subset of $\bR^d$; it is also non-empty,
since for any $\theta\in\bT$, $\bar A(\theta)\in\Averages(\theta)$,
where $\bar A(\theta) = \mu_\theta(A(\cdot,\theta))$
(observe that $\bar A(\theta)$ is deterministic, \ie it is a non-random vector).
Additionally, condition~\ref{eq:coboA} implies (see Lemma~\ref{lem:domainZ} for details) that $\bar A(\theta)\in\intr\Averages(\theta)$ for any
$\theta\in\bT$.  Let us now define
the $d\times d$ matrix
\[
 \begin{split}
      \Sigma^2(\theta) = &\mu_\theta\left(\hat A(\cdot,\theta)\otimes \hat A(\cdot,\theta)\right)
      + \sum_{m=1}^{\infty} \mu_\theta\left( \hat A(f_\theta^m(\cdot),\theta)\otimes \hat
      A(\cdot,\theta)\right)\\
     & +\sum_{m=1}^{\infty}\mu_\theta\left(  \hat
      A(\cdot,\theta)\otimes \hat A(f_\theta^m(\cdot),\theta)\right),
    \end{split}
    \]
    where $\hat A=A-\bar A$.  Then $\Sigma\in\cC^1(\bT,M_d)$,\footnote{ It follows from
      the fact that $\Sigma$ can be seen as the second derivative of the eigenvalue of an
      appropriate transfer operator~\eqref{e_secondDerivativeChi}, which, in turn is
      differentiable by Lemma~\ref{eq:theta-derivative-all}.}  where $M_d$ is the space of
    $d\times d$ symmetric non negative matrices.  If~\ref{eq:coboA} holds, then $\Sigma$
    is invertible (see Lemma~\ref{lem:coboundary}).

    In the following statement (and in the rest of
    the paper) we adopt the convention that $\inf\emptyset = +\infty$ (\resp $\sup\emptyset = -\infty$).
\begin{prop}[\emph{Asymptotic} Large Deviation Principle]\label{p_rateFunctionBasicProperties}\ %
  Let $\theta\in\bT$, $\mu$ be a measure with smooth density on $\bT$ and
  $\ppath_\mu = (\gpoly)_*(\mu\times\delta_\theta)$.  There exists a lower semicontinuous
  function $\ratef_{\theta}: \gspace\to\bRp\cup\{+\infty\}$ (see~\eqref{eq:rate2} for an
  explicit definition) so that $\ppath_\mu$ satisfies the Large Deviation Principle with
  rate function $\mathscr{I}_{\theta}$, that is: given any event $Q\subset\gspace$ we have
  \begin{equation}\label{e_largeDeviationPrincipleKifer}
  \begin{split}
    -\inf_{\gamma\in \intr Q}\ratef_{\theta}(\gamma)&\le%
    \liminf_{\ve\to0}\ve\log \ppath_\mu(Q)\\
    &\le  \limsup_{\ve\to 0}\ve\log \ppath_\mu(Q)\le%
    -\inf_{\gamma\in \overline Q}\ratef_{\theta}(\gamma).
  \end{split}
  \end{equation}
  Note that $\ratef_{\theta}$ is not necessarily convex, yet it satisfies the following properties:
  \begin{enumerate}
  \item the \emph{effective domain}
    $\fkD(\ratef_\theta):=\{\gamma\in \gspace\st\ratef_\theta(\gamma)<\infty\}$ consists of
    Lipschitz paths such that $\gamma(0)=0$ and, for almost all $t\in[0,T]$, the
    vector\footnote{ Recall that by Rademacher's Theorem, any Lipschitz function is a.e.\
      differentiable.} $\gamma'(t)\in\Averages(\theta^\gamma(t))$, where
    $\theta^\gamma(t,\theta) = \theta+(\gamma(t))_1$.\footnote{ Here $(\gamma(s))_1$ is
      the first component of the vector $\gamma(s)$: the one that corresponds to the
      $\theta$ motion. Also remark that, to ease notation, we will often suppress the $\theta$ dependency if no confusion arises.}  In particular, this implies $\|\gamma'\|_{L^\infty}\le\|A\|_{L^\infty}$
    for any $\gamma\in\fkD(\ratef_\theta)$.
   \item for any $\gamma\in\fkD(\ratef_\theta)$, the rate function $\ratef_{\theta}$ satisfies the
    following expansion:
    \begin{equation}\label{eq:mathscrI}
      \begin{split}
       &\left| \mathscr{I}_{\theta}(\gamma)-\frac 12\int_0^T\hskip-6pt\langle \gamma'(s)-\bar A(\theta^\gamma(s)), \left[\Sigma^2(\theta^\gamma(s))\right]^{-1}\left[\gamma'(s)-\bar A(\theta^\gamma(s))\right]\rangle \deh{}s\right|\\
        &\quad\leq \Const \|\gamma'-\bar A\circ\theta^\gamma)\|_{L^3}^3.
      \end{split}
    \end{equation}
  \end{enumerate}
\end{prop}
The above is the usual {\em asymptotic} large deviation principle, similar to what can be found
in~\cite{Kifer09}, although in a different setting.  We are, however, interested in
stating estimates valid for all, sufficiently small, $\ve$ not just in the limit
$\ve\to 0$.

In order to properly state results in the needed generality, we define, for each
$\theta_0\in\bT$, a set $\measures_\ve(\theta_0)$ of \emph{good} probability measures that
are supported in a $\ve$-neighborhood of $\theta_0$.  We refer to
Section~\ref{s_stdpairs}, in particular~\eqref{eq:good-measure}, for the precise
definition, but, as an example, $\mu\times \delta_{\theta_0}\in\measures_\ve(\theta_0)$
where $\mu$ is a measure on $\bT$ with a smooth distribution $\rho$, and the derivative of
$\ln\rho$ is bounded by some fixed constant.  Here is a useful, but minimal, example of
the kind of results we are after.
\begin{prop} \label{p_largeDevzQuadraticBound}\ %
  There exists $\Tmax, \ve_0\in (0,1)$,  $\bar C, c_\star>0$ such that, for all $\ve\leq \ve_0$, $T \in [\ve_0^{-4}\ve, \Tmax]$,
  $R\ge\bar C\sqrt{\ve T}$ and $\theta^*_0\in\bT$, if we set
  \begin{align*}
  Q_R=\{\gamma\in\gspace\st \|\gamma(\cdot)-\gavg(\cdot,\theta^*_0)\|\nc0\ge R\},
  \end{align*}
  then, for any $\mu\in\measures_\ve(\theta^*_0)$ and recalling $\ppath_\mu = ({\gpoly})_{*}\mu$, we have
  \begin{align*}
    \ppath_\mu(Q_R)&\le\expo{-c_\star \vei T^{-1} R^2}.
  \end{align*}
\end{prop}
Proposition~\ref{p_largeDevzQuadraticBound} is similar to~\cite[Theorem
6(b)]{DimaAveraging}, although in a different setting: our goal is to obtain stronger
results encompassing the above ones.  In particular, the previous results will be mere byproducts (see Section~\ref{s_proofCorollaries}).

In order to properly present such result we introduce a slightly modified rate functions and we will state the result by saying that the probability of an event is controlled from above by the inf of
the rate function on a slightly larger set and from below by the inf on a slightly smaller
set. Also, if an event describes a small deviation from the average, then we can obtain effective bounds only if it is not too wild on a small given scale. Unfortunately, it is a bit tricky to make quantitatively precise these notions, so we ask for the
reader patience.

First, for any $\lseparation>0$, we introduce functionals $\mathscr{I}_{\theta_0,\lseparation}^\pm$  so that $\mathscr{I}_{\theta_0,\lseparation}^-\le\mathscr{I}_{\theta_0}\le\mathscr{I}_{\theta_0,\lseparation}^+$ but agree with $\mathscr{I}_{\theta_0}$
outside a $\lseparation$ neighborhood of $\partial\fkD(\mathscr{I}_{\theta_0})$.\footnote{ Essentially $\mathscr{I}_{\theta_0,\lseparation}^+=\infty$ in a
  $\lseparation$-neighborhood of $\partial\fkD(\mathscr{I}_{\theta_0})$ while
  $\mathscr{I}_{\theta_0,\lseparation}^-<\infty$ in the same neighborhood,
  see~\eqref{eq:zeta-reg-def},~\eqref{eq:rate2}, Section~\ref{subsec:rate} and Lemma~\ref{lem:entropy} for precise definitions.}
Remark that we consider $\fkD(\ratef_{\theta_0})$ as a subset of the Lipschitz functions with the associated topology, see Remark~\ref{rem:effective} for more details.   In Lemma~\ref{lem:rate-lower} we prove
\begin{align*}
\lim_{\lseparation\to
  0}\mathscr{I}_{\theta_0,\lseparation}^-=\mathscr{I}_{\theta_0}\le\mathscr{I}_{\theta_0}^+=\lim_{\lseparation\to
  0}\mathscr{I}_{\theta_0,\lseparation}^+,
\end{align*}
where $\mathscr{I}_{\theta_0}^+$ agrees with $\mathscr{I}_{\theta_0}$ everywhere apart
from $\partial\fkD(\mathscr{I}_{\theta_0})$ where it has value $+\infty$.  Second, let
$\theta_0\in\bT$, $\hgamma(t)=\gamma(t)-\gavg(t,\theta_0)$ and define
$\Rad\pm:\gspace\to\bRp$ by
\begin{align}\label{e_defRadPm}
  \Rad-(\gamma)&= \ve^{\frac 1{2}}&
  \Rad+(\gamma)& =C_{\lseparation,T}\left\{\ve\eefrac17\|\hgamma\|\nl\infty\eefrac57+\sqrt\ve \right\},
\end{align}
for some appropriate constant $C_{\lseparation,T}$.
Then, for each $Q\subset \cC^0([0,T],\bR^d)$ let
\begin{equation}\label{eq:Qdef}
  Q^-=\{\gamma\in Q\st \gball{}(\gamma, \Rad-(\gamma))\subset Q\}\;;
  \hskip1cm
  Q^+=\bigcup_{\gamma\in Q}\gball{}(\gamma, \Rad+(\gamma))
\end{equation}
where $\gball{}(\gamma,r)$ is the
standard $C^0$-ball in $\gspace$. Obviously
$Q^-\subset \intr Q \subset\overline Q\subset Q^+$. Finally, we want to make precise what
do we mean by event that are not too wild on a given scale. Let
\begin{equation}\label{eq:last-resort}
\begin{split}
&\varrho(\theta_0,Q)=\inf_{\gamma\in Q}\|\hgamma\|_\infty,\\
&\Clip(\gamma)=T^{-\efrac {11}7}\ve^{-\efrac 27}\|\hgamma\|_{L^\infty}^{\efrac{11}7}\\
&\vestep(\gamma)=\sqrt\ve\left(\frac {T^2\ve}{\|\hgamma\|_{L^\infty}^2}\right)^{\efrac 1{14}}.
\end{split}
\end{equation}
Given a measure $\bP$ on $\cC^0$, we call an event a $Q\subset \cC^0$ {\em $\bP$-regular} if for $\bP$-almost all $\gamma\in Q$ we have

\begin{equation}\label{eq:lip-bound}
|s-s'|\leq \frac{\vestep(\gamma)}{2\Clip(\gamma)}\quad \Longrightarrow\quad \|\gamma(s)-\gamma(s')\|\leq \frac{\vestep(\gamma)}4.
\end{equation}
In other words, for each $\beta\in(0,1]$, points at a distance $\frac{\beta\vestep(\gamma)}{2\Clip(\gamma)}$ yield a Lipschitz constant bounded by $\Clip(\gamma)/(2\beta)$.

We are now ready to state our first main result. Essentially, it is a quantitative version Proposition~\ref{p_rateFunctionBasicProperties} which allows to, rather precisely, estimate the probability of events when $\ve$ is small, but non zero. In particular, it provides bounds for the speed at which the limits in Proposition~\ref{p_rateFunctionBasicProperties} take place.
\begin{thm}[Large and Moderate deviations]\label{thm:large}\ %
Let $T>0$; for all $\lseparation>0$, $\ve$ small
  enough (depending on $T$ and $\lseparation$), $\theta_0\in\bT$, $\mu\in\measures_\ve(\theta_0)$, and for any  $\ppath_\mu$-regular event $Q_\ve$ (possibly depending
  on $\ve$), we have
  \begin{equation}\label{eq:chepalle}
  \begin{split}
    &\ppath_\mu (Q_\ve)\le e^{-\ve^{-1}\left[\left(1-C_{\lseparation,T}\ve\eefrac1{7}\varrho(\theta_0,Q_\ve)\eefrac{-2}{7}\right)\inf_{\gamma\in
          Q_{\ve}^+}\mathscr{I}_{\theta_0,\lseparation}^-(\gamma)\right]}
    \\
    &\ppath_\mu (Q_\ve)\ge e^{-\ve^{-1}\left[(1+C_{\lseparation,T}\ve^{\efrac 12})\inf_{\gamma\in
          Q_{\ve}^-}\mathscr{I}_{\theta_0,\lseparation}^+(\gamma)+C_{\lseparation,T} \ve^{\efrac
          1{8}}\right]}.
  \end{split}
\end{equation}
\end{thm}
The proof can be found in Section~\ref{subsec:large-long-Q}.
\begin{rem}\label{rem:shortT}
Note that $\ppath_\mu$ almost surely the paths have Lipschitz constant bounded by $\|A\|_{L^\infty}$. Hence, if $\varrho(\theta_0,Q)\geq \Const\ve^{\efrac 2{11}}$ (that is, the deviation is large enough)  then $Q_\ve$ is always $\ppath_\mu$ regular.\newline
Also, if $\inf_{\gamma\in Q_{\ve}^+}\mathscr{I}_{\theta_0,\lseparation}^-(\gamma)\leq C_{\lseparation, T}\sqrt\ve$, then it must be $\varrho(\theta_0,Q_\ve)\leq C_{\lseparation, T}\sqrt\ve$ (see Lemma \ref{l_Jgamma-bound}), and the coefficient in front of the rate function in the first of the \eqref{eq:chepalle} becomes positive, therefore making the estimate empty.\newline
Finally, note that, using the results of Section~\ref{sec:upper} (in particular
  Lemmata~\ref{lem:upper} and~\ref{lem:lower}) one could state the theorem in the
  case of a small $T$ depending on $\ve$.
  This is in fact not necessary: indeed any event in $\cC^0([0,\Const \ve^\alpha],\bR^d)$,
  $\alpha\in [0,1)$, can be seen as an event in $\cC^0([0,T],\bR^d)$.  One can then check,
  using~\eqref{eq:mathscrI}, that times larger than $\Const \ve^\alpha$ do no contribute
  to the $\operatorname{inf}$, since any such event contains trajectories
  for which $\gamma'=\bar A$ for all $t\ge \Const \ve^\alpha$.
\end{rem}
The statement of Theorem~\ref{thm:large}, due to its precise quantitative nature, may feel
a bit cumbersome.  To help the reader understand its force we spell out few easy
consequences in a form of corollaries.  Their proof can be found in
Section~\ref{s_proofCorollaries}.

We already mentioned that Theorem~\ref{thm:large} implies
Proposition~\ref{p_rateFunctionBasicProperties}; yet the {\em finite size} version
provided by Theorem~\ref{thm:large} implies much more.  Also note that, although the
statement of Proposition~\ref{p_rateFunctionBasicProperties} looks very clean, it is not
very easy to use since the $\inf$ involved is often very hard to compute, even for a
simple event like
$Q= \{\gamma\in\gspace \st \|\gamma(s)-\gavg(s,\theta_0)\|\ge C s, s\in [0,T]\}$.  For
deviations that are not too large, one can get some more explicit estimates using the
expansion of $\mathscr{I}_{\theta_0}$ stated in~\eqref{eq:mathscrI}.  The following corollary provides precise asymptotic estimates for paths that deviate from the average by at most $\Const\ve^\beta$, where $\beta\in(0,1/2)$.
\begin{cor}[Moderate deviations]\label{cor:large-dev2}
Let $T>0$ and $\ve_0$ small enough. For each $\ve\leq \ve_0$, $\theta_0\in\bT$, $\beta\in (0,\frac 12)$ and Lipschitz bounded set $Q\subset\cC^0([0,T],\bR^d)$, i.e. the Lipschitz constant is uniformly bounded, define\footnote{ Hence there exists $C>0$ such that, if $\gamma\in Q_\ve$, then $\|\gamma-\gavg\|_{L^\infty}\leq C\ve^{\beta}$.}
  \begin{align*}
  Q_\ve=\{\ve^{\beta}\gamma(\cdot)+(1-\ve^{\beta})\gavg(\cdot,\theta_0)\}_{\gamma\in Q}.
  \end{align*}
  Then,  for all $\mu\in\measures_\ve(\theta_0)$, $\ppath_\mu=(\gpoly)_*\mu$, we have
  \begin{align*}
  \limsup_{\ve\to 0}\ve^{1-2\beta}\log \ppath_\mu (Q_\ve)\le -\inf_{\gamma\in\overline
    Q}\ratefLin_{\theta_0}(\gamma),
  \end{align*}
  where
  \begin{align*}
  \ratefLin_{\theta_0}(\gamma)=\frac 12\int_0^T\langle \gamma'(s)-\bar
  A(\bar\theta(s, \theta_0)), \Sigma^2(\bar\theta(s, \theta_0))^{-1}\left[\gamma'(s)-\bar
    A(\bar\theta(s, \theta_0))\right]\rangle \deh{}s.
  \end{align*}
  If, additionally, $\beta<\frac 1{16}$ then\footnote{ Our techniques should allow to establish a similar lower bound also for $\beta\in [1/2, \frac 1{16}]$, but at the price of further work. As is, if $Q$ is open, we have only $\log \ppath_\mu (Q_\ve)\ge- \Const \ve^{-\efrac{7}{8}}$, for $\beta\leq \frac 12$.}
  \begin{align*}
    \liminf_{\ve\to 0}\ve^{1-2\beta}\log \ppath_\mu (Q_\ve)\ge -\inf_{\gamma\in\intr
    Q}\ratefLin_{\theta_0}(\gamma).
  \end{align*}

\end{cor}

In fact, Theorem~\ref{thm:large} allows to estimate the probability of even smaller
deviations, up to the scale of the Central Limit Theorem, whereby providing a strong refinement of Proposition~\ref{p_largeDevzQuadraticBound}.
\begin{cor}[Small deviations]\label{cor:large-dev3}
For each  $T>0$ and $\lprecisione\in(0,1)$  there exists $\ve_0, C_*>0$ such that, for each $\ve\in (0,\ve_0)$, $\mu\in\measures_\ve(\theta_0)$, $\ppath_\mu=(\gpoly)_*\mu$, Lipschitz bounded event
  $Q\subset\gspace$ such that $\varrho(\theta_0,
  Q)\ge C_*$ and setting $Q_\ve=\{\ve^{\frac 12}\gamma(\cdot)+(1-\ve^{\frac
    12})\gavg(\cdot,\theta_0)\}_{\gamma\in Q}$, we have
  \[
 \ppath_\mu(Q_\ve)\le e^{ -\lprecisione\inf_{\gamma\in\hat
      Q^{+}}\mathscr{I}_{\operatorname{Lin},\theta_0}(\gamma) },
  \]
  where $\hat Q^+=\bigcup_{\gamma\in Q}B(\gamma,  \lprecisione\|\gamma-\gavg(\cdot,\theta_0)\|_\infty)$.
\end{cor}
\subsection{Local Central Limit Theorem}
Given that in many cases we have seen that the upper bound in Proposition~\ref{p_largeDevzQuadraticBound} is
sharp, one expects that typical deviations are of order $\sqrt{\ve T}$.  It is then
natural to wonder about their distribution.  It is possible to prove (see
e.g.~\cite[Theorem 5]{DimaAveraging}, where a related class of systems is
investigated, or \cite{DeL3} for a pedagogical exposition of the present case) that the
deviation of $z_\ve$ from the average, when rescaled by $\ve^{-\frac 12}$, converges
towards a diffusion process.  To simplify matters we will discuss only the case $d=1$, but
similar results hold for any $d$.

Let us describe the above statement more precisely. Once again fix $\thetas_0$, let $x$ be
random and define
$\deviatio^{\!\ve}(t) = \ve^{-\efrac12}\left[\theta_\ve(t)-\bar\theta(t,\thetas_0)\right]$.
Then, as $\ve\to 0$, the deviation $\deviatio^{\!\ve}(t)$ converges weakly to
$\deviatio(t)$, the solution of:
\begin{equation}\label{e_deviationDiffusion}
\begin{split}
&   \deh\deviatio(t) = \bar\omega'(\bar\theta(t,\thetas_0))\deviation(t)\deh t + \greenkubo(\bar\theta(t,\thetas_0))\deh B(t)\\
&    \deviatio(0) =0,
\end{split}
\end{equation}
where $B(t)$ is a standard Brownian motion and\footnote{ Observe that this is nothing else
  that the matrix element $\Sigma^2_{1,1}$, which appeared in the moderate deviations.}
\begin{equation}\label{e_definitionBarChi}
 \bVar(\theta) = \mu_\theta\left(\ho^2(\cdot,\theta) + 2 \sum_{m=1}^{\infty}
    \ho(f_\theta^m(\cdot),\theta)\ho(\cdot,\theta)\right).
\end{equation}
Our next result provides a dramatic sharpening of the above statement.
  \begin{thm}\label{thm:lclt}\ %
    For any $T>0$, there exists $\ve_0>0$ so that the following holds.  For any $\beta>0$,
    compact interval $I\subset\bR$, real numbers $\shiftPar>0$, $\ve\in(0,\ve_0)$ and
    $t\in[\ve^{1/2000},T]$, any fixed $\theta^*_0\in\bT^1$ and $\mu\in\measures_\ve(\theta^*_0)$, we have:
  \begin{equation} \label{e_indicatorlclt}
  \begin{split}
\left|  \frac{\ppath_{\mu}(\deviatio^{\!\ve}(t)\in \veh I +
      \shiftPar) }{\sqrt\ve}-     \frac{e^{-\shiftPar^2/2\Var_t^2(\theta^*_0)}}{\Var_t(\theta^*_0)\sqrt{2\pi}}\Leb\, I \right|
      \leq& \,C_{T,\beta} \ve^{\efrac 12-7\beta}\Leb\, I \\
      &+C_{T,\beta} \ve^{\efrac 12-\beta},
\end{split}
  \end{equation}
  where $\ppath_\mu = (\gpoly)_*\mu$ and the variance $\Var_t^2(\theta)$ is given by
  \begin{equation}\label{e_variancelclt} \Var_t^2 (\theta)= \int_0^t
e^{2\int_s^t\bar\omega'(\bar\theta(r,\theta))\deh r}\bVar(\bar\theta(s,\theta))\deh s.
  \end{equation}
\end{thm}
Note that the Gaussian in equation~\eqref{e_indicatorlclt} is indeed the solution
of~\eqref{e_deviationDiffusion}.\footnote{ If in doubt, see \cite{DeL3} for details.} The
proof of Theorem~\ref{thm:lclt} is given in Section~\ref{subsec:prooflclt}.
\begin{rem} If an Edgeworth expansion for~\eqref{e_indicatorlclt} would hold, then one
  would expect the next term to be $\cO(\veh \Leb I)$, see \cite{Feller2}. Thus our error
  term is just slightly bigger than the expected first term in the Edgeworth expansion. In
  fact, with the technology put forward in this paper it should be possible to obtain such
  a correction to the CLT at the price of explicitly computing the main contribution of
  some terms that we have just estimated and considered errors.  Unfortunately this,
  although feasible, is computationally heavy and we decided to avoid it to keep the
  length and readability of the paper (somewhat) under control.
\end{rem}
\section{Standard pairs and families}\label{s_stdpairs}
In this section we introduce standard pairs and families for our system.  As mentioned in
the introductory section, this tool proved quite powerful in obtaining quantitative
statistical results in systems with some degree of hyperbolicity.  The first step is thus
to establish some hyperbolicity result.

\subsection{Dominated splitting}\label{subsec:splitting}\ \newline
Let us start with a preliminary inspection of the geometry of our system: for
$\Kconeu,\Kconec > 0$ to be specified later, let us define the \emph{unstable cone} and
the \emph{center cone} as, respectively:
\begin{align}\label{e_definitionCones}
  \coneu&=\{(\xi,\eta)\in\bR^2\;:\; |\eta|\leq\ve \Kconeu|\xi|\}&
  \conec&=\{(\xi,\eta)\in\bR^2\;:\; |\xi|\leq \Kconec|\eta|\}.
\end{align}
We claim that there exist $\Kconeu,\Kconec$ such that, if $\ve$ is small enough, $\deh
F_{\ve}\coneu\subset\coneu$ and $\deh F_{\ve}\invr\conec\subset\conec$.  In fact, let us
compute the differential of $F_\ve$:
\begin{equation}\label{e_formulaDF}
  \deh F_\ve=\begin{pmatrix} \partial_{x}f &\partial_\theta f\\ \ve\partial_x\omega &1+\ve\partial_\theta\omega\end{pmatrix};
\end{equation}
consequently, if we consider the vector $(1,\ve u)$
\begin{align}
  \deh_pF_\ve(1,\ve u) &= (\partial_{x}f(p) +\ve u\partial_\theta f(p),\ve\partial_x\omega(p) +\ve u+\ve^2 u\partial_\theta\omega(p))\notag\\
  &=\partial_{x}f(p)\left(1 +\ve \frac{\partial_\theta
      f(p)}{\partial_{x}f(p)}u\right)\cdot (1,\ve \Xi_p(u))\label{e_formulaExpansionX}
\end{align}
where
\begin{equation}\label{e_evolutionXi}
  \Xi_p(u)=\frac{\partial_x\omega(p) + (1+\ve\partial_\theta\omega(p))u}{\partial_{x}f(p) +\ve \partial_\theta f(p)u},
\end{equation}
from which we obtain our claim, choosing for instance
\begin{align}\label{e_definitionKconeuc}
  \Kconeu&=2\|\partial_x\omega\|_\infty& \textrm{ and }& &
  \Kconec&=2\|\partial_\theta f\|_\infty.
\end{align}
In fact, for any $\lambda'$ so that
  $\lambda > \lambda' > 3/2$, we can choose $\ve$ so small
  that if $|u| <\Kconeu$:
  \begin{align*}
    |\Xi_p(u)| < \frac{\|\partial_x\omega\|_{\infty} + |u|}{\lambda'} < \Kconeu,
  \end{align*}
  which proves invariance of $\coneu$ under $dF_\ve$.  Invariance of $\conec$ can be
  similarly established.  Hence, for any $p\in\bT^2$ and $n\in\bN$, we can define the
quantities $\enu_n,u_n,\stable_n,\emu_n$ as follows:
\begin{align}\label{e_defineSlopes}
  \deh_p F_{\ve}^n (1,0)&=\enu_n ( 1, \ve u_n )&\deh_p F_{\ve}^n (\stable_n ,1) &= \emu_n (0,1)
\end{align}
with $|u_n|\le c$ and $|\stable_n|\le K^{-1}$.  Notice that $\deh_p F_\ve(\stable_n(p),1) =
\emu_n/\emu_{n-1}(\stable_{n-1}(F_ \ve p),1)$; therefore, there exists a constant $\expb$
such that:
\begin{equation}\label{e_trivialBound}
  \expo{-\expb\ve}\le \frac{\emu_n}{\emu_{n-1}} \le \expo{\expb\ve}.
\end{equation}
Furthermore, define $\Gamma_n=\prod_{k=0}^{n-1}\partial_x f\circ F_\ve^k$, and let
\[
\expa=c\left\|\frac{\partial_\theta f}{\partial_x f}\right\|_\infty.
\]
Clearly
\begin{equation}\label{e_***}
  \Gamma_n\expo{-\expa\ve n} \le\enu_n\le \Gamma_n\expo{\expa\ve n}.
\end{equation}

\subsection{Standard pairs: definition and properties}\ \newline\label{sec:stp}
We now proceed to define standard pairs for our system: we begin by introducing real
standard pairs (which are just special probabilities measures), and then proceed to extend our definitions to complex standard pairs.

\subsubsection{Real standard pairs}
Let us fix a small $\delta>0$, and $\dtt0, \dtt1>0$ large to be specified later; for
any $\spc1>0$ let us define the set of functions
\begin{align*}
  \Sigma_{\spc1}=\{G\in \cC^3([a,b], \bT^1)\st&
  a,b\in\bT^1, b-a\in[\delta/2,\delta],\\&
  \|G'\|\le \ve \spc1,\, \|G''\|\le \ve \dtt0 \spc1,\,\|G'''\|\le \ve \dtt1 \spc1\}.
\end{align*}
Let us associate to each $G\in \Sigma_{\spc1}$ the map $\bG(x)=(x,G(x))$ whose image is a
curve --the graph of $G$-- which will be denoted by $\gamma_{\bG}$; such curves are called
\emph{standard curves}.  For any $\spc2,\spc3>0$ define the set of $(\spc2,\spc3)$-\emph{standard}
probability densities on the standard curve $\gamma_\bG$ as
\begin{align*}
  D^\bR_{\spc2,\spc3}(G)=\left\{\rho\in \cC^2([a,b],\bR_{>0})\st \int_a^b\rho(x)\deh x=1,\
  \left\|\frac{\rho'}{\rho}\right\|\le \spc2,\,\left\|\frac{\rho''}{\rho}\right\|\le
  \spc3\right\}.
\end{align*}
A \emph{real $(\spc1,\spc2,\spc3)$-standard pair} $\ell$ is given by $\ell=(\bG,\rho)$ where
$G\in\Sigma_{\spc1}$ and $\rho\in D^\bR_{\spc2}(G)$.  
A real standard pair
$\ell=(\bG,\rho)$ induces a probability measure $\mu_\ell$ on $\bT^2$ defined as follows:
for any Borel-measurable function $g$ on $\bT^2$ let
\[
\mu_\ell(g):=\int_a^{b} g(x,G(x))\rho(x) \deh x.
\]
We define\footnote{ We remark that this is not the most general definition of standard
  family, yet it suffices for our purposes and it allows to greatly simplify our
  notations.} a \emph{standard family} $\stdf=\SF{\ell_j}{\fm}$ as a (finite or) countable
collection of standard pairs $\{\ell_j\}$ endowed with a finite factor measure $\fm$,
i.e.\ we associate to each standard pair $\ell_j$ a positive weight $\fm_{\ell_j}$ so that
$\sum_{\ell\in\stdf}\fm_\ell<\infty$.  A standard family $\stdf$ naturally induces a
finite measure $\mu_\stdf$ on $\bT^2$ defined as follows: for any Borel-measurable
function $g$ on $\bT^2$ we let
\[
\mu_\stdf(g):=\sum_{\ell\in\stdf}\fm_{\ell}\mu_\ell(g).
\]
A standard family is a standard probability family if the induced measure is a probability
measure (i.e.\ if $\fm$ is itself a probability measure).  Let us denote by $\sim$ the
equivalence relation induced by the above correspondence i.e.\ we let $\stdf\sim\stdf'$ if
and only if $\mu_\stdf=\mu_{\stdf'}$.The key property of the above objects is that the pushforward of a standard family is a standard family \cite[Proposition 5.2]{DeL2}.

Unfortunately, to study large deviations we will need to consider a more general pushforward in which the density is first multiplies by some real positive function (called {\em weight}, which logarithm is called {\em potential}) and then pushforwarded (see equation~\eqref{eq:push-pot}). This is analogous to the use of twisted transfer operators so useful  in the analytic approach to the statistical properties of dynamical systems \cite{Baladibook}. Yet, for the study of the CLT not even this suffices: we need to multiply the density by a complex phase. It is then necessary to generalize the above concepts to the complex setting. As the proofs for complex and real weights are essentially the same, we proceed directly in introducing complex potentials and prove the needed generalization of \cite[Proposition 5.2]{DeL2} : Proposition~\ref{p_invarianceStandardPairs}.

\subsubsection{ Complex standard pairs} We now proceed to  introduce \emph{complex
  standard pairs}. Let us first define the set of complex standard densities:
\begin{align}\label{eq:complex-st}
  D^\bC_{\spc2,\spc3}(G)=\left\{\rho\in \cC^2([a,b],\bC^*)\st \int_a^b\rho(x)\deh x=1,\
  \left\|\frac{\rho'}{\rho}\right\|\le \spc2,\,\left\|\frac{\rho''}{\rho}\right\|\le
  \spc3\right\},
\end{align}
where we denote $\bC^*=\bC\setminus\{0\}$.  Yet, this time, for technical reasons,
we cannot chose the length fixed once an for all. So we will consider standard curves $\Sigma_{\spc1}^\bC$ made of curves of length $b-a\in [\deltacomplex/2, \deltacomplex]$ for some $\deltacomplex\in (0,\delta)$.
We then require $\spc2\deltacomplex\le\pi/10$.  A \emph{complex standard pair} is then
given by $\ell=(\bG,\rho)$ where $G\in\Sigma_{\spc1}^\bC$ and $\rho\in D_{\spc2,\spc3}^\bC(G)$; a
complex standard pair induces a natural complex measure on $\bT^2$.  A complex standard
family $\stdf$ is defined as its real counterpart, but now we allow $\ell_j$'s to be
complex standard pairs and $\fm$ to be a complex measure so that
$\sum_{\ell\in\stdf}|\fm_\ell|<\infty$.  Clearly, a complex standard family naturally
induces a complex measure on $\bT^2$.

We will say that $b-a$ is the {\em length} of the standard pair and we will say that a family $\stdf$ has length $\deltacomplex$ if each $\ell\in \stdf$ has length $b_\ell-a_\ell\in [\deltacomplex/2, \deltacomplex]$.

\begin{lem}[Variation]\label{l_variation}
  Let $G\in\Sigma_{\spc1}^\bC$ be a standard curve and $\rho\in D_{\spc2,\spc3}^\bC(G)$; if $\delta$ is
  sufficiently small we have:
  \begin{align*}
    \textup{Range}\,\rho&\subset %
    \{z=re^{i\thetaArgument}\in\bC\st e^{-2\spc2\deltacomplex}< r (b-a) <
    e^{2\spc2\deltacomplex},\,|\thetaArgument|<\spc2\deltacomplex\}.
  \end{align*}
\end{lem}
\begin{proof}
  Observe that, by definition of standard density we have $\|(\log \rho)'\|\le \spc2$; since
  we are assuming $\spc2\deltacomplex\le\pi/10$, we can unambiguously define the function
  $\log\rho$, which is contained in a square $S\subset\bC$ of side $\spc2\deltacomplex$.  Thus,
  $\textup{Range}\,\rho\subset\exp S$, which is an annular sector.  The normalization
  condition $\int\rho=1$ and the mean value theorem imply that $\exp S$ must non-trivially
  intersect the sets $\{\Re\, z=(b-a)\invr\}$ and $\{\Im\, z=0\}$; these two conditions
  immediately imply that $\exp S\subset\{re^{i\thetaArgument}\in\bC\st |\thetaArgument|<\spc2\deltacomplex\}$.  It
  is then immediate to show that
  \[
  \exp S\subset\left\{re^{i\thetaArgument}\st e^{-\spc2\deltacomplex}<r
  (b-a)<\frac1{\cos(\spc2\deltacomplex)} e^{\spc2\deltacomplex}\right\},\] which concludes the
  proof.
\end{proof}
\begin{rem}
  The above lemma also implies a uniform $\cC^2$ bound on standard densities given by
  $\|\rho\|\nc2\le e^{2\spc2\deltacomplex}\spc3\deltacomplex\invr$.  Moreover, we have
  \begin{align}\label{e_trivialTotalVariationBound}
    |\mu_\ell|<e^{2\spc2\deltacomplex},
  \end{align}
  where $|\mu_{\ell}|$ is the standard total variation norm.
\end{rem}

The key property of the class of real standard pairs is its invariance under push-forward
by the dynamics; we are now going to prove a more general result.  Let
$\potFamily\subset\cC^2(\bT^2,\bC)$ be a family of smooth functions with uniformly bounded
$\cC^2$-norm; we denote by $\|\potFamily\|\nc{r} =
\sup_{\pot\in\potFamily}\|\pot\|\nc{r}$.  For any $\pot\in\potFamily$ define the
operator $F_{\ve *,\pot}$ acting on a complex measure $\mu$ as follows: for any
measurable function $g$ of $\bT^2$
\begin{equation}\label{eq:push-pot}
  \left[F_{\ve *,\pot}\mu\right](g) := \mu(e^{\pot}\cdot g\circ F_{\ve}).
\end{equation}
We call $F_{\ve *,\pot}$ the \emph{weighted push-forward operator with potential $\pot$};
observe that $F_{\ve *,0} = F_{\ve *}$ is the usual push-forward.
\begin{prop}[Invariance]\label{p_invarianceStandardPairs}\ %
Given a family of complex potentials $\potFamily$, there
  exist $\spc1$, $\spc2$, $\spc3$ and $\delta$ such that the following holds.  For any
  $\pot\in \potFamily$ and complex $(\spc1,\spc2,\spc3)$-standard family $\stdf$ of length $\deltacomplex\leq \min\{\delta,\pi/( 10\, \spc2)\}$, the complex measure $F_{\ve*,\pot}\mu_{\stdf}$ can be decomposed in complex
  $(\spc1,\spc2,\spc3)$-standard pairs, i.e.\ there exists a complex
  $(\spc1,\spc2,\spc3)$-standard family $\stdf'_\pot$, of length $\deltacomplex$, such that
  $F_{\ve*,\pot}\mu_{\stdf}=\mu_{\stdf'_\pot}$.  We say that $\stdf'_\pot$ is a
  \emph{$(\spc1,\spc2,\spc3)$-standard decomposition of} $F_{\ve*,\pot}\mu_{\stdf}$ and we
  write --with a little abuse of notation-- $\stdf'_\pot\sim F_{\ve*,\pot}\stdf$.
  Moreover, the constant $\spc1$ does not depend on $\potFamily$, whereas the constants
  $\spc2$ and $\spc3$ (and consequently $\delta$) can be chosen as follows: 
  \begin{align}\label{e_estimatec2c3}
    \spc2    &\ge \Const(1+\|\potFamily\|\nc1)&
    \spc3    &\ge \Const(1+\|\potFamily\|\nc2+\|\potFamily\|\nc1^2).
  \end{align}
\end{prop}
\begin{proof}
  For simplicity, let us assume that $\stdf$ is given by a single complex standard pair
  $\ell$; the general case does not require any additional ideas and it is left to the
  reader.

  Let then $\ell=(\bG,\rho)$ be a complex $(\spc1,\spc2,\spc3)$-standard pair.  For any
  sufficiently smooth function $A$ on $\bT^2$, by the definition of standard curve, it is
  trivial to check that:
  \begin{subequations}\label{e_AG}
    \begin{align}
      \|(A\circ\bG)'\| &\le \|\deh A\| (1+\ve c_1)\label{e_AG'}\\
      \|(A\circ\bG)''\| &\le \ve\|\deh A\| \dtt0c_1+\|\deh A\|\nc1(1+\ve c_1)^2\label{e_AG''}\\
      \|(A\circ\bG)'''\| &\le \ve\|\deh A\| \dtt1c_1+\|\deh A\|\nc2(1+\ve(1+\dtt0) c_1)^3\label{e_AG'''}.
    \end{align}
  \end{subequations}
  Let us then introduce the maps $f_{\bG}= f\circ\bG$, $\omega_{\bG}= \omega\circ\bG$ and
  $\pot_{\bG}=\pot\circ\bG$.  We will assume $\ve$ to be small enough (depending on our
  choice of $c_1$) so that $f_{\bG}'\ge \lambda-\ve c_1 \|\partial_\theta f\|>3/2$; in
  particular, $f_{\bG}$ is an expanding map.  Provided $\delta$ has been chosen small
  enough, $f_\bG$ is invertible.  Let $\vf(x)=f_{\bG}\invr(x)$.  Differentiating we obtain
  \begin{align}\label{e_estimatesVf}
    \vf'   &= \frac1{f_\bG'}\circ\vf &%
    \vf''  &= -\frac{f_\bG''}{f_\bG'^3}\circ\vf &%
    \vf''' &= \frac{3f_\bG''^2-f_\bG'''f_\bG'}{f_\bG'^5}\circ\vf.
  \end{align}

  Then, by definition, for any measurable function $g$:
  \begin{align*}
    F_{\ve*,\pot}\mu_\ell(g) &= \mu_{\ell}(e^\pot\cdot g\circ F_{\ve})\\&= \int_{a}^{b}%
    g(f_\bG(x),\bar G(x))\cdot %
    e^{\pot_\bG(x)} %
    \rho(x)%
    \deh x,
  \end{align*}
  where $\bar G(x)=G(x)+\ve\omega_{\bG}(x)$.  Then, fix a partition (mod $0$)
  $[f_{\bG}(a),f_{\bG}(b)]=\bigcup_{j\in\cJ}[a_{j},b_{j}]$, with
  $b_{j}-a_{j}\in[\deltacomplex/2,\deltacomplex]$ and $b_{j}=a_{j+1}$.  We can thus write:
  \begin{equation}\label{e_induction0}
    F_{\ve*,\pot}\mu_\ell(g) =\sum_j\int_{a_{j}}^{b_{j}} g(x,G_j(x))\tilde\rho_j(x)\deh x,
  \end{equation}
  provided that $G_j=\bar G\circ\vf_j$ and $\tilde\rho_j(x)= e^{\pot_\bG\circ\vf_j}\cdot
  \rho\circ\vf_j\cdot\vf_j'$, where $\vf_j=\vf|_{[a_j,b_j]}$.

  In order to conclude our proof it suffices to show that
  \begin{enumerate*}[label=(\roman*), itemjoin*={{, and }}]
  \item \label{i_geometryc1} there exists $\spc1$ large enough
    so that if $G\in\Sigma_{\spc1}$, then $G_j\in\Sigma_{\spc1}$
  \item \label{i_potentialc2} there exist $\spc2,c_3$ large
    enough and $\delta$ small enough so that if $\rho\in D_{\spc2,c_3}^\bC(G)$, $\tilde\rho_j$
    can be normalized to a complex standard density belonging to $D_{\spc2,c_3}^\bC(G_j)$.
  \end{enumerate*}

  Item~\ref{i_geometryc1} follows from routine computations: differentiating the above
  definitions and using~\eqref{e_estimatesVf} we obtain
  \begin{subequations}\label{e_C1}
    \begin{align}
      G_j' &=%
      \frac{\bar G'}{f'_\bG}\circ\vf_j\label{e_c1'}\\%
      G_j'' &=%
      \frac{\bar G''}{f'^2_\bG}\circ\vf_j%
      -G_j'\cdot\frac{f_\bG''}{f_\bG'^2}\circ\vf_j\label{e_c1''}\\%
      G_j''' &=%
      \frac{\bar G'''}{f'^3_\bG}\circ\vf_j%
      -3 G''_j\cdot\frac{f_\bG''}{f_\bG'^2}\circ\vf_j-
      G'_j\cdot\frac{f_\bG'''}{f_\bG'^3}\circ\vf_j\label{e_c1'''}%
    \end{align}
  \end{subequations}
  Using~\eqref{e_c1'}, the definition of $\bar G$ and~\eqref{e_AG'} we obtain, for small
  enough $\ve$:
  \begin{align*}
    \|G'_j\| &\le \left\|\frac{ G' + \ve\omega_\bG'}{f_\bG'}\right\|
    \le\frac23(1+\ve\|\deh\omega\|)\ve c_1 + \frac23\ve\|\deh\omega\| \le\frac34\ve c_1+ \ve D_1%
  \end{align*}
  where $D_1=\frac23\|\deh\omega\|$.  We can then fix $c_1$ large enough so that the right
  hand side of the above inequality is less than $c_1$.  Next we will use $C_*$ for a
  generic constant depending on $c_1,D_1,D_1'$ and $\Const$ for a generic constant
  depending only on $F_\ve$.  Then, we find\footnote{ The reader can easily fill in the
    details of the computations.}
  \begin{align*}
    \|G''_j\|&\le \frac34\ve [ c_1 D_1+\Const]+\ve^2 C_*\,;\\%
    \|G'''_j\|&\le \frac34\ve\left[ c_1( D_1'+D_1\Const+\Const)+\Const\right]+\ve^2 C_*.
  \end{align*}
  We can then fix $c_1,D_1'$ sufficiently large and then $\ve$ sufficiently small to
  ensure that the $\bG_{j}$'s are $c_1$-standard pairs.
  We now proceed with item~\ref{i_potentialc2}; by differentiating the definition of
  $\tilde\rho_j$ we obtain
  \begin{subequations}\label{e_C2}
    \begin{align}
      \label{e_c2'} 
      \frac{\tilde\rho_j'}{\tilde\rho_j}&=\frac{\rho'}{\rho\cdot
        f'_\bG}\circ\vf_j-\frac{f_\bG''}{f_\bG'^2}\circ\vf_j + \frac{\pot_\bG'}{f_\bG'}\circ\vf_j \\
      \label{e_c2''} 
      \frac{\tilde\rho_j''}{\tilde\rho_j}&=%
      \frac{\rho''}{\rho\cdot f'^2_\bG}\circ\vf_j
      -3\frac{\tilde\rho'_j}{\tilde\rho_j}\cdot \frac{f_\bG''}{f_\bG'^2}\circ\vf_j
      -\frac{f_\bG'''}{f_\bG'^3}\circ\vf_j+\\&\phantom=+\notag%
      2\frac{\rho'\pot_\bG'}{\rho{f_\bG'{}^2}}\circ\vf_j +
      \frac{\pot_\bG'{}^2}{f_\bG'{}^2}\circ\vf_j
      +\frac{\pot_\bG''{}^2}{f_\bG'{}^2}\circ\vf_j.
    \end{align}
  \end{subequations}
  From the first of the above expressions and~\eqref{e_AG'} we gather:
  \begin{align*}
    \left\|\frac{\tilde\rho_j'}{\tilde\rho_j}\right\|\nc0\le \frac23
    \left\|\frac{\rho'}{\rho}\right\|\nc0 + D + \Const\|\pot\|\nc1,
  \end{align*}
  where $D$ is a uniform constant related to the \emph{distortion} of the maps
  $f(\cdot,\theta)$, which can be obtained using our uniform bounds on $\|G'_j\|$ and $\|G''_j\|$.
  The above expression implies that we can choose $\spc2=\cO(1+\|\pot\|\nc1)$ so that if
  $\|\rho'/\rho\|\nc0\le \spc2$, then $\|\tilde\rho_j'/\tilde\rho_j\|\nc0\le \spc2$.  A
  similar computation, using~\eqref{e_c2''}, yields:
  \[
  \left\|\frac{\tilde\rho_j''}{\tilde\rho_j}\right\|\le   \frac49%
  \left\|\frac{\rho''}{\rho}\right\| + \Const(\|\pot\|\nc2+\|\pot\|\nc1^2+\spc2(\|\pot\|\nc1+D)+D'),
  \]
  where, once again, $D'$ is uniformly bounded thanks to our bounds on $\|G'_j\|$,
  $\|G''_j\|$ and $\|G'''_j\|$.  As
  before, this implies the existence of $c_3=\cO(1+\|\pot\|\nc2+\|\pot\|\nc1^2)$ so
  that if   $\left\|\rho''/\rho\right\|\nc0\le c_3$, then
  $\|\tilde\rho_j''/\tilde\rho_j\|\nc0\le c_3$.

  We are now left to show that, using our requirement on $\deltacomplex$:
  \begin{equation}\label{e_normalizationPossible}
    \fmm_j:=\int_{a_j}^{b_j}\tilde\rho_j\deh x\not =0;
  \end{equation}
  this implies that $\rho_j:=\fmm_j\invr\tilde\rho_j\in D_{\spc2,c_3}^\bC(G_j)$, which
  concludes our proof: in fact, define the standard family $\stdf'_\pot$ given by
  $\SF{\ell_j}{\fm}$, where $\fm_{\ell_j}=\fmm_j$; then we can
  rewrite~\eqref{e_induction0} as follows:
  \[
  F_{\ve *,\pot}\mu_\ell(g)=%
  \sum_{\tell\in\stdf'_\pot}\fm_\tell\mu_{\tell}(g)=%
  \mu_{\stdf'_\pot}(g).
  \]
  The proof of~\eqref{e_normalizationPossible} follows from arguments similar to the ones
  used in the proof of Lemma~\ref{l_variation}: in fact $\deltacomplex$ is sufficiently small so
  that the function $\log\tilde\rho_j$ can be defined and it is contained in a square of
  side $\spc2\deltacomplex$.  Therefore, $\textup{Range}\,\tilde\rho_j$ is contained in an annular
  sector of small aperture, whose convex hull is bounded away from $0$; this implies that
  $\fmm_j\not = 0$.
\end{proof}
\begin{rem}
  Assume $\stdf$ to be a standard probability family and $\pot\in\cC^2(\bT^2,\bR)$: then
  $\stdf'_{\pot}$ is also a real standard family.  Moreover, $\stdf'_0$ is a
  standard probability family.
\end{rem}
\begin{rem}\label{rem:prestandard}
A quick inspection to the proof of Proposition~\ref{p_invarianceStandardPairs}  shows that we can choose the standard family $\stdf'_{\pot}$ to be of length $\frac 32 \deltacomplex$, provided $\frac 32 \deltacomplex\leq\min\{\delta,\pi/( 10\, \spc2)\}$.
\end{rem}
\begin{rem}\label{rem:complextoreal}
Note that if $\ell=(\bG, \rho)$ is  a complex standard pair, then $\ell_\bR=(\bG,|\rho|)$ is also a complex standard pair, and indeed is a regular standard pair if $\deltacomplex=\delta$. Moreover form the arguments in the proof of Proposition~\ref{p_invarianceStandardPairs} it follows that, for $\delta$ small enough, if $\stdf_{\pot}=\{\ell\}$, calling $\stdf'_{\pot}=\{\{\ell'\}_{\ell'\in \stdf'_{\pot}}, \fm_{\ell'}\}$ and  $\stdf'_{\Re(\pot)}=\{\{\ell'_\bR\}_{\ell'\in \stdf'_{\pot}}, \fm_{\bR,\ell'}\}$ the family obtained applying Proposition~\ref{p_invarianceStandardPairs} to  $\{\ell_\bR\}$ we have $|\fm_{\ell'}|\geq \const \fm_{\bR,\ell'}$.
\end{rem}
We say that $\ell$ is a \emph{$\potFamily$-standard pair} if $\spc1, \spc2, \spc3$ and
$\deltacomplex$ are so that Proposition~\ref{p_invarianceStandardPairs} holds with respect to the
family $\potFamily$.  Given a $\potFamily$-standard pair $\ell$ and a sequence of
potentials $(\pot_k)_{k\in\bN} = \potSequence\in\potFamily^\bN$, we denote (again with an
abuse of notation) by $\stdf^{(n)}_{\ell,\potSequence}$ a standard decomposition of
$F^{(n)}_{\ve *,\potSequence}\mu_\ell=F_{\ve *,\pot_{n-1}}\cdots F_{\ve
  *,\pot_0}\mu_\ell$, which we obtain by iterating the above proposition.  By definition,
therefore, we have, for any sufficiently smooth function $g$ of $\bT^2$:
\begin{equation}\label{eq:preparatory0}
  F_{\ve *,\potSequence}^{(n)}\,\mu_\ell(g)=%
  \sum_{\tell\in \stdf^{(n)}_{\ell,\potSequence}}\fm_{\tell}\mu_{\tell}(g)=%
  \mu_\ell\left(e^{S_n\potSequence}g\circ F_\ve^n\right),%
\end{equation}
where we have defined the ``Birkhoff sum'' $S_n\potSequence=\sum_{k=0}^{n-1}\pot_k\circ
F_\ve^k$.  In particular,~\eqref{eq:preparatory0} implies that \(
\mu_\ell\left(e^{S_n\potSequence}\right) = \sum_{\tell\in\stdf^{(n)}_{\ell,\potSequence}}\fm_\tell.  \)
\begin{rem}\label{rem:fm-bound}
  The proof of Proposition~\ref{p_invarianceStandardPairs} allows to define, for any
  $\tell\in\stdf^{(n)}_{\ell,\potSequence}$ the corresponding characteristic function
  $\Id_{\tell}$, that is a random variable on $\ell$ which equals $1$ on points which are
  mapped to $\tell$ by $F_\ve^n$ and $0$ elsewhere.  This allows to write:
  \begin{subequations}\label{eq:def-zero-step}
    \begin{align}
      \fm_{\tell}&=\mu_{\ell}\left( e^{S_n\potSequence}
        \Id_{\tilde \ell}\right)\label{e_def-zero-step-nu}\\
      \mu_{\tell}(g)&=\fm_{\tell}\invr\mu_\ell\left( e^{S_n\potSequence} \Id_{\tell} \cdot
        g\circ F_\ve^n \right)\label{e_def-zero-step-mu},
    \end{align}
  \end{subequations}
\end{rem}
Observe that~\eqref{e_def-zero-step-nu} and~\eqref{e_trivialTotalVariationBound}
immediately implies that, for any $n\in\bN$:
\begin{equation}\label{eq:trivial-nu.basic}
  \sum_{\tell\in\stdf_{\ell,\potSequence}^{(n)}}|\fm_\tell|\le \expo{\sum_{k=0}^{n-1} \max\Re\,\pot_k+2\spc2\deltacomplex}.
\end{equation}
Moreover,
\begin{equation}\label{eq:trivial-nu.iterated}
  \sum_{\ell_1\in\stdf_{\ell,\potSequence}^{(n)}}%
  \sum_{\ell_2\in\stdf_{\ell_1,\potShift^n\potSequence}^{(n)}}%
  \cdots%
  \sum_{\ell_{m}\in\stdf_{\ell_{m-1},\potShift^{n(m-1)}\potSequence}^{(n)}}%
  \prod_{j=1}^{m}
  |\fm_{\ell_j}|\le e^{\sum_{k=0}^{nm-1}\max\Re\,\pot_k}e^{2\spc2\deltacomplex},
\end{equation}
where $\potShift$ is the one-sided shift acting naturally on $\potFamily^\bN$.  In fact,
the above is just a special choice of standard decomposition for
$F_{\ve*,\potSequence}^{(nm)}\,\mu_\ell$, indexed by a $m$-tuple of standard pairs
$(\ell_1,\cdots,\ell_{m-1})$ selected at intermediate steps of length $n$.

\begin{rem}\label{r_randomVariables}
  Given a standard pair $\ell=(\bG,\rho)$, we will interpret $(x_k,\theta_k)$ as random
  variables defined as $(x_k,\theta_k)=F_\ve^k(x, G(x))$, where $x$ is distributed
  according to $\rho$.  We would like do the same for complex standard pairs. Of course,
  in this case $(x_k,\theta_k)$ will be random variables under $\Re(\rho)$ only, so we
  will simply say that they are functions distributed according to $\rho$, or, for
  brevity, functions on $\ell$.
\end{rem}

Finally, let us define the set ``good probability measures" mentioned in
Section~\ref{sec:results}.  Fix $C_\star>0$ large enough; given $\theta^*_0\in\bT$, we
define
\begin{equation}\label{eq:good-measure}
\measures_\ve(\theta^*_0)=\{\mu_\stdf\;:\; \mu_\stdf(G_\ell)=\theta^*_0,\;\sup_{\ell\in\stdf} |G_\ell-\theta^*_0|\leq C_\star\ve,\; \stdf \in \textrm{standard families}\},
\end{equation}
where $\mu_\stdf(G_\ell):=\sum_{\ell\in\stdf}\fm_{\ell}\int_{\bT} G_\ell(x)\rho_\ell(x) dx$.





\section{Averaging}\label{sec:averaging}
This section is devoted to the proof of Theorem~\ref{t_averaging}.  The aim of this
section is mostly notational and didactic; therefore, we keep things as simple as possible
we provide the proof only for the variable $\theta$ since the argument for $z$ is exactly
the same.

In the following, given a standard pair $\ell=(\bG, \rho)$, we will use the notation
\begin{align}\label{eq:thetastartdef}
  \thetasl=\mu_\ell(\theta_0) \quad\textrm{and}\quad \thetaslk{k}=\bar\theta(\ve k,\thetasl)
\end{align}
where, according to Remark~\ref{r_randomVariables}, we consider
  $\theta_0 = G$ to be a random variable on the standard pair $\ell$ and we denote with
$\bar\theta(t,\thetas)$ the unique solution of~\eqref{eq:averageeq} for initial condition
$\bar\theta(0)=\thetas$.
\begin{rem}\label{eq:inital-cond}
We find it convenient to prove the theorem for slightly more general initial condition:  standard pairs $\ell_\ve$ s.t. $\mu_{\ell_\ve}\in \measures_\ve(\theta_0)$. In the following we will drop the subscript $\ve$ in the standard pair since this does not create confusion.
\end{rem}
\subsection{Deterministic approximation}\ \newline
First, we provide a preliminary useful approximation result, which allows to compare the
true dynamics with a fixed one for times of order $\ve^{-1/2}$: for fixed
$\theta^*\in\bT^1$, let us introduce\footnote{ The reader should not confuse the notation
  $F_*$, which is a map of $\bT^2$, with the push-forward $F_{\ve *}$ introduced in the
  previous section.} the map $F_*(x,\theta)=(f_*(x),\theta)$, where
$f_*(x)=f(x,\theta^*)$.
\begin{lem}\label{lem:shadow}
  Consider a standard pair $\ell=(\bG,\rho)$ and fix $\theta^*\in\bT^1$ at a distance at
  most $\epsilon$ from the range of $G$.\footnote{ We will typically
      apply this Lemma to the case $\theta^*=\thetasl$.}  For any $n\in\bN$ so that
  $\epsilon n+n^2\ve\le \Const$, there exists a diffeomorphism $Y_n:[a,b]\to[a^*, b^*]$
  such that
  $(x_n,\theta_n)=F_\ve^n(x, G(x))=(f_*^n(Y_n(x)),\theta_n)=F_*^n(Y_n(x),\theta_n)$.  In
  addition, for all $k\in\{0,\cdots,n\}$ and setting $x_k^*=f_*^k(Y_n(x))$,
  \begin{align*}
    \left\|\theta_k-\theta^*-\ve\sum_{j=0}^{k-1}\omega(x_j^*,\theta^*)\right\|\nc0 &\le \Const [\epsilon+\ve^2k^2],\\
    \left\| x^*_k-x_k+\ve\sum_{j=k}^{n-1}\Lambda^*_{k,j}\partial_\theta f(x^*_j,\theta^*)
    \sum_{l=0}^{j-1}\omega(x^*_l,\theta^*)\right\|\nc0 &\le \Const[\epsilon+\ve^2k^2],\\
    \left\|1-Y_n' \prod_{k=0}^{n-1} \frac{f_*' (x^*_k)}{\partial_x
      f(x_k,\theta_k)}\right\|\nc0 &\le \Const n\ve,
  \end{align*}
  where we defined $\Lambda^*_{k,j}=\prod_{l=k}^jf'_*(x^*_l)\invr
  =(f_*^{j-k+1})'(x^*_k)^{-1}\le \lambda^{k-j-1}$ and $\|\cdot\|\nc0$ denotes the usual
  $\sup$-norm of the random variables seen as functions of $x\in[a,b]$.
\end{lem}
\begin{proof}
  Let us denote with $\pi_x:\bT^2\to\bT$ the canonical projection on the $x$ coordinate;
  then, for $x,z\in\bT$ and $\convexP\in [0,1]$, define
  \begin{align*}
    \imFun_{n}(x, z; \convexP)&=\pi_x F_{\convexP\ve}^n(x,\theta^*+\convexP(G(x)-\theta^*)) - f_*^n(z).
  \end{align*}
  Note that, $\imFun_{n}(x,x; 0)=0$, in addition, for any $x,\convexP$:
  \begin{align*}
    \partial_{z}\imFun_{n}(x,z;\convexP)=- (f_*^n)'(z) \neq 0.
  \end{align*}
  Accordingly, by the implicit function theorem, for any $n\in\bN$ and
  $\convexP\in [0,1]$, there exists a diffeomorphism $Y_n(\cdot;\convexP)$ such that
  $\imFun_{n}(x,Y_n(x;\convexP);\convexP)=0$; from now on $Y_n(x)$ stands for $Y_n(x;1)$.
  Observe moreover that
  \begin{align}\label{e_expressionY'}
    Y_n' = \frac{(\pi_xF_\ve^n\circ\bG)'}{(f_*^n)'\circ Y_n} =
    \frac{(1-G'\stable_n)\enu_n}{(f_*^n)'\circ Y_n},
  \end{align}
  where we have used the notations introduced in~\eqref{e_defineSlopes}.

  Next, we want to estimate to which degree ${\{(x^*_k, \theta^*)\}}_{k=0}^n$ shadows the true
  trajectory.  Observe that
  \[
  \theta_k=\ve\sum_{j=0}^{k-1}\omega(x_j,\theta_j)+\theta_0
  \]
  thus $|\theta_k-\theta^*|\le \Const \ve k +\epsilon$.  Accordingly, let us set $\xi_k=x_k^*-x_{k}$;
  then by the mean value theorem we obtain, for some $x,\theta\in\bT$:
  \begin{align*}
    |\xi_{k+1}| %
    &= |f_*'(x)\cdot\xi_k + \partial_\theta f(x_k,\theta)\cdot(\theta_k-\theta^*)| \ge
    \lambda|\xi_k| - \Const (\theta_k-\theta^*)
  \end{align*}
  which, by backward induction, using the fact that $\xi_n=0$ and our previous estimates
  on $|\theta_k-\theta^*|$, yields $|\xi_k|\le\Const (\epsilon+\ve k)$.  We thus obtain:
  \begin{align*}
    \theta_k-\theta^* &=
    \theta_0-\theta^*+\ve\sum_{j=0}^{k-1}\omega(x^*_j,\theta^*)+\cO(\ve(\epsilon k + \ve
    k^2))\\
    \xi_k &=-\sum_{j=k}^{n-1}\Lambda^*_{k,j}\partial_\theta f(x^*_j,\theta^*)\left
      (\ve\sum_{l=0}^{j-1}\omega(x^*_l,\theta^*)+\cO(\epsilon+\ve^2j^2)\right).
  \end{align*}
  Finally, recalling~\eqref{e_expressionY'},~\eqref{e_***} and using invariance of the
  center cone, we have
  \begin{align*}
    e^{-\const \ve n}\prod_{k=0}^{n-1} \frac{\partial_x f(x_k,\theta_k)}{f_*'(x^*_k)}\le%
    \left| \frac{(1-G'\stable_n)\enu_n}{(f_*^n)'}\right| \le %
    e^{ \const \ve n}\prod_{k=0}^{n-1} \frac{\partial_x f(x_k,\theta_k)}{f_*'(x^*_k)}.
  \end{align*}
  Accordingly, $Y_n$ is invertible with uniformly bounded derivative, since we assume
  $n\epsilon+n^2\ve\le\Const$.\footnote{ On the contrary, the reader can easily check that
  $\|Y''_n\|_\infty\sim\lambda^n$.\label{f_hugeYDerivative}}
\end{proof}
\subsection{Proof of the Averaging Theorem}\label{eq:subsec-ave}\ \newline
Let us now ready to prove our first result.
\begin{proof}[{\bf Proof of Theorem~\ref{t_averaging}}]
  Let $\ell$ be a standard pair; recall that we defined
  $\hat\omega(x,\theta)=\omega(x,\theta) - \bar\omega(\theta)$; for any $t,h > 0$ define
  $H = H(t,h) = \pint{(t+h)\vei}-\pint{t\vei}$; observe that $|H(t,h)-\pint{h\vei}|\le1$.
  Let us start by computing
  \begin{align}\label{eq:correlations}
    \notag\mu_\ell &\left({\left[%
          \ve\sum_{k=\pint{t\vei }}^{\pint{(t+h) \vei}-1}
          \ho(x_k,\theta_k)\right]}^2\right) = %
                     \sum_{\ell_1\in\stdf_{\ell}^{(\lfloor t\vei\rfloor)}}
                     \sum_{k=0}^{H-1} \ve^2\fm_{\ell_1}\mu_{\ell_1}(\ho^2\circ F^{k}_\ve)+\\
                   &+\sum_{\ell_1\in\stdf_{\ell}^{(\lfloor t\vei\rfloor)}}%
                     2\sum_{j=0}^{H-1} \sum_{k=j+1}^{H-1}\sum_{\ell_2\in\stdf_{\ell_1}^{(j)}}%
                     \ve^2\fm_{\ell_1}\fm_{\ell_2}\mu_{\ell_2}(\ho\circ F_\ve^{k-j}\,\cdot \ho),
  \end{align}
  where we repeatedly used Proposition~\ref{p_invarianceStandardPairs} and the notation
  introduced before~\eqref{eq:preparatory0} without $\potSequence$, since in this case
  $\potSequence=0$.

  Next, using Lemma~\ref{lem:shadow} we introduce, for any standard pair
  $\tilde\ell = (\tilde\bG,\tilde\rho)$, the diffeomorphisms $Y=Y_H$ and let
  $[a^*, b^*]=Y([a,b])$.  Let us call
  $\rho^*=\frac{\tilde\rho\circ Y^{-1}}{Y'\circ Y^{-1}}$ the push-forward of
  $\tilde\rho$ by $Y$, also let $\theta_{\tilde\ell}^*=\mu_{\tilde\ell}(\theta)$.  For any
  functions $\vf,g\in\cC^1(\bT^2)$ and $k\in\bN$, Lemma~\ref{lem:shadow} implies
  \begin{align*}
    \mu_{\tilde\ell}(g\circ F^k_\ve\cdot\vf)&= \int_{a^*}^{b^*}
    \rho^*(x)\vf(Y\invr(x),\theta_{\tilde\ell}^*)\cdot
    g(f_{\theta_{\tilde\ell}^*}^k(x),\theta_{\tilde\ell}^*)\deh x%
    + \cO(k\ve \|g\|\nc1\|\vf\|\nc1)\\
    &=\int_{a}^{b} \tilde\rho(x)\vf(x,\theta_{\tilde\ell}^*)\cdot
    g(f_{\theta_{\tilde\ell}^*}^k(x), \theta_{\tilde\ell}^*)\deh x + %
    \cO(k^2\ve \|g\|\nc1\|\vf\|\nc1).
  \end{align*}
  To continue we introduce one of the main tools in the study of hyperbolic systems: the transfer operator (for now, without potential).
  Let
  \begin{align*}
  \tO_\theta g(x)=\sum_{y\in f_\theta\invr(x)}\frac {g(y)}{f'_\theta(y)}.
  \end{align*}
  The basic properties of these operators are well known (see e.g.~\cite{Baladibook}) but
  in the following we need several quite sophisticated facts that are either not easily
  found or absent altogether in the literature.  To help the reader we have collected all
  the needed properties in Appendix~\ref{subsec:transfer}.\footnote{ For the time being we
    need only that $\int g\tO_\theta \phi=\int g\circ f_\theta \phi$ and that, seen as an
    operator acting on $BV$, $\tO_\theta$ has $1$ as a maximal eigenvalue, a spectral gap,
    and $h_\theta$ (the eigenfunction associated to the eigenvalue $1$) is the $\cC^{r-1}$
    density of the unique absolutely continuous invariant measure of $f_\theta$.  In other
    words $\tO_\theta$ has the spectral decomposition $\tO_\theta g=h_\theta \int g+Rg$,
    where the spectral radius of $R$ is smaller that some $\tau\in (0,1)$. } We can thus
  estimate the quantity in the second line of~\eqref{eq:correlations} as
  \begin{align*}
    \mu_{\ell_2}(\hat \omega\circ F_\ve^l\,\cdot \hat\omega) &=%
    \int [\tO^l_{\theta_{\ell_2}^*}(\Id_{[a,b]}\rho\hat\omega)](x,\theta_{\ell_2}^*)\cdot
    \ho(x,\theta_{\ell_2}^*)\deh x+\cO(l^2\ve) = \\%
    &= \int_{\bT^1} h_{\theta_{\ell_2}^*}(x)\ho(x,\theta_{\ell_2}^*)\deh{}x
    \int_a^b\rho(x)\ho(x,\theta_{\ell_2}^*)\deh x+\cO(l^2\ve+\tau^l) = \\
    &=\cO(l^2\ve+\tau^l),
  \end{align*}
  where we used the fact that $\mu_\theta(\hat\omega(\cdot,\theta))=0$ by construction and
  $\tau\in(0,1)$ where $1-\tau$ is a lower bound on the spectral gap of $\tO_\theta$ for
  any $\theta\in\bT$

  Collecting all the above considerations we obtain
  \begin{align}\label{eq:correlations1}
    \mu_\ell\left(\left[\ve\sum_{k=\pint{t\vei}}^{\pint{(t+h)\vei}-1}\ho(x_k,\theta_k)\right]^2\right)
    &\le \Const \ve^2 \sum_{k=0}^{H-1}\left[1+\sum_{j=1}^{H -k-1}\{\tau^j +j^2\ve\}\right]\\
    &\le \Const[\ve h+\vei h^4]\le\Const \ve\eefrac53,\notag
  \end{align}
  where at the very last step we have chosen $h = \ve\eefrac23$, which optimizes the
  estimate.  Recall now the definition of the random element
  $\theta_\ve\in C^0([0,T],\bR)$, defined in~\eqref{eq:z-path}.  As previously observed,
  the functions $\theta_\ve$ are uniformly Lipschitz of constant $\|A\|_{\cC^0}$.  Using
  the Cauchy--Schwarz inequality and~\eqref{eq:correlations1}:
  \begin{align*}
    \mu_\ell
    &\left(\left|\theta_\ve(t)-
\theta_\ve(0)-\int_0^{t}\bar\omega(\theta_\ve(s))\deh{}s\right|^2\right)\\
    &\le \pint{th^{-1}}\sum_{r=0}^{\pint{th\invr}-1}\mu_\ell\left(\left|\theta_\ve((r+1)h)-\theta_\ve(rh)-\int_{rh}^{(r+1)h}
      \bar\omega(\theta_\ve(s))\deh{}s\right|^2\right)\\
    &\le \pint{th^{-1}}\sum_{r=0}^{\pint{th\invr}-1}\mu_\ell\left(\left|\ve\sum_{k=\pint{rh\vei}}^{\pint{(r+1)h\vei}-1}\hat\omega(x_k,\theta_k)+\cO(\ve h)\right|^2\right)
       \end{align*}
     \begin{align*}
    &\le \Const t^2 [\ve h^{-1}+\vei h^2+\ve]\le \Const t^2 \ve^{\frac 13},
  \end{align*}
  where at the very last step we have chosen $h=\ve^{\efrac 23}$.  Chebyshev inequality
  then implies, for any $t\le T$:
  \begin{equation}\label{eq:Chebyshev}
    \mu_\ell\left(\left\{
      \left|\theta_\ve(t)-\theta_\ve(0)-\int_0^{t}\bar\omega(\theta_\ve(s))\deh s\right|\ge
      \Const \ve^{\efrac 18}\right\}\right) \le \Const T^2 \ve^{\efrac 13}\ve^{-\efrac 14}.
  \end{equation}
  Let us partition the interval $[0,T]$ in $N = \pint{T\ve^{-1/24}}$ intervals of
  endpoints
  \begin{align*}
    0 = t_0 < t_1 < \cdots < t_{N} = T,
  \end{align*}
  where for any $k\in\{0,\cdots,N-1\}$ we have
  $\ve^{1/24}\le t_{k+1}-t_{k} < 2\ve^{1/24}$.  Since $\theta_\ve$ is uniformly Lipschitz
  and using~\eqref{eq:Chebyshev}, we conclude:
  \begin{equation}\label{eq:close-zero}
  \begin{split}
    &\mu_\ell\left(\left\{
      \sup_{t\in[0,T]}\left|\theta_\ve(t)-\theta_\ve(0)-\int_0^{t}\bar\omega(\theta_\ve(s))\deh
        s\right|\ge \Const \ve^{\efrac 1{24}}\right\}\right)\\
    &\le\mu_\ell\left(\bigcup_{k = 0}^{N-1}\left\{
      \left|\theta_\ve(t_k)-\theta_\ve(0)-\int_0^{t_k}\bar\omega(\theta_\ve(s))\deh
        s\right|\ge \Const \ve^{\efrac 1{24}}\right\}\right)\\
        &\le \Const T^3\ve^{\efrac 1{24}}.
  \end{split}
  \end{equation}
  Since $\theta_\ve$ are a uniformly Lipschitz family of paths, they form a compact set by
  Ascoli--Arzel\`a Theorem.  Consider then any converging subsequence $\theta_{\ve_j}$;
  choosing $\ve=\ve_j$ and taking the limit of~\eqref{eq:close-zero} for $j\to \infty$ it
  follows that all accumulation points of $\theta_\ve$ are solutions of the integral
  version of~\eqref{eq:averageeq}.  Since such differential equation admits a unique
  solution, we conclude that the limit exists and it is given by the solution
  of~\eqref{eq:averageeq}.

  If we consider now the initial conditions of the Lemma, which allow to consider all
  random variables on the same probability space, we immediately have the result for
  $\theta$.  The results for $z$ is more of the same.
\end{proof}
\begin{rem}
  Note that it may be possible to obtain this result \emph{almost surely} rather than in
  probability.  We do not push this venue since it is irrelevant for our
  purposes.\footnote{ But see~\cite{Kifer09} for a discussion of possible
    counterexamples.}
\end{rem}
\begin{rem} The bound~\eqref{eq:close-zero} was obtained by estimating the second moment.
  This gave us a simple argument, but not sufficient for our later needs.  To
  get sharper bounds we will need to estimate the exponential moment, which is tantamount
  to studying large deviations.
\end{rem}

\section{Moment generating function}\label{subsec:markov}\
We now begin the study of deviations from the average behavior described in the previous
section.  We start with the problem of investigating large and moderate deviations.  It is
well known that such information can be obtained from precise estimates of the exponential
moment generating function.  Hence our next goal is the study of this object.  In order to
do so it turns out to be helpful to have an approximate description of the dynamics that
is more refined than the one presented in Lemma~\ref{lem:shadow}.  This is achieved in the
next subsection.
\subsection{Random approximation}\label{subsec:randomapp}\ \newline
In order to obtain our main results Theorem~\ref{thm:large} and Theorem~\ref{thm:lclt}, we
will need to control deviations from the average with resolution up to order $\ve$; this
requires very fine bounds which we proceed to obtain in this subsection.

For later reference, we find convenient to state such estimates in a slightly more general
form than needed for our immediate purposes; we introduce two different notions of
\emph{deviation from the average}: let $\ell$ be a standard pair; recall the notation
$\thetaslk{k} = \bar\theta(\ve k,\thetasl)$ introduced in~\eqref{eq:thetastartdef}, where
$\thetasl = \mu_\ell(G_\ell)$.  Let us also define the functions
$\bar\theta_{k}(\theta) = \bar\theta(\ve k,\theta)$ (observe that
$\thetaslk{k} = \bar\theta_{k}(\thetasl)$).  Then we define two corresponding notions of
deviation:%
\begin{subequations}%
  \begin{align}
    \label{e_deviationsDef}\deviationsl_k(x,\theta) &= \theta_k(x,\theta)-\thetaslk{k}\\
    \label{e_deviationeDef}\deviatione_k(x,\theta) &= \theta_k(x,\theta)-\bar\theta_k(\theta).
  \end{align}
\end{subequations}
Since $|\theta_k-\theta_0|\le\Const\ve k$, $|\thetaslk{k}-\thetasl |\le\Const\ve k$,
$|\bar\theta_k-\theta|\le\Const\ve k$ and $|\theta_0 -\thetasl|\le\Const\ve$, we trivially
find
\begin{align}\label{eq_trivialBoundDeviations}
  |\deviationsl_k|&\le \Const \ve (k+1)
& |\deviatione_k|&\le \Const \ve k.
\end{align}
Moreover, observe that
\begin{align*}
  \theta_{k+1} - \theta_k &= \ve \bar\omega(\theta_k) + \ve \hat \omega(x_k,\theta_k)\\
  \thetaslk{k+1} - \thetaslk{k} &= \ve \bar\omega(\thetaslk{k}) +
                                  \frac12\ve^2\bar\omega'(\thetaslk{k})\bar\omega(\thetaslk{k}) + \cO(\ve^3),\\
  \bar\theta_{k+1} - \bar\theta_{k} &= \ve \bar\omega(\bar\theta_{k}) +
                                      \frac12\ve^2\bar\omega'(\bar\theta_{k})\bar\omega(\bar\theta_{k}) + \cO(\ve^3),
\end{align*}
where, recall $\bar\omega(\theta) = \mu_\theta(\omega(\cdot,\theta))$ and
$\ho(\cdot,\theta) = \omega(\cdot,\theta)-\bar\omega(\theta)$.  The above equations yield the difference equations:
\begin{subequations}
  \begin{align}\notag
    \deviationsl_{k+1}-\deviationsl_k
    &= \ve\ho(x_k,\theta_k) + \ve\bar\omega'(\thetaslk{k})\deviationsl_k+\\\label{eq:Delta-recursion}
    &\pheq + \frac{\ve}2\bar\omega''(\thetaslk{k})(\deviationsl_k)^2- \frac{\ve^2}2\bar\omega'(\thetaslk{k})\bar\omega(\thetaslk{k})+ \cO(\ve(\deviationsl_k)^3+\ve^3).\\\notag
    \deviatione_{k+1}-\deviatione_k &= \ve\ho(x_k,\theta_k) + \ve\bar\omega'(\bar\theta_{k})\deviatione_k+\\&\pheq + \frac{\ve}2\bar\omega''(\bar\theta_{k})(\deviatione_k)^2- \frac{\ve^2}2\bar\omega'(\bar\theta_{k})\bar\omega(\bar\theta_k{k})+ \cO(\ve(\deviatione_k)^3+\ve^3).\label{eq:Deltae-recursion}
  \end{align}
\end{subequations}
Define now the auxiliary functions:
\begin{subequations}
  \begin{align}\label{e_defEtaRef}
    \etaRefs_{k} &=\sum_{j=0}^{k-1}\etaExps_{j,k}\left[\ho(x_j,\theta_j)-\frac{\ve}2\bar\omega'(\thetaslk{j})\bar\omega(\thetaslk{j})\right],\\
    \label{e_defEtaRef2}\etaRefb_{k} &=\sum_{j=0}^{k-1}\etaExpb_{j,k}\left[\ho(x_j,\theta_j)-\frac{\ve}2\bar\omega'(\bar\theta_{j})\bar\omega(\bar\theta_{j})\right],
  \end{align}
\end{subequations}
where
\begin{align}\label{e_defEtaExp}
  \etaExps_{j,k} &=\prod_{l=j+1}^{k-1}\left[1+\ve\bar\omega'(\thetaslk{l})\right]
& \etaExpb_{j,k} &=\prod_{l=j+1}^{k-1}\left[1+\ve\bar\omega'(\bar\theta_{l})\right].
\end{align}
We are now finally ready to state and prove the needed approximation.  The following lemma
is a refinement of Lemma~\ref{lem:shadow}.
 \begin{lem}\label{l_boundEtaRef}%
   For any $T > 0$, $0\le k\le T\vei$ and standard pair $\ell$, we have
   \begin{subequations}
     \begin{align}\label{e_boundEtaRefs}
       \deviationsl_k-\ve \etaRefs_{k}
       &= \ve\sum_{j=0}^{k-1}\etaExps_{j,k}\left[\frac{\bar\omega''(\thetaslk{j})}{2}(\deviationsl_j)^2+\cO((\deviationsl_j)^3+\ve^2)\right]\\
       &\phantom{=}\;+\etaExps_{-1,k}\deviationsl_0\notag\\
       \label{e_boundEtaRefb}
       \deviatione_k-\ve \etaRefb_{k}
       &= \ve\sum_{j=0}^{k-1}\etaExpb_{j,k}\left[\frac{\bar\omega''(\bar\theta_{j})}{2}(\deviatione_j)^2+\cO((\deviatione_j)^3+\ve^2)\right].
     \end{align}
   \end{subequations}
\end{lem}
\begin{proof}
  Observe that~\eqref{e_defEtaRef} implies $|\etaRefs_{k}|\le\Const k$.  Also, it is
  immediate to check that, by definition, $\etaRefs_{k}$ satisfies the following recurrence
  equation:
  \begin{align*}
    \etaRefs_{k+1}
    &= \ho(x_k,\theta_k)
      +(1+\ve\bar\omega'(\thetaslk{k})) \etaRefs_{k}-%
      \frac{\ve}2\bar\omega'(\thetaslk{k})\bar\omega(\thetaslk{k}).
  \end{align*}
  Hence, by~\eqref{eq:Delta-recursion},
  \begin{align*}
     \deviationsl_{k+1}- \ve \etaRefs_{k+1}%
      &=(1+ \ve\bar\omega'(\thetaslk{k}))\left[\deviationsl_k-\ve\etaRefs_{k}\right]+
      \frac{\ve}2\bar\omega''(\thetaslk{k})(\deviationsl_k)^2 \\
      &\phantom = +\cO(\ve(\deviationsl_k)^3+\ve^3).
  \end{align*}
  The first statement of the lemma then follows by induction, since $\etaRefs_0 = 0$.  The
  second statement follows by identical computations and the observation that, by
  definition, we have $\deviatione_0 = 0$.
\end{proof}
\subsection{Computation of the exponential moment}\label{ss_exponentialMoment}\ \newline
We can now proceed to the main result of this section, which is the precise computation of
the exponential moment.  The goal is to compute it with an error much smaller than
currently available in the literature.  This will allow to obtain precise information not
only on large, but also on moderate deviations, as will be shown in the next two sections.

In this section, given a $(\spc1,\spc2,c_3)$-standard pair $\ell=(\bG,\rho)$, we will call
it simply a $\spc2$-standard pair, since $\spc1$ will be always fixed (as in the rest of
the paper) and $c_3$ is irrelevant for the estimates in this section.  %
Recall that we fixed $A = (A_1,\cdots,A_d)\in\cC^3(\bT^2,\bR^d)$, with $A_1=\omega$; we also
introduced the notation $\bar A(\theta)=\mu_\theta(A(\cdot,\theta))$ and
$\hat A=A-\bar A$.  Recall that $\thetasl=\mu_\ell(G)=\int_a^b \rho(x) G(x) \deh x$ (hence
it belongs to the range of the standard pair $\ell$).  Moreover, recall that we are under
the standing assumption~\ref{eq:coboA}.
\begin{rem}\label{rem:variationdef}
  In this section we will use the notation $\BV([0,T],\bR^d)$ to denote the space of
  functions in $\bR^d$ whose components are bounded variation functions.  Recall that,
  given a $L^1$ function $\varphi:I\to\bR^d$, its $\BV$-norm is defined as:
  \begin{align*}
    \|\varphi\|_{\BV(I)} = \|\varphi\|_{L^1(I)}+V_I(\varphi),
  \end{align*}
  where $V_I(\varphi)$ is the \emph{total variation} of $\varphi$ on the interval $I$,
  given by:
  \begin{align*}
    V_I(\varphi) =
    \sup_{\substack{\psi\in\cC_\text{c}^1(I,\bR^d)\\ \|\psi\|_\infty = 1}}\int_I\langle\varphi(x),\psi'(x)\rangle dx,
  \end{align*}
  where $\cC_\text{c}^1(I)$ is the space of $\cC^1$ functions that are $0$ in a
  neighborhood of the boundary of $I$.  Moreover in this section, given $I\subset\bR$, we
  will denote $\|f\|\nl\infty = \sup_{x\in I}|f(x)|$.  As usual, if the set $I$ is not
  specified, it is understood to be the domain of the function.
\end{rem}
\newcommand{\blockStep}{h} \newcommand{\lmgf}{\Lambda} 
\newcommand{\ssp}{\mathscr{S}}
\begin{rem}\label{r_lmgf}
  Before giving the main result of this section (an estimate for the \emph{logarithmic
    moment generating functional}), as an attempt to illustrate its statement, let us
  consider the following simple example.  Let us fix $\ve > 0$; consider a
  (non-stationary) Markov chain on the state space $\ssp$ described at time $n$ by the
  transition matrix $P_n$; assume that (in an appropriate sense) $P_{n+1}$ is $\ve$-close
  to $P_n$.  Let us fix an arbitrary observable $A\in\bR^\ssp$ (which we identify with a
  column vector $A(x) = A^x$) and $x\in\ssp$ we can define the \emph{logarithmic moment
    generating functional of $A$ associated to the Markov chain with initial state $x$}
  (denoted by $\lmgf$) as follows: for any function $\sigma\in\BV([0,T],\bR)$
  \begin{align*}
    \lmgf_{x}(\sigma) = \ve\log\bE_x\left[\exp\vei\int_0^T \sigma(s)A(X_{\pint{\vei s}}) ds\right].
  \end{align*}
  where $X_n$ is a realization of the Markov chain with initial state $X_0 = x$ and
  $\bE_x$ denotes the expectation conditioned to having initial state $x$. If $\sigma$
  were a constant and $P_n = P$ for all $n$, then it would be possible to express this
  expectation as follows: let $P_{\sigma A}$ be the transition matrix \emph{twisted with
    potential $\sigma A$}, that is $[P_{\sigma A}]^{xy} = [P]^{xy}\exp(\sigma A^y)$.  Then
  \begin{align*}
    \bE_x\left[\exp\left(\sum_{n = 0}^{\pint {T\vei}-1}\sigma A(X_n)\right)\right]=
    \sum_{y\in\ssp} \left[P_{\sigma A}^{\pint{T\vei}}\right]^{xy}.
  \end{align*}
  The leading contribution to the logarithmic moment generating functional is thus given
  by the spectral radius of the matrix $P_{\sigma A}$, that is, its leading
  eigenvalue $e^{\chi_A(\sigma)}$. We then obtain
  \begin{align*}
    \lmgf_{x}(\sigma) = \pint {T\vei}\chi_A(\sigma)+\cR_x
  \end{align*}
  where $\cR$ is a remainder term that hopefully can be neglected.  If, on the other hand,
  $\sigma$ and $P$ are not constant, then, heuristically, we can choose
  $\ve \ll \blockStep \ll T$ and assume $\sigma$ and $P$ to be constant in each block of
  length $h$ in $[0,T]$.  Arguing in this way we can expect
  \begin{align*}
    \lmgf_{x}(\sigma) = \int_{0}^T \chi_{A}(\sigma(s),\pint{s\vei})ds+\cR_x
  \end{align*}
  where $e^{\chi_{A}(\sigma,n)}$ is the leading eigenvalue of the transition matrix $P_n$
  twisted with the potential $\sigma A$ and $\cR_x$ is a remainder term which remains to
  be estimated.

  The main result of this section is the proof of a formula similar to the above, for our
  deterministic system.  The Markov chain will be replaced by the fast dynamics,
  which changes in time according to the evolution of the slow variable.
\end{rem}
The first object that we need to define is the class of transfer operators associated to the
function $A\in\cC^2(\bT^2,\bR^d)$ and a parameter $\sigma\in\bR^d$. These will play the role of the $P_{\sigma A}$ in the example of Remark~\ref{r_lmgf}. For any $\cC^1$ (or
$\BV$) function $g$, define
\begin{align}\label{eq:Lopzero}
  [\tO_{\theta,\langle\sigma, A\rangle}\,g](x)
  & := \sum_{f_\theta(y)=x}\frac{e^{\langle\sigma, A(y,\theta)\rangle}}{f_\theta'(y)} g(y)
    = e^{\langle\sigma, \bar A(\theta)\rangle}[\tO_{ \theta, \langle\sigma, \hatA\rangle}\,g](x),
\end{align}
where, recall, we have defined $\bar A(\theta)=\mu_\theta(A(\cdot,\theta))$,
$\hatA=A-\bar A$ and $\mu_\theta$ denotes the unique absolutely continuous invariant
probability of $f_\theta$.  The above operators are of Perron--Frobenius type when acting
on $\cC^1$ (see Lemma~\ref{lem:quasicompact}), and the same is true for sufficiently small
$\sigma$ when acting on $\BV$ (see Remark~\ref{rem:bv-contraction}).  In other words, they
have a simple maximal eigenvalue and a spectral gap.  Let $e^{\chi_A(\sigma,\theta)}$ and
$e^{\hat\chi_A(\sigma,\theta)}$ be their maximal eigenvalues, respectively.
By~\eqref{eq:Lopzero} it follows
\begin{align}\label{e_relationChiHatChi}
  \chi_A(\sigma,\theta)
  = \langle\sigma, \bar A(\theta)\rangle+\hat\chi_A(\sigma,\theta).
\end{align}
Moreover, it is well known (see e.g.~\cite[Remark 2.5]{Baladibook}) that
\begin{align*}
  \chi_A(\sigma,\theta) = P_{\text{top}}(f_\theta,\langle \sigma,A\rangle-\log f'_\theta),
\end{align*}
where $P_\text{top}$ denotes the topological pressure.  Also the results of
Appendices~\ref{subsec:pert-paramL} and~\ref{subsec:pert-theta} imply
$\chi_A\in\cC^2(\bR^d\times\bT,\bR)$.

  Given $\sigma\in\BV([0,T],\bR^d)$, and $n\in\bN$, we introduce the notation
\begin{equation}\label{eq:signa-def}
  \sigma_n = \vei\int_{\ve n}^{\ve (n+1)}\sigma(s)ds;
\end{equation}
observe that for any $s\in[\ve n,\ve (n+1)]$ we have
$|\sigma(s)-\sigma_n|\le\|\sigma\|_{\BV([\ve n,\ve(n+1)])}$.

For any standard pair $\ell$, $\ve > 0$, $T > 0$ and $\sigma\in\BV([0,T],\bR^d)$, we now
proceed to obtain some information on the logarithmic moment generating functional
\begin{align}\label{e_logmgfDefinitionSigma}
  \logmgf(\sigma)
  &= \ve\log\mu_{\ell} \left[\exp\left({ \sum_{n=0}^{\pint{T\vei}-1}
    \langle\sigma_n,A\circ F_\ve^n\rangle}\right)\right].
\end{align}
Remark~\ref{r_lmgf} suggests that
$\Lambda_{\ell,\ve}(\sigma)\sim\int_0^T\chi_A(\sigma(s),\bar\theta(s,\thetasl))ds$; it is
therefore natural to define the quantity:
\begin{align}\label{eq:almost-equality}
  \restoPalla(\sigma) = \logmgf(\sigma) -
  \int_0^{T} \chi_A(\sigma(s),\bar\theta(s,\avgtheta{\ell})) \deh s.
\end{align}
The main result of this section is a bound on the remainder term defined above.
\begin{prop}\label{lem:exponential}
  There exists $\ve_0 > 0$, such that, for any $\ve\in (0,\ve_0)$, $L\in\left[\frac{1}{\ve_0}, \frac{\ve_0}{\sqrt\ve}\right]$ and $T\in[\ve L,\Tmax]$,
  \begin{enumerate}
    \item\label{i_nonPerturbativeBoundOnR} for any $\sigma\in\BV([0,T],\bR^d)$ we have
    \begin{align*}
      |\restoPalla(\sigma)|
      &\le \Const\big(\ve L\|\sigma\|\nBV+\ve L T  \\
      &\phantom\le+\left[L\invr+ \min\{T,\|\sigma\|\nl1+\ve L\|\sigma\|\nBV\}\right]\|\sigma\|\nl1\big);
    \end{align*}
    \item\label{i_perturbativeBoundOnR} there exists $\sigma_* = \sigma_*(\Tmax)>0$ so that, if
    $\|\sigma\|\nl\infty< \sigma_*$, then
     \begin{align*}
      |\restoPalla(\sigma)|
      \le \Const\left(\ve\|\sigma\|\nBV+ \ve LT+\ve (L+T\invr)\|\sigma\|\nl1+
      \|\sigma\|\nl1^2+ L\invr\|\sigma\|\nl2^2\right).
    \end{align*}
  \end{enumerate}
\end{prop}
The proof of the Proposition~\ref{lem:exponential} relies on the spectral properties of the transfer
operators~\eqref{eq:Lopzero}. It is then natural that our ability to bound the size of the
remainder term $\cR$ depends on the size of $\sigma$.  Without any assumption on $\sigma$
we cannot use perturbation theory of the associated transfer operators.  This allows only
a rough bound, which is stated in item~\ref{i_nonPerturbativeBoundOnR}.  On the other
hand, if $\|\sigma\|\nl\infty$ is sufficiently small, then the corresponding transfer operators
are guaranteed to be of uniform Perron--Frobenius type and can be treated using
perturbation theory.  This enables us to give the much sharper bounds stated
in~\ref{i_perturbativeBoundOnR}.

As hinted in Remark~\ref{r_lmgf}, the main (quite standard) idea of the proof is to
introduce a block decomposition: consider a partition of the set
$\{0,\cdots,\pint{T\vei}\}$ in $K$ blocks of length $L$, where $K=\pint{T\vei}L\invr$ (in
Remark~\ref{r_lmgf} we have $h \sim L\ve$).\footnote{ For simplicity of notation we ignore
that $K$ may not be an integer, as such a problem can be fixed  trivially.}

By~\eqref{eq:preparatory0}, we have that, for any $g\in L^\infty(\bT^2,\bR)$,\footnote{ In
  this section we will use~\eqref{eq:unfold-mgf} only in the case $g=1$; yet in
  Section~\ref{subsec:large-long} this more general formulation will be needed.}
\begin{align}\label{eq:unfold-mgf}
  \mu_{\ell}\left(e^{\sum_{n=0}^{\pint{T\vei}-1} \langle\sigma _n,A\circ F_\ve^n\rangle}
  g\circ F_\ve^{\pint{T\vei}}\right)%
  &=\sum_{\ell_1\in \stdf_{\ell}^L}\hskip -4pt \cdots\hskip -4pt \sum_{\ell_K\in
    \stdf_{\ell_{K-1}}^L}\prod_{i = 1}^K\fm_{\ell_i}\mu_{\ell_K}(g),
\end{align}
where, to ease the notation, we dropped the subscript potentials
$\langle\sigma_k,\hat A\circ F_\ve^k\rangle$ from the symbols for standard families.  To
further shorten notation, given a standard pair $\ell$, we use $\rho_\ell$, $G_\ell$,
$a_\ell$ and $b_\ell$ to denote the corresponding data.

Recall from Section~\ref{sec:stp} that $\rho_\ell$ is a $\cC^2$ probability density over
$[a_\ell,b_\ell]$; yet for our future purposes it is more convenient to deal with
functions that are defined on the whole $\bT^1$; to this end we introduce the extension
$\rrho_\ell$ of $\rho_\ell$ to $\bT^1$ which we indicate by the (slightly abusing)
notation $\rrho_\ell=\Id_{[a_\ell,b_\ell]}\rho_\ell$.
\begin{rem}\label{r_standardDensityBV}
  Observe that if $\rho_\ell$ is a $c_*$-standard density, then $\rrho_\ell$ is a $\BV$
  function and its $\BV$ norm is bounded by:
  \begin{align*}
    \|\rrho_\ell\|\nBV &\le \|\rrho_\ell\|\nl1+\sup_{\substack{\psi\in\cC^1(\bT,\bR)\\\|\psi\|_\infty = 1}}\left|\int\psi'(x) \rrho_\ell(x)dx\right| \le 1+2\|\rho_\ell\|\nl\infty+c_*\\
    &\leq (1+2|b_\ell-a_\ell|^{-1})e^{c_*}.
  \end{align*}
\end{rem}
The next lemma, whose proof we briefly postpone, is our basic computational
tool:  it contains an estimate of the contribution of each of
the blocks of length $L$ appearing in~\eqref{eq:unfold-mgf}.\footnote{ The $\{\sigma_0,\cdots,\sigma_{L-1}\}$ in Lemma~\ref{lem:one-step-moment}
correspond to an arbitrary block $\{\sigma_{jL},\cdots,\sigma_{(j+1)L-1}\}$ in equation~\eqref{eq:unfold-mgf}. Recall that $\sigma_j$ is defined in~\eqref{eq:signa-def}.}
\newcommand{\indBlock}{j}
\newcommand{\fixBlock}{{j^*}}
\newcommand{\osigma}{\bar\sigma_*}
\begin{lem}\label{lem:one-step-moment}
There exists $\ve_0>0$, such that, for any
  $\ve\in(0,\ve_0)$, any standard pair $\ell$, $L\in[\ve_0^{-1}, \ve_0 \veih]$,
  $\sigma\in\BV([0,\ve L],\bR^d)$ and $\auxPot\in\cC^2(\bT,\bR)$:
  \begin{enumerate}
  \item \label{i_one-step-moment-rough} the following bound holds
    \begin{align*}
      \Leb\left[\sum_{\ell'\in \stdf_{\ell}^L}\fm_{\ell'} \mr\rho_{\ell'}(\cdot)
      e^{\vei\auxPot\circ G_{\ell'}}\right] %
      &=e^{\vei\auxPot(\bar\theta(\ve L,\avgtheta\ell)) + \sum_{\indBlock=0}^{L-1}
        \chi_A(\sigma_\indBlock,\bar\theta(\ve \indBlock,\thetasl))}e^{\vei\cS(\sigma)}
    \end{align*}
    where
    \begin{align*}
      |\cS(\sigma)|
      \le &\Const\Big(
      \ve L\|\sigma\|_{\BV([0,\ve L])} + L\invr\|\sigma\|_{L^1([0,\ve L])}+\ve^2 L^2\\
      & +\ve\|\auxPot\|\nc2 \left[ 1+L\min\{1,\vei L\invr\|\sigma\|_{L^1([0,\ve L])} +
        \|\sigma\|_{\BV([0,\ve L])}+\|\auxPot\|\nc2 \}\right]\Big).
    \end{align*}
    Recall that $e^{\chi_A(\sigma,\theta)}$ denotes the maximal eigenvalue of
    $\tO_{\theta,\langle\sigma,A(\cdot,\theta)\rangle}$.
    \item\label{i_one-step-moment-sharp} There exists $\osigma > 0$ so that if
    $\|\sigma\|\nl\infty+2\|\auxPot\|\nc2\le \osigma$, then
    \begin{align*}
      &\sum_{\ell'\in \stdf_{\ell}^L} \fm_{\ell'}\rrho_{\ell'}(x)
        e^{\vei\auxPot(G_{\ell'}(x))} = \harb(x) \\
      &\phantom{\sum_{\ell'\in \stdf_{\ell}^L} \fm_{\ell'}\rrho_{\ell'}(x)}
        \cdot  \marb\left(\rrho_\ell(\cdot) e^{\ve^{-1}\auxPot(\bar \theta(\ve L,
        G_\ell(\cdot)))+\sum_{\indBlock=0}^{L-1}\chi_A(\sigma_\indBlock,\bar\theta(\ve \indBlock,G_\ell(\cdot)))
        }\right)e^{\vei\cS(\sigma,x)}
    \end{align*}
    where
    \begin{align*}
      \|\cS(\sigma,\cdot)\|\nl\infty
      \le &\Const\Big(\ve\|\sigma\|_{\BV([0,\ve L])}
            +\ve^2 L^2(1+\|\auxPot\|\nc2)+L\ve\|\sigma\|_{L^1([0,\ve L])}\\
          & +L\invr\|\sigma\|^2_{L^2([0,\ve L])}+\|\auxPot\|\nc2
            \left(\ve L\|\auxPot\|\nc2+\|\sigma\|_{L^1([0,\ve L])}\right)\Big),
    \end{align*}
    where $\harb$ is the right eigenvector of the transfer operator
    $\tO_{\thetas_1,\langle\sigmas_1,\hat A(\cdot,\thetas_1)\rangle}$ and $\marb$ is the
    left eigenvector of $\tO_{\thetas_2,\langle\sigmas_2,\hat A(\cdot,\thetas_2)\rangle}$,
    where $\thetas_1,\thetas_2$ can be chosen arbitrarily with
    $|\thetasl-\thetas_i|\le\Const\ve L$ and $\sigmas_1,\sigmas_2$ can be chosen
    arbitrarily in the \emph{essential range}\footnote{ Recall that the essential range of
      a function $\sigma$ is the ``range modulo null sets'', \ie the intersection of the
      closure of the image of all functions which agree a.e. with $\sigma$.} of $\sigma$.
  \end{enumerate}
\end{lem}
\begin{proof}[\bf{Proof of Proposition~\ref{lem:exponential}}]
  For $k\in \{0, \cdots, K\}$ and $\theta\in\bT$, $T=K\ve L$, define
  \begin{align*}
    \auxPot_{k}(\theta)=\int_{\ve k L}^{\ve K L}\chi_A(\sigma(s),\bar\theta(s-\ve kL,\theta))
    \deh{}s.
  \end{align*}
  It follows from equation~\eqref{e_estimatesecondpartchi} that
  $\auxPot_k\in C^2(\bT,\bR)$. Also ~\eqref{e_derivativeChi} implies that, for any $k$, we
  have $\|\auxPot_k\|\nc2\le \Const\|\sigma\|_{L^1([0,T])}$.  First of all notice that, by
  definition
  \begin{align*}
    \logmgf(\sigma) = \auxPot_0(\avgtheta\ell) + \restoPalla(\sigma).
  \end{align*}
  Moreover, let us define $J_k=[\ve k L,\ve(k+1) L]$ and
  recall~\eqref{e_estimatesecondpartchi0},~\eqref{eq:derivative-m} together with
  Remark~\ref{rem:variationdef}; then
  \begin{align*}
    \auxPot_{k}(\theta)-
    &\auxPot_{k+1}(\bar\theta(\ve L,\theta))=\int_{J_k}\chi_A(\sigma(s),\bar\theta(s-\ve kL,\theta))\\
    &=\sum_{j=kL}^{(k+1)L-1}\int_{\ve j}^{\ve(j+1)}\chi_A(\sigma(s),\bar\theta(\ve(j- kL),\theta))+\cO(\ve\|\sigma\|_{L^1(J_k)})\\
    &=\ve\sum_{j=0}^{L-1}\chi_A(\sigma_{j+kL},\bar\theta(\ve j,\theta))+\cO(\ve\|\sigma\|_{\BV(J_k)}),
  \end{align*}
which is the quantity appearing in Lemma~\ref{lem:one-step-moment}.

We first proceed to prove item~\ref{i_nonPerturbativeBoundOnR}: let us fix
  conventionally $\ell_0 = \ell$.  Consider~\eqref{eq:unfold-mgf} with $g = 1$ and
  isolate the last term:
  \begin{align*}
    \mu_{\ell}%
    &\left(e^{\sum_{n=0}^{KL-1}\langle\sigma_n, A\circ F_\ve^n\rangle}\right)= \sum_{\ell_1\in
      \stdf_{\ell_0}^L}\hskip -4pt \cdots\hskip -4pt\sum_{\ell_{K-1}\in
      \stdf_{\ell_{K-2}}^L}\prod_{k=1}^{K-1}\fm_{\ell_k}\Leb\left[\sum_{\ell'\in\stdf^L_{\ell_{K-1}}}\fm_{\ell'}\mathring\rho_{\ell'}\right].
  \end{align*}
  We then apply Lemma~\ref{lem:one-step-moment}-\ref{i_one-step-moment-rough} with
  $\auxPot = \auxPot_{K} = 0$ to the term in brackets and obtain:
  \begin{align*}
    \mu_{\ell}&\left(e^{\sum_{n=0}^{KL-1}\langle\sigma _n, A\circ F_\ve^n\rangle}\right)
    = \sum_{\ell_1\in \stdf_{\ell_0}^L}\hskip -4pt \cdots\hskip -4pt\sum_{\ell_{K-1}\in
      \stdf_{\ell_{K-2}}^L}
      \prod_{k=1}^{K-1}\fm_{\ell_k}\Leb(\rrho_{\ell_{K-1}}) e^{\vei\auxPot_{K-1}(\avgtheta{
      \ell_{K-1}})}e^{\vei\tilde\cS(\sigma)}
    \\&\quad= \sum_{\ell_1\in \stdf_{\ell_0}^L}\hskip -4pt \cdots\hskip -4pt
        \sum_{\ell_{K-2}\in \stdf_{\ell_{K-3}}^L}
        \prod_{k=1}^{K-2}\fm_{\ell_k}%
        \Leb \left[\sum_{\ell'\in \stdf_{\ell_{K-2}}^L}\fm_{\ell'}
        \rrho_{\ell'} e^{\vei\auxPot_{K-1}\circ G_{\ell'}}\right] e^{\vei\tilde\cS(\sigma)}
  \end{align*}
  where $\tilde\cS$ stands for an arbitrary function on $\BV([0,T])$ satisfying the bound
  \begin{align*}
    |\tilde\cS(\sigma)| &\leq \Const\left(\ve L\|\sigma\|_{\BV(J_{K-1})}+ L\invr\|\sigma\|_{L^1(J_{K-1})} +\ve^2 L^2\right).
  \end{align*}
  Also, we used the fact that $\Leb(\rrho_{\ell_{K-1}}) = 1$ and
  Lemma~\ref{lem:non-pert-theta} to change the argument in $\auxPot_{K-1}$.  We now apply
  Lemma~\ref{lem:one-step-moment}-\ref{i_one-step-moment-rough} to the term in brackets
  and iterate.  This proves item~\ref{i_nonPerturbativeBoundOnR} since, recalling that
  $K = T\vei L\invr$, we obtain
   \begin{align*}
    \mu_{\ell_0}\left(e^{\sum_{n=0}^{KL-1}\langle\sigma _n,A\circ F_\ve^n\rangle}\right)%
    &= e^{\vei\auxPot_0(\avgtheta{\ell_0})%
    + \cO\left(L\|\sigma\|\nBV+ \vei L\invr\|\sigma\|\nl1+ LT\right)}\\
    &\phantom\le\cdot e^{\|\sigma\|\nl1\cO\left(T\vei
      L\invr+\min\{T\vei,(1+T)\vei\|\sigma\|\nl1+L\|\sigma\|\nBV\}\right)}.
  \end{align*}

  To prove item~\ref{i_perturbativeBoundOnR} note that, if we assume
  $(1+\Const T)\sigma_*< \osigma$, it possible to obtain a sharper estimate using
  Lemma~\ref{lem:one-step-moment}-\ref{i_one-step-moment-sharp} and carefully keeping track
  of the error terms.  More precisely: for any $k\in\{0,\cdots,K-1\}$ define
  $m_{\ell_k} = m_{\theta^*_{\ell_k},\langle\sigma_{(k+1)L}
    ,\hatA(\cdot,\theta^*_{\ell_k})\rangle}$.  For each standard
  pair $\ell$,~\eqref{e_basicEstimatem-BV-Lebesgue} and Remark~\ref{r_standardDensityBV} yield
  \begin{align}\label{e_sillyClaim}
| m_{\ell_k}(\rrho_\ell) -\Leb(\rrho_\ell)|\leq \Const|\sigma_{(k+1)L}|\le \Const(T\invr\|\sigma\|\nl1+\|\sigma\|\nBV).
  \end{align}
Then we claim that for any $k\in\{1,\cdots,K\}$:
  \begin{align}\label{eq:iterative-step}
    \mu_\ell&\left(e^{\sum_{n=0}^{KL-1}\langle\sigma_n,A\circ
              F_\ve^n\rangle}\right)=\sum_{\ell_1\in
              \stdf_{\ell_0}^L}\!\!\cdots\!\!\sum_{\ell_k\in
              \stdf_{\ell_{k-1}}^L}\prod_{i=1}^k\fm_{\ell_i}
              m_{\ell_{k-1}}\left(\rrho_{\ell_k}
              e^{\vei\auxPot_k(G_{\ell_k}(\cdot))}\right)\cdot e^{\vei\cS_k(\sigma)}
  \end{align}
  where
  \begin{align*}
    |\cS_k(\sigma)| \leq& \Const\big(
    \ve\|\sigma\|_{\BV([\ve kL,\ve KL])}+(K-k)L^2\ve^2(1+\|\sigma\|\nl1)+L\invr\|\sigma\|^2_{L^2([\ve kL,\ve KL])}\\
    &+L\ve\|\sigma\|_{L^1([\ve kL,\ve KL])}+\ve\|\sigma\|\nl1^2L(K-k) +\|\sigma\|\nl1\|\sigma\|_{L^1([\ve kL,\ve KL])}\\
    &+\ve T\invr\|\sigma\|\nl1+\ve\|\sigma\|\nBV\big).
  \end{align*}
Let us give an inductive proof of~\eqref{eq:iterative-step}.  The base case is $k=K$:
  choosing $g = 1$ in~\eqref{eq:unfold-mgf} yields
  \begin{align*}
    \mu_\ell&\left(e^{\sum_{n=0}^{KL-1}\langle\sigma_n,A\circ F_\ve^n\rangle}\right)
              =\sum_{\ell_1\in \stdf_{\ell_0}^L}\!\!\cdots\!\!\sum_{\ell_K\in
              \stdf_{\ell_{K-1}}^L}\prod_{i=1}^K\fm_{\ell_i}
              \Leb\left(\rrho_{\ell_K}\right)
  \end{align*}
  and~\eqref{eq:iterative-step}, for $k = K$, follows by~\eqref{e_sillyClaim} (\ie
  $\|\cS_K\|\nl\infty \leq \Const(\ve T\invr\|\sigma\|\nl1+\ve\|\sigma\|\nBV)$).

  Next, we proceed by backward induction to prove~\eqref{eq:iterative-step} for $k<K$.
  Suppose that the estimate holds for $k+1\le K$, then we need to compute
  \begin{align}\label{eq:ind-step-ld}
    \sum_{\ell_{k+1}\in \stdf_{\ell_{k}}^L}&\fm_{\ell_{k+1}}m_{\ell_{k}}\left(\rrho_{\ell_{k+1}}e^{\ve^{-1}\auxPot_{k+1}(G_{\ell_{k+1}}(\cdot))}\right) = \\
    & =m_{\ell_{k}}\left(\sum_{\ell_{k+1}\in \stdf_{\ell_{k}}^L}\fm_{\ell_{k+1}}\rrho_{\ell_{k+1}}e^{\ve^{-1}\auxPot_{k+1}(G_{\ell_{k+1}}(\cdot))}\right).\notag
  \end{align}
  Apply Lemma~\ref{lem:one-step-moment}-\ref{i_one-step-moment-sharp} with
  $\auxPot = \auxPot_{k+1}$, $\marb = m_{\ell_{k-1}}$ and
  $\harb = h_{\avgtheta{\ell_{k}},\langle\sigma_{(k+1)L},\hat
    A(\cdot,\avgtheta{\ell_{k}})\rangle}$.  Since by design $m_{\ell_k}(\harb) = 1$, we
  obtain~\eqref{eq:iterative-step} at step $k$, which concludes the proof
  of~\eqref{eq:iterative-step} for any $k\in\{1,\cdots,K\}$.

  In particular, choosing $k = 1$ we have:
  \begin{align}\label{eq:last-iterative-step}
    \mu_\ell\left(e^{\sum_{n=0}^{KL-1}\langle\sigma_n,A\circ F_\ve^n\rangle}\right)
    = m_{\ell_{0}}\left(\sum_{\ell_1\in \stdf_{\ell_0}^L}\fm_{\ell_1}
      \rrho_{\ell_1} e^{\vei\auxPot_1(G_{\ell_1}(\cdot))}\right)\cdot e^{\vei\cS_1(\sigma)}.
  \end{align}
  We now apply once again Lemma~\ref{lem:one-step-moment}-\ref{i_one-step-moment-sharp}
  with $\auxPot = \auxPot_1$, $\marb = m_{\ell_0}$ and
  $\harb = h_{\avgtheta{\ellz},\langle\sigma_{L},\hat A(\cdot,\avgtheta{\ell_{0}})\rangle}$.
  We conclude that
  \begin{align*}
    \mu_\ell\left(e^{\sum_{n=0}^{KL-1}\langle\sigma_n,A\circ F_\ve^n\rangle}\right)
    = m_\ellz(\index{$\rrho$}\rrho_\ellz(\cdot)e^{\vei\auxPot_0(G_\ellz(\cdot))})e^{\vei\cS_0(\sigma)}.
  \end{align*}
  Since $\|\auxPot_0\|\nc2 = \Const\|\sigma\|\nl1$ and $\ellz$ is a standard pair, we
  conclude that $\|\auxPot_0(G_\ell(\cdot))-\auxPot_0(\thetasl)\| < \ve\|\sigma\|\nl1$;
  hence, using once again~\eqref{e_sillyClaim} to estimate $m_\ellz(\rrho_\ellz)$ we
  obtain:
  \begin{align*}
    \mu_\ell\left(e^{\sum_{n=0}^{KL-1}\langle\sigma_n,A\circ F_\ve^n\rangle}\right) =
    e^{\vei (\auxPot_0(\thetasl)+\cS_0(\sigma))}
  \end{align*}
  which concludes the proof of item~\ref{i_perturbativeBoundOnR}.
\end{proof}
\begin{proof}[{\bf Proof of Lemma~\ref{lem:one-step-moment}}]
  \newcommand{\lhs}{\sum_{\ell'\in \stdf_{\ell}^L}\fm_{\ell'} \mu_{\ell'}\left(\testfunc\,
      e^{\vei\auxPot\circ G_{\ell'}}\right)} For any standard pair $\ell = (G,\rho)$
  supported on $[a,b]$, recall that we consider $x_\indBlock$ and $\theta_\indBlock$ to be
  random variables on $\ell$.  Let $\testfunc\in \cC^0(\bT^1,\bRp)$ be an arbitrary
  non-negative test function; using~\eqref{eq:preparatory0}, we can write:
  \begin{align}%
    \notag
    \lhs
    &= \mu_\ell\left(\testfunc(x_L) e^{\vei\auxPot(\theta_L)}e^{\sum_{\indBlock=0}^{L-1}\langle\sigma _\indBlock,A(x_\indBlock,\theta_\indBlock)\rangle}\right)%
    \\
    \label{eq:markov-1}
    &=\int_{a}^{b} \testfunc(x_L(x)) \rho(x)  e^{\vei\auxPot(\theta_L(x))
      + \sum_{\indBlock=0}^{L-1}\langle\sigma_\indBlock,A(x_\indBlock(x),\theta_\indBlock(x))\rangle} dx.
  \end{align}
  First of all, observe that, if $\theta_0$ is distributed according to $\ell$, then
  \begin{equation}\label{eq:exchange-t}
    |\auxPot(\theta_0)-\auxPot(\avgtheta{\ell})|\le \|\auxPot\|\nc1 \ve.
  \end{equation}
  Next, let us define the random variable
  $\bar\theta_\indBlock = \bar\theta(\ve\indBlock,\theta_0)$; and recall the notation
  $\bar\theta_{\ell,\indBlock}^* = \bar\theta(\ve\indBlock,\thetasl)$.  Observe that
  \begin{align*}
    \left\|\bar\theta_\indBlock-\bar\theta^*_{\ell,\indBlock}\right\|\nc0\le \Const\ve
  \end{align*}
Also, by Lemma~\ref{l_boundEtaRef} (more precisely~\eqref{e_boundEtaRefb}) we have, for any
  $\indBlock\in\{0,\cdots, L-1\}$,
   \begin{equation}\label{eq:theta-dev}
     \|\theta_\indBlock-\bar\theta_\indBlock-\ve H_{\indBlock}\| \le \Const \indBlock^3\ve^3.
   \end{equation}
   Hence, we conclude that (recall the definition of $\etaExps$ given
   in~\eqref{e_defEtaExp}:
   \begin{align*}
     \vei\auxPot(\theta_L(x))
     &=\vei\auxPot(\bar\theta_L(x)) + \vei\auxPot'(\bar\theta_L(x))\cdot (\theta_L-\bar\theta_L)+\cO(\|\auxPot\|\nc2\ve L^2)\\
     &=\vei\auxPot(\bar\theta_L(x)) + \auxPot'(\bar\theta_{L})\cdot
       \sum_{\indBlock=0}^{L-1}\etaExp_{\indBlock,L}\hat\omega(x_\indBlock,\theta_\indBlock) +\cO(\|\auxPot\|\nc2\ve L^2).\\
     &=\vei\auxPot(\bar\theta_L(x)) + \auxPot'(\bar\theta^*_{\ell,L})\cdot
       \sum_{\indBlock=0}^{L-1}\etaExps_{\indBlock,L}\hat\omega(x_\indBlock,\theta_\indBlock) +\cO(\|\auxPot\|\nc2\ve L^2).
   \end{align*}
   We now proceed to incorporate the first term 
   of the above expression in the density; the 
   second term will be incorporated as a potential and the third term is small enough to
   be considered as an error term.  Let us introduce the notation
  \begin{align*}
    \Gamma_{\auxPot,\indBlock}=\auxPot'(\bar\theta^*_{\ell,L})(\etaExps_{\indBlock,L}, 0)\in\bR^d
  \end{align*}
  and let:
  \begin{align*}
    \rho_\auxPot(x) &= \rho(x) e^{\vei \left[\auxPot(\bar\theta_L(x))
                      -\auxPots\right]}& \text{where }
     \auxPots &=\ve\log\left[ \int_a^b\rho(x) e^{\vei
                \auxPot(\bar\theta_L(x))
                }\right].
  \end{align*}
  Observe that $\rho_\auxPot$ is a $(\spc2+\Const \|\auxPot\|\nc2)$-standard probability
  density and that
  \begin{equation}\label{eq:pot0-change}
  |\auxPots-\auxPot(\bar\theta^*_{\ell,L})| \leq \Const \|\auxPot\|\nc2\ve.
  \end{equation}
   We can then rewrite~\eqref{eq:markov-1} as
  \begin{align}\label{eq:preparatory1.5}
    \notag\lhs =&e^{\vei\auxPots+ \cO(\|\auxPot\|\nc2\ve L^2)}\\
 &\times\int_a^b\rho_\auxPot(x) \testfunc(x_L)
   e^{\sum_{\indBlock=0}^{L-1}\left[\langle\sigma_\indBlock,A(x_\indBlock,\theta_\indBlock)\rangle+\langle\Gamma_{\auxPot,\indBlock},\hat
   A(x_\indBlock,\theta_\indBlock)\rangle\right]}dx.
  \end{align}
  It is then convenient to defined
  \begin{align}\label{e_definitionOmega}
    \Omg_\indBlock(x,\theta) &= \langle\sigma_\indBlock ,A (x,\theta)\rangle+\langle\Gamma_{\auxPot,\indBlock},\hat  A(x,\theta)\rangle\\
                             &= \langle\sigma_\indBlock ,\bar A(\theta)\rangle+\langle\sigma_\indBlock + \Gamma_{\auxPot,\indBlock},\hat  A(x,\theta)\rangle\notag.
  \end{align}
To estimate the integral in~\eqref{eq:preparatory1.5} we use Lemma~\ref{lem:shadow} and write,
  using the notations introduced there,\footnote{ Choosing $n=L$ and setting $Y=Y_L$.} for
  some $\theta^*$, $|\thetas-\thetasl| < \Const \ve L$, to be chosen later:\footnote{ To ease notation, for the duration of the proof $L^1$ will denote
    $L^1([0,\ve L])$ (and similarly for $\BV$, $L^\infty$ and $L^2$) unless a different
    domain is explicitly written. }
  \begin{equation}    \label{eq:preparatory2}
  \begin{split}
&\phantom{=`} \int_a^b\rho_\auxPot(x) \testfunc(x_L)e^{ \sum_{\indBlock=0}^{L-1}\left[\langle\sigma_\indBlock ,
      A(x_\indBlock,\theta_\indBlock)\rangle+\langle\Gamma_{\auxPot,\indBlock},\hat A(x_\indBlock,\theta_\indBlock)\rangle\right]}\deh x\\
    & =\int_a^b\rho_\auxPot(x) \testfunc(x_L)e^{\sum_{\indBlock=0}^{L-1} \Omg_\indBlock(x^*_\indBlock,\thetas)
      + \cO(L\|\auxPot\|\nc2+\vei\|\sigma\|\nl1 )\ve L}\deh x\\
    &=\int_{a^*}^{b^*}\frac{\rho_\auxPot\circ Y^{-1}(x)}{Y'\circ
      Y^{-1}(x)}\testfunc\circ f^L_*(x)e^{\sum_{\indBlock=0}^{L-1}\Omg_\indBlock(f_*^j(x),\thetas)+ \cO(\ve L^2\|\auxPot\|\nc2+
      L\|\sigma\|\nl1)}\deh x.
  \end{split}
  \end{equation}
  Next, we let $\smg_\indBlock = \sigma_\indBlock+\Gamma_{\auxPot,\indBlock}$ and write
  $\Omg_\indBlock(x,\theta) = \bOmega_\indBlock(\theta)+\hOmg_\indBlock(x,\theta)$,
  where
  \begin{align}\label{e_definitionBarOmega}
    \bOmega_\indBlock(\theta) &= \langle\sigma_j,\bar A(\theta)\rangle
    &\hOmg_\indBlock(x,\theta) &=  \langle\smg_j,\hat A(x,\theta)\rangle.
  \end{align}
  It is then natural to introduce the $\BV$-function
  \begin{align*}
    \smg(s) = \sigma(s)+\Gamma_{\auxPot,\pint{s\vei}}
  \end{align*}
  so that $\smg_\indBlock = \vei\int_{\ve\indBlock}^{\ve(\indBlock+1)}\smg(s)ds$.
  Observe that the definition of $\Gamma_{\auxPot,\indBlock}$ and our upper bound on $L$
  imply, if $\ve$ is sufficiently small:
  \begin{subequations}\label{e_boundsFor-smg}
    \begin{align}
      \|\smg-\sigma\|\nl\infty &\le \|\auxPot\|\nc2(1+\ve L)\le 2\|\auxPot\|\nc2 \\
      \|\smg\|_{\BV([\ve \indBlock,\ve\indBlock'])}
                               &\le \|\sigma\|_{\BV([\ve \indBlock,\ve\indBlock'])}+\Const\|\auxPot\|\nc2\ve|j'-j|\\
      \|\smg\|_{L^1([\ve \indBlock,\ve\indBlock'])}
                               &\le \|\sigma\|_{L^1([\ve\indBlock,\ve\indBlock'])}+\Const\|\auxPot\|\nc2\ve|j'-j|.
    \end{align}
  \end{subequations}
  Combining~\eqref{eq:preparatory2} and~\eqref{eq:preparatory1.5} and using the above
  definitions we can thus write:
  \begin{align}\notag
    \notag\lhs &= e^{\vei\auxPots+\sum_{\indBlock = 0}^{L-1}\bOmega_\indBlock(\bar\theta^*_{\ell,j})}\notag\\
    &\phantom{=}\cdot\int_{a^*}^{b^*}\frac{\rho_\auxPot\circ Y^{-1}(x)}{Y'\circ  Y^{-1}(x)} \testfunc(f^L(x,\thetas))
    e^{\sum_{\indBlock=0}^{L-1}\hOmg_\indBlock(f^\indBlock(x,\thetas),\thetas)}\deh x\label{eq:preparatory2.5}\\
               &\phantom = \cdot e^{ \cO\left(\ve L^2\|\auxPot\|\nc2+L\|\sigma\|\nl1\right)},\notag
  \end{align}
  where, in the above estimate, we also used:
  \begin{align*}
    \sum_{\indBlock = 0}^{L-1}\bOmega_\indBlock(\thetas) = \sum_{\indBlock =
    0}^{L-1}\bOmega_\indBlock(\bar\theta^*_{\ell,j}) +\cO(L\|\sigma\|\nl1).
  \end{align*}
  The problem with expression~\eqref{eq:preparatory2.5} is that $Y'$ has a very large derivative (see
  footnote~\ref{f_hugeYDerivative}) and hence it cannot be effectively treated as a $\BV$
  function.  In Section~\ref{sec:one-small} we will deal with this problem in a more
  sophisticated way; here it suffices the following rough estimate based, again, on Lemma~\ref{lem:shadow}:
  \begin{align}\label{eq:ur-E}
    \frac 1{Y'}&= e^{\cO(\ve L)}\prod_{\indBlock=0}^{L-1}
                             \frac{\partial_xf(x^*_\indBlock,\thetas)}{\partial_x f(x_\indBlock,\theta_\indBlock)}\notag\\
                           &= e^{\cO(\ve L)+\sum_{\indBlock=0}^{L-1}\left[ \log \partial_x
                             f(x^*_\indBlock,\thetas)-\log \partial_x f (x_\indBlock,
                             \theta_\indBlock)\right]}=e^{\cO\nl\infty(\ve L^2)}.
  \end{align}
  Also note that, setting
  \begin{align}\label{e_definitionTildeRho}
    \tilde\rho_\auxPot &= \frac{\rho_\auxPot\circ Y^{-1}}{\rho_\auxPot^*}\Id_{[a^*,b^*]} =
                      \frac{\rho_\auxPot\,\Id_{[a,b]}}{\rho_\auxPot^*}\circ Y\invr  &\text{where }\rho_\auxPot^* &=
                                                                                                          \int_{a^*}^{b^*}\rho_\auxPot\circ Y\invr,
  \end{align}
  we have that $\tilde\rho_\auxPot$ is a $\Const(\spc2+\|\auxPot\|\nc2)$-standard probability
  density and $\|\tilde\rho_\auxPot\|\nBV\le \Const\|\rho_\auxPot\|\nBV$.  Observe moreover
  that~\eqref{eq:ur-E} implies that $\rho_\auxPot^* = e^{\cO(\ve L^2)}$.  Collecting the above estimate together with
 ~\eqref{eq:ur-E} and~\eqref{eq:preparatory2.5}  we have
  \begin{align}\label{eq:preparatory3}
    \notag\lhs =& e^{\vei\auxPots+\sum_{\indBlock = 0}^{L-1}\bOmega_\indBlock(\bar\theta^*_{\ell,j})}\\
    &\cdot \int_{\bT}\tilde\rho_\auxPot(x)
                 \testfunc(f^L(x,\thetas))e^{\sum_{\indBlock=0}^{L-1}\hOmg_\indBlock(f^\indBlock(x,\thetas),\thetas)}\deh x\\
               & \cdot e^{ \cO\left(\ve L^2(1+\|\auxPot\|\nc2)+L\|\sigma\|\nl1\right)}.\notag
  \end{align}
Such integrals can be computed by introducing the weighted transfer operators
  \begin{align*}
    [\tO_{\theta,\hOmg_j}g](x)=\sum_{f_\theta(y)=x}\frac{e^{\hOmg_j(y,\theta)}}{f_\theta'(y)}  g(y),
  \end{align*}
which allow to rewrite the integral in~\eqref{eq:preparatory3} as
  \begin{equation}\label{eq:preparatory4}
  \begin{split}
\int_{\bT}\tilde\rho_\auxPot(x) &\testfunc(f^L(x,\thetas))\exp\left[\sum_{\indBlock=0}^{L-1} \hOmg_\indBlock(f^\indBlock(x,\thetas),\thetas)\right]\deh x\\
&= \Leb\left( \testfunc\tO_{\thetas,\hOmg_{L-1}} \cdots \tO_{\thetas,\hOmg_0}[\tilde\rho_\auxPot]\right).
 \end{split}
  \end{equation}
  Such a quantity can be computed in terms of
  $\hat\chi_{\theta,\Omg_j} = \chi_{\theta,\hOmg_j}$, the logarithm of the maximal
  eigenvalue of $\tO_{\theta,\hOmg_j}$ when acting on $\cC^1$.  Observe that by
  definition, remembering~\eqref{e_relationChiHatChi},~\eqref{e_definitionBarOmega}, and
  by~\eqref{eq:pvarrhochi}, Lemma~\ref{lem:quasicompact} we can write\footnote{ Also
    recall the normalization
    $m_{\theta, \langle\sigma_\indBlock+s\Gamma_{\auxPot,\indBlock},\hat A\rangle}(
    h_{\theta, \langle\sigma_\indBlock+s\Gamma_{\auxPot,\indBlock},\hat A\rangle})=1$.}
 \[
 \begin{split}
    \hat\chi_{\theta,\Omg_\indBlock}&= \hat\chi_A(\smg_\indBlock,\theta)
    = \hat\chi_A(\sigma_\indBlock,\theta)
    +\int_0^1m_{\theta, \langle\sigma_\indBlock+s\Gamma_{\auxPot,\indBlock},\hat A\rangle}(\langle\Gamma_{\auxPot,\indBlock},\hat A\rangle h_{\theta, \langle\sigma_\indBlock+s\Gamma_{\auxPot,\indBlock},\hat A\rangle})ds\\
    &=\hat\chi_A(\sigma_\indBlock,\theta)+\cO(\|\auxPot\|\nc2).
\end{split}
\]
Also, by  Lemma~\ref{l_derivativess} and since
  $m_{\theta,0}(\hat A(\cdot,\theta)h_{\theta,0})=0$, we have, for $\|\smg_j\|$ small,
\[
m_{\theta, \langle\sigma_\indBlock+s\Gamma_{\auxPot,\indBlock},\hat A\rangle}(\langle\Gamma_{\auxPot,\indBlock},\hat A\rangle h_{\theta, \langle\sigma_\indBlock+s\Gamma_{\auxPot,\indBlock},\hat A\rangle})=\cO(\|\smg_j\|\|\auxPot\|\nc2).
\]
Collecting the above facts, yields
  \begin{align}\label{eq:chi-g}
    \hat\chi_{\theta,\Omg_\indBlock}%
    &= \hat\chi_A(\sigma_\indBlock,\theta)
    +\cO(\min\{1,\|\smg_j\|\}\|\auxPot\|\nc2).
  \end{align}


  \begin{irem}
  We will now adopt the following strategy: we first obtain a rather crude bound for~\eqref{eq:preparatory4}
  (see~\eqref{e_roughPreliminaryEstimate}): this bound will be valid for arbitrary
  $\sigma$ and $\auxPot$.  We then proceed to obtain a sharper bound, which is however
  valid only for $\sigma$ with a relatively small $L^\infty$ norm; the sharper bound
  will enable us to improve the previously found rough bound to obtain item~\ref{i_one-step-moment-rough}
  and to prove~\ref{i_one-step-moment-sharp}.
\end{irem}
  We obtain the rough bound by replacing the potential $\hOmg_{\indBlock}$ for
  $j\in\{0,\cdots,L-1\}$ with a fixed $\fixhOmega = \hOmg_\fixBlock$ for some
  $\fixBlock\in\{0,\cdots,L-1\}$ chosen arbitrarily.  Notice in fact that, for
  $g\ge 0$ and any $\indBlock\in\{0,\cdots,L-1\}$:
  \begin{align}\label{e_approxBigBlock}
    \tO_{\thetas,\hOmg_{\indBlock}} g %
    &= e^{\cO\nl\infty(\|\smg_\indBlock-\smg_\fixBlock\|)}\tO_{\thetas,\fixhOmega}g
      = e^{\cO\nl\infty(\|\smg\|\nBV)}\tO_{\thetas,\fixhOmega}g,
  \end{align}
  whence:
  \begin{align}\label{e_roughPreparatoryPreliminary}
    \tO_{\thetas,\hOmg_L}\cdots\tO_{\thetas,\hOmg_0}[\tilde\rho_\auxPot]%
    &=e^{\cO\nl\infty(L\|\smg\|_{BV([0,\ve L])})}\tO_{\thetas,\fixhOmega}^L[\tilde\rho_\auxPot].
  \end{align}
  Since $\tilde\rho_\auxPot$ is a $\BV$ function and $\sigma$ can be arbitrarily large, we
  cannot guarantee that $\tO_{\thetas,\fixhOmega}$ is of Perron--Frobenius type (see
  Remark~\ref{rem:bv-contraction}).  We thus proceed as follows: recall that
  $\tilde\rho_\auxPot$ is supported on an interval $[a^*,b^*]$ of size at least
  $\delta/4$; since $f(\cdot,\thetas)$ is uniformly expanding there exists
  $q_0 \sim \log \delta = \cO(1)$ so that $f([a^*,b^*],\thetas)\supset \bT^1$; by
  definition of $\tO_{\thetas,\fixhOmega}$:
  \begin{align*}
    e^{-\Const(1+\|\auxPot\|\nc2+q_0\|\fixhOmega\|\nc0)} \le \tO_{\thetas,\fixhOmega}^{q_0} \tilde\rho_\auxPot\le e^{\Const(1+\|\auxPot\|\nc2+q_0\|\fixhOmega\|\nc0)}.
  \end{align*}
  By positivity of the transfer operator, and since it is of Perron--Frobenius type when
  acting on $\cC^1$ densities (here we want to apply it to the constant functions), we can apply $\tO_{\thetas,\fixhOmega}^{L-q_0}$ to the previous inequalities and, by~\eqref{eq:hilbert-contraction}, obtain
  \begin{align*}
    \tO_{\thetas,\fixhOmega}^L[\tilde\rho_\auxPot] =
    e^{L\,{\hat\chi_{\thetas,\fixOmega}}+\cO(1+\|\auxPot\|\nc2+\|\fixhOmega\|\nc1)} \bar h_{\thetas,\fixhOmega},
  \end{align*}
  where $\bar h_{\thetas,\fixhOmega}$ is the eigenfunction associated to the maximal eigenvalue and normalized so that $\Leb(\bar h_{\thetas,\fixhOmega})=1$.
  Thus, using~\eqref{e_roughPreparatoryPreliminary},
  \begin{align}\label{e_roughPreliminaryEstimate}
    \Leb\tO_{\thetas,\hOmg_L}\cdots\tO_{\thetas,\hOmg_0}[\tilde\rho_\auxPot]
    &= e^{L\,{\hat\chi_{\thetas,\fixOmega}}+\cO(L\|\smg\|\nBV+\|\sgm\|\nl\infty+1+\|\auxPot\|\nc2)},
  \end{align}
  where we used that  $\|\fixhOmega\|\nc1\le\|\sgm\|\nl\infty$.  This is our announced
  preliminary rough bound, which holds for any $\sigma$ and $\auxPot$.

  In order to obtain a sharper bound we need to subdivide $\{0,\cdots,L-1\}$ into
  smaller sub-blocks and replace $\hOmg_{j}$ on each sub-block with a potential that
  is constant on the corresponding sub-block.

  Let us now assume $\|\sigma\|\nl\infty+2\|\auxPot\|\nc2<\osigma$ (hence $\|\smg\|_{L^\infty}<\osigma$) for some fixed
  $\osigma\le\sigma_2$ (from Lemma~\ref{lem:m-estimate-0}) sufficiently small to be
  chosen shortly.

  Observe that, by definition~\eqref{e_definitionBarOmega}, we have that
  $\|\hOmg_\indBlock\| < \Const\osigma$ for any $j\in\{0,\cdots,L-1\}$ and thus each
  $\tO_{\theta,\hOmg_\indBlock}$ is a perturbation of the Perron--Frobenius operator
  $\tO_{\theta,0}$.

  Lemma~\ref{lem:quasicompact} implies that can fix $\sstep \in\bN$ such that $q = \cO(1)$
  and $\tO_{\theta,0}^{\sstep}=\tOpb_{\theta}+\tOqb_{\theta}$ where $\tOpb_\theta$ is a
  projector, $\tOpb_{\theta}\tOqb_{\theta}=\tOqb_{\theta}\tOpb_{\theta}=0$ and
  $\|\tOqb_{\theta}\|\nc1\le \frac 14$.  %

  As announced, we now partition $\{0,\cdots,L-1\}$ in $L' = L\sstep\invr$
  sub-blocks\footnote{ Once again we ignore the issue that $L'$ may not be an integer.}
  of length $\sstep$.  Let us fix arbitrarily $\fixsstep\in\{0,\cdots,\sstep-1\}$; for
  any $l\in\{1,\cdots,L'\}$ define $\sbOm_l = \hOmg_{(l-1)\sstep+\fixsstep}$ and
  let $\tOb_{\theta,l} = \tO_{\theta,\sbOm_l}^\sstep$.  Then in each sub-block, for any
  $g\ge0$, similarly to~\eqref{e_approxBigBlock}:
  \begin{align}\label{e_preparatory8}
    \tO_{\thetas,\hOmg_{l\sstep-1}}\cdots\tO_{\thetas,\hOmg_{(l-1)\sstep}}[g] =
    e^{\cO\nl\infty(\|\sgm\|_{\BV([\ve(l-1)\sstep,\ve l\sstep-1])})}\tOb_{\theta,l}[g]
  \end{align}

  By Lemma~\ref{lem:quasicompact}, each $\tOb_{\theta,l}$ has a simple maximal
  eigenvalue, which we denote $e^{\chib_{\theta,l}}$.  Observe that by definition
  $\chib_{\theta,l} ={\sstep\chi_{\theta,\sbOm_l}}$.  Moreover, we can write
  $\tOb_{\theta,l}=e^{\chib_{\theta,l}} \tOpb{}_{\theta,l}+\tOqb_{\theta,l}$ where
  $\tOpb{}_{\theta,l}^2=\tOpb{}_{\theta,l}$,
  $\tOpb{}_{\theta,l}\tOqb{}_{\theta,l}=\tOqb{}_{\theta,l}\tOpb{}_{\theta,l}=0$ and the
  theory of Section~\ref{subsec:pert-paramL} implies that
  $\|\tOqb{}_{\theta,l}\|\nc1\le \frac 12e^{\chib_{\theta,l}}$, provided that $\osigma$
  is sufficiently small and $Q$ has been chosen large enough.  Let us write
  $\tOpb{}_{\theta,l}g=h_{\theta,l}m_{\theta,l}(g)$, normalized as in
  Lemma~\ref{l_normalization}.

  The main advantage in defining the iterated operators $\tOb_{\theta,l}$ is that the
  bound on the \emph{norm} of $\tOqb$ (as opposed to the bound on the mere spectral
  radius which is available for the operator $\tOq$ acting on a single iterate) makes
  them well behaved under composition.


  \begin{sublem}\label{lem:itera-trivial}
    Using the above notation, if $\osigma>0$ is sufficiently small and $\|\sgm\|\nl\infty < \osigma$:
    \begin{align*}
      \tOb_{\theta,L'} \cdots \tOb_{\theta,1} [\tilde\rho_\auxPot]=
      h_{\theta,L'}(x)m_{\theta,1}(\tilde\rho_\auxPot)e^{\sum_{l=
      1}^{L'}\chib_{\theta,l}} e^{\cO\nl\infty(\|\sgm\|\nBV)}.
    \end{align*}
  \end{sublem}

  The above sub-lemma, whose proof is postponed after the end of the current proof, allows to refine the rough estimate~\eqref{e_roughPreliminaryEstimate}.  Observe that,
  using~\eqref{e_estimatePartialh} and~\eqref{e_basicEstimatem-BV-Lebesgue}:
  \begin{align*}
    \Leb(h_{\theta,L'})
    & = e^{\cO(\|\sgm\|\nl\infty)}\\
    |m_{\theta,1}\tilde\rho_\auxPot|
    & = \Leb\,\tilde\rho_\auxPot+\cO(\|\sgm\|\nl\infty(1+\|\auxPot\|\nc2))
      = e^{\cO(\|\sgm\|\nl\infty(1+\|\auxPot\|\nc2))};
  \end{align*}
  moreover, by~\eqref{e_derivativeChi}:
  \begin{align*}
    \sum_{l = 1}^{L'}\chib_{\theta,l} =
    L\hat\chi_{\theta,\fixOmega}+\cO(L\|\sgm\|\nBV).
  \end{align*}
  We thus conclude that if $\|\sgm\|\nl\infty < \osigma$, using~\eqref{e_preparatory8}:
  \begin{align*}
    \Leb\tO_{\thetas,\hOmg_L} \cdots \tO_{\thetas,\hOmg_0}[\tilde\rho_\auxPot] =
    e^{L\hat\chi_{\thetas,\fixOmega} + \cO(L\|\sgm\|\nBV + \|\sgm\|\nl\infty(1+\|\auxPot\|\nc2))}.
  \end{align*}
  Combining the above equation with~\eqref{e_roughPreliminaryEstimate}, we conclude that
  for arbitrary $\sgm$:
    \begin{align*}
    \Leb\tO_{\thetas,\hOmg_L} \cdots \tO_{\thetas,\hOmg_0}[\tilde\rho_\auxPot]
    &= e^{L\hat\chi_{\thetas,\fixOmega}
      + \cO(L\|\sgm\|\nBV + \|\sgm\|\nl\infty + \min\{1,\|\sgm\|\nl\infty\}\|\auxPot\|\nc2)}.
  \end{align*}
  Applying~\eqref{eq:chi-g} we thus obtain
  \begin{align*}
    \Leb\tO_{\thetas,\hOmg_L} \cdots \tO_{\thetas,\hOmg_0}[\tilde\rho_\auxPot]
    = e^{L\hat\chi_A(\sigma_\fixBlock,\thetas) +
    \cO(L\min\{1,\|\sgm\|\nl\infty\}\|\auxPot\|\nc2+L\|\sgm\|\nBV + \|\sgm\|\nl\infty)}.
  \end{align*}
  At last, setting $\testfunc = 1$ and substituting the latter equation
  in~\eqref{eq:preparatory4} and~\eqref{eq:preparatory3},
  \begin{align*}
    &\left[\sum_{\ell'\in\stdf_\ell^L}\fm_{\ell'}\mu_{\ell'}\left(e^{\vei\auxPot\circ G_{\ell'}}\right)\right] e^{-\vei\auxPots-\sum_{\indBlock = 0}^{L-1}\bOmega_\indBlock(\bar\theta^*_{\ell,j})}
      = e^{L\hat\chi_A(\sigma_\fixBlock,\thetas)}\\
    &\cdot e^{\cO(L\min\{1,\|\sgm\|\nl\infty\}\|\auxPot\|\nc2+L\|\sgm\|\nBV +
      \|\sgm\|\nl\infty +\ve L^2( 1+\|\auxPot\|\nc2)+L\|\sigma\|\nl1)}.
  \end{align*}
  Choosing $\thetas = \bar\theta^*_{\ell,\fixBlock}$ and taking the geometric mean of the above
  expressions for $\fixBlock \in\{0,\cdots,L-1\}$, we conclude
  \begin{align*}
    &\sum_{\ell'\in\stdf_\ell^L}\fm_{\ell'}\mu_{\ell'}\left(e^{\vei\auxPot\circ
      G_{\ell'}}\right)
      =  e^{\vei\auxPots+\sum_{\indBlock = 0}^{L-1}\chi_A(\sigma_j,\bar\theta^*_{\ell,j})}\\
    &\cdot e^{\cO(L\min\{1,\|\sgm\|\nl\infty\}\|\auxPot\|\nc2+L\|\sgm\|\nBV +
      \|\sgm\|\nl\infty + \ve L^2(1+\|\auxPot\|\nc2)+L\|\sigma\|\nl1)}.
  \end{align*}
  Item~\ref{i_one-step-moment-rough} then follows using~\eqref{e_boundsFor-smg},~\eqref{eq:pot0-change} and $\|\sigma\|_{L^\infty}\leq (\ve L)^{-1}\|\sigma\|_{L^1}+\|\sigma\|_{\BV}$.

  We now proceed to the proof of item~\ref{i_one-step-moment-sharp}, which follows from a more
  careful application of Sub-Lemma~\ref{lem:itera-trivial}.

  Recall that, by definition, $h_{\theta,L'} = h_{\theta,\fixhOmega_{L'}}$, where
  $\fixhOmega_{L'} = \langle\sgm_{(L'-1)\sstep+\fixsstep}, \hat A(\cdot,\theta)\rangle$.
  Then, notice that for any $\sigma^*_1$ in the essential range of $\sigma$ we have,
  using bounds~\eqref{e_boundsFor-smg}, that
  $\sigma_1^*-\sgm_{(L'-1)\sstep+\fixsstep}\le\|\sigma\|_{\BV([0,\ve L])}+\|\auxPot\|\nc2
  \ve L^2$.  By~\eqref{e_estimatePartialh} and~\eqref{e_estimatePartialThetah} we thus
  conclude that for any $\thetas_1$ so that $|\thetas_1-\theta| < \Const\ve L:$
  \begin{equation}\label{eq:hbar-chage}
    \harb = h_{\thetas_1,\langle\sigma^*_1,\hat
    A(\cdot,\thetas_1)\rangle} = h_{\theta,L'}e^{\cO\nl\infty(\|\sigma\|_{\BV([0,\ve
    L])}+\|\auxPot\|\nc2\ve L^2 +\ve L)}.
  \end{equation}
  Likewise, for any $\sigma^*_2$ in the essential range of $\sigma$,
  using~\eqref{e_estimatePartialm-BV} we gather,
  \begin{equation}\label{eq:marb-chage}
    \marb(\tilde\rho_\auxPot) = m_{\theta,\langle\sigma^*_2,\hat
    A(\cdot,\theta)\rangle}(\tilde\rho_\auxPot) =
    m_{\theta,1}(\tilde\rho_\auxPot)e^{\cO(\|\sigma\|_{\BV([0,\ve L])}+\|\auxPot\|\ve L^2)}.
  \end{equation}
Also, by~\eqref{e_derivativeChi} and using~\eqref{eq:chi-g} and~\eqref{e_boundsFor-smg} we obtain
  \begin{align*}
    \sum_{l = 1}^{L'}\chib_{\theta,l} %
    &= \sstep\sum_{l = 1}^{L'}\hat\chi_{\theta,\sbOm_l} %
      = \sum_{\indBlock = 0}^{L-1}\hat\chi_{\theta,\Omg_\indBlock}+\cO(\|\sgm\|_{\BV([0,\ve L])})\\
    &= \sum_{\indBlock = 0}^{L-1}\hat\chi_A(\sigma_\indBlock,\theta)+\cO(\|\sigma\|_{\BV([0,\ve
      L])}+L\|\auxPot\|\nc2(\ve +\vei L^{-1}\|\sigma\|_{L^1}+\|\auxPot\|\nc2))
  \end{align*}
and using~\eqref{e_estimatesecondpartchi0}:
  \begin{align*}
    \sum_{\indBlock = 0}^{L-1}\hat\chi_A(\sigma_j,\theta) = \sum_{\indBlock =
    0}^{L-1}\hat\chi_A(\sigma_j, \bar\theta^*_j)+\cO(\|\sigma\|_{L^2}^2 L).
  \end{align*}
  Hence, using Sub-Lemma~\ref{lem:itera-trivial} and
  equations~\eqref{e_preparatory8},~\eqref{eq:hbar-chage},~\eqref{eq:marb-chage}, we
  conclude:
  \begin{align}\label{e_preparatoryk}
    \tO_{\thetas,\hOmg_{L-1}}&\cdots\tO_{\thetas,\hOmg_0}[\tilde\rho_\auxPot]
    = \harb \marb(\tilde\rho_\auxPot)e^{\sum_{\indBlock = 0}^{L-1}\hat\chi_A(\sigma_j,\bar\theta^*_\indBlock)}\\%
    &\phantom = \cdot e^{\cO\nl\infty(\|\sigma\|\nBV+\ve L^2\|\auxPot\|\nc2+\|\auxPot\|\nc2(\|\auxPot\|\nc2 L+\vei\|\sigma\|\nl1)+\ve L+\|\sigma\|_{L^2}^2L )}.\notag
  \end{align}
  In order to proceed we need to compare $\marb(\tilde\rho_\auxPot)$ with
  $\marb(\rrho_\auxPot)$.  Recall that by~\eqref{e_definitionTildeRho},
  $\tilde\rho_\auxPot = (\rrho_\auxPot\circ Y\invr)/\rho^*_\auxPot$, where
  $\mr\rho_{\auxPot} = \rho_\auxPot\,\Id_{[a,b]}$ and $\rho^*_\auxPot = 1+\cO(\ve L^2)$.  We claim
  that
  \begin{align}\label{e_marbClaim}
    \marb(\rrho_\auxPot\circ Y\invr)
    &= \marb(\mr\rho_\auxPot) e^{\cO(\ve L^2+\|\sigma\|_{\BV([0,\ve L])}
      +\vei L\invr\|\sigma\|\nl2^2)}.
  \end{align}
  Observe that substituting~\eqref{e_marbClaim} into~\eqref{e_preparatoryk},
  item~\ref{i_one-step-moment-sharp} follows by~\eqref{eq:preparatory4}
  and~\eqref{eq:preparatory3} since $\testfunc$ is arbitrary.  In order to conclude, we
  therefore only need to prove~\eqref{e_marbClaim}.  First of all, recall that
  $\rho_\auxPot$ is a $(\spc2+\Const\|\auxPot\|\nc2)$-standard probability density and
  that, by hypotheses, $\|\auxPot\|\nc2\le\Const\osigma$; hence
  Remark~\ref{r_standardDensityBV} implies that $\|\rrho_\auxPot\|\nBV\le\Const$.  Hence,
  if $\osigma$ is sufficiently small and since $\Leb\, \rrho_\auxPot = 1$,
  Lemma~\ref{lem:m-estimate} yields:
  \begin{align*}
    \marb(\mr\rho_\auxPot) = e^{\cO(\osigma |\log \osigma\,|)}.
  \end{align*}
  Let us proceed to estimate $\marb(\mr\rho_\auxPot \circ Y\invr - \mr\rho_\auxPot)$: if
  $\osigma$ is sufficiently small, Lemma~\ref{lem:m-estimate} ensures that
  \begin{align*}
    \marb(\mr\rho_\auxPot \circ Y\invr - \mr\rho_\auxPot)
    &=\Leb(\mr\rho_\auxPot \circ Y\invr - \mr\rho_\auxPot)+ \|\mr\rho_\auxPot \circ Y\invr - \mr\rho_\auxPot\|\nl1\cO(\|\sigma^*_2\||\log\|\sigma^*_2\|\,|)\\
    &\phantom = +\cO(\|\sigma^*_2\|^2\|\rrho_\auxPot\|_{\BV}) \\
    &\le \|\mr\rho_\auxPot \circ Y\invr - \mr\rho_\auxPot\|\nl1\left(1+\cO(\|\sigma^*_2\||\log\|\sigma^*_2\|\,|)\right)\\
    & +\cO(\|\sigma^*_2\|^2\|\rho_\auxPot\|_{\BV})
  \end{align*}
  Next, we proceed to estimate the $L^1$ norm; recall that for any bounded $\psi$:
  \begin{align*}
    \|\psi\|\nl1 = \sup_{\substack{\vf\in L^\infty\\\|\vf\|\nl\infty = 1}}\int\psi\vf = \sup_{\substack{\vf\in \cC^0\\\|\vf\|\nc0 = 1}}\int\psi\vf,
  \end{align*}
  where in the last equality we used the fact that continuous functions are dense in
  $L^1$. For any $\vf\in\cC^0$ we have
  \begin{align*}
    \left|\Leb(\vf\left[ \mr\rho_\auxPot\circ Y\invr - \mr\rho_\auxPot\right])\right|
 &=\left|\int_\bT \left[\vf( Y(x))\cdot Y'(x)-\vf(x)\right]\rho_\auxPot(x)dx\right|\\
 &=\left|\int_\bT [\testfuncp\circ Y-\testfuncp]'(x)\rrho_\auxPot(x)dx\right|\\
 &\le\|\rrho_\auxPot\|\nBV\|\testfuncp\circ Y-\testfuncp\|\nc0,
  \end{align*}
  where $\testfuncp'=\vf$ on $[0,1]$.  Since
  $|Y-\text{Id}|_{\infty}\le \Const \ve L^2$ (by~\eqref{eq:ur-E}), and
  $\|\rrho_\auxPot\|\nBV\le\Const$, we conclude that $|\Leb(\vf[ \mr\rho_\auxPot\circ Y\invr - \mr\rho_\auxPot])|
  \le \Const\|\vf\|\nc0\ve L^2$, which implies
  \begin{align*}
    \|\mr\rho_\auxPot\circ Y\invr - \mr\rho_\auxPot\|\nl1
    &\le \Const\ve L^2.
  \end{align*}
  Accordingly, putting together the above estimates:
  \begin{align*}
    \marb(\mr\rho_\auxPot\circ Y\invr)
    &= \marb(\mr\rho_\auxPot) e^{\cO(\ve  L^2+\|\sigma^*_2\|^2)}.
  \end{align*}
  Observe that
  \begin{align*}
    \|\sigma_2^*\|^2
    &= \frac1{\ve L}\int_0^{\ve L} \|\sigma(s)+\sigma_2^*-\sigma(s)\|^2
      \le \frac1{\ve L}\|\sigma\|\nl2^2 + 4\osigma\|\sigma\|_{\BV([0,\ve L])},
  \end{align*}
  and thus we have
  \begin{align*}
    \marb(\mr\rho_\auxPot\circ Y\invr)
    &= \marb(\mr\rho_\auxPot)
      e^{\cO(\ve L^2+\|\sigma\|_{\BV([0,\ve L])} +\vei L\invr\|\sigma\|\nl2^2)}.
  \end{align*}
  which gives~\eqref{e_marbClaim} and concludes the proof of the Lemma.
\end{proof}
\begin{proof}[Proof of Sub-lemma~\ref{lem:itera-trivial}]%
  First of all observe that, using~\eqref{e_estimatePartialh}
  and~\eqref{e_estimatePartialm-BV},
  \begin{subequations}\label{e_blockEstimates}
    \begin{align}
      \label{e_blockEstimate-h}
      \|h_{\theta,l}-h_{\theta,l+1}\|
      &\le \cR_{\theta,l+1}\\\label{e_blockEstimate-m}
      |m_{\theta,l}(g)-m_{\theta,l+1}(g)|
      &\le \cR_{\theta,l+1}\|g\|\nBV
    \end{align}
  \end{subequations}
  where\footnote{ The proposed estimate of $\cR_{\theta,l}$ may seem a bit cumbersome.
    The reason is that the second possibility is good locally to verify the condition
    $\cR_l\le 2\sstep \osigma$ below but is otherwise a
    bad choice since it gives a too large cumulative mistake.}
  \begin{align}\label{e_cRbound}
    \cR_{\theta,l+1}=C_\sstep\min\left\{\|\sgm\|_{\BV([\ve (l-1)\sstep, \ve (l+1)\sstep])},
    \vei\|\sgm\|_{L^1([\ve (l-1) \sstep ,\ve (l+1)\sstep])}\right\}
  \end{align}
  where the first term in the $\min$ comes from comparing the potential in one block to
  the potential in the next one and the second term comes from comparing the potential
  in each block with the zero potential.  We assume conventionally $\cR_{\theta,0}=1$.

  \newcommand{\irho}[1]{\rho_{(#1)}}
  Let $\irho0 = \tilde\rho_\auxPot$ and define for $l\in\{1,\cdots,L'\}$:
  \begin{align*}
    \irho{l}:=\tOb_{\theta,l}\irho{l-1};
  \end{align*}
  observe in particular that $\irho{L'} = \tOb_{\theta,L'}\cdots\tOb_{\theta,1}[\tilde\rho_\auxPot]$.

  Let us now define $\gamma_l = m_l(\irho l)\geq 0$ and $\varphi_l = (1-\tOpb_l)\irho l$ so that
  $\irho{l}=\gamma_{l} h_{l}+\vf_{l}$; in particular
  $\|\irho l\|\nBV\le \Const\gamma_l+\|\varphi_l\|\nBV$ and $m_l(\vf_{l})=0$.  Then
  \begin{subequations}
    \begin{align}\label{eq:gamma-itera}
      \gamma_{l+1}&=m_{l+1}(\rho_{l+1})=e^{\chib_{\theta,l+1}}m_{l+1}(\rho_{l})\\
      &=e^{\chib_{\theta,l+1}}\left(\gamma_{l}-m_{l}(\irho{l})+m_{l+1}(\irho{l})\right)\notag\\
                  &= e^{\chib_{\theta,l+1}}\gamma_{l}+e^{\chib_{\theta,l+1}}[m_{l+1}-m_{l}](\irho{l})\;;\notag\\
      \vf_{l+1}&=\irho{l+1}-\gamma_{l+1}h_{l+1}=\tOb_{\theta,l+1}(\irho{l}-m_{l+1}(\irho {l}) h_{l+1})\\
                  &=\tOqb_{\theta,{l+1}}(\irho{l}-m_{l+1}(\irho{l}) h_{l+1}).\notag
    \end{align}
  \end{subequations}
  By~\eqref{e_blockEstimate-m} we have
  $|[m_{l+1}-m_{l}](\irho{l+1})|\le \Const\cR_{l+1}\|\irho {l+1}\|\nBV$.  Accordingly,
  since $\|\tOqb_{\theta,l}\|\nBV < \frac12e^{\chib_{\theta,l}}$ and setting
  $\alpha=\log 2$:
  \begin{subequations}
    \begin{align}
      \gamma_{l+1}&=e^{\chib_{\theta,l+1}}\gamma_{l}+\cO(e^{\chib_{\theta,l+1}}\cR_{l+1}\|\irho{l}\|\nBV)\label{eq:cone-iterag}\\
                  &= e^{\chib_{\theta,l+1}+\cO(\cR_{l+1})}\gamma_{l}+\cO(e^{\chib_{\theta,l+1}}\cR_{l+1}\|\vf_{l}\|\nBV)\;;\nonumber\\
      \|\vf_{l+1}\|\nBV\label{eq:cone-itera}
                  &\le e^{\chib_{\theta,l+1}-\alpha}\|\irho{l}-m_{l+1}(\irho{l})h_{l+1} \|\nBV\\\notag
                  &\le e^{\chib_{\theta,l+1}-\alpha}\|m_{l}(\irho{l})h_{l}+\varphi_{l}-m_{l+1}(\irho{l})h_{l+1} \|\nBV\\\notag
                  &\le e^{\chib_{\theta,l+1}-\alpha}\left[\|\vf_{l}\|\nBV +C_\#\cR_{l+1} \|\irho{l}\|\nBV\right]\\
                  &\le e^{\chib_{\theta,l+1}-\alpha}\left[e^{\Const\cR_{l+1}}\|\vf_{l}\|\nBV+\Const \cR_{l+1}\gamma_{l}\right].\nonumber
    \end{align}
  \end{subequations}
  Next, we prove, by induction, that there exists $C_*>1$ such that, for any
  $l\in\{1,\cdots, L'\}$
  \begin{equation}\label{phi-induction}
    \|\vf_l\|\nBV\le C_* \gamma_l\sum_{j=0}^le^{-(l-j)\alpha/2}\cR_j.
  \end{equation}
 Since $\irho{0}$ is a standard density we have
  $\gamma_0=m_0(\irho0)\ge C_\# \|\irho0\|\nBV\ge C_\#\|\vf_0\|\nBV$, thus the relation
  is satisfied for $l=0$ provided $C_*$ is chosen large enough.  Next,
  combining~\eqref{eq:cone-iterag} with~\eqref{phi-induction} and observing that for any
  $j$ we have by definition $\cR_j\le 2\sstep \osigma$, we obtain:
  \begin{align*}
    \gamma_{l+1}&\ge e^{\chib_{\theta,l+1}-\Const\cR_{l+1}}\gamma_{l}-\Const e^{\chib_{\theta,l+1}}\cR_{l+1}\left[
                  C_*\gamma_l\sum_{j=0}^le^{-(l-j)\alpha/2}\sstep\osigma\right]\\
                &\ge e^{\chib_{\theta,l+1}-\Const \sstep\osigma-\Const C_*\sstep^2\osigma^2}\gamma_l.
  \end{align*}
  Plugging the above estimate into~\eqref{phi-induction} and combining
  with~\eqref{eq:cone-itera} yields,
  \begin{align*}
    \|\vf_{l+1}\|\nBV&\le e^{-\alpha+\Const \sstep\osigma+\Const\sstep^2
                       C_*\osigma^2} \gamma_{l+1}C_* \sum_{j=0}^l e^{-(l-j)\alpha/2}\cR_j +C_\#\cR_{l+1}\gamma_{l+1}\\
                     &\le C_* \gamma_l\sum_{j=0}^{l+1}e^{-(l+1-j)\alpha/2}\cR_j
  \end{align*}
  provided $C_*$ is chosen large enough and
  $\Const \sstep \osigma+\Const C_*\sstep^2\osigma^2\le \frac{\alpha}{2}$, which can always be
  satisfied by choosing $\osigma$ small enough.  This concludes the proof
  of~\eqref{phi-induction}. Combining this estimate with~\eqref{e_cRbound} we conclude
  that
  \begin{align*}
    \|\varphi_{L'}\|\nBV\le\Const\gamma_{L'}\|\sgm\|\nBV.
  \end{align*}
Finally, substituting again~\eqref{phi-induction} into~\eqref{eq:cone-iterag} implies
  \begin{align*}
    \gamma_{L'}=e^{\chib_{\theta,L'}+\cO(\cR_{L'})}\gamma_l = e^{\sum_{l = 1}^{L'}[\chib_{\theta,l}+\cO(\cR_{l})]}\gamma_0.
  \end{align*}
  since $\gamma_0 = m_0(\irho0) = m_0(\tilde\rho_\auxPot)$, we gather that
  \begin{align*}
    \irho{L'}(x) = e^{\sum_{l = 1}^{L'}[\chib_{\theta,l}+\cO(\cR_{l})]}m_0(\tilde\rho_\auxPot)(h_{L'}(x)+\Const\|\sgm\|\nBV).
  \end{align*}
Hence, recalling, from Lemma~\ref{lem:quasicompact} (or more
  precisely~\eqref{eq:cone-est}) that $h_{L'}\ge e^{\const}$, we can conclude the proof of
  the sub-lemma, since $\sum_{l = 1}^{L'}\cR_{\theta,l}\le \Const\|\sgm\|\nBV$.
\end{proof}

\subsection{Regularizing moment generating functional}\ \\
\label{s_regularizedlogmgf}
The discussion of the previous section tells us that, provided the error term
$\restoPalla(\sigma)$ is somewhat under control, the logarithmic moment generating
function $\logmgf(\sigma)$ is well described by
$\int_0^T\chi_A(\sigma(s),\bar\theta(s,\thetasl))ds$.  Proposition~\ref{lem:exponential},
however, shows that any kind of control on the remainder term $\restoPalla(\sigma)$
might fail if the $\BV$-norm of $\sigma$ is much larger than its $L^1$-norm (\ie for
rapidly oscillating functions).  We will then need to consider regularizations of
$\sigma$ whose $\BV$-norm is controlled by their $L^1$-norm.

Given a step size $h=T/N_h$, for $N_h\in\bN$ suitably large, define the
projector $\regPi{h}$ given by averaging on each interval of size $h$:
\begin{align}\label{e_definitionRegPi}
  [\regPi{h}\sigma](t)=h\invr\int_{h\pint{th\invr}}^{h(\pint{th\invr}+1)}\sigma(s)\deh s.
\end{align}
We collect in the following sub-lemma the basic properties of $\regPi{h}$; their proof is
elementary and it is left to the reader.
\begin{sublem}\label{sl_propertiesPi} The operator $\regPi{h}$ satisfies the following properties
  \begin{enumerate}
  \item $\regPi{h}1 = 1$
  \item $\int f \cdot \regPi{h} \sigma = \int \regPi{h} f\cdot \sigma$;
  \item $\regPi{h}$ is a contraction in the $\BV$ and $L^p$-norms if $p\in[1,\infty]$;
  \item $\|\regPi{h}\sigma\|\nBV\le \Const h\invr\|\sigma\|\nl1$.
  \end{enumerate}
\end{sublem}
We then proceed to define the \emph{regularized moment generating functional} as
\begin{align}\label{e_defReguLogmgf}
  \regu\logmgf{h} =\logmgf\circ\regPi{h}
\end{align}
and as in the previous section we can define
\begin{align}\label{e_defReguPalla}
  \regu\restoPalla{h}(\sigma) = \regu\logmgf{h}(\sigma)-\int_0^T\chi_A(\sigma(s),\bar\theta(s,\thetasl))ds
\end{align}
\begin{lem}\label{lem:exponential-h}
  There exists $\ve_0 > 0$, such that if $\ve\in (0,\ve_0)$,  $L\in\left[\vei_0, \ve_0{\ve^{-1/2}}\right]$ and $T\in[\ve L,\Tmax]$:
  \begin{enumerate}
    \item\label{i_qnonPerturbativeBoundOnR-h} for any $\sigma\in\BV([0,T],\bR^d)$, the
    following upper bound holds:
    \begin{align*}
      \regu\restoPalla{h}(\sigma)
      &\le \Const(\ve LT+\\
      &\phantom\le + \left[\ve Lh\invr + h + L\invr+\min\{T,(1+\ve L h\invr)\|\sigma\|\nl1\}\right]\|\sigma\|\nl1);
    \end{align*}
    \item\label{i_perturbativeBoundOnR-h} there exists $\sigma_* = \sigma_*(\Tmax)>0$ so that if
    $\|\sigma\|\nl\infty< \sigma_*$, the following upper bound holds:
    \begin{align*}
      \regu\restoPalla{h}(\sigma) \le \Const\left(\ve LT+\ve[L+ T\invr + h\invr +
      h\vei]\|\sigma\|\nl1 + \|\sigma\|\nl1^2+ L\invr\|\sigma\|\nl2^2\right).
    \end{align*}
  \end{enumerate}
\end{lem}
\begin{proof}
  Observe that, by definition
  \begin{align*}
    \regu\restoPalla{h}(\sigma) = \restoPalla(\regPi{h}\sigma)+ \int_0^T\chi_A(\regPi{h}\sigma(s),\bar\theta(s,\thetasl))ds
    -\int_0^T\chi_A(\sigma(s),\bar\theta(s,\thetasl))ds.
  \end{align*}
  But since $\chi_A(\cdot,\theta)$ is convex, Jensen inequality yields:
  \begin{equation}\label{eq:regular-Pi}
    \int_0^T\chi_A(\regPi{h}\sigma(s),\bar\theta(s,\thetasl))ds
    \le \int_0^T\!\! ds\int_{h\pint{s\vei}}^{h(\pint{s\vei}+1)}\frac{\hat\chi_A(\sigma(r),\bar\theta(s,\thetasl))}hdr .
 \end{equation}
 Next, by~\eqref{e_partialThetachi},~\eqref{eq:derivativeL} and since, by
 Lemma~\ref{lem:quasicompact}, $m_\theta$ is a measure, holds the normalization
 $m_\theta(h_\theta)=1$ and $|h_\theta'|\leq \Const \|\sigma\|h_\theta$, we have
\[
|\hat\chi_A(\sigma,\bar\theta(s,\theta))-\hat\chi_A(\sigma,\bar\theta(r,\theta))| \leq \Const h\|\sigma\|.
\]
Hence
   \begin{align*}
    \int_0^T\chi_A(\regPi{h}\sigma(s),\bar\theta(s,\thetasl))ds
    &\le \int_0^T[\regPi{h}\chi_A(\sigma(\cdot),\bar\theta(\cdot,\thetasl))](s)ds +\Const h\|\sigma\|\nl1\\
    &\le \int_0^T\chi_A(\sigma(s),\bar\theta(s,\thetasl))ds +\Const h\|\sigma\|\nl1,
  \end{align*}
where we used items (a-b) of Sub-lemma~\ref{sl_propertiesPi} to conclude that
  $\int\regPi{h} f = \int f$.  The lemma readily follows from items (c-d) of
  Sub-lemma~\ref{sl_propertiesPi} and Proposition~\ref{lem:exponential}.
\end{proof}
\section{Deviations from the average: the rate function}\label{sec:upper}
\noindent Here we study the deviations from the average behavior described in
Section~\ref{sec:averaging}.
\begin{rem}
  The results in Sections~\ref{sec:upper} and~\ref{sec:LargeDeviations} are in the spirit
  of~\cite{DimaAveraging} although more precise, insofar in~\cite{DimaAveraging} only a
  rough upper bound on the rate function is provided.  Regarding the classical Large
  Deviations Principle, the exact rate function was derived in~\cite{Kifer09}, but with an
  estimate of the error largely insufficient to handle moderate deviations, as the
  function was computed with a mistake of order $o(1)$.  Here we estimate the error much
  more precisely and we are therefore able to study accurately also deviations of order
  $\ve^\alpha$, with $\alpha< 1/2$.  In addition, contrary to~\cite{DimaAveraging}, we
  derive not only an upper bound but a lower bound as well, at least for deviations larger
  than $\ve^{\efrac 1{96}}$.  We refrain from obtaining completely optimal results (which
  may be obtained using the techniques developed later in this paper) only to keep the
  length of the paper (somewhat) under control.
\end{rem}

Recall that we fixed $d\in\bN$ and $A = (A_1,\cdots,A_d)\in \cC^2(\bT^2,\bR^d)$, with
$A_1(x,\theta)=\omega(x,\theta)$.  Recall moreover that we are always under the standing
assumption~\ref{eq:coboA}.  Finally, note that, for convenience, we will often implicitly
lift $\theta\in\bT$ to its universal cover $\bR$.

The fundamental object in the theory of large deviations is the {\em rate function}.
Because its definition is a bit involved, we start by discussing it in some detail.  The
reader that is not familiar with the meaning and the use of such a function may want to
review the discussion in Sections~\ref{sec:trivial},~\ref{ss_largeModerateDeviationIntro} and have a
preliminary look at Section~\ref{sec:LargeDeviations} where it is made clear the role of
the rate function in the statements of the various large and moderate deviations results.

\subsection{Definition and properties: the preliminary rate function}\label{subsec:rate}\ \newline
We start by discussing a rate function that is expressed in terms of the averaged
trajectory of $\theta$ and therefore turns out to be accurate only for short
times. However its discussion entails all the quantities and ideas needed for the general
case.

Recall that $e^{\chi_A(\sigma,\theta)}$ (\resp $e^{\hat\chi_A(\sigma,\theta)}$) denotes
the maximal eigenvalue of the transfer operator $\tO_{\theta,\langle\sigma,A\rangle}$
(\resp $\tO_{\theta,\langle\sigma, \hatA\rangle}$) which has been introduced in
Section~\ref{ss_exponentialMoment}.  Recall also that
$\chi_A(\sigma,\theta)=\langle \sigma, \bar{A} (\theta)\rangle+\hat\chi_A(\sigma,\theta)$;
finally, observe that~\eqref{e_secondDerivativeChia},~\eqref{e_secondDerivativeChi} and
Lemma~\ref{lem:coboundary}, together with assumption~\ref{eq:coboA} imply that
$\chi_A(\cdot,\theta)$ is a strictly convex function.

For any $\sigma,b\in\bR^d$ and $\theta\in\bT^1$, define
\begin{equation}\label{eq:kappaZ}
  \zet(\sigma,b,\theta)
  = \langle \sigma, b\rangle -\chi_A(\sigma,\theta)
  = \langle \sigma, b-\bar A(\theta)\rangle-\hat\chi_A(\sigma,\theta),
\end{equation}
and define the function $\cZ: \bR^d\times\bT\to\bR\cup\{+\infty\}$ as
\begin{equation}\label{eq:M-0}
  \cZ(b,\theta)=\sup_{\sigma\in\bR^d}\zet(\sigma,b,\theta).
\end{equation}
Observe that $\cZ(\cdot,\theta)$ is the Legendre transform of $\chi_A(\cdot,\theta)$.  We
are now able to give a first preliminary definition of the \emph{rate function}; for any
$\thetas\in\bT$, let
\begin{align}
  \notag
  \Ifirst_{\thetas}&:\cC^0([0,T],\bR^d)\to\bR\cup\{+\infty\}\\\label{eq:rate1}
  \Ifirst_{\thetas}&(\gamma)=\begin{cases}+\infty&\textrm{ if }\gamma \textrm{ is not Lipschitz, or $\gamma(0)\neq0$}\\
    \int_0^T \cZ(\gamma'(s), \bar \theta(s,\thetas))\,\deh s\quad&\textrm{ otherwise.}
  \end{cases}
 \end{align}

Our next task is to investigate the properties of $\Ifirst_{\thetas}$ or, equivalently, of
$\cZ$.  Let $\bD(\theta)=\{b\in\bR^d\;:\:\cZ(b,\theta)<+\infty\}$ be the \emph{effective
  domain} of $\cZ(\cdot,\theta)$,
\begin{lem}\label{lem:domainZ}
  Assume~\ref{eq:coboA} (\ie for all $\sigma\in\bR^d$ and $\theta\in\bT^1$,
  $\langle\sigma,\hatA(\cdot, \theta) \rangle$ is not an $f_\theta$-coboundary).
  Then the following properties hold:
  \begin{enumerate}[label = (\arabic*),start = 0]
  \item \label{i_Zconvex} $\cZ(\cdot,\theta)$ is a convex lower semi-continuous function;
    in particular $\bD(\theta)$ is convex.
  \item \label{i_bDstar} let $\bD_*(\theta)=\partial_\sigma\chi_A(\bR^d, \theta)$; then
    $\bD_*(\theta) = \intr{\bD}(\theta)$; in particular $\bD_*(\theta)$ is convex;
  \item \label{i_analyticity} let
    $U=\{(b,\theta)\;:\; \theta\in\bT,\,b\in\, {\bD_*}(\theta)\}$; $\cZ\in\cC^2(U,\bRp)$
    and it is analytic in $b$;
  \item \label{i_nbhdSRB} $\bD(\theta)$ contains a neighborhood of $\bar A(\theta)$;
  \item \label{i_goodSRB} $\cZ(\bar A(\theta),\theta)=0$, $\partial_b  \cZ(\bar A(\theta),\theta)=0$, $\partial_\theta  \cZ(\bar A(\theta),\theta)=0$, and $\cZ \ge 0$;
  \item \label{i_GreenKubo} $\partial_b^2\cZ(b,\theta)>0$, and setting
    $\left[ \partial_b^2\cZ(\bar A(\theta),\theta)\right]^{-1}=\Sigma^2(\theta)$ we have
    \[
      \Sigma^2(\theta)= \mu_\theta\left(\hat A(\cdot,\theta)\otimes\hat A(\cdot,\theta)\right)+ 2 \sum_{m=1}^{\infty}\mu_\theta\left( \hat A(f_\theta^m(\cdot),\theta)\otimes\hat A(\cdot,\theta)\right).
    \]
  \end{enumerate}
\end{lem}
\begin{proof}
  Item~\ref{i_Zconvex} follows since, for each $\theta$, $\cZ(\cdot,\theta)$ is the
  (convex) conjugate function of a proper function, hence a convex lower semi-continuous
  function (see~\cite[Theorems 10.1, 12.2]{Rockafellar70}).

  Since $\chi_A$ is a strictly convex function, $\partial_\sigma\chi_A$ in an injective
  map and hence, by the theorem of Invariance of Domain, we conclude that $\bD_*(\theta)$
  is open.  The equality $\intr{\bD}(\theta)= \bD_*(\theta)$ follows then
  from~\cite[Theorem 23.4, Corollary 26.4.1]{Rockafellar70}.  We have thus proved
  item~\ref{i_bDstar}.

  Observe now that if $b\in\bD_*(\theta)$, then $\cZ(b,\theta)=\zet(\bar\sigma,b,\theta)$
  where $\bar\sigma=\bar\sigma(b,\theta)$ is the unique solution of
  $b=\partial_\sigma\chi_A(\bar\sigma(b,\theta), \theta)$.  Item~\ref{i_analyticity}
  follows by the implicit function theorem and the perturbation theory results collected
  in Appendix~\ref{subsec:pert-paramL} and~\ref{subsec:pert-theta}.

  By~\eqref{e_derivativeChi}
  $\partial_\sigma \chi_A(\sigma,\theta)=\mnu_{\theta,\langle\sigma,A\rangle}(
  A(\cdot,\theta))$, where $\mnu_{\theta,\langle\sigma, A\rangle}$ is the invariant
  probability measure associated to the operator~\eqref{eq:Lopzero}; in particular
  $\mnu_{0,\theta}=\mu_\theta$.  Then $\partial_\sigma \chi_A(0,\theta)=\bar A(\theta)$,
  which implies that $\bar A(\theta)\in \bD_*(\theta)$, hence proving item~\ref{i_nbhdSRB}.

  Next, let us prove item~\ref{i_goodSRB}. First $\cZ(\bar A(\theta),\theta)=0$ and $\cZ(b,\theta)\ge -\hat\chi_A(0,
  \theta)=0$.  Also, a direct computation shows that
  \begin{equation}\label{eq:zeta-sigma-rel}
    \partial_b \cZ(b,\theta)=\bar\sigma(b,\theta);
  \end{equation}
  in particular $\partial_b \cZ(\bar A(\theta),\theta)=0$.  Next,
  $(\partial_\theta\cZ)(\bar A(\theta),\theta)=-\partial_\theta\hat\chi_A(0,\theta)=0$
  by~\eqref{e_estimatesecondpartchi0}.

  Finally, by~\cite[Theorem 26.5]{Rockafellar70},
  $\partial_b^2\cZ(b,\theta)=[\partial_\sigma^2\chi_A(a(b,\theta),\theta)]^{-1}$.  This
  and ~\eqref{e_secondDerivativeChia} imply item~\ref{i_GreenKubo}.
\end{proof}

\begin{rem}\label{rem:z-concrete}
Arguing as in~\eqref{e_derivativeChi} it follows that
  $\|\partial_\sigma\chi_A\|_{\infty}\le \|A\|_\infty$, thus $\bD_*(\theta)$ is
  uniformly (in $\theta$) bounded.  It follows that if $b\in \bD_*(\theta)$ and $\sigma_b$ is the
  solution of
  \begin{equation}\label{eq:stationary-sigma}
    b = \partial_\sigma\chi_A(\sigma, \theta),
  \end{equation}
  then
  \begin{equation}\label{eq:good-zeta}
    \cZ(b,\theta)=\zet(\sigma_b, b, \theta),
  \end{equation}
  and $\partial_b\cZ(b,\theta)=\sigma_b$.
\end{rem}
\begin{rem} Using the above facts, it would be possible to show that $\Ifirst_{\thetas}$ is lower
  semi-continuous with respect to the uniform topology.  We refrain from proving it here
  because the proof will be given later in Lemma~\ref{lem:lws}.
\end{rem}
We conclude this subsection with a useful estimate.
\begin{lem}\label{l_gamma-bound}
  Fix $\theta\in\bT$ and let $\dualsigma,\sigma\in\bR^d$ so that
  $\dualsigma = \partial_\sigma\hat\chi_A(\sigma,\theta)$; then there exists $\Constgs >
  1$ so that
  \begin{enumerate}
    \item $ \|b\|\le \Constgs \|\sigma\|$
    \item $ \Constgs\invr \min\{\|\sigma\|, \|\sigma\|^2\}%
    \le \langle \sigma,\dualsigma\rangle %
    \le \Constgs \min\{\|\sigma\|, \|\sigma\|^2\}.$
  \end{enumerate}
\end{lem}
\begin{proof}
  The first item follows by the definition, equations~\eqref{eq:derivative-m} and the fact
  that $\partial_{\sigma^2}\hat\chi_A(\cdot,\theta)$ is bounded (as a quadratic form).  We
  proceed to prove the second item.  Let $\hat\sigma=\|\sigma\|\invr\sigma$; then, by
  definition:
\begin{align} \label{eq:gamma-bound1}%
  \langle \sigma,\dualsigma\rangle%
  &=\int_0^{\|\sigma\|}\langle\sigma,\partial_{\sigma}^2\hat\chi_A(\lambda\hat\sigma,\theta)\hat\sigma\rangle\deh\lambda.
\end{align}
Moreover,~\eqref{e_secondDerivativeChi} and~\ref{eq:cobo} imply that
$\partial_\sigma^2\hat\chi_A(\cdot,\theta)\ge \Consti \Id$ (as quadratic forms) for
$\|\sigma\|\le \sigma_\#$ for some $\sigma_\#$ sufficiently small.  Observe that
$\sigma_\#$ depends on $f$ and $A$ only.  Hence by~\eqref{eq:gamma-bound1} we gather
\begin{align*}
  \langle\sigma,\dualsigma\rangle
  &\ge \Consti
    \int_0^{\min\{\|\sigma\|,\sigma_\#\}}\langle\sigma,\hat\sigma\rangle\deh\lambda %
    \ge \Consti \min\{\|\sigma\|,\|\sigma\|^2\},
\end{align*}
which gives the lower bound.

On the other hand, since $\dualsigma = \partial_\sigma\hat\chi_A(\sigma,\theta)$, we have
$\dualsigma\in\bD_*(\theta)$; hence, by Remark~\ref{rem:z-concrete}, $\dualsigma$ is
uniformly bounded and thus we obtain $\langle\sigma,\dualsigma\rangle\le\Const\|\sigma\|$.

Moreover, using once again~\eqref{eq:gamma-bound1} and since $\partial_\sigma^2\hat\chi_A$
is locally bounded from above (see~\eqref{e_secondDerivativeChi0}), $\langle\sigma,\dualsigma\rangle\le\Const\|\sigma\|^2$, for all $\sigma\leq\sigma_\#$, which concludes the proof.
\end{proof}
\subsection{Entropy characterization}\label{subsec:rate-entropy}\ \newline
As already mentioned, $\cZ$ can also be expressed in terms of entropy (see
e.g.~\cite{Kifer09}).
\begin{lem}\label{lem:entropy} For any $\theta\in\bT^1$ and $b\in\bR^d$, let
  $\cM_\theta(b)=\{\mnu\in\cM_\theta\st \mnu(A(\cdot,\theta))=b\}$, where $\cM_\theta$
  denotes the set of $f_\theta$-invariant probability measures.  Then:
  \begin{align}\label{e_entropy}
    \cZ(b,\theta)= -\sup_{\mnu\in\cM_\theta(b)}\{ \kse_\theta(\mnu) - \mnu(\log f_\theta')\},
  \end{align}
  where $\kse_\theta(\mnu)$ is the Kolmogorov-Sinai metric entropy of the measure $\mnu$
  with respect to the map $f_\theta$. In particular\footnote{ Recall that we adopt the
    convention $\sup\,\emptyset = -\infty$. },
  $\bD(\theta)=\{b\in\bR^d\st \cM_\theta(b)\neq \emptyset\}$.
\end{lem}
\begin{proof}
  It is well known (see e.g.~\cite[Remark 2.5]{Baladibook}), that
  \begin{equation}\label{rem:top-ent}
    \begin{split}
      \chi_A(\sigma, \theta) &= \sup_{\mnu\in\cM_\theta}\left\{ \kse_\theta(\mnu) + \mnu(\langle \sigma, A\rangle-\log f_\theta')\right\}\\
      &=\kse_\theta({\mnu_{\theta,\langle \sigma, A\rangle}})+\mnu_{\theta,\langle \sigma,
        A\rangle}(\langle \sigma, A\rangle-\log f_\theta')
    \end{split}
  \end{equation}
  where
  $\mnu_{\theta,\langle \sigma, A\rangle}(g)=m_{\theta,\langle \sigma,
    A\rangle}(g\,h_{\theta,\langle \sigma, A\rangle})$ and
  $m_{\theta,\langle \sigma, A\rangle}$ and $h_{\theta,\langle \sigma, A\rangle}$ are
  respectively the left and right eigenvectors of $\tO_{\theta,\langle\sigma, A\rangle}$
  corresponding to the eigenvalue $e^{\chi_A(\sigma,\theta)}$, normalized so that
  $\mnu_{\theta,\langle \sigma, A\rangle}$ is a probability measure.  We record, for
  future use, some properties of the entropy: since each $f_\theta$ is expanding,
  $\kse_\theta$ is (as a function of $\mnu$) an upper-semicontinuous function with respect
  to the weak topology (see~\cite[Theorem 4.5.6]{Keller98}).  Also, $\kse_\theta$ is a
  convex affine function\footnote{ i.e.\ it is both convex and concave.} by~\cite[Theorem
  3.3.2]{Keller98}. Incidentally, this implies that the $\sup$ in~\eqref{rem:top-ent}
  would be the the same if taken only on ergodic measures, see~\cite[Theorem
  4.3.7]{Keller98}.  Then, using the definition~\eqref{eq:M-0}:
  \begin{align}\notag
    \cZ(b,\theta)
    &= \sup_{\sigma\in\bR^d}\left\{\langle\sigma,b\rangle -\sup_{\mnu\in\cM_\theta}[
      \kse_\theta(\mnu)+\mnu(\langle \sigma, A\rangle-\log f_\theta')]\right\}\\
    &\le \sup_{\sigma\in\bR^d}\left\{\langle\sigma,b\rangle -\sup_{\mnu\in\cM_\theta(b)}[
      \kse_\theta(\mnu)+\mnu(\langle \sigma, A\rangle-\log f_\theta')]\right\}\notag\\
    &\le -\sup_{\mnu\in\cM_\theta(b)}\{ \kse_\theta(\mnu)-\mnu(\log f_\theta')\}.\label{e_rateFunctionLowerBoundEntropy}
  \end{align}
  In particular, the above implies that if $\cZ(b,\theta) = \infty$,
  then~\eqref{e_entropy} holds.  We may thus assume that $\cZ(b,\theta) < \infty$.
  Observe that:
  \begin{align*}
    \cZ(b,\theta)
    &= \sup_{\sigma\in\bR^d} \left\{- \kse_\theta({\mnu_{\theta,\langle \sigma, A\rangle}})+\mnu_{\theta,\langle \sigma, A\rangle}(\langle \sigma, b-A\rangle)+\mnu_{\theta,\langle \sigma, A\rangle}(\log f_\theta')\right\}.
  \end{align*}
  Note that the first term on the right hand side is bounded by the topological entropy
  \cite[Theorem 4.2.3]{Keller98}, while the last term is bounded because $f_\theta'>1$.
  Thus, since Lemma~\ref{lem:domainZ}\ref{i_goodSRB} implies that $\cZ\ge 0$ and we assume
  $\cZ(b,\theta) < \infty$, we conclude that
  \begin{align}\label{e_assumptionLambda}
    \sup_{\sigma\in\bR^d}|\mnu_{\theta,\langle \sigma, A\rangle}(\langle \sigma,
    b-A\rangle)|<\infty.
  \end{align}
  For any $\lambda\in\bR$ consider the function $K_\lambda\in\cC^0(\bR^d,\bR^d)$ defined
  by $K_\lambda(\sigma)=\mnu_{\theta, \lambda\langle \sigma, A\rangle}( b-A)$.  Since
  $\mnu_{\theta, \lambda\langle \sigma, A\rangle}$ is a probability measure, we have
  $K_\lambda(\bR^d)\subset B=\{x\in\bR^d\st\|x\|\le \|b-A\|_\infty\}$.  By Brouwer
  fixed-point theorem it follows that there exists $\sigma_\lambda\in B$ such that
  $K_\lambda(\sigma_\lambda)=\sigma_\lambda$.  Accordingly, for any non-negative sequence
  $(\lambda_j)$ with $\lambda_j\to +\infty$:
  \begin{align*}
    \langle\lambda_j\sigma_{\lambda_j}, \mnu_{\theta, \langle
    \lambda_j\sigma_{\lambda_j}, A\rangle}(b-A)\rangle=\lambda_j\| \mnu_{\theta,
    \langle \lambda_j\sigma_{\lambda_j}, A\rangle}(b-A)\|^2\ge 0.
  \end{align*}
  Since the left hand side is bounded, see~\eqref{e_assumptionLambda},
  $\lim_{j\to\infty}\mnu_{\theta, \langle \lambda_j\sigma_{\lambda_j}, A\rangle}(b-A)=0$.
  By passing to a subsequence $\{j_k\}$ we can assume, setting
  $\bar\sigma_k=\lambda_{j_k}\sigma_{\lambda_{j_k}}$, that
  $\mnu_{\theta, \langle\bar\sigma_k, A\rangle}$ weakly converges to a measure $\mnu_*$.
  Moreover, for any $k\in\bN$
  \begin{align*}
    \cZ(b,\theta)\ge - \kse_\theta({\mnu_{\theta, \langle \bar\sigma_k, A\rangle}})+\mnu_{\theta, \langle \bar\sigma_k, A\rangle}(\log f_\theta').
  \end{align*}
  Since $\kse_\theta$ is upper-semicontinuous, we conclude that
  \begin{align*}
    \cZ(b,\theta)\ge -\limsup_{k\to\infty}\left[ \kse_\theta({\mnu_{\theta, \langle \bar\sigma_k, A\rangle}})-\mnu_{\theta, \langle \bar\sigma_k, A\rangle}(\log f_\theta')\right]\ge -\kse_\theta({\mnu_*})+\mnu_*(\log f_\theta').
  \end{align*}
  Finally, notice that $\mnu_*\in\cM_\theta$ and $\mnu_*(b-A)=0$, hence $\mnu_*\in\cM_\theta(b)$.  Thus we have
  \begin{align*}
    \cZ(b,\theta)\ge -\sup_{\mnu\in\cM_\theta(b)}\{ \kse_\theta({\mnu})-\mnu(\log f_\theta')\},
  \end{align*}
  which together with~\eqref{e_rateFunctionLowerBoundEntropy} concludes the proof of the
  lemma.
\end{proof}
The entropy characterization allows to add two useful properties to those listed in Lemma~\ref{lem:domainZ}.
\begin{lem}\label{lem:hausdorff}
The following properties hold:
  \begin{enumerate}[label = (\arabic*),start = 6]
 \item\label{i_compactD} $\bD(\theta)$ is a compact set for all $\theta\in\bT$;
 \item\label{i_lipschitzD} The map $\theta\mapsto\bD(\theta)$ is Lipschitz in the Hausdorff metric.
 \end{enumerate}
\end{lem}
\begin{proof}
  If $\{b_n\}\subset \bD(\theta)$, then there exists $\{\nu_n\}\subset \cM_\theta$ such
  that $\nu_n(A)=b_n$.  Since $\cM_\theta$ is compact in the weak topology, by extracting
  a convergent subsequence, item~\ref{i_compactD} follows.

  To prove item~\ref{i_lipschitzD}, note that all the maps $f_\theta$ are topologically
  conjugated to $f_0$ by a homeomorphism $\xi(\cdot,\theta)=\xi_\theta(\cdot)$ with the
  property\footnote{ This is folklore, e.g. it can be proven using shadowing and keeping
    track of the constants.}
  \[
    \|\xi_\theta-\xi_{\theta'}\|\nc0+\|\xi_\theta^{-1}-\xi_{\theta'}^{-1}\|\nc0\leq \Const |\theta-\theta'|.
  \]
  Accordingly, using the notation of Lemma~\ref{lem:entropy},
  $\cM(\theta)=(\xi_{\theta})_*\cM(0)$. Hence, for each $\theta, \theta'\in\bT$ and
  $b\in\bD(\theta)$ there exist $\nu\in\cM(0)$ such that
  $b=\nu(\bar A(\xi_{\theta}(\cdot),\theta))$ and
  $b'=\nu(\bar A(\xi_{\theta'}(\cdot),\theta'))\in \bD(\theta')$. Then
  \[
    \|b-b'\|\leq \|\nu(\bar A(\xi_{\theta}(\cdot),\theta))-\nu(\bar A(\xi_{\theta'}(\cdot),\theta'))\|\leq \Const |\theta-\theta'|.
  \]
  Thus $b$ must belong to a $\Const |\theta-\theta'|$ neighborhood of $\bD(\theta')$ and
  exchanging the role of $\theta, \theta'$, the item follows.
\end{proof}
Lemma~\ref{lem:entropy} allows to specify exactly the effective domain of $\Ifirst_{\thetas}$:
\begin{align*}
\fkD(\Ifirst_{\thetas}) = \{\gamma\in C([0,T])\st \Ifirst_{\thetas}(\gamma) < \infty\}.
\end{align*}
In fact $\Ifirst_{\thetas}(\gamma)<\infty$ if and only if $\gamma(0)=0$, $\gamma$ is
Lipschitz and $\cM_{\bar \theta(t,\thetas)}(\gamma'(t))\neq \emptyset$ for almost all
$t\in [0,T]$. Having fixed $\thetas\in\bT$, we will call {\em s-admissible} the paths
such that $\gamma\in \fkD(\Ifirst_{\thetas})$.  To use effectively this definition, it would
be convenient if one could characterize s-admissibility in terms of periodic orbits.  To
this end, given a periodic orbit $p$, let $\mnu_p$ the measure determined by the average
along the orbit of $p$.

\begin{lem}\label{lem:unconstrained} Given $\theta\in\bT$ and $b\in\bR^d$, $b\in\intr\bD(\theta)$ if
  and only if there exist $d+1$ periodic orbits $\{p_i\}$ of $f_\theta$ such that the
  convex hull of $\mnu_{p_i}(b- A(\cdot,\theta))$ contains a neighborhood of zero.  Also
  if there exists $n\in\bN$ such that
  \begin{align*}
    \inf_{x\in\bT}\left|\langle b,b-\frac 1n
    \sum_{k=0}^{n-1}A(f_\theta^k(x),\theta)\rangle\right|>0,
  \end{align*}
  then $b\not\in\bD(\theta)$.
\end{lem}
\begin{proof}
  If the convex hull contains a neighborhood of zero, then there exists $\delta>0$ such that, for all $b'\in\bR^d$, $\|b-b'\|<\delta$,
  there exists $\{\alpha_i\}_{i=1}^{d+1}\subset \bRp$, $\sum_{i=1}^{d+1}\alpha_i=1$ such
  that $\sum_{i=1}^{d+1}\alpha_i\mnu_{p_i}(A(\cdot,\theta))=b'$, hence $b'\in\bD(\theta)$
  and $b\in\bD_*(\theta)$.  On the other hand if $b\in\bD_*(\theta)$ then there are
  $\{b_i\}_{i=1}^{d+1}\subset \bD_*(\theta)$ such that $b$ belongs to the interior of
  their convex hull.  Hence there exists $\mnu_i\in\cM_\theta(b_i)$ such that their convex
  combination gives an element of $\cM_\theta(b)$.  Since the measures supported on
  periodic orbits are weakly dense in $\cM_\theta$ (see\footnote{
    In fact the proof in~\cite{Parthasarathy61} is for the invertible case but it applies
    almost verbatim to the present one.}~\cite{Parthasarathy61}) it is possible to find periodic orbits $\{p_i\}$
  such that the convex hull of $\mnu_{p_i}(A(\cdot,\theta))$ contains a neighborhood of
  $b$, hence the necessity of the condition.

  To prove the other necessary condition, note that, by~\eqref{e_derivativeChi},~\eqref{eq:stationary-sigma} reads
  \begin{equation}\label{eq:unconstraint}
    b=\mnu_{\theta,\langle \sigma, A\rangle}(A(\cdot,\theta)).
  \end{equation}
  Thus $b\in \bD_*(\theta)$ if and only if~\eqref{eq:unconstraint} has a solution.  If the second
  condition in the lemma is satisfied, then for each invariant measure $\mnu$ we have
  $|\langle b,b-\mnu(A(\cdot,\theta))\rangle|>0$, hence equation~\eqref{eq:unconstraint}
  cannot be satisfied.  Moreover, the same conclusion holds for any $b'$ in a small
  neighborhood of $b$, hence the claim.
\end{proof}

\subsection{An equivalent definition}\label{subsec:rate1}\ \newline
Unfortunately, in our subsequent discussion, the rate function will appear first in a much
less transparent form, a priori different from the definition~\eqref{eq:rate1}.  Namely,
recall the notation~\eqref{eq:kappaZ}; then for any $\sigma\in\bR^d$ and any Lipschitz
path $\gamma\in\cC^0([0,T],\bR^d)$, we introduce the shorthand notation
  \begin{align}\label{e_bzetgt_definition}
    \bzetgts(\sigma,s)= \zet(\sigma, \gamma'(s), \bar\theta(s,\thetas)).
  \end{align}
  \begin{rem}\label{rem:further-notice}
    For further use remark that Lemmata~\ref{lem:large-sup},~\ref{lem:lws}
    and~\ref{sublem:rate-f} hold verbatim if in the above definition of
    $ \bzetgts(\sigma,s)$ one substitutes $\bar\theta(s,\thetas)$ with some other
    continuous function of $s$.
  \end{rem}
  Also let us fix $C\ge 2\|A\|_{\cC^0}$ and define:
  \begin{align*}
    \Lip_{C,*}=\{\gamma\in \cC^0([0,T],\bR^d):\; \gamma(0)=0, \|\gamma(t)-\gamma(s)\|\le
    C|t-s|\;\forall\;t,s\in[0,T]\},
  \end{align*}
  Then, the functional $I_{\thetas}:\cC^0([0,T],\bR^d)\to \bR\cup\{+\infty\}$ will appear naturally, where
  \begin{equation}\label{eq:ratefunction}
    I_{\thetas}(\gamma)= \begin{cases}+\infty &\textrm{if } \gamma\not \in \Lip_{C,*}\\
      \sup\limits_{\sigma\in \BV}\int_0^T\bzetgts(\sigma(s),s) \deh{}s&\textrm{otherwise.}
    \end{cases}
  \end{equation}
  It is the task of this subsection to show that the two definitions~\eqref{eq:rate1}
  and~\eqref{eq:ratefunction} coincide.  At a superficial level, it amounts to prove that we
  can bring the $\sup$ inside the integral.  This will be proved essentially via a
  compactness argument.

  First, observe that $I_{\thetas}$ is convex, because it is the conjugate function of a
  proper function.  Moreover, $I_{\thetas}\ge 0$ (just consider $\sigma=0$ in the sup) and
  since $\hat\chi_A\ge 0$ we obtain $I_{\thetas}(\gavg(\cdot,\thetas))=0$, where recall
  that $\gavg(s,\theta)$ (defined in~\eqref{e_gavgDef}) satisfies the equation
  $\gavg'(t,\theta) = \bar A(\bar\theta(t,\theta))$.   Our first task is to show that we
  can replace the $\sup$ on $\sigma\in\BV$ with the $\sup$ on $\sigma\in L^1$.
  \begin{lem}\label{lem:large-sup}
    Let  $\gamma\in \Lip_{C,*}$; then:
    \begin{align*}
      I_{\thetas}(\gamma) &= \sup_{\sigma\in L^1}\int_0^T\bzetgts (\sigma(s),s)\deh{}s.
    \end{align*}
  \end{lem}
  \begin{proof}
    First, notice that~\eqref{e_derivativeChi} implies that
    $\|\partial_\sigma\bar \zet_{\gamma,\thetas}(\sigma,s)\|\le \Const(C+1)$ (and
    consequently $\|\bzetgts(\sigma,s)\|\le \Const(C+1) \|\sigma\|$) for
    all $s\in[0,T]$.  It follows that, for all $\gamma\in \Lip_{C,*}$ the functional
    $\sigma\mapsto\int_0^T\bzetgts(\sigma(s),s) \deh{}s$
    is continuous in the $L^1$ topology.

    Let $\sigma\in L^1$; since $\BV$ is dense in $L^1$, \cite[Theorem 2.16]{Lieb-Loss}, for any $\epsilon>0$ there exists
    $\sigma_\epsilon\in \BV$ such that $\|\sigma-\sigma_\epsilon\|\nl1 < \epsilon$ and
    thus
    \[
    \begin{split}
      \int_0^T\bzetgts(\sigma(s),s)\deh{s}
      & \le \; \Const \eps+\int_0^T\bzetgts(\sigma_\epsilon(s),s)\deh{s}
       \\
        & \le \Const \eps+\sup_{\bar\sigma\in\BV}\int_0^T\bzetgts(\bar\sigma(s),s)\deh{s}.
    \end{split}
   \]
    Taking the limit $\epsilon\to 0$ first and then sup on $\sigma\in L^1$ we have that the
    sup on $\BV$ equals the sup on $L^1$, proving the lemma.
  \end{proof}
  \begin{lem}\label{lem:lws}
    The functional $I_{\thetas}$ is lower semi-continuous on $\cC^0([0,T],\bR^d)$.
  \end{lem}
  \begin{proof}
    Consider a sequence $\{\gamma_n\}\subset\cC^0([0,T],\bR^d)$ converging uniformly to
    $\gamma$.  If $\liminf_{n\to\infty}I_{\thetas}(\gamma_n)=+\infty$, then obviously
    $\liminf_{n\to\infty}I_{\thetas}(\gamma_n)\ge I_{\thetas}(\gamma)$.  Otherwise, there
    exists a subsequence $\{\gamma_{n_j}\}$, $M>0$ and $j_0\in\bN$ such that
    \begin{align*}
      \liminf_{n\to\infty}I_{\thetas}(\gamma_n)=\lim_{j\to\infty}I_{\thetas}(\gamma_{n_j}),
    \end{align*}
    and $I_{\thetas}(\gamma_{n_j})\le M$ for all $j\ge j_0$.  This implies that if
    $j \ge j_0$, then $\gamma_{n_j}\in\Lip_{C,*}$; hence, we also conclude that
    $\gamma\in\Lip_{C,*}$.  This implies that, for any $\sigma\in L^1$,
    \begin{align*}
      \lim_{j\to\infty}\int_0^T\langle\sigma,\gamma_{n_j}'\rangle=\int_0^T\langle\sigma,\gamma'\rangle.
    \end{align*}
    In fact, for any $\epsilon>0$ there exists $\sigma_\epsilon\in\cC^1$, such that
    $\|\sigma-\sigma_\epsilon\|_{L^1}\le \epsilon$,  \cite[Theorem 2.16]{Lieb-Loss}. Then
    \begin{align*}
      \left|\int_0^T\langle\sigma,\gamma_{n_j}'\rangle-\int_0^T\langle\sigma,\gamma'\rangle\right|\le
      2C\epsilon+\left|\int_0^T\langle\sigma_\epsilon',\gamma_{n_j}-\gamma\rangle\right|+|\langle\sigma_\epsilon(T),\gamma_{n_j}(T)-\gamma(T)\rangle|.
    \end{align*}
    We conclude that, for any $\sigma\in L^1$,
    \begin{align*}
      \liminf_{n\to\infty}I_{\thetas}(\gamma_n)
      &\ge \lim_{j\to\infty}\int_0^T\left[\langle\sigma(s),\gamma'_{n_j}(s))\rangle-\chi_A(\sigma(s),\bar\theta(s,\thetas))\right] \deh{}s\\
      &=\int_0^T\left[\langle\sigma(s),\gamma'(s)\rangle-\chi_A(\sigma(s),\bar\theta(s,\thetas))\right] \deh{}s.
    \end{align*}
    The proof follows by taking the $\sup$ on $\sigma$.
  \end{proof}
  We can finally show that the definition of $I_{\thetas}$ given in the current section
  coincides with~\eqref{eq:rate1}.
  \begin{lem}\label{sublem:rate-f}%
    For any $\thetas\in\bT$, let $\gamma\in\cC^0([0,T],\bR^d)$, then:
    \begin{align*}
      I_{\thetas}(\gamma) = \Ifirst_{\thetas}(\gamma).
    \end{align*}
  \end{lem}
  \begin{proof}
    If $\gamma(0)\ne0$ or $\gamma$ is not Lipschitz, we have
    $\Ifirst_{\thetas}(\gamma) = \infty = I_{\thetas}(\gamma)$; we can thus assume $\gamma$
    to be a Lipschitz function so that $\gamma(0) = 0$.  Recall that in this case
    \begin{align*}
      \Ifirst_{\thetas}(\gamma) = \int_0^T\cZ(\gamma'(s),\bar\theta(s,\thetas)) \deh{}s.
    \end{align*}
    If $\gamma\not\in\Lip_{C,*}$, then, provided $C$ has been chosen large enough, there
    is a positive measure set in which $\gamma'(t)\not\in \bD( \bar\theta(t,\thetas))$ and hence
    $\int_0^T\cZ(\gamma'(s),\bar\theta(s,\thetas)) \deh{}s=\infty$, which coincides with
    $I_{\thetas}$.  We can then assume $\gamma\in\Lip_{C,*}$.

    Observe that by definition we have $I_{\thetas}(\gamma)\le\Ifirst_{\thetas}(\gamma)$;
    it just suffices to prove the reverse inequality.

    Suppose first that $\Ifirst_{\thetas}(\gamma) = \infty$: we want to show that
    $I_{\thetas}(\gamma) = \infty$. Define $z_*(s)=\cZ(\gamma'(s),\bar\theta(s,\thetas))$;
    by assumption $z_*\not\in L^1[0,T]$.  Let us fix arbitrarily $M>0$; by Lusin Theorem
    and Lebesgue monotone convergence Theorem there exists $\lambda>0$ and a compact set
    $E$ such that $\gamma'$ and $\min\{\lambda,z_*(t)\}$ are continuous on $E$ and
    $\int_E \min\{\lambda,z_*(t)\}\deh{}t\ge M$.  Then for $t\in E$ let
    $\sigma_\lambda(t)$ be such that
    $\bzetgts(\sigma_\lambda(t), t)\ge \frac 12\min\{\lambda, z_*(t)\}$.  Since
    $\bzetgts(\sigma_\lambda(t), s)$ is continuous in $s\in E$, it follows that, for all $t\in E$,
    there exists and open set $U(t)\ni t$ such that
    $\bzetgts(\sigma_\lambda(t), s)\ge \frac 14\min\{\lambda, z_*(s)\}$ for all
    $s\in U(t)\cap E$.  We can then extract a finite sub cover $\{U(t_i)\}$ of $E$ and
    define
    \begin{align*}
      \bar\sigma_\lambda(s) =
      \begin{cases}
        \sigma_\lambda(t_{k(s)})&\textrm{if $s\in E$, where $k(s)=\inf \{i: s\in U(t_i)\}$}\\
        0&\textrm{if $s\not\in E$}.
      \end{cases}
    \end{align*}
    By construction $\bar\sigma_\lambda\in L^\infty$ and
    $\bzetgts(\bar\sigma_\lambda(t), t)\ge \frac 14 \min\{\lambda, z_*(t)\}$ for each
    $t\in E$.  Accordingly, setting $z_\lambda(t)=\Id_E(t)\cdot\min\{\lambda, z_*(t)\}$,
    by Lemma~\ref{lem:large-sup} we have
    \begin{align*}
      I_{\thetas}(\gamma)\ge \int_0^T \bzetgts(\bar\sigma_\lambda(s), s)\deh{}s\ge\frac
      14\int_0^Tz_\lambda(s) \deh{}s\ge \frac M4.
    \end{align*}
    By the arbitrariness of $M$ it follows $I_{\thetas}(\gamma)=+\infty$.

    On the other hand, if $z_*\in L^1$, then by Lemma~\ref{lem:large-sup}
    \begin{equation}\label{eq:upperZ}
      I_{\thetas}(\gamma)=\sup_{\sigma\in L^1}\int_0^T \zet(\sigma(s),
      \gamma'(s), \bar\theta(s,\thetas)) \deh{}s\le \int_0^T z_*(s) \deh{}s<+\infty,
    \end{equation}
and $\gamma'(s)\in \bD(\bar\theta(s,\thetas))$ for almost every $s\in [0,T]$.
    For $\convexP\in(0,1)$ and $s\in[0,T]$ let us define the convex combination
    \begin{align*}
      \gamma_\convexP(s)=(1-\convexP)\gamma(s)+\convexP\gavg(s,\thetas).
    \end{align*}
    Since $\gavg'(s,\thetas) = A(\bar\theta(s,\thetas))\in\intr\bD(\bar\theta(s,\thetas))$ it follows that, for any $\convexP\in(0,1)$ and $s\in [0,T]$, there exists a compact set
    $K(\convexP, s)\subset\bD_*(\bar\theta(s,\theta))$, such that $\gamma'_\convexP(s)\in K(\convexP, s)$ for
    almost all $s\in [0,T]$.  By Lemma~\ref{lem:hausdorff} such compacts depend continuously on $s$. Since the inverse of $\partial_\sigma\chi_A(\cdot,\theta)$
    is a continuous function with depends continuously on $\theta$, it follows that the preimages of $K(\convexP, s)$ are all contained in a fixed compact
    set $K_\convexP$.  Hence, there exists $\sigma_\convexP\in L^\infty$ such that
    $\gamma_\convexP'(s)=\partial_\sigma\chi_A(\sigma_\convexP(s),\bar\theta(s,\thetas))$
    for almost all $s\in[0,T]$.
    \begin{align*}
      I_{\thetas}(\gamma)
      &=\lim_{\convexP\to 0}(1-\convexP) I_{\thetas}(\gamma)\ge \liminf_{\convexP\to 0} I_{\thetas}(\gamma_\convexP)\ge \liminf_{\convexP \to 0} \int_0^T \zet(\sigma_\convexP(s), \gamma_\convexP'(s), \bar\theta(s,\thetas)) \deh{}s\\
      &=\liminf_{\convexP\to 0} \int_0^T \cZ(\gamma_\convexP'(s), \bar\theta(s,\thetas)) \deh{}s\ge \int_0^T \liminf_{\convexP\to 0}\cZ(\gamma_\convexP'(s), \bar\theta(s,\thetas)) \deh{}s\\
      &\ge \int_0^T \cZ(\gamma'(s), \bar\theta(s,\thetas)) \deh{}s = \Ifirst_{\thetas}(\gamma),
    \end{align*}
    where we have used the convexity of $I_{\thetas}$ first, then
    Lemma~\ref{lem:large-sup}, then equation~\eqref{eq:good-zeta}, then Fatou's Lemma and
    finally Lemma~\ref{lem:domainZ}\ref{i_Zconvex}.  The above concludes the proof.
  \end{proof}
\subsection{Definition and properties: the rate function}\label{subsec:ratefinal}\ \newline
Lemma~\ref{lem:unconstrained} tells us that it might be difficult to exactly determine
the boundary of the effective domain of the rate function, i.e.\ to distinguish between
impossible and almost impossible paths. To circumvent this problem it is convenient to
thicken the boundary of the effective domain by slightly modifying the rate function.

For any $\epsilon>0$ let  $\partial_\epsilon\bD(\theta)=\{b\in\bR^d\;:\; \dist(b,\partial\bD(\theta))< \epsilon\}$ and define
  \begin{equation}\label{eq:zeta-reg-def}
    \begin{split}
      &\cZ^+_\epsilon(b,\theta)=\begin{cases} \cZ(b,\theta) \quad&\textrm{ if } b\not\in  \partial_\epsilon\bD(\theta)\\
        +\infty&\textrm{ otherwise,}
      \end{cases}\\
      &\cZ^-_\epsilon(b,\theta)=\begin{cases} \cZ(b,\theta) \quad&\textrm{ if }b\not\in  \overline{\partial_\epsilon\bD(\theta)}\\
        \cZ(\bar A(\theta) + \convexP_{b,\theta}(b-\bar A(\theta)),\theta) & \textrm{ otherwise,}
      \end{cases}
    \end{split}
  \end{equation}
  where
  $\convexP_{b,\theta}=\sup\{\convexP > 0\st \bar A(\theta)+\convexP(b-\bar A(\theta))\in
  \bD(\theta)\setminus \partial_\epsilon\bD(\theta)\}$.  We will conveniently assume that
  $\epsilon$ is so small that for any $\theta\in\bT^1$ and
  $b\in\partial_\epsilon\bD(\theta)$ we have $\convexP_{b,\theta} > 1/2$.  Note that, by
  Lemma~\ref{lem:unconstrained}, the set $\partial_\epsilon\bD(\theta)$ can be explicitly
  determined for arbitrarily small $\epsilon$ by computing longer and longer periodic
  orbits and ergodic averages (see \cite{Collier-Morris} for a discussions on the speed of
  such approximation).  Moreover, for any $\eps' < \eps$ we have:
  \begin{align*}
    \cZ_\eps^- < \cZ_{\eps'}^- < \cZ < \cZ_{\eps'}^+ < \cZ^+_{\eps}.
  \end{align*}
\begin{rem}\label{rem:quadratic} Note that
  Lemmata~\ref{lem:domainZ} and~\ref{lem:entropy} show that $\bD(\theta)$ is a convex
  compact non-empty set (in fact,~\eqref{e_entropy} implies
  $\sup_{\theta\in\bT^1}\sup_{b\in\bD(\theta)}\cZ(b,\theta)<\infty$). Moreover, they
  characterize $\partial\bD(\theta)$ as those values that can be attained as averages of
  $A$ with respect to an invariant measure which is not associated to a transfer operator
  of type~\eqref{eq:Lopzero}.  Hence, again by Lemma~\ref{lem:domainZ}, there exists
  $\Sigma^+> \Sigma^->0$ (as quadratic forms) such that, for any $\epsilon$ small enough,
  \begin{align*}
    \cZ^+_\epsilon(b,\theta)&\ge \langle b-\bar A(\theta), \Sigma^- (b-\bar A(\theta))\rangle\\
    \cZ^-_\epsilon(b,\theta)&\le \langle b-\bar A(\theta), \Sigma^+ (b-\bar A(\theta))\rangle.
  \end{align*}
\end{rem}
The rate function defined in the previous sections would suffice to describe deviations
from the average behavior for relatively short times.  If we want to study longer times,
then we must consider slightly different rate functions.  That is, for any $\thetas\in\bT$
and $\gamma\in\cC^0([0,T];\bR^d)$, let $\theta^\gamma(s,\thetas) = \thetas+(\gamma(s))_1$
(where recall that $(\gamma(s))_1$ denotes the first component of the vector $\gamma(s)$).
Then for any $\epsilon>0$:
\begin{equation}\label{eq:rate2}
\begin{split}
&  \ratef_{\thetas}(\gamma)=\begin{cases} +\infty&\textrm{if }\gamma \not\in\textrm{ Lipschitz, or $\gamma(0)\neq0$}\\
    \int_0^T \cZ(\gamma'(s), \theta^\gamma(s,\thetas))\, \deh s\;&\textrm{otherwise};
  \end{cases}\\
&  \ratef_{\thetas,\epsilon}^\pm(\gamma)=\begin{cases} +\infty&\textrm{if }\gamma \not\in\textrm{ Lipschitz, or $\gamma(0)\neq0$}\\
    \int_0^T \cZ^\pm_\epsilon(\gamma'(s), \theta^\gamma(s,\thetas))\, \deh s\;&\textrm{otherwise}.
  \end{cases}
\end{split}
\end{equation}
Finally, we define
\begin{align}\label{e_defRatefPm}
  \ratef^\pm_{\thetas}(\gamma)=\lim_{\epsilon\to 0}
  \ratef_{{\thetas},\epsilon}^\pm(\gamma).
\end{align}
We stress the important difference with the definition of $I_\theta$ which comes from the
fact that the function $\cZ$ is now calculated along the actual path
$\theta^{\gamma}$ rather than the averaged path $\bar\theta$.

\begin{rem}\label{rem:effective}
  Note that, in general, $\ratef_{\thetas}$ is not convex.  The effective domain
  $\fkD(\ratef_{\thetas})=\{\gamma\in\cC^0\st\ratef_{\thetas}(\gamma)<\infty\}$ is given
  by those paths $\gamma$ that are $C$-Lipschitz, $\gamma(0)=0$ and such that
  $\cM_{\gamma_1(t)}(\gamma'(t))\neq \emptyset$ for almost all $t\in[0,T]$.  By lower
  semicontinuity (which we prove shortly) $\fkD(\ratef_{\thetas})$ is closed, and hence
  compact by Ascoli-Arzel\`a, in $\cC^0$; however it has empty interior. It is therefore
  more convenient to consider $\fkD(\ratef_{\thetas})$ as a subset of the Lipschitz
  functions with the associated topology. Then the interior and the boundary are non
  trivial and this is the topology we will always consider for the effective domains
  otherwise differently stated. The effective domain of $\ratef^+_{\thetas}$ is given by
  $\intr\fkD(\ratef_{\thetas})$.  If $\gamma\in\fkD(\ratef^+_{\thetas})$ we say that
  $\gamma$ is {\em admissible}.  Note that the two functionals only differ on
  $\partial \fkD(\ratef_{\thetas})$.
\end{rem}

\begin{lem}\label{lem:rate-lower}
  For any $\thetas\in\bT$ and $\epsilon$ sufficiently small, the rate functions
  $\ratef_{\thetas}, \ratef^-_{\thetas,\epsilon}$ are lower-semicontinuous and
  $\ratef_{\thetas}=\ratef^-_{\thetas}$. Also, $\partial\fkD(\ratef_{\thetas})$ has empty
  interior and, for each $\gamma \in \partial\fkD(\ratef_{\thetas})$, there exists a
  sequence $\{\gamma_n\}\subset \intr\fkD(\ratef_{\thetas})$ such that
  $\lim_{n\to\infty}\ratef_{\thetas}(\gamma_n)=\ratef_{\thetas}(\gamma)$.
\end{lem}
\begin{proof}
  Let $\{\gamma_n\}\in \cC^0([0,T],\bR^d)$ be a converging sequence and call $\gamma_*$
  its limit. We start proving that:
  \[
    \liminf_{n\to\infty}\ratef_{\theta^*}(\gamma_n)\geq \ratef_{\theta^*}(\gamma_*).
  \]
  If the left hand side equals $+\infty$, then the statement is obviously true; if not,
  then there exists a subsequence such that $\{\gamma_{n_j}\}\subset \Lip_{C,*}$, hence
  $\gamma_*\in \Lip_{C,*}$.  From now on the proof follows very closely the argument in
  Section~\ref{subsec:rate1} (recall Remark~\ref{rem:further-notice}): for
  $\gamma\in \Lip_{C,*}$ define
  \[
    \ratef_{\theta_*,\textrm{pre}}(\gamma)= \sup_{\sigma\in L^1}\int_0^T\langle \sigma(s),\gamma'(s)\rangle-\chi_A(\sigma(s),\theta^\gamma(s))\deh{}s
  \]
  Then arguing as in Lemma~\ref{lem:lws} it follows that $\ratef_{\theta_*,\textrm{pre}}$ is
  lower semicontinuous. The only difference being in the last display of the proof, since
  now the second argument of $\chi_A$ depends on $\gamma$, which can be controlled using
  Lemma~\ref{lem:non-pert-theta}.  Then the equivalent of Lemma~\ref{sublem:rate-f} holds
  verbatim whereby establishing $\ratef_{\theta_*,\textrm{pre}}=\ratef_{\theta_*}$. The
  argument for $\ratef^-_{\thetas,\epsilon}$ (for arbitrary $\epsilon > 0$) is more of the
  same.

  Next, by the convexity of $\cZ(\cdot, \theta)$,
  $\ratef^-_{\thetas,\epsilon}\leq \ratef^-_{\thetas}$ and
  $\ratef^-_{\thetas,\epsilon}(\gamma)=\int_0^T\cZ(\gamma_\epsilon',\theta^{\gamma})$,
  where $\lim_{\epsilon\to 0}\|\gamma'_\epsilon-\gamma'\|\nc0=0$.  Fatou Lemma and the
  lower semicontinuity of $\cZ(\cdot, \theta)$ then imply that
  $\ratef_{\thetas}=\ratef^-_{\thetas}$.

  Further, suppose $\gamma\in \partial\fkD(\ratef_{\thetas})$.  Fix an
  arbitrary $\delta>0$ and, for all $\alpha>1$ we define the path $\gamma_\alpha$ as the
  unique solution of the ODE
  \[
    \begin{split}
      &\gamma_\alpha'(s)=(1-\delta e^{-\alpha(T-s)})\gamma'(s)+ \delta e^{-\alpha(T-s)}\bar A(\theta^{\gamma_\alpha}(s))\\
      &\gamma_\alpha(0)=0.
    \end{split}
  \]
  Then $\|\gamma_\alpha-\gamma\|\nc0\leq 2C\delta\alpha^{-1}e^{-\alpha (T-s)}$ and $\|\gamma'_\alpha-\gamma'\|\nc0\leq 2C\delta e^{-\alpha (T-s)}$. Also, by Lemma ~\ref{lem:hausdorff},
  \[
    d(\gamma_\alpha,\partial\bD(\theta^{\gamma_\alpha}(s)))\geq d(\gamma_\alpha,\partial\bD(\theta^{\gamma}(s))-c_42C\delta\alpha^{-1}e^{-\alpha (T-s)}.
  \]
  On the other hand by Lemma~\ref{lem:domainZ}-\ref{i_nbhdSRB} and Lemma~\ref{lem:hausdorff} the distance between $\bar A(\theta^{\gamma}(s))$ and $\partial\bD(\theta^\gamma(s))$ is continuous and strictly positive, hence it has a minimum $\tau>0$. Accordingly,  provided $\delta$ is small enough,\footnote { Just note that $\bD(\theta^{\gamma}(s))$ must contain a right triangle with vertexes $\gamma'(s)$ and $\bar A(\theta^{\gamma}(s))$ and base of length at least $\tau$.}
  \[
    \begin{split}
      d(\gamma_\alpha,\partial\bD(\theta^{\gamma}(s)))\geq & d((1-\delta e^{-\alpha(T-s)})\gamma'(s)+ \delta e^{-\alpha(T-s)}\bar A(\theta^{\gamma}(s)),\partial\bD(\theta^{\gamma}(s)))\\
      &-\Const \delta^2\alpha^{-1} e^{-\alpha (T-s)}\geq \Const \tau \delta e^{-\alpha(T-s)}.
    \end{split}
  \]
  Thus, by choosing $\alpha$ large enough, we have $d(\gamma_\alpha,\partial\bD(\theta^{\gamma_\alpha}(s)))\geq c_T\delta$.
  Accordingly, $\gamma_\alpha\in  \fkD(\ratef_{\thetas})\setminus\partial \fkD(\ratef_{\thetas})$, thus $\intr\partial \fkD(\ratef_{\thetas})=\emptyset$.

  Finally, by~\eqref{eq:good-zeta},~\eqref{eq:kappaZ}, Lemma~\ref{lem:non-pert-theta} and
  Lemma~\ref{l_gamma-bound}, setting $\lambda(s)=\delta e^{-\alpha(T-s)}$ and choosing
  $\alpha$ large enough, we have
  \[
    \begin{split}
      \ratef_{\thetas}(\gamma_\alpha)&=\int_0^T\langle \sigma_\alpha,\hat\gamma'_\alpha\rangle-\hat\chi_A(\sigma_\alpha,\theta^{\gamma_\alpha})\\
      & \leq\int_0^T (1-\lambda)\langle \sigma_\alpha,\hat\gamma'\rangle-\hat\chi_A(\sigma_\alpha,\theta^{\gamma})+\Const \min\{\|\sigma_\alpha^2\|,\|\sigma_\alpha\|\}\,\|\gamma-\gamma_\alpha\|\\
      &\leq\int_0^T\langle \sigma_\alpha,\hat\gamma'\rangle-\hat\chi_A(\sigma_\alpha,\theta^{\gamma})-\Const\min\{\|\sigma_\alpha^2\|,\|\sigma_\alpha\|\}\,\left[\lambda -\Const\lambda \alpha^{-1}\right]\\
      &\leq \sup_{\sigma\in L^1}\int_0^T\langle \sigma, \gamma'\rangle-\chi_A(\sigma,\theta^{\gamma})=\ratef_{\thetas}(\gamma).
    \end{split}
  \]
  The then lemma follows by the lower semicontinuity of $\ratef_{\thetas}$.
\end{proof}

We conclude this section with a useful estimate:
\begin{lem}\label{l_Jgamma-bound}
  For any $\theta\in\bT$ and $\gamma\in\Lip_{C,*}([0,T],\bR^d)$:
  \begin{align}
    \ratef_{\theta,\epsilon}^\pm(\gamma)%
    &\ge \Const\|\gamma' - \bar A(\theta^\gamma(\cdot,\theta))\|\nl2^2.
  \end{align}
\end{lem}
\begin{proof}
  Let us fix $\gamma$ and introduce the shorthand notation
  $\hat\chi_{A}(\sigma) = \hat\chi_A(\sigma(\cdot),\theta^\gamma(\cdot))$; let $\vgamma$
  be so that $\vgamma' = \gamma'(s) - \bar A(\theta^\gamma(s))$.  Observe that if
  $\ratef^\pm_{\theta,\epsilon}(\gamma) = \infty$ then the statement trivially holds;
  hence let us assume that this is not the case and fix $s\in[0,T]$ so that
  $\cZ_\epsilon^-(\gamma'(s),\theta^\gamma(s)) < \infty$.  Then, for any
  $\convexP\in[0,1/2)$, by the definition~\eqref{eq:zeta-reg-def} of $\cZ_\epsilon^-$ and
  the smallness condition on $\epsilon$, we can define $\bar\sigma_\convexP(s)$ to be the
  solution of $\convexP\vgamma'(s)=\partial_\sigma\hat\chi_A(\bar\sigma_\convexP(s))$.
  Define, moreover
  \begin{align*}
    \vf(s,\convexP) = \langle \bar\sigma_\convexP(s), \convexP \vgamma'(s)
    \rangle - \hat\chi_A(\bar\sigma_\convexP(s)).
  \end{align*}
  Observe that $\bar\sigma_0=0$ and
  $\partial_\convexP\vf(s,\convexP)=\langle
  {\bar\sigma_\convexP(s)},{\vgamma'(s)}\rangle$, finally:
  \begin{align*}
    \partial_\convexP^2\vf(s,\convexP) %
    = \langle {\partial_\convexP\bar\sigma_\convexP(s)},{\vgamma'(s)}\rangle%
    = \langle \partial_\convexP\bar\sigma_\convexP(s)
    ,\partial_\sigma^2\hat\chi_A(\bar\sigma_\convexP)\partial_\convexP\bar\sigma_\convexP(s)\rangle \ge 0.
  \end{align*}
  In particular $\vf(s,\cdot)$ is increasing; hence, using once again the
  definition~\eqref{eq:zeta-reg-def} of $\cZ_\epsilon^-$ and the smallness condition on
  $\epsilon$, we conclude that $\cZ_\epsilon^-(\gamma(s),\theta^\gamma(s))\ge \vf(s,1/2)$.

  Recall moreover there exists $c>0$ such that
  $\inf_{\|\sigma\|\leq 1}\partial_\sigma^2\hat\chi_A(\sigma)\geq c\Id$ (as quadratic
  forms).  Let
  \begin{align*}
    \convexP_0 = \max\{\convexP\in[0,1/2]\st \|\bar\sigma_{\convexP'}(s)\|\le 1 \text{ for all $\convexP'\in[0,\convexP]$}\}.
  \end{align*}
  Since $\gamma\in\Lip_{C,*}([0,T],\bR^d)$ we have
  $ c\|\partial_\convexP\bar\sigma_\convexP(s)\|\le \|\vgamma'(s)\|\le 2C$ for all
  $\convexP\in[0,\convexP_0]$.  Hence, either $\convexP_0 = 1/2$ or, otherwise,
  $1=\|\bar\sigma_{\convexP_0}(s)\|\leq 2Cc^{-1}\convexP_0$.  In any case we have
  $\convexP_0\geq \Const$. We thus conclude:
  \begin{align*}
    \vf(s,1/2)\ge\vf(s,\convexP_0)
    &= \int_0^{\convexP_0}\deh\convexP\int_0^\convexP \deh\mnu \langle \vgamma'(s),\partial^2_\sigma\hat\chi_A(\bar\sigma_\mnu(s))\invr\vgamma'(s)\rangle \ge \Const \|\vgamma'(s)\|^2.
  \end{align*}
  Since $\ratef^-_{\theta,\epsilon}(\gamma) < \infty$ we conclude that $s$ can be chosen
  in a full-measure set in $[0,T]$; hence the above estimate holds a.e., which concludes
  the proof of our lemma since $\ratef^+_{\theta,\epsilon}\ge\ratef^-_{\theta,\epsilon}$.
\end{proof}
Note that a similar, but simpler, argument shows that
\begin{align}\label{eq:Igamma-bound}
  I_\theta(\gamma)\ge\Const\|\gamma'-\bar A(\bar\theta(\cdot,\theta))\|\nl2^2.
\end{align}
\section{Deviations from the average: Large Deviations}\label{sec:LargeDeviations}
We are now at last ready to precisely state and prove our Large Deviations results.

Let us recall the definition~\eqref{e_gpolyDef} of the random element $\gpoly(t)$; observe
that equivalently, we have:
\begin{equation}\label{eq:path-def}
  \gpoly(t)=\ve\sum_{j=0}^{\pint{t\vei}-1} A\circ F_ \ve^j(x,\theta) +
  (t-\ve\pint{t\vei})A\circ F_ \ve^{\pint{t\vei}}(x,\theta).
\end{equation}
Recall that
$\gpoly\in\cC_*^0([0,T],\bR^d):=\{\gamma\in\cC^0([0,T],\bR^d)\st \gamma(0)=0\}$; moreover
$\gpoly(t)=(\theta_\ve(t)-\theta,\zeta_\ve(t))$ and the family $\{\gpoly\}$ is
uniformly Lipschitz of constant $\|A\|\nc0$: in fact it is differentiable at all $t\not\in\ve\bZ$.

Given a standard pair $\ell$ we can consider $\gpoly$ as a random element of
$\cC_*^0([0,T],\bR^d)$ by assuming that $(x,\theta)$ are distributed according to $\ell$.
In fact, as already mentioned in Section~\ref{sec:results}, it is more convenient to work
directly in the probability space $\cC^0([0,T],\bR^d)$ endowed with the probability
measure $\ppath_\ell$ determined by the law of $\gpoly$ under $\ell$, that is
$\ppath_\ell = (\gpoly)_*\mu_\ell$.
  In particular, for any function
$g\in\cC^0(\bR^{d},\bR)$, $k\in\bN$ and standard pair $\ell$:\footnote{ According to the
  usual probabilistic notation $\gamma(t)$ stands both for the numerical value of the path
  $\gamma$ at time $t$ and for the evaluation functional
  $\gamma(t):\cC^0([0,T],\bR^d)\to \bR^d$ defined by
  $\gamma(t)(\tilde \gamma)=\tilde\gamma(t)$, for all
  $\tilde\gamma\in \cC^0([0,T],\bR^d)$.}
\begin{align*}
  \epath_\ell(g\circ \gamma(k\ve)) =
  \mu_\ell(g(\gpoly(k\ve))=\mu_\ell(g\circ \bF_{\ve}^k)
\end{align*}
where $\epath_\ell$ is the expectation associated to the probability $\ppath_\ell$ and
$\bF_\ve$ is defined in~\eqref{eq:map-full}.
\begin{rem}\label{r_lipschitzPath}
  By the above mentioned Lipschitz property of the paths $\gpoly$ and since
  $\gpoly\in\cC_*^0([0,T];\bR^d)$ we conclude that the support of $\ppath_\ell$ is
  contained in a compact set that is independent on $\ve$ and $\ell$; in particular the
  family $\{\ppath_\ell : \ve > 0, \ell \textrm{ standard pair}\}$ is tight.  More
  precisely, for any $C>\|A\|\nc0$ we have that $\ppath_\ell(\Lip_{C,*})=1$ where
  \begin{align*}
    \Lip_{C,*}=\{\gamma\in \cC_*^0([0,T],\bR^d):\; \|\gamma(t)-\gamma(s)\|\le
    C|t-s|\;\forall\;t,s\in[0,T]\} .
  \end{align*}
\end{rem}
 For any standard pair $\ell$ recall (see~\eqref{eq:thetastartdef}) that we defined
$\thetasl=\int_a^b\rho(x)G(x) \deh{}x$, to be the average of the random variable $\theta_0$.
If we let $\ve\to 0$ and consider standard pairs $\ell_\ve$ (each standard with respect to
the corresponding $\ve$) with fixed $\thetasl$, Theorem~\ref{t_averaging} implies that
$\ppath_{\ell_\ve}$ converges, as $\ve\to 0$, to a measure supported on the single path
$\gavg(t)$ defined in~\eqref{e_gavgDef}.  The goal of this section is to establish
estimates for the deviations from this path.

We begin with Sections~\ref{subset:rate-up-short} and~\ref{ss_lowerBound} where
we establish large deviations results that are optimal only for relatively short times.
Then in Section~\ref{subsec:large-long} we use such preliminary results to prove
Theorem~\ref{thm:large}; finally in Section~\ref{s_proofCorollaries} we prove the
remaining propositions stated in Section~\ref{sec:results}.
\subsection{Upper bound for arbitrary sets (short times)}\label{subset:rate-up-short}\ \newline
For any $\gamma\in C^0([0,T];\bR^d)$, let $\gball{}(\gamma,r)$ denote the $\cC^0$-ball of
radius $r$ centered at $\gamma$.\footnote{ We prefer not to write the explicit dependence of $B$ on $T$
  since this can be recovered by the fact that $\gamma\in\cC^0([0,T],\bR^d)$.}  For each measurable set (event) $Q\subset \cC^0([0,T],\bR^d)$ define $\QLipC= Q\cap\Lip_{C,*}$.  By
Remark~\ref{r_lipschitzPath}, we conclude that for any $C$ sufficiently large:
\begin{equation}\label{eq:lip-supported}
  \ppath_\ell(Q\setminus \QLipC)=0.
\end{equation}
\newcommand{\centgamma}{\underline\gamma}
\newcommand{\csigma}{{\underline\sigma}}
\newcommand{\gengamma}{\gamma}
\begin{lem}[Upper bound]\label{lem:upper}
  There exist $C_0 > 0$ and $\ve_0,\Tmax\in(0,1]$ such that, for all $\ve\le\ve_0$,
  $T\in [\ve_0^{-4}\ve,\Tmax]$,
  and $Q\subset\cC^0([0,T],\bR^d)$, for any $\theta\in\bT$ and standard pair $\ell$
  with $\thetasl = \theta$.
  \begin{align*}
    \ve\log\ppath_\ell(Q)
    &\le {-\inf_{\gengamma\in Q_{\ve,+}} I_\theta(\gengamma)},
  \end{align*}
  where $I_\theta$ is defined in~\eqref{eq:ratefunction},
  $Q_{\ve,+}=\bigcup_{\gamma\in \overline Q}\gball{}(\gamma, R_\ve(\gamma))$ with
\begin{align*}
      R_\ve(\gamma) = C_0
      \max\left\{(\ve^{\efrac 14}T^{-\efrac 14}+T)\|\hgamma\|\nl\infty,%
      \min \left\{ \ve\eefrac14 T\eefrac34,
      (\ve T)\eefrac16\|\hgamma\|\nl\infty\eefrac23\right\},
      \sqrt{\ve T}\right\}
    \end{align*}
    and $\hgamma=\gamma-\gavg(\cdot,\theta)$, the latter being defined in~\eqref{e_gavgDef}.
\end{lem}
\begin{proof}
  For any linear functional
  $\gfunctional\in \cM^d([0,T])=\cC^0([0,T],\bR^d)'=\left[\cC^0([0,T],\bR)'\right]^d$,
  recalling~\eqref{eq:lip-supported}, we have
  \begin{align}\label{eq:largedev1}
    \ppath_\ell(Q)
    & \le \epath_\ell\left( \Id_{\QLipC}e^{\gfunctional-\inf_{\gengamma\in \QLipC}\gfunctional(\gengamma)}\right)\\
    &\le \exp\left[{-\inf_{\gengamma\in \overline\QLipC}\gfunctional(\gengamma)}\right] \epath_\ell\left(e^{\gfunctional}\right),\notag
  \end{align}
  where $\epath_\ell(e^\gfunctional)$ denotes the expectation of
  $\gamma\mapsto \exp(\gfunctional(\gamma))$ with respect to the probability
  $\ppath_\ell$.  Let $\logmgf$ be the \emph{logarithmic moment generating function} and
  $\logmgf^*$ be its convex conjugate function, i.e.:\footnote{ We will see shortly,
    in~\eqref{eq:largedev3}, that $\logmgf$ agrees with the previous
    definition~\eqref{e_logmgfDefinitionSigma}, hence justifying the abuse of notations
    (in one case we have a functional on measures, in the other a functional on $\BV$).}
  \begin{align}\label{eq:legendre}
    \logmgf(\gfunctional) &=\ve\log\epath_\ell\left( e^{\gfunctional}\right);
    & \logmgf^*(\gamma)&=\sup_{\gfunctional\in\cM^d([0,T])}(\ve\gfunctional(\gamma) -\Lambda_{\ell,\ve}(\gfunctional)).
  \end{align}
  Note that $|\logmgf(\gfunctional)|\le \ve \Const\|\gfunctional\|<\infty$, hence $\logmgf$ is a
  proper convex function.\footnote{ The first assertion follows by~\eqref{eq:lip-supported} which implies $\|\gamma\|_\infty\leq  CT$, $\ppath_\ell$-a.s.. The second follows from the H\"older inequality since, for all
    $t\in[0,1]$, and $\gfunctional,\gfunctional'\in \cM([0,T])$,
    \begin{align*}
      \logmgf(t\gfunctional+(1-t)\gfunctional')
      &=\ve\log\bE_{\ell,A,\ve}\left( [e^{\gfunctional}]^t[e^{\gfunctional'}]^{1-t}\right)\\
      &\le \ve\log\left[\bE_{\ell,A,\ve}\left(e^{\gfunctional}\right)^t\bE_{\ell,A,\ve}\left(e^{\gfunctional'}\right)^{1-t}\right]=
        t\logmgf(\gfunctional)+(1-t)\logmgf(\gfunctional').
    \end{align*}}
  Since $\logmgf(0)=0$, we have $\logmgf^*\ge 0$.
  Moreover   $\logmgf^*:\cC^0([0,T],\bR^d)\to\bR\cup\{+\infty\}$ is convex as well and lower
  semi-continuous (with respect to the $C^0$ topology), since it is the conjugate
  function of a proper function.  We can then follow the strategy of~\cite[Exercise
  4.5.5]{DemboZeitouni}.  Note that
  $(\gfunctional,\gamma)\mapsto\ve\gfunctional(\gamma)-\logmgf(\gfunctional)$ is a
  function concave in $\gfunctional$, continuous in $\gamma$ with respect to the $\cC^0$
  topology, also it is convex in $\gamma$ for any $\gfunctional\in \cM([0,T])$.  Finally,
  $\overline\QLipC$ is compact in $\cC^0$.  Thus, the Minimax Theorem
  (\cite{Sion58}, but see~\cite{Komiya88} for an elementary proof) guarantees that
  \begin{align*}
    \sup_{\gfunctional\in\cM^d}\inf_{\gamma\in \overline\QLipC}[\ve\gfunctional(\gamma)-\logmgf(\gfunctional)]=\inf_{\gamma\in \overline\QLipC}\sup_{\gfunctional\in\cM^d}[\ve\gfunctional(\gamma)-\logmgf(\gfunctional)].
  \end{align*}
  The above implies, taking the $\inf$ on $\gfunctional$ in~\eqref{eq:largedev1},
  \begin{equation}\label{eq:largedev2}
    \begin{split}
      \ve\log\ppath_\ell(Q)&\le {-\inf_{\gamma\in \overline\QLipC}\sup_{\gfunctional\in\cM^d}\ve\gfunctional(\gamma) -\logmgf(\gfunctional) }\\
      &\le {-\inf_{\gamma\in \overline\QLipC}\logmgf^*(\gamma)}.
    \end{split}
  \end{equation}
  The above estimate looks indeed quite promising, but unfortunately it is completely
  useless without sharp information on $\logmgf^*$.

  We are thus left with the task of computing $\logmgf^*$.  It turns out to be convenient
  to associate to $\gfunctional$ the function $\sigma$ defined as:
  \begin{align}\label{e_cdlg}
    \sigma(s)=\ve\gfunctional((s,T]),
  \end{align}
  where the right hand side is interpreted by applying the Jordan Decomposition to
  $\gfunctional$.  Note that, by definition, $\sigma=(\sigma_1,\cdots,\sigma_d)$ with
  $\sigma_i\in \BV$, thus $\sigma\in \BV^d$, that we will simply call $\BV$ to ease
  notation.  By definition $\|\gfunctional\|=\vei \|\sigma\|\nBV$ and, for any
  $\gamma\in \Lip_{C,*}$,
  \begin{equation}\label{eq:sigma-nu}
    \gfunctional(\gamma)=\vei\int_0^T\langle{\sigma(s)},{\gamma'(s)}\rangle \deh{}s.
  \end{equation}
  On the other hand, for each $\sigma\in \BV$ there exists $\gfunctional\in\cM^d([0,T])$ such that~\eqref{eq:sigma-nu} holds, see \cite[Section 5.1, Theorem 1]{EvansGariepy}.
  Moreover (recall
  definition~\eqref{eq:path-def}):\footnote{ As we often do in this work, we are neglecting the contribution
    of the fact that $T\vei$ may not be an integer.}
  \begin{align*}
    \gfunctional(\gpoly)&=\vei\int_0^T\langle{\sigma(s)},{\gpoly'(s)}\rangle \deh s\\
                &=\sum_{k=0}^{\pint{T\vei}-1}\langle {\vei\int_{k\ve}^{(k+1)\ve}\sigma(s)\deh s},{A\circ F_\ve^{k}}\rangle
                =\sum_{k=0}^{\pint{T\vei}-1}\langle {\sigma_k},{A\circ F_\ve^{k}}\rangle.
  \end{align*}
  where, as in Section~\ref{ss_exponentialMoment}, we introduced the notation
  \begin{align*}
    \sigma_n = \vei\int_{n\ve}^{(n+1)\ve}\sigma(s)ds.
  \end{align*}
Hence, we conclude that for any $\gfunctional\in\cM^d$ and for the corresponding
  $\sigma\in \BV$:
  \begin{align}\notag
    \logmgf(\gfunctional)
    &=\ve\log\mu_\ell\left(e^{\sum_{k=0}^{\pint{T\vei}-1}\langle\sigma_k
      , A\circ F_\ve^k\rangle}\right)= \logmgf(\sigma)\\
    &=  \int_0^{T}\chi_A(\sigma(s),\bar\theta(s,\thetasl))\deh
      s+\restoPalla(\sigma).\label{eq:largedev3}
  \end{align}
  where $\restoPalla(\sigma)$ (which is defined by the equation above,
  see~\eqref{eq:almost-equality}) satisfies the estimates given in
  Proposition~\ref{lem:exponential}.

  The above implies (recall~\eqref{eq:legendre},~\eqref{eq:sigma-nu} and the
  definition~\eqref{e_bzetgt_definition} of $\bzetgt$):
  \begin{equation}\label{eq:newLambda}
  \begin{split}
    \logmgf^*(\gamma) &=\sup_{\sigma\in\BV}\left[\int_0^T\langle\sigma(s),\gamma'(s)\rangle-\logmgf(\sigma)\right]\\
    &= \sup_{\sigma\in\BV}\left[
      \int_0^T\bzetgtsl(\sigma(s),s)\deh s - \restoPalla(\sigma)\right].
  \end{split}
  \end{equation}
  Formula~\eqref{eq:newLambda} closely resembles the definition of rate function given
  in~\eqref{eq:ratefunction}.  Unfortunately there is an obvious obstacle: we need to
  ensure that the first term dominates $\restoPalla(\sigma)$.  As already observed in
  Subsection~\ref{s_regularizedlogmgf}, this can be taken care of by some regularization
  procedure for $\sigma$; we will now describe the dual procedure, \ie a regularization
  procedure for the paths $\gamma$.

  Given $\gamma$ and $h=T/N_h$, for $N_h\in\bN$ suitably large to be chosen later, we
  denote with $\gamma_h = \udPi_{(h)}\gamma$ the polygonalization of $\gamma$ over a mesh of size $h$.  In
  other words, we define $\gamma_h\in \cC_*^0([0,T],\bR^d)$ so that $\gamma_h' = \regPi{h}(\gamma')$ where
  $\regPi{h}$ has been defined in~\eqref{e_definitionRegPi}.  Recall
  (see~\eqref{eq:lip-supported}) that it suffices to consider paths $\gamma\in\Lip_{C}$
  and thus, since $\regPi{h}$ is a contraction in $L^\infty$ (see Sub-lemma~\ref{sl_propertiesPi}), we
  conclude that $\gamma_h\in\Lip_{C,*}$.  Moreover, for $t\in [nh, (n+1)h]$, we have
  \begin{equation}\label{eq:poly-close}
    |\gamma_h(t)-\gamma(t)|\le \int_{nh}^{(n+1)h}|\gamma_h'(s)-\gamma'(s)|\deh{}s\le
    2\int_{nh}^{(n+1)h}|\gamma'(s)|\deh s,
  \end{equation}
  which implies that $\gamma_h\in Q_h=\bigcup_{\gamma\in \overline \QLipC}B(\gamma, 2Ch)$.
  Then, by~\eqref{eq:legendre} and recalling the definition of
  $\regu\logmgf{h}$ given in~\eqref{e_defReguLogmgf}:
  \begin{align*}
    \logmgf^*(\gamma)
    &\ge     \sup_{\sigma\in\BV}\left[\int_0^T\langle\regPi{h}\sigma(s),\gamma'(s)\rangle-\regu\logmgf{h}(\sigma)\right] \\
    &\ge\sup_{\sigma\in\BV}\left[\int_0^T\langle\sigma(s),\gamma_h'(s)\rangle-\regu\logmgf{h}(\sigma)\right] = \regu\logmgf{h}^*(\gamma_h)
  \end{align*}
  where $\regu\logmgf{h}^*$ is the Legendre transform of the regularized moment
  generating functional $\regu\logmgf{h}$. In particular, we have
  $\regu\logmgf{h}^*\ge0$ because $\regu\logmgf{h}(0) = 0$.

  Hence, we conclude that
  \begin{equation}\label{eq:upper-eq-0}
    \inf_{\gamma\in \overline\QLipC}\logmgf^*(\gamma)
    \ge \inf_{\gamma\in Q_h}\regu\logmgf{h}^*(\gamma).
  \end{equation}
Observe that, by~\eqref{e_defReguPalla},
\begin{equation}\label{eq:newLambda-h}
  \regu\logmgf{h}^*(\gamma) = \sup_{\sigma\in\BV}\left[
    \int_0^T\bzetgtsl(\sigma(s),s)\deh s - \regu\restoPalla{h}(\sigma)\right],
\end{equation}
where $\regu\restoPalla{h}$ satisfies the estimates obtained in
Lemma~\ref{lem:exponential-h}.  The aim of the above regularization is to gain control on
$\regu\restoPalla{h}$ even for very rough $\sigma$.  This is the content of the next
sub-lemma.
\begin{sublem} \label{lem:omegah}%
  There exists $\CJ5 > 0$ and $\ve_0,\Tmax\in(0,1]$, such that, for all
  $\ve\in [0,\ve_0]$, $T\in[\ve_0^{-4}\ve, \Tmax]$
  and $\gamma\in\Lip_{C,*}([0,T],\bR^d)$, we have for any $\theta\in\bT$ and standard pair
  $\ell$ with $\thetasl = \theta$, setting $h=\sqrt{\ve T}$:
  \begin{align}\label{e_omegah}
    \regu\logmgf{h}^*(\gamma)
    \ge
    \inf_{\tilde\gamma\in B(\gamma,\tilde R_\ve(\gamma))} I_\theta(\tilde\gamma)
    \end{align}
    where we define
    \begin{align*}
      \tilde R_\ve(\gamma) = \CJ5
      \max\left\{(\ve\eefrac 14T^{-\efrac 14}+T)\|\hgamma\|\nl\infty,%
      \min \left\{ \ve\eefrac14 T\eefrac34,
      (\ve T)\eefrac16\|\hgamma\|\nl\infty\eefrac23\right\},
      \sqrt{\ve T}\right\}
    \end{align*}
and recall, $\hgamma = \gamma-\gavg(\cdot,\theta)$.
\end{sublem}
Observe that, together with~\eqref{eq:largedev2} and~\eqref{eq:upper-eq-0}, the above
sub-lemma immediately allows to conclude the proof of Lemma~\ref{lem:upper} choosing $C_0
= \CJ5+2C$.
\end{proof}
\begin{proof}[{\bf Proof of Sub-Lemma~\ref{lem:omegah}}]
  Let us fix $\CJ5 > 1$ large enough to be specified later; we begin by observing that if
  $\|\hgamma\|\nl\infty < \CJ5\sqrt{\ve T}$, then
  $B(\gamma,\tilde R_\ve(\gamma))\ni\gavg$, which implies that
  $\inf_{\tilde\gamma\in B(\gamma,\tilde R_\ve(\gamma))}I_\theta(\tilde\gamma) = I_\theta(\gavg) = 0$ and
  the sub-lemma holds trivially since $\regu\logmgf{h}^*(\gamma)\ge 0$.  In the rest of
  the proof, we will therefore always assume that
  $\|\hgamma\|\nl\infty \ge \CJ5\sqrt{\ve T}$ provided
  \begin{equation}\label{eq:cond-R-one}
    \tilde R_\ve(\gamma)\geq \CJ5\sqrt{\ve T}.
    \end{equation}
 Recall~\eqref{eq:newLambda-h}:
  \begin{align*}
    \regu\logmgf{h}^*(\gamma) = \sup_{\sigma\in\BV}\left[
    \int_0^T\bzetgtsl(\sigma(s),s)\deh s - \regu\restoPalla{h}(\sigma)\right],
  \end{align*}
  and that, by definition~\eqref{eq:ratefunction} (since $\gamma\in\Lip_{C,*}$):
  \begin{align*}
    I_\theta(\gamma) = \sup_{\sigma\in\BV}\int_0^T\bzetgt(\sigma(s),s)\deh s,
  \end{align*}
  where, by definition~\eqref{e_bzetgt_definition},
  \begin{equation}\label{eq:kappas-def}
    \begin{split}
      \bzetgt(\sigma(s),s)    &= \langle\sigma(s),\gamma'(s)\rangle-\chi_A(\sigma(s),\bar\theta(s,\theta)) \\
      &= \langle\sigma(s),\hgamma'(s)\rangle-\hat\chi_A(\sigma(s),\bar\theta(s,\theta)).
    \end{split}
  \end{equation}
  We will proceed as follows: by convexity of the rate function, in any ball $B(\gamma)$
  around $\gamma$, we can find paths which are more likely than $\gamma$ itself; in
  particular $\inf I_\theta(B(\gamma)) < I_\theta(\gamma)$ (the inequality is strict since
  $\gamma \ne \gavg$).  The idea is then to take the ball $B$ to be so large that the
  decrease in the rate function compensates for the remainder term $\regu\restoPalla{h}$.  Of
  course, by choosing larger balls, we obtain worse bounds for the error in the final
  estimate: the key technical point of the sub-lemma rests exactly in finding a good
  compromise for the size of $B$.

  For any $\convexP\in[0,1]$, let us define the convex interpolation
  $\gamma_\convexP = (1-\convexP)\gamma+\convexP\gavg$; since
  $I_\theta(\gamma_{\convexP = 1}) = I_\theta(\gavg) = 0$ and by convexity of
  $I_\theta$, we conclude that $I_\theta(\gamma_{\convexP})$ is decreasing in $\convexP$.
  Hence we want to find $\convexP$ sufficiently large so that the decrease compensates
  for the remainder term. The choice of $\convexP$ will in fact depend on the distance of
  $\gamma$ from $\gavg$, that is on $\|\hgamma\|\nl\infty$.
  We carry out the estimate using two different strategies as they yield optimal bounds in different regimes.\\

  \noindent\textbf{Case I: non-perturbative estimate}\\
  Note that, since $\hgamma'_\convexP = (1-\convexP)\hgamma'$, we have
  \begin{equation}\label{eq:kappa-r-s}
    \zet(\sigma,\gamma',\theta) =  \zet(\sigma,\gamma'_\convexP,\theta)
    + \convexP \langle\sigma,\hgamma'\rangle.
  \end{equation}
  Collecting~\eqref{eq:newLambda-h},~\eqref{eq:kappas-def} and~\eqref{eq:kappa-r-s} we
  obtain, for any $\sigma\in\BV$ and $\convexP\in[0,1]$,
  \begin{align}\label{e_collection}
    \regu\logmgf{h}^*(\gamma)
    &\ge \int_0^T \zet(\sigma(s),\gamma_\convexP'(s),\bar\theta(s,\theta))ds
      + \convexP\int_0^T\langle\sigma(s),\hgamma'(s)\rangle ds-\regu\restoPalla{h}(\sigma).
  \end{align}
  We then use the estimate for $\regu\restoPalla{h}$ given by
  Lemma~\ref{lem:exponential-h}-\ref{i_qnonPerturbativeBoundOnR-h} with the choice
  \[
  L=T\eefrac 14\ve\eefrac {-1}4\in[\ve_0^{-1}, \min\{\vei T, \ve_0\ve^{-\efrac 12}\}],
  \]
  where the inclusion follows from our conditions on $\ve, T$ and noticing that $[\ve_0^{-4}\ve,\Tmax]$ is empty for $\ve> \Tmax\ve_0^4$.
  Recalling that $h = \sqrt{\ve T}$ we obtain:
  \begin{align}\label{eq:Rh-ok}
    \regu\restoPalla{h}(\sigma)
    &\le\CJdue%
      \left[T\eefrac54\ve\eefrac34+[T\eefrac{-1}4\ve\eefrac14 +
      \min\{T,\|\sigma\|\nl1\}]\|\sigma\|\nl1\right].
  \end{align}
  Let us fix $\CJ0 > 0$ large to be specified later and let
\begin{equation}\label{eq:convexP1}
\convexP_0 = \CJ0 \{\ve^{\efrac 14}T^{-\efrac 14}+2T\}.
\end{equation}
To fix ideas, we may  assume $\ve_0, \Tmax$ to be so small that
  $\convexP_0 \le \frac 13$.
  Let us first examine the possibility $I(\gamma_{\convexP_0})=\infty$; in this case we claim
  that $\regu\logmgf{h}^*(\gamma) = \infty$, which trivially implies the sub-lemma.  In
  fact: let us fix arbitrarily $M > 1$; by~\eqref{eq:ratefunction}, since
  $\gamma\in\Lip_{C,*}$, there exists $\bar\sigma_M\in\BV$ such that
  \begin{equation}\label{eq:large-M}
  \begin{split}
    M&\le \int_0^T\zet(\bar\sigma_M(s),\gamma_{\convexP_0}'(s),\bar\theta(s,\theta))ds\\
    & = \int_0^T\langle\bar\sigma_M(s),\hgamma_{\convexP_0}'(s)\rangle -   \hat\chi_A(\bar\sigma_M(s),\bar\theta(s,\theta))ds.
\end{split}
\end{equation}
  By the properties of $\hat\chi_A$ it follows\footnote{ Equations~\eqref{eq:derivative-m}
    and~\eqref{e_secondDerivativeChi} imply that $\hat\chi_A(\cdot,\theta)$ is convex and
    has its minimum in $\sigma = 0$ where $\hat\chi_A(0,\theta) = 0$} that if
  $\|\sigma\|\ge\sigma_*$, then
\newcommand{\setAbove}{\Sigma_*}
\newcommand{\setAboveM}{\Sigma_M}
\newcommand{\setAbovec}{\setAbove^\text{c}}
  \begin{align*}
    \hat\chi_A(\sigma,\theta) \ge \Const\|\sigma\|.
  \end{align*}
Define the set $\setAboveM = \setAboveM(\convexP_0,M) = \{s\in[0,T]\st \|\bar\sigma_M(s)\|\ge \sigma_*\}$.  Then
  \begin{align*}
    M\le  \int_0^T \langle \bar\sigma_M,\hgamma_{\convexP_0}'\rangle-\Const\int_{\setAboveM} \|\bar\sigma_M\|\le\int_0^T
    \langle\bar \sigma_M,(1-\convexP_0)\hgamma'\rangle-\Const (\|\bar\sigma_M\|\nl1-T\sigma_*).
  \end{align*}
  Assuming $M\geq \Const \sqrt T$ to be sufficiently large, it follows
  \begin{equation}\label{eq:lower-b-braket}
    \int_0^T\langle\bar\sigma_M,\hgamma'\rangle > \frac M2 + \Const\|\bar\sigma_M\|\nl1.
  \end{equation}
Using~\eqref{e_collection},~\eqref{eq:Rh-ok} together with~\eqref{eq:large-M} and~\eqref{eq:lower-b-braket} yields
  \begin{align}\label{eq:Lambdah-0}
    \regu\logmgf{h}^*(\gamma)-M
    &\ge \convexP_0\int_0^T\langle \bar\sigma_M,\hgamma'\rangle -
      \CJdue \left[\ve\eefrac 34T\eefrac 54+\left\{\frac{\ve\eefrac 14}{T\eefrac {1}4}+T\right\}\|\bar\sigma_M\|\nl1\right]\\
    &\ge 0\notag,
  \end{align}
  provided $\ve$ is small enough, $M$ is large enough, and $\CJ0$, thus $ \convexP_0$, is
  large enough.  By the arbitrariness of $M$ it follows
  $ \regu\logmgf{h}^*(\gamma)=\infty$.

  We can therefore assume that $I(\gamma_{\convexP_0})<\infty$: since $I(\gamma_\convexP)$
  is decreasing in $\convexP$, this implies that $I(\gamma_\convexP)<\infty$ for any
  $\convexP\in[\convexP_0,1]$.  In particular, recalling also Lemma~\ref{lem:hausdorff}, for each $\convexP\in(\convexP_0,1]$ and $\theta\in \bT$, $s\in [0,T]$  there exists an open set $s\in U\subset [0,T]$ and a compact set $\cK = \cK(\theta,\convexP,s)\subset \cap_{s'\in U}\bD_*(\bar\theta(s',\theta))$ such that $\gamma'_\convexP(s')\in\cK$ for almost all $s'\in U$ and thus that
  $(1-\convexP)\hgamma'=\hgamma'_\convexP(s)=\partial_\sigma\hat\chi_A(\sigma)$ has a (unique) solution for
  almost all $s\in [0,T]$, which we denote with $\bar\sigma_\convexP$. Since $\partial_\sigma\hat\chi_A$ is a homeomorphism, see the proof of Lemma~\ref{lem:domainZ}, we have $\bar\sigma_\convexP\in L^\infty$.

By construction, $\bar\sigma_\convexP$ realizes the $\sup$
  in~\eqref{eq:ratefunction}; in particular~\eqref{e_collection} reads:
  \begin{align*}
    \regu\logmgf{h}^*(\gamma) - I_\theta(\gamma_\convexP)
    &\ge \convexP\int_0^T\langle\bar\sigma_\convexP(s),\gamma'(s)\rangle ds-\regu\restoPalla{h}(\bar\sigma_\convexP).
  \end{align*}
  Let us assume $\convexP\in (\convexP_0,\frac 12]$. Define the set
  $\setAbove=\setAbove(\convexP) = \{s\in[0,T]\st \|\bar\sigma_\convexP(s)\|\ge
  \sigma_*\}$ and let us denote by $\setAbovec=[0,T]\setminus \setAbove$ its complement.
  Then, plugging~\eqref{eq:Rh-ok} into the above inequality and recalling
  Lemma~\ref{l_gamma-bound}, we find:
  \begin{align}\label{e_collection2}
    \notag\regu\logmgf{h}^*(\gamma) - I_\theta(\gamma_\convexP)
    &\ge \Const\convexP\left[\int_{\setAbove}\|\bar\sigma_\convexP\|+\int_{\setAbovec}\|\bar\sigma_\convexP\|^2\right]\\
    &\phantom = -\CJdue \left[T\eefrac 54\ve\eefrac 34 +\left\{\ve\eefrac 14T\eefrac
      {-1}4+\min\{T,\|\bar\sigma_\convexP\|\nl1\}\right\}\|\bar\sigma_\convexP\|\nl1\right].
  \end{align}
  First, we claim that:
  \begin{equation}\label{eq:min-bound}
    \min\{T,\|\bar\sigma_\convexP\|_{L^1}\}\|\bar\sigma_\convexP\|_{L^1} \le 2T\int_{\setAbove}\|\bar\sigma_\convexP\|+2T\int_{\setAbovec}\|\bar\sigma_\convexP\|^2.
  \end{equation}
  In fact, assume that 
  $\int_{\setAbove}\|\bar\sigma_\convexP\|\geq \int_{\setAbovec}\|\bar\sigma_\convexP\|$;
  then $\|\bar\sigma_\convexP\|\nl1\leq 2\int_{\setAbove}\|\bar\sigma_\convexP\|$ and
  \begin{align*}
    \min\{T,\|\bar\sigma_\convexP\|\nl1\}\|\bar\sigma_\convexP\|\nl1\le 2 T\int_{\setAbove}\|\bar\sigma_\convexP\|.
  \end{align*}
  If, on the other hand,
  $\int_{\setAbove}\|\bar\sigma_\convexP\|\leq \int_{\setAbovec}\|\bar\sigma_\convexP\|$,
  we have
  \begin{align*}
    \min\{T,\|\bar\sigma_\convexP\|\nl1\}\|\bar\sigma_\convexP\|\nl1\le  T\int_{\setAbove}\|\bar\sigma_\convexP\|+2\left[\int_{\setAbovec}\|\bar\sigma_\convexP\|\right]^2
    \le T\int_{\setAbove}\|\bar\sigma_\convexP\|+2T\int_{\setAbovec}\|\bar\sigma_\convexP\|^2.
  \end{align*}
Which proves~\eqref{eq:min-bound}. Plugging it into~\eqref{e_collection2} we
  obtain:
  \begin{align*}
    \regu\logmgf{h}^*(\gamma)& - I(\gamma_\convexP)
    \ge \Const\convexP\left[\int_{\setAbove}\|\bar\sigma_\convexP\|+\int_{\setAbovec}\|\bar\sigma_\convexP\|^2
      \right]-\CJdue%
      \Bigg[\ve\eefrac34T\eefrac54+\\
      &+\{\ve\eefrac 14T\eefrac {-1}4+2T\}\int_{\setAbove}\|\bar\sigma_\convexP\|+ (\ve T)\eefrac
      14 \left[{\int_{\setAbovec}\|\bar\sigma_\convexP\|^2}\right]\eefrac12+2
      T\int_{\setAbovec}\|\bar\sigma_\convexP\|^2 \Bigg]\\
    &\ge \Const\convexP \left[\int_{\setAbove}\|\bar\sigma_\convexP\|+\int_{\setAbovec}\|\bar\sigma_\convexP\|^2
      \right]\\
    &\phantom{\ge} -\CJdue \left[ \ve\eefrac 34 T\eefrac54 +
      (\ve T)\eefrac 14 \left[{\int_{\setAbovec}\|\bar\sigma_\convexP\|^2}\right]\eefrac12\right],
  \end{align*}
  provided $\CJ0$ has been chosen large enough (recall~\eqref{eq:convexP1} and that
  $\convexP \ge \convexP_0$).

  We conclude that,
 \[\begin{split}
 &\textrm{ if }\int_{\setAbovec}\|\bar\sigma_\convexP\|^2\ge \Const\convexP^{-2}\ve\eefrac 12T\eefrac 12\\
 &\textrm{ or } \int_{\setAbove}\|\bar\sigma_\convexP\|\ge \Const \convexP\invr\max\left\{\ve\eefrac
  34T\eefrac54,(\ve T)\eefrac14\left[{\int_{\setAbovec}\|\bar\sigma_\convexP\|^2}\right]\eefrac12\right\},
\end{split}
\]
then $\regu\logmgf{h}^*(\gamma) - I(\gamma_\convexP)\geq 0$.  Otherwise,
  \begin{equation}\label{eq:last-case}
    \begin{split}
      \int_{\setAbovec}\|\bar\sigma_\convexP\|^2
      &\le \Const \convexP^{-2}(\ve T)\eefrac12\\
      \int_{\setAbove}\|\bar\sigma_\convexP\|
      &\le \Const\convexP\invr\max\left\{\ve\eefrac34T\eefrac54,(\ve T)\eefrac14
      \left[{\int_{\setAbovec}\|\bar\sigma_\convexP\|^2}\right]\eefrac12\right\}\\
      &\le \Const \convexP^{-2}(\ve T)\eefrac12,
    \end{split}
  \end{equation}
  provided $\ve$ is small enough.
  Note that Lemma~\ref{l_gamma-bound} implies that
  $\|\hat\gamma_\convexP'(s)\|\le\Constgs\|\bar\sigma_\convexP(s)\|$;
  then estimates~\eqref{eq:last-case} and definition~\eqref{eq:convexP1} imply, since $\convexP> \convexP_0$,
  \begin{align}\label{eq:hgamma-ub}
    \|\hgamma\|\nl\infty
    &\leq 2\|\hgamma_\convexP'\|\nl1\le 2\Constgs\|\bar\sigma_\convexP\|\nl1\\
    &\le 2\Constgs\left[ \int_{\setAbove}\|\bar\sigma_\convexP\| +
      \left[ T{\int_{\setAbovec}\|\bar\sigma_\convexP\|^2}\right]\eefrac12 \right]
      \le \Const \convexP\invr\ve\eefrac14 T\eefrac34.\notag
  \end{align}
Hence, if we choose $\CJ0$ sufficiently large and
  \begin{equation}\label{eq:convexP2}
    \convexP =\max\{2\convexP_0, 2\CJ0\ve\eefrac14 T\eefrac34\|\hgamma\|\nl\infty\invr\},
  \end{equation}
  then~\eqref{eq:hgamma-ub} cannot hold true and necessarily $\regu\logmgf{h}^*(\gamma)\ge I(\gamma_\convexP)$. Thus,  we
  obtain~\eqref{e_omegah}, provided that $\tilde R_\ve(\gamma) > \convexP\|\hgamma\|\nl\infty$.\\

 \noindent\textbf{Case II: perturbative estimate}\ \\
  In this case we plan to apply the more refined estimate given in
  Lemma~\ref{lem:exponential-h}-\ref{i_perturbativeBoundOnR-h}; yet it may not be possible
  to do so with the path $\gamma_\convexP$ since in general $\setAbove \ne \emptyset$.  In
  order to circumvent this problem we define another path that is sufficiently close to
  $\gamma_\convexP$ and to which we can apply the mentioned estimate.  Let
  $\tgamma\in \Lip_{C,*}$ so that $\tgamma(0)=0$ and $\tgamma'(s)=\gavg'(s)$ for
  $s\in \setAbove$ and $\tgamma'(s)=\gamma'(s)$ otherwise.  Define
  $\tgamma_\convexP = (1-\convexP)\tgamma+\convexP\gavg$, hence
   \begin{equation}\label{eq:gammarho}
     \hat\tgamma'_\convexP = (1-\convexP)\hat\tgamma' = \Id_{\setAbovec}\hgamma_\convexP'.
   \end{equation}
   Let us now choose
   \begin{align}\label{e_optimalConvexP}
     \convexP = \CJ0\max\{(\ve T)\eefrac16\|\hgamma\|\nl\infty\eefrac{-1}3, \ve^{\efrac 14}T^{-\efrac 14}+2 T\}.
   \end{align}
   Then, either $\regu\logmgf{h}^*(\gamma)\ge I(\gamma_\convexP)$ (and
   then~\eqref{e_omegah} holds provided
   $\tilde R_\ve(\gamma) \geq \convexP\|\hgamma\|\nl\infty$), or estimates~\eqref{eq:last-case} hold; in
   the latter case Lemma~\ref{l_gamma-bound} yields
  \begin{align}\label{eq:tgamma-bound}
    \|\gamma_\convexP-\tgamma_\convexP\|\nl\infty
    &\le \int_{\setAbove} \|\hgamma_\convexP'(s)\| ds
      \le \Constgs \int_{\setAbove} \|\bar\sigma_\convexP(s)\| ds 
    \le \Const\convexP^{-2}(\ve T)\eefrac12.
  \end{align}
  In particular,
  $\|\gamma-\tgamma_\convexP\|\nl\infty \le\convexP\|\hgamma\|\nl\infty +
  \Const\convexP^{-2}(\ve T)\eefrac12$.  Observe that, assuming $\CJ0$ large
  enough,~\eqref{e_optimalConvexP} implies that (this justifies the
  choice of the first term in~\eqref{e_optimalConvexP}, which optimizes the above inequality)
  \begin{align}\label{e_boundRadius}
    \|\gamma-\tgamma_\convexP\|\nl\infty
    &\le 2\convexP\|\hgamma\|\nl\infty.
  \end{align}
  Let $\underline\sigma_\convexP(s)$ be the unique solution of
  $\hat\tgamma'_\convexP(s)=\partial_\sigma\hat\chi_A(\sigma,\bar\theta(s,\theta))$. By
  definition,
  \begin{equation}\label{eq:sigmarho}
    \underline\sigma_\convexP(s)= \Id_{\setAbovec}(s) \bar \sigma_\convexP(s)\le \sigma_*;
  \end{equation}
  in particular, $\underline\sigma_\convexP(s)$ satisfies the hypotheses of
  Lemma~\ref{lem:exponential-h}-\ref{i_perturbativeBoundOnR-h}. Also it satisfies the first inequality of~\eqref{eq:last-case}. We can now proceed as
  before but using the path $\tgamma_\convexP$: more precisely, equations~\eqref{eq:newLambda-h},~\eqref{eq:kappas-def},~\eqref{eq:gammarho},
 ~\eqref{eq:sigmarho} and Lemma~\ref{l_gamma-bound} imply
  \begin{align*}
    \regu\logmgf{h}(\gamma)
    &\ge \int_0^T\bzetgtsl(\underline\sigma_\convexP(s),s)\deh s - \regu\restoPalla{h}(\underline\sigma_\convexP) \\
    &= \int_0^T\kappa_{\underline\gamma_\convexP,\theta}(\underline\sigma_\convexP(s),s)\deh s+\frac{\convexP}{1-\convexP}\int_0^T \langle\underline\sigma_\convexP(s), \hat\tgamma_\convexP'(s)\rangle - \regu\restoPalla{h}(\underline\sigma_\convexP)\\
    &\ge I(\tgamma_\convexP)+\Const\convexP\|\underline\sigma_\convexP\|\nl2^2-
      \regu\restoPalla{h}(\underline\sigma_\convexP).
  \end{align*}
  Before continuing note that, choosing $\tilde R_\ve(\gamma) \ge 3\convexP\|\hgamma\|\nl\infty$, we can ensure
  that~\eqref{e_omegah} trivially holds if $\convexP > 1/3$. Hence we can assume
  $\convexP \le 1/3$. Next, we claim that
  \begin{equation}\label{eq:sigma-bound}
\frac{\CJ5\Constgs}3\sqrt\ve\leq \|\underline\sigma_\convexP\|\nl2\leq \Const \CJ0\invr\sqrt T.
 \end{equation}
  Indeed, by the first of~\eqref{eq:last-case} and recalling the choice~\eqref{e_optimalConvexP}
  $\|\underline\sigma_\convexP\|\nl2\le \Const \CJ0\invr\sqrt T$.
  Moreover, recall that we are assuming that $\|\hgamma\|\nl\infty\ge\CJ5\sqrt{\ve T}$ and
  $\convexP \le 1/3$, which by~\eqref{e_boundRadius} yields
  $\|\underline\hgamma_\convexP\|\nl\infty\ge(\CJ5/3)\sqrt{\ve T}$; in turn, by Lemma~\ref{l_gamma-bound},  this implies
  that $\|\underline\sigma_\convexP\|\nl2\ge\frac{\CJ5\Constgs}3\sqrt\ve$.

  We now make the choice  $L=\Const {\|\underline\sigma_\convexP\|\nl2}(\ve T)\eefrac{-1}2$. Let us check when it satisfies the hypotheses of Lemma~\ref{lem:exponential-h}-\ref{i_perturbativeBoundOnR-h}. By~\eqref{eq:sigma-bound}, choosing $\CJ5$ and $\CJ0$ sufficiently large,
we have $L\in[\ve_0\invr,\ve_0\ve\eefrac{-1}2]$. On the other hand the condition $\ve L\leq T$ is satisfied only if $\Const \|\underline\sigma_\convexP\|\nl2\leq T^{\efrac 32}\ve^{-\efrac 12}$, which is automatically ensured only if $T\geq \sqrt\ve$, provided  $\CJ0$ is sufficiently large.

Thus, if $T\leq \sqrt\ve$, then our choice it is not good. In such a case we make the choice $L=\ve_0\ve^{-1} T$. Again, we have  $L\in[\ve_0\invr,\ve_0\ve\eefrac{-1}2]$, and $\ve L\leq T$, thus the hypotheses of Lemma~\ref{lem:exponential-h}-\ref{i_perturbativeBoundOnR-h} are satisfied.

We are now ready to estimate $\regu\restoPalla{h}$ using Lemma~\ref{lem:exponential-h}-\ref{i_perturbativeBoundOnR-h}, recall our choices $h =\sqrt{\ve T}$ and $T\in (\ve L, \Tmax)$.  We have to treat the two above regimes separately.

 If $T\geq \sqrt\ve$, then, using Schwarz inequality,
  \begin{equation}\label{eq:resto-palla-h-II}
    \begin{split}
      \regu\restoPalla{h}(\underline\sigma_\convexP) \leq\CJdue\left[5\sqrt{\ve}\,\|\underline\sigma_\convexP\|\nl2+2T\|\underline\sigma_\convexP\|^2\nl2
      \right].
    \end{split}
  \end{equation}
  We thus conclude, assuming $\CJ0$ in~\eqref{e_optimalConvexP} to be sufficiently large:
  \begin{align}\label{e_finalCaseEstimate}
    \regu\logmgf{h}(\gamma)-I(\tgamma_\convexP)
    &\ge \Const\convexP\|\underline\sigma_\convexP\|\nl2^2-
      5\CJdue      \sqrt{\ve}\|\underline\sigma_\convexP\|\nl2.
  \end{align}
  Since Lemma~\ref{l_gamma-bound} implies $\Constgs\sqrt T \|\sigma\|\nl2\geq \|\hgamma\|\nl\infty\geq \CJ5\sqrt{\ve T}$, the first term in the max of~\eqref{e_optimalConvexP} implies that~\eqref{e_omegah} holds by choosing
  $\tilde R_\ve(\gamma) = 3\convexP\|\hgamma\|\nl\infty$.

 Next, we consider the case  $T\leq \sqrt\ve$ and apply again Lemma~\ref{lem:exponential-h}-\ref{i_perturbativeBoundOnR-h}:
    \begin{equation}\label{eq:resto-palla-h-III}
    \begin{split}
      \regu\restoPalla{h}(\underline\sigma_\convexP) &\leq\CJdue\left[T^{2}\CJdue^{-\efrac12}+4\sqrt \ve\|\underline\sigma_\convexP\|\nl2+2\ve T^{-1} \|\underline\sigma_\convexP\|^2\nl2\right]\\
      &\leq\CJdue\left[5\sqrt{\ve}\,\|\underline\sigma_\convexP\|\nl2+2\ve T^{-1}\|\underline\sigma_\convexP\|^2\nl2
      \right],
    \end{split}
  \end{equation}
  where we have chosen $\CJdue$ large enough and, in the last line, we have used~\eqref{eq:sigma-bound}. Then~\eqref{e_omegah} follows again since, on the one hand, $\ve^{\efrac14}T^{-\efrac 14}\geq \ve T^{-1}$, and, on the other hand, $(\ve T)^{\efrac 16}\|\hgamma\|_{L^\infty}^{-\efrac 13}\|\underline\sigma_\convexP\|\nl2\geq \Const \sqrt\ve$ is implied, again, by $\Constgs\sqrt T\|\underline\sigma_\convexP\|\nl2\geq \|\hgamma\|_{L^\infty}$ and $\|\hgamma\|_{L^\infty}\geq \CJ5 \sqrt{\ve T}$.
  The sub-lemma then follows by
  recalling~\eqref{eq:cond-R-one}, that $\convexP$ must be larger than $\convexP_0$, defined in
 ~\eqref{eq:convexP1}, as well as larger than either~\eqref{eq:convexP2} or
 ~\eqref{e_optimalConvexP}, and provided that we choose $\CJ5\ge3\CJ0$.
\end{proof}
  \begin{rem}
    The above Lemma is based on a trade-off: for small deviations (up to the ones
    predicted by the Central Limit Theorem) it gives very rough estimates, but up to times
    of order one; while for larger deviations it provides much sharper results although only
    for short times.  Obtaining sharper results for small deviations would entail more work, in
    particular a sharper version of Lemma~\ref{lem:exponential} (which can be achieved by
    using the techniques that we will employ in the next sections to prove a local CLT).
    On the other hand, in order to extend the above sharp results to longer times one can simply
    divide the time interval in shorter ones and use Lemma~\ref{lem:upper} repeatedly (see
    Theorem~\ref{thm:large} for an implementation of this strategy).
  \end{rem}
  \subsection{Lower bound for balls (short times)}\label{ss_lowerBound}\ \newline %
  In this section we proceed to obtain a lower bound for short times.  It turns out that
  in~\cite{DeL2} we need this type of estimates only in the case of large deviations;
  therefore we do not insist here in obtaining optimal bounds for all moderate deviations,
  since this would require considerable additional work.  On the other hand, we can, and
  will, obtain results for deviations that are not exceedingly small at essentially no
  extra cost.  %
  Also, in an attempt to simplify the exposition, we will consider trajectories that are not
  arbitrarily close to being impossible, \ie so that their derivatives belong to the
  domains $\bD_\eps(\theta)=\bD(\theta)\setminus\partial_\eps\bD(\theta)$ for some
  arbitrarily small, but fixed, $\eps$.
  \newcommand{\centsigma}{\bar\sigma}
  \newcommand{\argamma}{\gamma}
  \begin{lem}[Lower bound]\label{lem:lower}
    For any $\delta\in(0,\frac 14)$ there exists $\ve_\delta, K_\delta,\Tmax>0$ such that, for
    any $\theta\in\bT$, $\ve\le \ve_\delta$, $\centgamma\in\Lip_{C,*}([0,T])$, with
    $\centgamma'(s)\in\bD_\delta(\bar\theta(s,\theta))$ for almost all $s\in[0,T]$, and for
    any standard pair $\ell$ so that $\thetasl = \theta$:
    \begin{align*}
    &\ve\log\ppath_\ell(B(\centgamma, h))\ge - I_\theta(\centgamma)- c_\delta (T\sqrt{\ve/h}+ T^2+T\raggio/h))\\
   & \raggio=  K_\delta [T^{\efrac 32}+\max\{\ve^{\efrac 14}T^{\efrac 34}, K_\delta^{-\efrac 12}T(\ve/h)^{\efrac 14}\} ]
    \end{align*}
   provided $T\in[\ve_\delta^{-1}\ve,\Tmax]$, $h\in[\raggio,T]$.
  \end{lem}
  \begin{proof}
    Before describing the core of the proof we need some preparation: we must define a
    new reference path having several special properties. To this end we first perform the
    same polygonalization done just before~\eqref{eq:poly-close}, with step $h=T/N_h$,
    $N_h\in\bN$ and $h\geq 2\ve$. Let $\centgamma_{h}$ the resulting path. Clearly
    $\|\centgamma-\centgamma_{h}\|\nc0\leq Ch\leq \delta/4$, provided $h$ is small enough,
    but it also has another important property.
    \begin{sublem}\label{sublem:poly}
      For all $\delta>0$ and $\centgamma\in\Lip_{C,*}([0,T])$, with
      $\centgamma'(s)\in\bD_\delta(\bar\theta(s,\theta))$ for almost all $s\in[0,T]$, there
      exists $h_\delta\in (0,T]$ such that, for all $h\leq h_\delta$,
      $\centgamma'_{h}(s)\in \bD_{\delta/4}(\bar\theta(s,\theta))$ for almost all $s\in[0,T]$.
    \end{sublem}
    \begin{proof}
      By Lemma~\ref{lem:hausdorff}, $\bD(\bar\theta(s,\theta))$ varies continuously, in
      the Hausdorff topology, with respect to $s$.  Hence, for each $\delta>0$ there
      exists $h_\delta>0$ such that
      $\bD_{\delta}(\bar\theta(s',\theta))\subset
      \bD_{\delta/2}(\bar\theta(s,\theta))\subset \bD_{\delta/4}(\bar\theta(s',\theta))$
      for all $|s'-s|\leq h_\delta$.  Accordingly, for all $h\leq h_\delta$,
      $k\in\{0,\cdots,N_h\}$ and $s\in (kh,(k+1)h)$ we have
      $\centgamma'(s)\in \bD_{\delta/2}(\bar\theta(kh,\theta))$. By the convexity of the
      set it follows
      $\centgamma'_{h}(s)\in \bD_{\delta/2}(\bar\theta(kh,\theta))\subset
      \bD_{\delta/4}(\bar\theta(s,\theta))$, hence the sub-lemma.
    \end{proof}
    The above sub-lemma implies that the equation
    $\centgamma_{h}'(s)=\partial_\sigma\chi_A(\sigma(s),\bar\theta(s,\theta))$ has a
    unique solution $\csigma_h(s)\in \BV$. Moreover, from
    Lemmata~\ref{lem:hausdorff},~\ref{lem:domainZ} it also follows that there exists
    $\raggiod>0$ such that
    $\partial_\sigma\chi_A(B(0,\raggiod),\bar\theta(s,\theta))\supset\bD_{\delta/4}(\bar\theta(s,\theta))
    $ for all $s\in[0,T]$.  This allows to obtain also some regularity of the function
    $\csigma_h\in\BV$. Indeed, for $k\in\{0,\cdots,N_h\}$ and $s\in (kh,(k+1)h)$ we have
    \[
  0=\partial^2_\sigma\chi_A(\csigma_h(s),\bar\theta(s,\theta))\csigma'_h(s)+\partial_\theta\partial_{\sigma}\chi_A(\csigma_h(s),\bar\theta(s,\theta))\bar
  A(\bar\theta(s,\theta).
\]
Since $\|\csigma_h\|_{L^\infty}\leq \raggiod$ and
$\inf_{\|\sigma\|\leq \raggiod, s}\partial^2_\sigma\chi_A>0$, there exists $C_\delta>0$
such that $\|\csigma'_h(s)\|\leq C_\delta$.

We have thus obtained a path with a controlled regularity. Unfortunately, we have a further
problem: it is not obvious how to compare effectively the rate function computed on the
regularized path and the rate function computed on the original one. To this end it turns
out to be convenient to further modify the path in a special way: consider a path
$\gammah\in\cC^0([0,T],\bR^d)$ such that $\gammah(kh)=\centgamma(k h)=\centgamma_{h}(kh)$
for all $k\in\{0,\cdots, N_h\}$ with the following extra property: a) $\gammah$ is
Lipschitz, b) the equation
$\gammah'(s) = \partial_\sigma\chi_A(\sigma(s),\bar\theta(s,\theta))$, which by (a) is
well defined for almost every $s\in[0,T]$, has a solution $\sigmah$ which is constant in
each interval $[nh, (n+1)h)$.

The reason for the above construction lies in item (c) of the next lemma.
\begin{sublem}\label{sublem:minimizing} There exists $h_\delta>0$ such that, for all $h\leq h_\delta\in (0,T]$, the above defined path $\gammah$ is well defined and unique. In addition,
\begin{enumerate}
\item $\gammah'(s)\in\bD_{\delta/8}(\bar\theta(s,\theta))$ for almost all $s\in[0,T]$;
\item $\|\centgamma-\gammah\|_{\cC^0}\leq 2C h$ ;\label{sublem:minimizing-b}
\item $I_{\theta}(\centgamma)\geq I_\theta(\gammah)$.\label{sublem:minimizing-c}
 \end{enumerate}
\end{sublem}
\begin{proof}
  We can change the path interval by interval. Assume that we have a path $\gamma_{h,k}$
  that has the wanted properties in the interval $[0,kh]$ and that agrees with
  $\centgamma_{h}$ in the interval $(kh,T]$, and let us consider the the interval
  $J_k=(kh, (k+1)h]$. Define the function
  $\Xi_k: L^{\infty}([0,T])\times \bR^{d+1}\to L^{\infty}(J_k)\times \bR$ given by
  \[
    \Xi_k(\eta,\zeta,\beta)=
    \left(\centgamma_{h}'(s)+\eta(s)-\partial_\sigma\chi_A(\zeta+(1-\beta)\csigma_h(s),\bar\theta(s,\theta)),
      \int_{J_k}\eta(s)ds\right).
  \]
  By definition we have $\Xi_k(0,0,0)=0$. We want to apply the implicit function theorem,
  hence we have to study the differential
  \[
    \partial_{\eta,\zeta}\Xi_k=\begin{pmatrix} \Id&-\partial^2_\sigma\chi_A(\zeta+(1-\beta)\csigma_h(s),\bar\theta(s,\theta))\\
      \Leb&0\end{pmatrix}.
  \]
  Using Lemmata \ref{lem:quasicompact}, \ref{lem:coboundary},
  \ref{eq:theta-derivative-all}, a direct computation shows that, provided
  $\|\zeta\|\leq 2\raggiod$, $\|(\partial_{\eta,\zeta}\Xi_k)^{-1}\|\leq C_\delta$. We can
  then apply the Implicit Function Theorem and obtain a solution
  $(\eta(\beta,s),\zeta(\beta))$ of $\Xi_k(\eta,\zeta,\beta)=0$. Such a solution is
  differentiable and satisfies
  \[
    \begin{split}
      &\partial_\beta\eta(s)-\partial^2_\sigma\chi_A(\zeta+(1-\beta)\csigma_h(s),\bar\theta(s,\theta))[\partial_\beta\zeta-\csigma_h(s)]=0\\
      &\int_{J_k}\partial_\beta\eta(s) ds=0.
    \end{split}
  \]
  Integrating the first and using the second equation yields
  \[
    \begin{split}
      &\partial_\beta\eta(s)=B(s,\zeta,\beta)[\partial_\beta\zeta-\csigma_h(s)]\\
      &\partial_\beta\zeta=\left[\int_{J_k}B(s',\zeta,\beta)ds'\right]^{-1}\int_{J_k}B(s',\zeta,\beta)\csigma_h(s')ds'\\
      &B(s,\zeta,\beta)=\partial^2_\sigma\chi_A(\zeta+(1-\beta)\csigma_h(s),\bar\theta(s,\theta)).
    \end{split}
  \]
  Recalling that $\|\csigma'_h(s)\|\leq C_\delta$ (which was the point of introducing the
  regularized path $\centgamma_{h}$ in the first place), we have
  $\|B(s,\zeta,\beta)-B(kh,\zeta,\beta)\|\leq C_\delta h$. Thus,
\[
\begin{split}
&\left\|\partial_\beta\zeta- h^{-1}\int_{J_k}\csigma_h(s)ds\right\|\leq C_\delta h\leq \raggiod\\
&\|\partial_\beta\eta(s)\|\leq C_\delta h\leq \delta/8.
\end{split}
\]
provided $h$ is small enough. Accordingly,
$\centgamma'_{h}(s)+\eta(\beta,s)\in \bD_{\delta/8}(\bar\theta(s,\theta))$ for all
$\beta\leq 1$.  The above shows that $\gamma_{h,k+1}$ and hence, by induction, $\gammah$
is uniquely defined and point (b) of the lemma follows as well for $\delta$ small enough.
To prove (c), we use Lagrange multipliers to find the local minimum among all the paths
$\gamma\in\cC^{0}([0,T],\bR^d)$ such that $\gamma(kh)=\centgamma(kh)$. This amount to
finding the stationary points of the functional
\[
I_L(\gamma, \phi_i)=I(\gamma)+\sum_{k=1}^{h^{-1}T}\langle \phi_k,\gamma(kh)-\centgamma(kh)\rangle.
\]
By~\eqref{eq:rate1} and~\eqref{eq:zeta-sigma-rel} it follows that, for each $\gamma$ such that $I(\gamma)<\infty$, we have, for each $\alpha\in\Lip$, $\alpha(0)=0$,
\[
\partial_\gamma I_L(\gamma)(\alpha)=\int_0^T\langle\sigma (s),\alpha'(s)\rangle ds+\sum_{k=1}^{h^{-1}T}\langle \phi_k,\alpha(kh)\rangle
=\int_0^T\langle(\sigma (s)-\sigma_\phi(s)),\alpha'(s)\rangle ds
\]
where $\gamma'(s) = \partial_\sigma\chi_A(\sigma(s),\bar\theta(s,\theta))$ and
$\sigma_\phi(s)=-\sum_{j=k}^{h^{-1}T}\phi_j$ for $s\in [(k-1)kh,kh)$.  Hence it must be
$\sigma=\sigma_\phi$, that is $\sigma$ is piecewise constant. But only $\gammah$ has such
a property, thus the convex function $t: \bR\to I_\theta(t\gamma_h+(1-t)\centgamma)$ has a
unique stationary point that must be a minimum, hence point (c) of the sub-lemma.
\end{proof}
Note that, by the above sub-lemma, the equation
$\gamma_{h}'(s)=\partial_\sigma\chi_A(\sigma(s),\bar\theta(s,\theta))$ has a unique
solution $\sigma_{h}(s)\in \BV$, $\|\sigma_h\|_{L^\infty}\leq C_\delta$,
$\|\sigma_h\|_{\BV}\leq C_\delta h^{-1} T$.  Also, if $C h>\raggio$, then we have
    \begin{equation}\label{eq:inclusion}
     B(\centgamma, 3C h)\supset B(\gammah,\raggio).
    \end{equation}
    This concludes our preparation; the rest of the proof follows the strategy strategy for
    proving the lower bound.  Consider the linear functional $\gfunctionalh\in\cM^d([0,T])$
    defined by
    \begin{equation}\label{eq:vfh}
      \gfunctionalh(\argamma) = \vei\int_0^T\langle\sigmah(s),\argamma'(s)\rangle ds,
    \end{equation}
    and introduce the measure $\ppath_{\gfunctionalh,\ell}$ on $\cC^0([0,T],\bR^d)$
    defined by
    \begin{align*}
      \epath_{\gfunctionalh,\ell}(\psi) =
      \frac{\epath_{\ell}(e^{\gfunctionalh}\psi)}{\epath_\ell(e^{\gfunctionalh})},
    \end{align*}
    for any continuous functional $\psi\in \cC([0,T],\bR^d)'$.
    \begin{sublem}\label{sub:small-measure}%
      There exists $\ve_\delta, K_\delta, \Tmax>0$ such that for any $0 < \ve\le \ve_\delta$:
      \begin{align*}
        \ppath_{\gfunctionalh,\ell}(B(\gammah,\raggio)) \ge\frac 12,
      \end{align*}
      provided $T\in[\ve_\delta^{-1}\ve,\Tmax]$, $\raggio\geq  K_\delta [T^{\efrac 32}+\max\{\ve^{\efrac 14}T^{\efrac 34}, K_\delta^{-\efrac 12}T(\ve/h)^{\efrac 14}\} ]$ and $h\in[\raggio/(3C),T]$.
    \end{sublem}
    \begin{proof}
      The idea is to cover the complement of $B(\gammah,\raggio)$ in the support of
      $\ppath_{\gfunctionalh,\ell}$ with finitely many sufficiently small balls whose measure
      we can estimate using the upper bound obtained in Lemma~\ref{lem:upper}.  In order to
      do so, let us partition the interval $[0,T]$ in subintervals of length $\const r_\ve$,
      $r_\ve<h$, so that any path in the support of the measure (and hence $C$-Lipschitz) can
      vary in any given subinterval by at most $2 r_\ve$.  This means that there exists a
      finite set $\Gamma = \{\gamma_i\}\subset \cC^0([0,T],\bR^d)$ of cardinality
      $|\Gamma| = \exp({\const r_\ve^{-1} T})$ so that\footnote{ In fact, consider a lattice
        of size $r_\eps \sqrt{4/d}$ in $\bR^d$. If a path $\gamma$ is in the support of the
        measure and $\gamma(s)$ belongs to the lattice, by the Lipschitz property
        $\gamma(s+\const r_\ve)$ belongs to the union of balls of radius $r_\ve$ centered at
        \emph{finitely many} points of the lattice.}
      \begin{align*}
        \supp \ppath_{\gfunctionalh,\ell} \subset \bigcup_{\gamma_i\in\Gamma} B(\gamma_i,r_\ve).
      \end{align*}
      Define $\Gamma_*=\{\gamma_i\in\Gamma\st\|\gammah-\gamma_i\|_\infty\ge \raggio/2\}$: then, by definition,
      \begin{align*}
        \ppath_{\gfunctionalh,\ell}(B(\gammah,\raggio))
        &\ge  1- \sum_{\gamma_i\in \Gamma_*}\ppath_{\gfunctionalh,\ell}(B(\gamma_i,r_\ve)).
      \end{align*}
      Let $r_\ve\in [ K_\delta^2 \raggio^{-2}\ve T^2, \raggio/8]$; observe that {the
        interval is not empty provided
        $\raggio\geq 2\ve^{\efrac 13}T^{\efrac 23} K_\delta^{\efrac 23}$}, which always
      holds if $\ve_\delta$ is    small enough.  We claim that, for all
      $\gamma_i\in \Gamma_*$,
      \begin{align}\label{e_boundComplementaryBall}
        \ppath_{\gfunctionalh,\ell}(B(\gamma_i,r_\ve)) < e^{- C_\delta\vei \raggio^2T^{-1}}.
      \end{align}
      Observe that the above estimate suffices to prove the sub-lemma: in fact, using our
      estimate on the cardinality of $\Gamma$, we have
      \begin{align*}
        \sum_{\gamma_i\in \Gamma_*}\ppath_{\gfunctionalh,\ell}(B(\gamma_i,r_\ve)) \le
        e^{-\Const([C_\delta-\const K_\delta^{-2}] \vei \raggio^2T^{-1}}\le \frac 12,
      \end{align*}
      provided we choose  $K_\delta$ large enough.

      We thus proceed to prove~\eqref{e_boundComplementaryBall}.
      Using~\eqref{eq:legendre} we gather
      \begin{align*}
        \ppath_{\gfunctionalh,\ell}(B(\gamma_i,r_\ve))
        &=e^{-\vei\logmgf(\gfunctionalh)}
          \epath_{\ell}\left(e^{\gfunctionalh}\Id_{B(\gamma_i,r_\ve)}\right)\\
        &\leq e^{-\vei\logmgf(\gfunctionalh)+\gfunctionalh(\gamma_i)+\ve^{-1}r_\ve\|\sigmah\|\nBV}
          \ppath_{\ell}\left(B(\gamma_i,r_\ve)\right),\notag
      \end{align*}
      where in the second line we used the fact that, for any
      $\gamma\in B(\gamma_i,r_\ve)$, we have by definition~\eqref{eq:vfh} (recall
      Remark~\ref{rem:variationdef}),
      \begin{equation}\label{eq:fi-h-est}
        \begin{split}
          |\gfunctionalh(\gamma)-\gfunctionalh(\gamma_i)| &=
          \vei\left|\int\langle\sigmah,\gamma'-\gamma'_i\rangle\right|
          \le \vei \|\sigmah\|\nBV\|\gamma-\gamma_i\|\nl\infty\\
          & \le \vei \|\sigmah\|\nBV\, r_\ve .
        \end{split}
      \end{equation}
      Then, using~\eqref{eq:sigma-nu} and~\eqref{eq:largedev3}:
      \begin{align}
        \ppath_{\gfunctionalh,\ell}(B(\gamma_i,r_\ve))
        &\leq e^{\vei \left[-\restoPalla(\sigmah)+ r_\ve\|\sigmah\|\nBV\right]}\notag\\
        &\phantom = \cdot e^{\vei\left[\int_0^T\langle\sigmah(s),\gamma_i'(s)\rangle
          - \chi_A(\sigmah(s),\bar\theta(s,\theta))ds\right]}
          \ppath_{\ell}\left(B(\gamma_i,r_\ve)\right).\label{eq:exhausted-1}
      \end{align}
      Next, let us set
      $\Xi_i=\inf_{\argamma\in B(\gamma_i, \raggio/4)}I_\theta(\argamma)$; we claim that
      \begin{align}\label{e_ppathEstimate}
        \ppath_\ell\left(B(\gamma_i,r_\ve)\right)\le e^{-\vei \Xi_i}.
      \end{align}
      In fact, if $\|\hgamma_i\|_\infty\ge 2CT$, then $\|\hgamma\|_\infty\geq 3/2 CT$ for all $\gamma\in B(\gamma_i, r_\ve)$, hence their  Lipschitz constant must be larger than $C$ and  hence $\ppath_\ell\left(B(\gamma_i,r_\ve)\right)=0$.
      Otherwise~\eqref{e_ppathEstimate} follows from Lemma~\ref{lem:upper}, provided that $R_\ve(\gamma_i) +r_\ve< \raggio/{4}$. This holds since one can check that $R_\ve(\gamma_i)\leq\Const \{\ve^{\efrac 14}T^{\efrac 34}+T^2\}$, hence $R_\ve(\gamma_i)\leq \raggio/8$ by choosing $K_\delta$ sufficiently large.
      Substituting~\eqref{e_ppathEstimate} in~\eqref{eq:exhausted-1},
      recalling~\eqref{eq:kappaZ},~\eqref{eq:M-0},~\eqref{eq:rate1}, Sub-Lemma~\ref{sublem:rate-f} and Remark~\ref{rem:z-concrete}:
      \begin{align}\notag
        \ppath_{\gfunctionalh,\ell}(
        &B(\gamma_i,r_\ve)) %
          \le e^{\vei\left[-\restoPalla(\sigmah)+ r_\ve\|\sigmah\|_{\BV}\right]}\\
        &\cdot e^{\vei\sup_{\argamma\in B(\gamma_i, \raggio/4)\cap \Lip}\int_0^T\left[\langle\sigmah(s),\gamma_i'(s)\rangle
          - \chi_A(\sigmah(s),\bar\theta(s,\theta)) - \cZ(\argamma'(s),\bar\theta(s,\theta))\right]ds}\notag\\
        &\le e^{\vei\left[-\restoPalla(\sigmah)+ r_\ve\|\sigmah\|_{\BV}\right]}\notag\\
        &\cdot e^{\vei\sup_{\argamma\in B(\gamma_i, \raggio/4)} \int_0^T
          \left[\langle\sigmah(s),\argamma'(s)-\gammah'(s)\rangle+\cZ(\gammah'(s),\bar\theta(s,\theta))-\cZ(\argamma'(s),\bar\theta(s,\theta))\right]ds}.\label{eq:exhausted-2}
      \end{align}
      Next, we proceed to estimate the argument of the $\sup$ appearing on the last line: note that we only
      need to consider paths $\gamma$ so that
      $s\mapsto\cZ(\gamma'(s),\bar\theta(s,\theta))$ is integrable on $[0,T]$, hence $\gamma'(s)\in\bD$ for almost all $s$.  Then, for any
      $\convexP\in[0,1]$, let us define the interpolating function
      \begin{align*}
        H(\convexP,s)
        &:= \langle \sigmah(s),\convexP(\gamma'(s)-\gammah'(s))\rangle\\
        &\phantom{: = }-\cZ(\gammah'(s)+\convexP(\gamma'(s)-\gammah'(s)),\bar
          \theta(s,\theta))+\cZ(\gammah'(s),\bar \theta(s,\theta)).
      \end{align*}
      Let $\sigma_\convexP$ be the solution of $\gammah'+\convexP(\gamma'-\gammah')=\partial_\sigma\hat\chi_A(\sigma_\convexP,\bar\theta(\cdot,\theta))$, which exists, for all $\convexP<1$, by the convexity of $\bD$. In addition, note that, if $\convexP\leq 1/2$, then $\gammah'+\convexP(\gamma'-\gammah')\in\bD_{\delta/16}$.  Note that $H(0,s) = 0$ and
      \[
        \begin{split}
          & \partial_\convexP H(0,\cdot) = \langle
          \sigmah-\centsigma,\gamma'-\gammah'\rangle -
          \langle\centsigma,\partial_b\cZ(\gammah'(s),\bar \theta(s,\theta)),\gamma'-\gammah'\rangle=0,\\
          & \partial_{\convexP\convexP} H(\convexP,\cdot)=-\langle
          [\partial_\sigma^2\hat\chi_A(\sigma_\convexP)]\invr(\gamma'-\gammah'),\gamma'-\gammah'\rangle\leq 0,\\
          & \partial_{\convexP\convexP} H(\convexP,\cdot)\le -C_\delta \|\gamma'-\gammah'\|^2 \quad\textrm{ for all }\convexP\leq 1/2.
        \end{split}
      \]
      Thus
      \begin{align*}
        \langle \sigmah, \gamma'-\gammah'\rangle
        - \cZ(\gamma',\bar \theta(\cdot,\theta))
        + \cZ(\gammah',\bar \theta(\cdot,\theta))
        =H(1)\le -C_\delta\|\gamma'-\gammah'\|^2.
      \end{align*}
      The term containing $\restoPalla$ can be estimated by
      Proposition~\ref{lem:exponential}-\ref{i_nonPerturbativeBoundOnR}; by our bounds on $\sigma_h$ and choosing
      $L=\ve^{-\efrac 12}h^{\efrac 12}$ we obtain, since $\sqrt{\ve h}\leq T$,
      \begin{equation}\label{eq:R-estimate}
        |\restoPalla(\sigmah)|\le  c_*^{-1}C_\delta \ve L h^{-1}T+C_\delta(L^{-1}+T+\ve Lh^{-1}T)T
        \le C_\delta (T\sqrt {\ve/h}+ T^2).
      \end{equation}
      Thus,
      \[
        \ppath_{\gfunctionalh,\ell}(B(\gamma_i,r_\ve))\le e^{-C_\delta\vei\left[ \inf_{\gamma\in B(\gamma_i, \raggio/4)}\|\gamma'-\gammah'\|_{L^2}^2-C_\delta (T\sqrt {\ve/h}+ T^2)-r_\ve\|\sigmah\|_{\BV}\right]}.
      \]
      To conclude, note that, for $\gamma_i\in\Gamma_*$, if $\argamma\in B(\gamma_i,\raggio/4)$ we  have
      \[
        \raggio/4\le \|\gamma-\gammah\|_{L^\infty}\le \|\gamma'-\gammah'\|_{L^1}\le\|\gamma'-\gammah'\|\nl2\sqrt T,
      \]
      which, by choosing $K_\delta$ large and $r_\ve= K_\delta^2 \raggio^{-2}\ve T^2$, proves \eqref{e_boundComplementaryBall}.
    \end{proof}
    To conclude the proof of Lemma~\ref{lem:lower} it suffices to compare the measures
    $\ppath_{\gfunctionalh,\ell}$ and $\ppath_\ell$.  By
    Sub-Lemma~\ref{sub:small-measure}, equations~\eqref{eq:legendre},
   ~\eqref{eq:largedev3}, and using~\eqref{eq:R-estimate}, and arguing like in
   ~\eqref{eq:fi-h-est} we have
    \begin{align*}
      \frac 12
      &\le \ppath_{\gfunctionalh,\ell}(B(\gammah, \raggio)) =
        \frac{\epath_\ell\left(e^{\gfunctionalh}\Id_{B(\gammah,\raggio)}\right)}{\epath_\ell(e^{\gfunctionalh})}\\
      &\le \Const
        e^{\vei\left[\int_0^T\langle\sigmah,\gammah'\rangle-\chi_A(\sigmah,\bar\theta(s,\theta))\deh{}s
        + C_\delta (T\sqrt {\ve/h}+ T^2)\right]}\;\epath_{\ell}\left(e^{\gfunctionalh-\gfunctionalh(\gammah)}\Id_{B(\gammah,\raggio)}\right)\\
      &\le \Const
        e^{\vei\left[I_\theta(\gammah)
        + C_\delta (T\sqrt {\ve/h}+ T^2)+\|\sigma_h\|_\BV \raggio\right]}\;\ppath_\ell(B(\gammah,\raggio)).
    \end{align*}
    By~\eqref{eq:inclusion}, Sub-Lemma \ref{sublem:minimizing}-\ref{sublem:minimizing-c}
    and $h,\raggio$ as in Sub-Lemma \ref{sub:small-measure} we have
  \[
      \ppath_\ell(B(\centgamma, 3Ch))\ge    \ppath_\ell(B(\gamma_h, \raggio))\geq e^{-\vei(I(\gammah)+c_\delta (T\sqrt{\ve/h}+ T^2+T\raggio/h))}.
    \]
 The Lemma follows by choosing $\raggio$ as small as allowed and renaming $3Ch$ as $h$.
 \end{proof}
  \subsection{Large and moderate deviations for balls: long times}\label{subsec:large-long}\ \newline
  Lemmata~\ref{lem:upper} and~\ref{lem:lower} are the basic ingredients to prove
  Theorem~\ref{thm:large}.  Their major drawback is that they are really effective only
  for short times.  In order to proceed and obtain a Large Deviation estimate for times of
  order $1$ with the announced small error, we will subdivide a trajectory in shorter
  subintervals and apply the mentioned lemmata to each subinterval.  To this end, some
  type of Markov-like property is needed.  Before stating it in Lemma~\ref{lem:markov}, we
  need to introduce a bit of notation.

  For any $J\subset J'\subset [0,T]$, $\centgamma\in \cC^0(J',\bR^d)$ and $r > 0$, we introduce the notation
  \begin{equation}\label{eq:balls}
    \gball J(\centgamma, r) = \left\{\gamma\in\gspace\st \|\gamma -
    \centgamma\|_{L^\infty(J)}<r\right\},
  \end{equation}
  where $\|\gamma\|_{L^\infty(J)}=\sup_{s\in J}\|\gamma(s)\|$.  In other words,
  $\gball J(\centgamma,r)$ is a set of paths in $\gspace$ that are $r$-close to
  $\bar\gamma$ on $J$, but are otherwise arbitrary on $[0,T]\setminus J$.  Naturally, if
  $J = [0,T]$, the set $\gball{}(\centgamma,r) = \gball{[0,T]}(\centgamma,r)$ is the
  standard $C^0$-ball of radius $r$ around $\centgamma$.

  Let us fix $\theta\in\bT$; consider a path $\gamma\in \cC^0([0,T],\bR^d)$ and a number
  $r>0$.  For any standard pair $\ell = (\bG,\rho)$ and $t\in [0,T]$ we define
  $\Id^\pm_{\theta,\gamma,r}(\ell,t)$ as follows:
  \begin{subequations}
    \begin{align}
      \Id^-_{\theta,\gamma,r}(\ell,t)
      &=
        \begin{cases}
          1 &\text{ if $\inf_x\{|G_\ell(x)-\theta^{\gamma}(t)|\}< r$}\\
          0 &\text{ otherwise}
        \end{cases}\\
      \Id^+_{\theta, \gamma,r}(\ell,t)
      &=
        \begin{cases}
          1 &\text{ if $\sup_x\{|G_\ell(x)-\theta^{\gamma}(t)|\}< r $}\\
          0 &\text{ otherwise,}
        \end{cases}
    \end{align}
  \end{subequations}
  where we used the previously introduced notation
  $\theta^\gamma(s) = \theta^\gamma(s,\theta) = \theta+\gamma(s)_1$, and $\gamma(s)_1$
  denotes the first component of the vector $\gamma(s)$.  Observe that, by construction,
  $\Id^+_{\theta,\gamma,r}(\ell,t)\le\Id^-_{\theta,\gamma,r}(\ell,t)$.  Finally, for any
  interval $E=[s,t]\subset[0,T]$, let
 \begin{equation}\label{eq:porb-bound}
  \begin{split}
    P_-(E, \theta,\gamma, r)
    &=\sup_{\{\ell\;:\;\Id^-_{\theta,\gamma,r}(\ell,s)=1\}}\ppath_\ell(\gball{[0,t-s]}(\gamma(s+\cdot)-\gamma(s), r)),\\
    P_+(E,\theta, \gamma, r)
    &=\inf_{\{\ell\;:\;\Id^+_{\theta,\gamma, r}(\ell,s)=1\}}\ppath_\ell(\gball{[0,t-s]}(\gamma(s+\cdot)-\gamma(s), r));
  \end{split}
\end{equation}
observe that by construction $P_+(E,\theta, \gamma, r)\le P_-(E, \theta,\gamma, r)$.
  \begin{lem}\label{lem:markov}
    For $T > 0$, $k\in\{0,\cdots, K-1\}$ and $\tau=TK\invr\geq \ve$ let
    $E_k=[k\tau,(k+1)\tau]$.  Then, for any $\thetaref\in\bT$ and standard pair $\ellz$ with
    $\thetaslz = \thetaref$ we have:
        \begin{align*}
      \prod_{k=0}^{K-1}P_+(E_k,\thetaref, \centgamma, r-\Const \ve)
      \le \ppath_{\ellz}(\gball{}(\centgamma, r))%
      \le \prod_{k=0}^{K-1}P_-(E_k,\thetaref,\centgamma, r+\Const\ve),%
    \end{align*}
    where the first inequality holds for $\centgamma\in\cC^0$, while the second for
    $\centgamma\in\supp \ppath_{\ellz}$.\footnote{ In particular, $\centgamma$, is
      $C$-Lipschitz and piecewise linear in each interval $[k\ve, (k+1)\ve]$.}
  \end{lem}
\begin{proof} For each $k\in \bN$ let $\stdf^{(k)}_{\ellz, 0}$ be a standard family
  representing $(F^k_\ve)_*\mu_{\ell_0}$, see Proposition~\ref{p_invarianceStandardPairs}
  and Remark~\ref{rem:fm-bound} for exact definitions.  Let
  $\pi_z:\bT^2\times\bR^{d-1}\to\bT\times \bR^{d-1}$ be the projection on the last $d$
  (\ie slow) coordinates.  Let $\ell'\in \stdf^{(k)}_{\ell, 0}$, by Remark
  \ref{rem:fm-bound}, definition~\eqref{eq:map-full} and the expansivity of the $x$
  dynamics, it follows that, for all $j\leq k$,
\begin{equation}\label{eq:close-by}
\begin{split}
&\sup_{x,x'}\left\|\Id_{\ell'} (x)\Id_{\ell'} (x') [\pi_z \bF^{j}(x,G_{\ell}(x),0)-\pi_z \bF^{j}(x',G_{\ell}(x'),0)]\right\|\\
&=\sup_{x,x'}\left\|\Id_{\ell'}(x)\Id_{\ell'}(x') \ve\sum_{m=0}^{j-1}\left [ A\circ F^{m}(x,G_{\ell}(x))-A\circ F^{m}(x',G_{\ell}(x'))\right]\right\|\\
&\leq\Const\ve \lambda^{-k+j}.
\end{split}
\end{equation}
Let us define, for $j, k\in \bN$, $\ell'\in \stdf^{(k)}_{\ell, 0}$, $E=[j\ve,(j+k)\ve]$,
$\Id^-_{\gball{E}(\centgamma,r)}(\ell,\ell')=1$ if
\[
\sup_{m\leq k}\inf_x\|\Id_{\ell'}(x)[\centgamma(\ve(j+m))-\centgamma(\ve j)+(G_\ell(x),0) -\pi_z \bF^{m}(x,G_\ell(x),0)]\|\leq r,
\]
while $\Id^-_{\gball{E}(\centgamma,r)}(\ell,\ell')=0$ otherwise. Note that,
$\Id^-_{\gball{[0,t]}(\centgamma,r)}(\ell,\ell')\leq \Id^-_{\thetasl,\centgamma,r+\Const\ve}(\ell',t)$.
Also, $\Id^-_{\gball{[0,t]}(\centgamma,r+\Const\ve)}(\ell,\ell')\geq \Id^-_{\thetasl,\centgamma,r}(\ell',t)$
since, if $ \Id^-_{\thetasl,\centgamma,r}(\ell',t)=1$, then, by  \eqref{eq:close-by},
\[
\sup_{m\leq k}\sup_x\|\Id_{\ell'}(x)[\centgamma(\ve(j+m))-\centgamma(\ve j)+(G_\ell(x),0) -\pi_z \bF^{m}(x,G_\ell(x),0)]\|\leq r+\Const\ve.
\]
Setting $\centgamma_{k}(s)=\centgamma(k\tau+s)-\centgamma(k\tau)$, by
Proposition~\ref{p_invarianceStandardPairs} and~\eqref{eq:porb-bound} it follows that
\begin{equation}\label{eq:uffff}
 \begin{split}
 &\ppath_\ellz\left(\gball{}(\centgamma ,r)\right)\leq \mu_{\ellz}\left(\prod_{k=0}^{\vei T}\Id_{\gball{\{k\ve\}}(\centgamma+(\theta_0,0), r)}\circ \bF_\ve^k\right)\\
 &\leq \sum_{  \ell'\in \stdf^{K\tau/\ve}_{\ellz, 0}}\fm_{\ell'} \Id^-_{\gball{[0,K\tau]}(\centgamma,r)}(\ell_0,\ell')\\
 &\le  \sum_{  \ell'\in \stdf^{(K-1)\tau/\ve}_{\ellz, 0}}\fm_{\ell'} \Id^-_{\gball{[0,(K-1)\tau]}(\centgamma,r)}(\ell_0,\ell') \ppath_{\ell'}(\gball{[0,\tau]}(\centgamma_{K-1}, r+\Const \ve))\\
& \le  \sum_{  \ell'\in \stdf^{(K-1)\tau/\ve}_{\ellz, 0}}\fm_{\ell'} \Id^-_{\gball{[0,(K-1)\tau]}(\centgamma,r)}(\ell_0,\ell')
  P_-(E_{K-1}, \thetaref,\centgamma, r+\Const \ve).
\end{split}
\end{equation}
Iterating, and arguing similarly for the lower bound, the lemma follows.\footnote{ Note
  that if $\centgamma\in\supp \ppath_\ellz$, then in the first line of \eqref{eq:uffff}
  holds equality.}
  \end{proof}
We are finally ready to prove our main Large Deviations result; the proof is divided in two
  steps; the first step is to obtain upper and lower bounds for the probability of a small ball
  around a trajectory of length of order $1$ and it is given by Lemma~\ref{l_finalBounds}
  below.  The second will be carried out in the proof of Theorem~\ref{thm:large}.

  \newcommand{\cgamma}{\underline\gamma}
  \newcommand{\hcgamma}{\underline{\hat\gamma}}
  \newcommand{\ctheta}{\underline\theta}
  \begin{lem}\label{l_finalBounds}
    Let $T>0$, $\theta\in\bT$, $\cgamma\in\cC^0([0,T],\bR^d)$ and, recalling the notation in~\eqref{e_defRadPm},
    $\Rad+(\gamma) = C_{\epsilon,T}\left\{\ve\eefrac17\|\hgamma\|\nl\infty\eefrac57+\sqrt\ve \right\}$,  where $\hgamma = \gamma-\bar\gamma(\cdot,\theta)$; let moreover
    $\vestep =(T^{\efrac 17}\|\hcgamma\|\nl\infty^{-\efrac 17}\ve^{\efrac 47})$, $\vestep_-=\ve^{\frac{1}{2}}$.  Then for any $\epsilon > 0$, if $\ve$ is sufficiently
    small, we have, for any $\ellz$ with $\thetaslz = \thetaref$:
    \begin{subequations}\label{e_finalBounds}
      \begin{align}\label{e_finalBoundUpper}
        \ve\log\ppath_\ellz(\gball{}(\cgamma,\vestep/3))
        &\le {-\left[1-C_{\epsilon}\left(\frac{T^2\ve}{\|\hcgamma\|\nl\infty^2}\right)^{\efrac 17}\right]}\inf_{\gamma\in
          B(\cgamma,\Rad+(\cgamma))}
          \ratef^-_{\thetaref,\epsilon}(\gamma).\\
        \ve\log\ppath_\ellz(\gball{}(\cgamma,\vestep_-))
        &\ge -(1+C_\epsilon\ve\eefrac 12) \ratef^+_{\thetaref,\epsilon}(\cgamma)-C_{\epsilon, T} \ve\eefrac{1}{8}.
          \label{e_finalBoundLower}
      \end{align}
    \end{subequations}
  \end{lem}
  \begin{proof}
    In accordance with our conventions we will use $C_\epsilon$ or $C_T$ to designate an
    arbitrary constant depending only on the values of $\ve$ or $T$, respectively.

    We begin by estimating the upper bound~\eqref{e_finalBoundUpper}. First
    of all observe that if $\|\hcgamma\|\nl\infty\le C_{\epsilon,T}^4\sqrt\ve$,
    then $\Rad+(\cgamma) > \|\hcgamma\|\nl\infty$ and
    therefore $\bar\gamma\in B(\cgamma,\Rad+(\cgamma))$, which implies that~\eqref{e_finalBoundUpper}
    holds trivially, since $\ratef_{\thetaref,\epsilon}(\bar\gamma) = 0$.  Hence we can
    assume that $\|\hcgamma\|\nl\infty\ge C_{\epsilon,T}^4\sqrt\ve$ and, provided $C_{\epsilon,T}$ has been chosen large enough, $\vestep\leq \sqrt\ve$.

    As mentioned, we will divide $[0,T]$ in time intervals of length
    $ \tau\leq c_\epsilon$. Although, for convenience, most of the following argument is
    done for arbitrary $\tau$ satisfying said inequalities, we will eventually
    choose\footnote{ We do not claim this choice to be optimal; its motivation is to
      simplify~\eqref{eq:jet-lag-3}.}
  \begin{equation}\label{eq:taudef}
\vestep=  \sqrt{\ve \tau} \quad \textrm{that is }\quad \tau= \left(\frac{T^2 \ve}{\|\hcgamma\|_{L^\infty}^2}\right)^{\efrac 17}.
  \end{equation}

  Note that if $B(\cgamma,\vestep)\cap \Lip_{C,*}([0,T],\bR^d)=\emptyset$, then
  $\ppath_\ellz(\gball{}(\cgamma,\vestep))=0$ and~\eqref{e_finalBoundUpper} is trivially
  true. We can then assume that there exists
  $\tilde\cgamma\in B(\cgamma,\vestep)\cap \Lip_{C,*}([0,T],\bR^d)$. We then consider the
  piecewise linear path $\overline\cgamma$ such that
  $\overline\cgamma(k\ve)=\tilde\cgamma(k\ve)$ for all $k\in\bN$. Since
  $\|\tilde\cgamma-\overline\cgamma\|\nl\infty\leq C\ve$,
  $B(\cgamma,\vestep)\subset B(\overline\cgamma,3\vestep)$. We can thus assume, without
  loss of generality, that $\cgamma=\overline\cgamma$ by substituting $\vestep$ to
  $\vestep/3$.

  Let $E = [0,\tau]$ and recall the notation $E_k = [k\tau,({k+1})\tau)$ introduced in
  Lemma~\ref{lem:markov}.  By Lemma~\ref{lem:markov}, we gather that, provided $\ve$ is
  sufficiently small,
    \begin{align}\label{e_markovBoundOnP}
      \ppath_\ellz\left({B(\cgamma,\vestep)}\right)
      &\le \prod_{k = 0}^{K-1} P_-(E_k,\thetaref,\cgamma,2\vestep)
    \end{align}
    where recall that
    \begin{equation}\label{eq:vista1}
      P_-(E,\thetaref,\cgamma,r )
      = \sup_{\{\ell\st\Id^-_{\thetaref,\cgamma,r}(\ell,k\tau)=1\}}
         \ppath_\ell(\gball{E}(\cgamma_{k},r )),
    \end{equation}
    and $\cgamma_{k}\in\Lip_{C,*}(E,\bR^d)$ is the translation of the path
    $\cgamma$ defined by:
        \begin{align*}
      \cgamma_{k}(s)
      &= \cgamma(k\tau+s)-\cgamma(k\tau).
    \end{align*}
    Let us fix arbitrarily $k\in\{0,\cdots,K-1\}$ and let
    $\ctheta_k=\theta^{\cgamma}(k\tau,\thetaref)$; fix also $\ell$ so that
    $\Id^-_{\thetaref,\cgamma,\vestep}(\ell,k\tau)=1$, \ie we have
    $|\thetasl-\ctheta_k| \le \Const \vestep$.

    Let us define the shorthand notations $\bar \theta_\ell:E\to\bT$ as
    $\bar \theta_\ell(\cdot) = \bar\theta(\cdot,\thetasl)$ and correspondingly
    $\bar z_{\ell}:E\to\bR^d$ as $\bar z_\ell = \bar z(\cdot,\thetasl)$ (\ie $\bar z_\ell$
    satisfies the differential equation $\bar z_{\ell}'(s)=\bar A(\bar \theta_{\ell}(s))$
    with initial condition $\bar z_{\ell}(0)=(\thetasl,0)$).  For convenience, remembering~\eqref{e_gavgDef}, let us also
    introduce, for any $\gamma\in\Lip_{C,*}(E,\bR^d)$,
    \begin{equation}\label{eq:ell-gavgDef}
    \begin{split}
   & \gavg_{\ell}(s)=\bar z_{\ell}(s)-(\avgtheta\ell,0)\\
    &\hgamma_\ell = \gamma-\gavg_\ell.
\end{split}
\end{equation}
    Our first step is to estimate $\ppath_\ell(\gball{E}(\cgamma_k,2\vestep ))$.
    For further use  let us introduce the notations (recall
    definitions~~\eqref{eq:M-0},~\eqref{eq:rate1},~\eqref{eq:zeta-reg-def}
    and~\eqref{eq:rate2}):
    \begin{equation}\label{eq:rate-good}
      \begin{split}
        I_\theta|_E(\gamma)
        &= \begin{cases}
          +\infty &\textrm{ if }\gamma(0)\neq 0 \textrm{ or } \gamma\not\in \Lip\\
          \int_E \cZ(\gamma'(s), \bar\theta(s,\theta))\, \deh s&\textrm{ otherwise.}
        \end{cases}\\
        \ratef_{\theta,\epsilon}^\pm|^{\phantom+}_E(\gamma)
        &= \begin{cases}
          +\infty &\textrm{ if }\gamma(0)\neq 0 \textrm{ or } \gamma\not\in \Lip\\
          \int_E \cZ^\pm_\epsilon(\gamma'(s), \theta^{\gamma}(s,\theta))\, \deh s&\textrm{ otherwise.}
        \end{cases}
      \end{split}
    \end{equation}
    Of course, we have $I_\theta = I_\theta|_{[0,T]}$ and
    $\ratef_{\theta,\epsilon}^\pm =\ratef_{\theta,\epsilon}^\pm|^{\phantom+}_{[0,T]}$.

    Next, we apply Lemma~\ref{lem:upper} in the interval $E$: assume
    $\sqrt{\ve\tau}\geq \vestep$, $\tau\geq \ve_0^{-4}\ve$, let
\begin{equation}\label{eq:Rkell}
        R_{k,\ell} = C_0'
        \max\left\{((\ve/\tau)^{\efrac 14}+\tau)\|\hcgamma_{k,\ell}\|\nl\infty,%
        \min \left\{ \ve\eefrac14 \tau\eefrac34,
        (\ve \tau)\eefrac16\|\hcgamma_{k,\ell}\|\nl\infty\eefrac23\right\},
        \sqrt{\ve \tau}\right\}
      \end{equation}
      where $\hcgamma_{k,\ell} = \cgamma_k-\gavg_\ell$, $C_0' = \CJ1 C_0$, $C_0$ was introduced in Lemma~\ref{lem:upper} and the
      constant $\CJ1$ is chosen so large that
      $C_0'\sqrt{\ve\tau}\ge C_0\sqrt{\ve\tau}+2\vestep$, then
      \begin{align}\label{eq:renzinoooo2}
        \ve\log\ppath_\ell(\gball{E}(\cgamma_k,2\vestep))
        &\le  {-\inf_{\gamma\in \gball{E}(\cgamma_k, R_{k,\ell})}
          I_{\thetasl}|_{E}(\gamma)}.
      \end{align}
We now proceed to relate the rate function $I_{\thetasl}|_{E}$ appearing in~\eqref{eq:renzinoooo2} with the modified rate function $\ratef^-_{\ctheta, \epsilon/2}|_E$.
\begin{sublem}\label{l_firstHorribleClaim}
For each $\epsilon>0$ there exists $\ve_0, c_\epsilon>0$ such that, for all $\ve<\ve_0$, $c_\epsilon \geq \tau>0$, $E=[0,\tau]$, standard pair $\ell$ such that $\Id^-_{\thetaref,\cgamma,\vestep}(\ell,k\tau)=1$,
and $\gamma\in \Lip_{C,*}(E,\bR^d)$,
 \begin{align*}
 \ratef^-_{\ctheta_k,\epsilon/2}|_{E}(\gamma)
        &\le (1+C_\epsilon\tau)I_{\thetasl}|_{E}(\gamma)+
          C_\epsilon \vestep^2.
      \end{align*}
\end{sublem}
    \begin{proof}
      Since $\|\cdot\|_{L^1(E)}\le\sqrt \tau\|\cdot\|_{L^2(E)}$, for any
      $\gamma\in\Lip_{C,*}(E,\bR^d)$ we have
       \begin{align}
        \notag\|\bar\theta_\ell-\theta^\gamma(\cdot,\ctheta_k)\|_{L^{\infty}(E)}
        &\le \int_{E}\|\bar A(\bar\theta_\ell(s))-\gamma'(s)\|\deh{}s+\Const \vestep\\\notag
        &\le \sqrt\tau\| \bar A(\bar\theta_\ell)-\gamma'\|_{L^2(E)}+\Const \vestep
          \intertext{which yields, recalling~\eqref{eq:Igamma-bound}:}
          \label{eq:hoi-wei-wei}
          \|\bar\theta_\ell-\theta^\gamma(\cdot,\ctheta_k)\|_{L^{\infty}(E)}
        &\le \Const \sqrt{\tau I_{\thetasl}|_{E}(\gamma)}+\Const \vestep.
      \end{align}
      Similarly, we have:
      \begin{align}
        \notag\|\bar\theta_\ell-\theta^\gamma(\cdot,\ctheta_k)\|_{L^{\infty}(E)}
        &\le \int_E \|\bar A(\bar\theta_\ell(s)) -\bar A(\theta^\gamma(s,\ctheta_k))\| + \|\bar A(\theta^\gamma(s,\ctheta_k)) -\gamma'(s)\|\deh{}s+\Const \vestep\\\notag
        &\le
          \Const\tau\|\bar\theta_\ell-\theta^\gamma(\cdot,\ctheta_k)\|_{L^{\infty}(E)}+ \|\bar A(\theta^\gamma(s,\ctheta_k)) -\gamma'(s)\|_{L^1(E)}+\Const \vestep .
\end{align}
Then, provided $\Const \tau\leq 1/2$, Lemma \ref{l_Jgamma-bound} implies
\begin{equation}\label{eq:hoi-wej-wej}
\|\bar\theta_\ell-\theta^\gamma(\cdot,\ctheta_k)\|_{L^{\infty}(E)}\le\Const\sqrt{\tau \ratef^-_{\ctheta_k,\epsilon/2}|_E(\gamma)}+\Const\vestep.
\end{equation}
      Notice that if $I_{\thetasl}|_{E}(\gamma)=\infty$, then the sub-lemma holds
      trivially.
      We can then assume that, for a.e.\ $s\in E$ we have
      $\cZ(\gamma'(s),\bar\theta_\ell(s))<\infty$, \ie
      $\gamma'(s)\in \bD(\bar\theta_\ell(s))$.  Hence, if $\tau$ is sufficiently small (with respect to $\epsilon$), since
      $|\bar\theta_\ell(\cdot)-\theta^{\gamma}(\cdot,\ctheta_k)|\leq \Const\sqrt{ \tau}$, by Lemma~\ref{lem:hausdorff} we conclude that
      $\gamma'(s)\in \bD(\theta^{\gamma}(s,\ctheta_k))\cup\partial_{\epsilon/2}\bD(\theta^{\gamma}(s,\ctheta_k))$,
      that is, $\cZ^-_{\epsilon/2}(\gamma'(s),\theta^{\gamma}(s,\ctheta_k)) < \infty$.
      Thus, by definition~\eqref{eq:zeta-reg-def}, we have that, for any $s\in E$,
      \begin{align*}
        \cZ^-_{\epsilon/2}(\gamma'(s),\theta^\gamma(s,\ctheta_k))
        &=\cZ(b_\convexP(s),\theta^\gamma(s,\ctheta_k))
      \end{align*}
      where
      $b_\convexP(s) = \bar A(\theta^\gamma(s,\ctheta_k)) +\convexP(s)({\gamma'(s)}-{\bar
        A(\theta^\gamma(s,\ctheta_k))})$ and $ \convexP(s)\in[0,1]$ is the largest $\convexP$ such that $b_\convexP\not\in \partial_{\epsilon/2}\bD(\theta^\gamma(s,\ctheta_k))$.  Let
      $\bar b_\convexP(s) = \bar A(\bar\theta(s,\thetasl)) +\convexP(s)({\gamma'(s)}-\bar
      A(\bar\theta(s,\thetasl)))$; observe that, by definition,
      $\|\bar b_\convexP- b_\convexP\| <
      \Const|\bar\theta_\ell-\theta^\gamma(\cdot,\ctheta_k)|$.  By
      Lemma~\ref{lem:domainZ}-\ref{i_analyticity} we can expand $\cZ$ to second order
      obtaining:\footnote{ Notice that all second derivatives of $\cZ$ are uniformly
        bounded by some constant that depends on $\epsilon$, which we denote with
        $C_\epsilon$.}
      \begin{align*}
      &\cZ(b_\convexP(s),\theta^\gamma(s,\ctheta_k))-
        \cZ(\bar b_\convexP(s),\bar\theta_\ell(s)) =\cZ(b_\convexP(s),\theta^\gamma(s,\ctheta_k))-
          \cZ(b_\convexP(s),\bar\theta_\ell(s)))\\
&\phantom{\cZ(b_\convexP(s),\theta^\gamma(s,\ctheta_k)) \cZ(\bar b_\convexP(s),\bar\theta_\ell(s))= =,}
         +\cZ(b_\convexP(s),\bar\theta_\ell(s))-
          \cZ(\bar b_\convexP(s),\bar\theta_\ell(s))\\
          &=\partial_\theta\cZ(b_\convexP(s),\bar\theta_{\ell}(s))(\theta^\gamma(s,\ctheta_k)-\bar\theta_{\ell}(s) ) +\partial_b\cZ(\bar b_\convexP(s),\bar\theta_{\ell}(s))(b_\convexP-\bar b_\convexP)\\
        &\phantom = +C_\epsilon\cO(|\theta^\gamma(s,\ctheta_k)-\bar\theta_{\ell}(s)|^2).
      \end{align*}
      Next, we can expand $\partial_\theta\cZ$ and $\partial_b\cZ$ around the point
      $(\bar A(\bar\theta_{\ell}(s)), \bar\theta_{\ell}(s))$.  Recalling
      Lemma~\ref{lem:domainZ}-\ref{i_analyticity},\ref{i_goodSRB} we obtain:
      \begin{align*}
        |\partial_\theta\cZ(b_\convexP(s),\bar\theta_{\ell}(s))|
        &\le C_\epsilon\|\bar A(\bar\theta_{\ell}(s)) - b_\convexP\| \le \Const_\epsilon \|\bar A(\bar\theta_{\ell}(s))-\bar b_\convexP\|+\|\bar b_\convexP-
          b_\convexP \|\\
        &\le C_\epsilon\|\bar A(\bar\theta_{\ell}(s))-\gamma'(s)\|+C_\epsilon|\theta^\gamma(s,\ctheta_k)-\bar\theta_{\ell}(s)|.\\
        |\partial_b\cZ(\bar b_\convexP(s),\bar\theta_{\ell}(s))|
        &\le C_\epsilon\|\bar A(\bar\theta_{\ell}(s))-\bar b_\convexP\|\le C_\epsilon\|\bar A(\bar\theta_{\ell}(s))-\gamma'(s)\|.
      \end{align*}
      Since Lemma~\ref{lem:domainZ}-\ref{i_Zconvex},\ref{i_analyticity},\ref{i_goodSRB}
      imply that
      $\cZ(\bar b_\convexP(s),\bar\theta_\ell(s) )\le \cZ(\gamma'(s),\bar\theta_\ell(s)
      )$, we conclude that
      \begin{align*}
        \cZ^-_{\epsilon/2}(\gamma'(s)&,\theta^\gamma(s,\ctheta_k))
                                   \le \cZ(\gamma'(s),\bar\theta_{\ell}(s))\\
                                 &+ C_\eps\left[
                                   \|\bar A(\bar\theta_\ell(s))-\gamma'(s)\|\|\theta^\gamma(\cdot,\ctheta_k)-\bar\theta_{\ell}\|_{L^\infty(E)}
                                   +\|\theta^\gamma(\cdot,\ctheta_k)-\bar\theta_{\ell}\|^2_{L^\infty(E)}\right].
      \end{align*}
      Integrating the above expression over $E$ yields
      \begin{align*}
        \ratef^-_{\ctheta_k,\epsilon/2}|_{E}(\gamma)\le I_\thetasl|_E(\gamma) + C_\epsilon \Big[
        &\|\bar
          A(\bar\theta_\ell(s))-\gamma'\|_{L^1(E)}\|\theta^\gamma(\cdot,\ctheta_k)-\bar\theta_{\ell}\|_{L^\infty(E)}\\
        &+\tau\|\theta^\gamma(\cdot,\ctheta_k)-\bar\theta_{\ell}\|^2_{L^\infty(E)}\Big].
      \end{align*}
      Finally, using~\eqref{eq:hoi-wei-wei}, the relation
      $\|\cdot\|_{L^1(E)}\le\sqrt\tau\|\cdot\|_{L^2(E)}$ and~\eqref{eq:Igamma-bound} we
      obtain
      \begin{align*}
        \ratef^-_{\ctheta_k,\epsilon/2}|_{E}(\gamma)
        &\le (1+C_\epsilon\tau)I_{\thetasl}|_{E}(\gamma)+
          C_\epsilon\vestep\sqrt{\tau I_{\thetasl}|_{E}(\gamma)}+
          C_\epsilon \vestep^2\tau.
      \end{align*}
Since $2ab\leq a^2+b^2$, this concludes the proof of the sub-lemma.
    \end{proof}



    By  \eqref{eq:renzinoooo2} and Sub-Lemma \ref {l_firstHorribleClaim} it follows that
    for each standard pair $\ell$ such that
    $\Id^-_{\thetaref,\cgamma,\vestep}(\ell,k\tau)=1$, and
    $c_\epsilon \geq \tau\geq\ve^{-1}\vestep^2$,
      \begin{equation}\label{eq:upper-piece}
       \ve\log \ppath_\ell(\gball{E_k}(\cgamma_k,2\vestep ))
        \le -\inf_{\gamma\in \gball{E_k}(\cgamma_k,R_{k,\ell})}(1-C_\epsilon\tau)\ratef^-_{\ctheta_k,\epsilon/2}|_{E_k}(\gamma)+C_\epsilon\vestep^2.
      \end{equation}

    We have now obtained an estimate for each of the terms appearing in the product on the
    right hand side of~\eqref{e_markovBoundOnP} (see also \eqref{eq:vista1}). Next, we must join the above estimates for different time intervals
    $E_k$.  This can be done in various ways: we choose to
    control the trajectories at the endpoint of the intervals so that the paths
    corresponding to different time intervals will join naturally into a continuous path.
    To this end, for each $\tilde \gamma\in\cC^0([0,T],\bR^d)$ and $r>0$, let us define the sets
    \begin{align*}
      B^*|_{E}(\tilde\gamma,r)&=\{\gamma\in \gball{E}(\tilde\gamma,r)\st \gamma(0)=0\;;\gamma(\tau)=\tilde\gamma(\tau)\}.
    \end{align*}
    Then, for any $\gamma\in \gball{E}(\cgamma_k,R_{k,\ell})$ with $\gamma(0) = 0$, we
    define $\gamma^*\in B^*|_E(\cgamma_k,2R_{k,\ell})$ as:
    \begin{align*}
      \gamma^*(s)=\gamma(s)+\frac{\cgamma_k(\tau)-\gamma(\tau)}{\tau}s.
    \end{align*}
   Our next goal is to relate $\ratef^-_{\ctheta_k,\epsilon/2}|_{E}(\gamma)$ with
    $\ratef^-_{\ctheta_k,\epsilon}|_{E}(\gamma^*)$.
\begin{sublem}\label{sublem:divide-et-impera}
   There exist $c_\epsilon>0$ such that, for all  $c_\epsilon>\tau>c_\epsilon^{-1}\ve$,
\[
      \ratef^-_{\ctheta_k,\epsilon/2}|_{E}(\gamma)
      \geq \left(1-C_\epsilon\max\left\{\left(\frac\ve\tau\right)^{\efrac 14}+\tau,\left(\frac{\ve}{\ratef^-_{\ctheta_k,\epsilon}|_{E}(\gamma^*)}\right)^{\efrac 16}\right\}\right)
         \ratef^-_{\ctheta_k,\epsilon}|_{E}(\gamma^*)-C_\epsilon\ve,
\]
provided $\ve$ is small enough.
\end{sublem}
\begin{proof}
        Again, it suffices to consider the case
        $\ratef^-_{\ctheta_k,\epsilon/2}|_{E}(\gamma)<\infty$ (hence $\gamma'(s)\in \bD(\theta^\gamma(s,\ctheta_k))$ for all $s\in E$), the other case being trivial.
        Observe that $\|\hcgamma_{k,\ell}\|\leq 2C\tau$ hence, by  definition, $R_{k,\ell}\leq \epsilon\tau/4$ provided $\Const \epsilon^{-2}\ve\leq\tau\leq \Const \epsilon^2$.
        Accordingly, for a.e.\ $s\in E$.
    \begin{equation}\label{eq:gamma-e}
    \begin{split}
      &\|(\gamma^*)(s)-\gamma(s)\|\le R_{k,\ell}\\
      &\|(\gamma^*)'(s)-\gamma'(s)\|
      = \|{\cgamma_k(\tau)-\gamma(\tau)}\|/{\tau} \le R_{k,\ell}/\tau\leq \frac \epsilon 4.
    \end{split}
    \end{equation}
Since $\bD(\theta)$ varies continuously (see Lemma \ref{lem:hausdorff})  it follows that, for $\ve$ small enough, if,  for all $s$, $\gamma'(s)\not\in\partial_{3\epsilon/4}\bD(\theta^\gamma(s,\ctheta_k))$, then $(\gamma^*)'(s)\not\in\partial_{\epsilon/2}\bD(\theta^{\gamma^*}(s,\ctheta_k))$. We start to discuss such case as we will see that the other possibility can be essentially reduced to the present one.

We expand $\cZ$ to first order at  $({\gamma^*}',\theta^{\gamma^*}(s,\ctheta_k))$ using
    Lemma~\ref{lem:domainZ}-\ref{i_analyticity} and obtain
    \begin{align*}
      |\cZ
      (\gamma'(s),\theta^\gamma(s,\ctheta_k)) &-
        \cZ((\gamma^*)'(s),\theta^{\gamma^*}(s,\ctheta_k)|\le
      \|\partial_b\cZ((\gamma^*)',\theta^{\gamma}(s,\ctheta_k))\| \frac{R_{k,\ell}}\tau\\
       & +|\partial_\theta\cZ((\gamma^*)',\theta^{\gamma^*}(s,\ctheta_k))| R_{k,\ell}+ C_\epsilon(R_{k,\ell}/\tau)^2.
    \end{align*}
    Once again, expanding the derivatives to the first order at
    $(\bar A(\theta^{\gamma^*}(s,\theta)),\theta^{\gamma^*}(s,\theta))$ and using
    Lemma~\ref{lem:domainZ}-\ref{i_goodSRB} we conclude
    \begin{equation}\label{eq:zetachange}
    \begin{split}
      |\cZ(\gamma'(s),\theta^\gamma(s,\ctheta_k)) -
          &  \cZ((\gamma^*)'(s),\theta^{\gamma^*}(s,\ctheta_k)| \\
          &\le C_\epsilon\left(\|(\gamma^*)'-\bar A(\theta^{\gamma^*}(s,\theta))\|R_{k,\ell}/\tau + (R_{k,\ell}/\tau)^2\right).
    \end{split}
    \end{equation}
    Integrating over $E$, it follows that
    \begin{align}
      \notag\ratef^-_{\ctheta_k,\epsilon/2}|_{E}(\gamma)
      &= \ratef^-_{\ctheta_k,\epsilon/2}|_{E}(\gamma^*)+C_\epsilon\cO\left(
        \frac{R_{k,\ell}}{\tau}\int_{E}\|(\gamma^*)'-\bar
        A(\theta^{\gamma^*}(s,\theta))\|+\frac{R_{k,\ell}^2}{\tau}\right)\intertext{and, by Lemma~\ref{l_Jgamma-bound} and Remark~\ref{rem:quadratic},}
      &\geq\ratef^-_{\ctheta_k,\epsilon}|_{E}(\gamma^*)-C_\epsilon\left( \frac{R_{k,\ell}}{\sqrt \tau}\sqrt{\ratef^-_{\ctheta_k,\epsilon}|_{E}(\gamma^*)}+\frac{R_{k,\ell}^2}{\tau}\right).\label{e_secondHorribleBound}
    \end{align}
    Then, recalling that $\hat\cgamma_{k,\ell}=\cgamma_{k}-\gavg_\ell$ and~\eqref{eq:ell-gavgDef},~\eqref{eq:gamma-e},
    \begin{align*}
      \|\hat\cgamma_{k,\ell}\|\nl\infty
      &\le \|\cgamma_{k}-\gamma^*\|\nl\infty +\|\gamma^*-\bar\gamma_\ell\|\nl\infty \le 2
        R_{k,\ell}+\|\gamma^*-\bar\gamma_\ell\|\nl\infty,
\end{align*}
using the same argument as in~\eqref{eq:hoi-wej-wej} and since $R_{k,\ell}\geq \vestep$, we have
\begin{equation}\label{eq:hgamma-b}
        \|\hat\cgamma_{k,\ell}\|\nl\infty\le 3 R_{k,\ell}+ \Const\sqrt{\tau \ratef^-_{\ctheta_k,\epsilon}|_{E}(\gamma^*)}.
    \end{equation}
    Before continuing, note that if  $R_{k,\ell} \ge \|\hcgamma_{k,\ell}\|\nl\infty/4$, then the value of  $R_{k,\ell}$ cannot be given by $(\ve^{\efrac 14}\tau^{-\efrac 14}+\tau) \|\hcgamma_{k,\ell}\|\nl\infty$. If $R_{k,\ell}=\min \left\{ \ve\eefrac14 \tau\eefrac34, (\ve \tau)\eefrac16\|\hcgamma_{k,\ell}\|\nl\infty\eefrac23\right\}$, then
$ (\ve \tau)\eefrac16\|\hcgamma_{k,\ell}\|\nl\infty\eefrac23\geq \|\hcgamma_{k,\ell}\|\nl\infty/4$ implies
$ \|\hcgamma_{k,\ell}\|\nl\infty\leq 4^3\sqrt{\ve \tau}$. In turn, the later inequality implies, provided $\ve$ is small enough,
\[
(\ve \tau)\eefrac16\|\hcgamma_{k,\ell}\|\nl\infty\eefrac23\leq 4^2\sqrt{\ve\tau}\leq \ve^{\efrac 14}\tau^{\efrac 34}.
\]
We have thus seen that $\sqrt{\ve \tau}\leq R_{k,\ell} \le\Const \sqrt{\ve \tau}$ and
substituting it in~\eqref{e_secondHorribleBound} yields
\begin{align}\label{eq:first-up-short}
      \ratef^-_{\ctheta_k,\epsilon/2}|_{E}(\gamma)
      &\geq \left[1-C_\epsilon\left(\left[\frac{\ve}{\ratef^-_{\ctheta_k,\epsilon}|_{E}(\gamma^*)}\right]\eefrac12\right)\right]\ratef^-_{\ctheta_k,\epsilon}|_{E}(\gamma^*)-\Const \ve.
    \end{align}
We are thus left considering the case $R_{k,\ell} < \|\hcgamma_{k,\ell}\|\nl\infty/4$. By~\eqref{eq:hgamma-b}:
    \begin{align*}
      \|\hcgamma_{k,\ell}\|_{L^\infty}\le \Const\sqrt{\tau \ratef^-_{\ctheta_k,\epsilon}|_{E}(\gamma^*)}.
    \end{align*}
    Accordingly, we have:
    \begin{align*}
      R_{k,\ell}
      &\le \Const \max\left\{(\ve^{\efrac 14}\tau^{-\efrac 14}+\tau)\sqrt{\tau\ratef^-_{\ctheta_k,\epsilon}|_E(\gamma^*)}
        ,\ve\eefrac16\tau\eefrac12{\ratef^-_{\ctheta_k,\epsilon}|_E(\gamma^*)}\eefrac13
        ,\sqrt{\ve\tau}\right\}.
    \end{align*}
    If the first term realizes the $\max$, then by~\eqref{e_secondHorribleBound} we conclude
    \begin{align*}
      \ratef^-_{\ctheta_k,\epsilon/2, E}(\gamma)\geq
      (1-C_\epsilon\cO(\ve^{\efrac 14}\tau^{-\efrac 14}+\tau))\ratef^-_{\ctheta_k,\epsilon, E}(\gamma^*).
    \end{align*}
    Otherwise, if the second term realizes the $\max$,~\eqref{e_secondHorribleBound} gives:\footnote{ Just consider the two possibilities $\ratef^-_{\ctheta_k,\epsilon, E}(\gamma^*)>\ve$ and $\ratef^-_{\ctheta_k,\epsilon, E}(\gamma^*)<\ve$.}
\[
      \ratef^-_{\ctheta_k,\epsilon/2}|_{E}(\gamma)
      \geq \left[1-C_\epsilon\left(\left[\frac{\ve}{\ratef^-_{\ctheta_k,\epsilon}|_{E}(\gamma^*)}\right]\eefrac16\right)\right]\ratef^-_{\ctheta_k,\epsilon}|_{E}(\gamma^*)-C_\epsilon \ve.
\]
Finally, if the third term realizes the $\max$, then we have~\eqref{eq:first-up-short}
again.  This proves the sub-lemma in the case under consideration.

We are thus left with the case that, for some $s$, we have $\gamma'(s)\in\partial_{3\epsilon/4}\bD(\theta^\gamma(s,\ctheta_k))$. Let $S_\gamma\neq \emptyset$ be the collection of such $s$. Then we define  $b_\varrho(s)=\varrho(s)\gamma'(s)+(1-\varrho(s))\bar A(\bar\theta(s,\ctheta_k))$ where $\varrho(s)\in [0,1]$ is zero on the complement of $S_\gamma$ and such that $b_\varrho(s)$ belongs to the boundary of $\bD(\bar\theta(s,\ctheta_k))\setminus \partial_{3\epsilon/4}\bD(\bar\theta(s,\ctheta_k))$ otherwise. Also, we define $b_\varrho^*(s)=\varrho(s)(\gamma^*)'(s)+(1-\varrho(s))\bar A(\bar\theta(s,\ctheta_k))$. Note that $b_\varrho(s)\not \in \partial_{\epsilon/2}\bD(\bar\theta^{\gamma_\varrho}(s,\ctheta_k))$ and $b^*_\varrho(s) \not\in \partial_{\epsilon/2}\bD(\bar\theta^{\gamma^*_\varrho}(s,\ctheta_k))$ but, for $s\in S_\gamma$, $b^*_\varrho(s) \in \partial_{\epsilon}\bD(\bar\theta^{\gamma^*_\varrho}(s,\ctheta_k))$. By \eqref{eq:gamma-e} we have $\|b_\varrho-b^*_\varrho\|_{L^\infty}\leq R_{k,\ell}/\tau$ and
\[
\begin{split}
& \ratef^-_{\ctheta_k,\epsilon/2}|_{E}(\gamma)\geq  \int_{E}\cZ(b_\varrho(s),\bar\theta^\gamma(s,\ctheta_k))ds\\
& \int_{E}\cZ(b^*_\varrho(s),\bar\theta^{\gamma^*}(s,\ctheta_k))ds\geq\ratef^-_{\ctheta_k,\epsilon}|_{E}(\gamma^*).
\end{split}
\]
We can then conclude by expanding $\cZ$ as in \eqref{eq:zetachange} and, using the above relations, we obtain again
 \[
      \notag\ratef^-_{\ctheta_k,\epsilon/2}|_{E}(\gamma)
     \geq\ratef^-_{\ctheta_k,\epsilon}|_{E}(\gamma^*)-C_\epsilon\left( \frac{R_{k,\ell}}{\sqrt \tau}\sqrt{\ratef^-_{\ctheta_k,\epsilon}|_{E}(\gamma^*)}+\frac{R_{k,\ell}^2}{\tau}\right).
    \]
The argument is then concluded exactly in the same manner as before.
\end{proof}
 Using equations \eqref{e_markovBoundOnP}, \eqref{eq:vista1}, \eqref{eq:upper-piece} and Sub-Lemma \ref{sublem:divide-et-impera} we obtain:
    \begin{equation}\label{eq:opla'}
    \begin{split}
     & \ve\log\ppath_{\ellz}(\gball{}(\cgamma,\vestep))
      \leq-(1-\bK_{\epsilon,\ve,\tau})\left[\sum_k\inf_{\gamma\in
         B|_{E}(\cgamma_k,R_k)}\ratef^-_{\ctheta_k,\epsilon}|_{E}(\gamma^*)\right] \\
        &\phantom{ \ve\log\ppath_{\ellz}}
        +C_\epsilon\left(\frac{T \ve}\tau\right)^{\efrac16}\left[\sum_k\inf_{\gamma\in
         B|_{E}(\cgamma_k,R_k)}\ratef^-_{\ctheta_k,\epsilon}|_{E}(\gamma^*)\right]^{\efrac56}+C_\epsilon\frac{T\ve}{\tau}\\
         &\bK_{\epsilon,\ve,\tau}=C_\epsilon (\ve^{\efrac 14}\tau^{-\efrac 14}+\tau),
    \end{split}
    \end{equation}
    where, in the second line, we have used H\"older inequality, the assumption $\vestep^2\leq \ve$ and
    \begin{align*}
      R_k &= 2 \sup_{\{\ell\st\Id^-_{\thetaref,\cgamma,2\vestep}(\ell,k\tau)=1\}}\!\! R_{k,\ell}\;;\quad R = \max_k\{R_k\}.
    \end{align*}

Next, we must compute the sum in the square brackets. Let us define the sets   $\underline B^*_k = B^*|_{E}(\cgamma_k,R_k)$.
For each set of paths $\{\tilde \gamma_k\}_{k\in\{0,\cdots,{K-1}\}}$,
    $\tilde\gamma_k\in \underline B^*_k$, we can ``glue them together'' defining
    $\tilde\gamma(s)=\tilde\gamma_k(s-k\tau)+\cgamma(k\tau)$ for $s\in E_k$.  Clearly
    $\tilde \gamma\in \gball{}(\cgamma, R)$.  In addition,
    $\ratef^-_{\thetaref,\epsilon}(\tilde\gamma) =
    \sum_{k}\ratef^-_{\ctheta_k,\epsilon}|_{E}(\tilde\gamma_k)$, which yields
    \[
   \sum_k\inf_{\gamma\in
         B|_{E}(\cgamma_k,R_k/2)}\ratef^-_{\ctheta_k,\epsilon}|_{E}(\gamma^*)
         \geq \sum_k\inf_{\gamma\in\underline B^*_{k}}\ratef^-_{\ctheta_k,\epsilon}|_{E}(\gamma)\geq \inf_{\gamma\in
            B(\cgamma,R)}\ratef^-_{\thetaref,\epsilon}(\gamma)=:J_{*,R}.
   \]
Note that the right hand side of~\eqref{eq:opla'}  is bounded by $\Const C_\epsilon^5\ve\tau^{-1}T$ and the maximum is achieved for $\sum_k\inf_{\gamma\in
        \underline B|_{E}(\cgamma_k,R_k)}\ratef^-_{\ctheta_k,\epsilon}|_{E}(\gamma)$ proportional to $C_\epsilon \ve\tau^{-1}T$. Hence, for $J_{*,R}\geq  C_\epsilon \ve\tau^{-1}T$, the right hand side of~\eqref{eq:opla'} is a decreasing function of the quantity in square brackets. Accordingly, by eventually increasing the value of $C_\epsilon$,
    \begin{equation}\label{eq:sonno}
      \ve\log\ppath_{\ellz}(\gball{}(\cgamma,\vestep))
      \le -(1-\bK_{\epsilon,\ve,\tau})J_{*,\Rad+}+C_\epsilon\left(\frac{T \ve}\tau\right)^{\efrac16}J_{*,\Rad+}^{\efrac56}+C_\epsilon\ve\tau^{-1}T,
    \end{equation}
provided that $\Rad+:=\Rad+(\cgamma)\geq R = \max_k\{R_k\}$.
Further,
\[
\begin{split}
\|\hcgamma_{k,\ell}(s)\|&=\|\cgamma(k\tau+ s)-\cgamma(k\tau) -\bar z(s,\thetasl)+(\thetasl,0)\|\\
&\leq\|\hcgamma(k\tau+s)\|+\|\hcgamma(k\tau)\|+\|\bar z(k\tau+s,\thetaslz)-\bar z(k\tau,\thetaslz)-\bar z(s,\thetasl)+(\thetasl,0)\|\\
&\leq 2\|\hcgamma\|\nl\infty+\left\|\int_{0}^s\left[\bar A(\bar\theta(s',\bar\theta_{\ellz}(k\tau)))-\bar A(\bar\theta(s', \thetasl))\right]ds'\right\|.
\end{split}
\]
By continuity with respect to the initial conditions and recalling the assumption
$\|\hcgamma\|_{L^\infty}\geq C_{\epsilon,T}\sqrt \ve $, hence
$\|\hcgamma\|_{L^\infty}\geq \Const \vestep \tau$, it follows
\begin{equation}\label{eq:short-to-long}
\|\hcgamma_{k,\ell}\|_{L^\infty}\leq \Const \|\hcgamma \|_{L^\infty}.
\end{equation}
Also, by Lemma \ref{l_Jgamma-bound}, for each $\gamma\in B(\cgamma, \Rad+)$ we have
$\|\hcgamma\|_{L^\infty}\le 2 \|\hgamma\|_{L^\infty}\le\Const \sqrt{T
  \ratef^-_{\thetaref,\epsilon}(\gamma)}$. Hence
$\|\hcgamma\|_{L^\infty}^2\leq\Const TJ_{*,\Rad+}$ and, by \eqref{eq:sonno},
\begin{equation}\label{eq:bisonno}
\ve\log\ppath_{\ellz}(\gball{}(\cgamma,\vestep))
\le -\left[1-\bK_{\epsilon,\ve,\tau}-C_\epsilon\left(\frac{T^2 \ve}{\tau\|\hcgamma\|_{L^\infty}^2}\right)^{\efrac16}-C_\epsilon \frac{T^2 \ve}{\tau\|\hcgamma\|_{L^\infty}^2}\right]J_{*,\Rad+}.
\end{equation}

To validate \eqref{eq:bisonno} we still need to verify that $\Rad+\geq R$. To this end,
notice (see the beginning of the proof of Lemma~\ref{sublem:divide-et-impera}) that
$R_{k}\leq \tau$, provided $\tau\leq C_\epsilon$, which is implied by \eqref{eq:taudef}
when $\|\hcgamma\|_{L^\infty}\geq C_{\epsilon, T}\sqrt\ve$.  Thus, using \eqref{eq:Rkell},
\eqref{eq:short-to-long} and the choice \eqref{eq:taudef}, we have
\[
 \begin{split}
 R_{k}&\leq C_T
        \max\left\{\frac{\|\hcgamma\|\nl\infty^{\efrac {15}{14}}\ve^{\efrac 3{14}}}{T^{\efrac 1{14}}}+\|\hcgamma\|\nl\infty^{\efrac 57}T^{\efrac 27}\ve^{\efrac 17},%
        \|\hcgamma\|\nl\infty^{\efrac{13}{21}}T^{\efrac 1{21}}\ve^{\efrac 4{21}},
        \frac{\ve^{\efrac 47}T^{\efrac 17}}{\|\hcgamma\|\nl\infty^{\efrac 17}}\right\}.
\end{split}
\]
One can then check that
\[
R_{k}\leq C_T\left[\|\hcgamma\|\nl\infty^{\efrac 57}\ve^{\efrac 17}+\sqrt\ve\right].
\]
The above implies the claim $\Rad+\geq R $ and $\Rad+\leq\max\{ \|\hcgamma \|\nl\infty, C_{\epsilon,T}\sqrt\ve\}$.
Substituting the choice \eqref{eq:taudef} in equation \eqref{eq:bisonno}, yields
     \begin{equation}\label{eq:jet-lag-3}
      \ve\log\ppath_{\ellz}(\gball{}(\cgamma,\vestep))
      \le -\left[1-C_\epsilon\left(\frac{T^2 \ve}{\|\hcgamma\|_{L^\infty}^2}\right)^{\efrac17}\right]J_{*,\Rad+}.
    \end{equation}

    To obtain the lower bound~\eqref{e_finalBoundLower} we argue along the same lines
    (with a different choice of $\vestep$ and $\tau$), but the argument turns out to be a
    bit simpler.  To further simplify our discussion we are not going to pursue optimal
    results.  We use Lemma~\ref{lem:lower} with $\vestep_-=h=K_\delta^{-2} T$ and
    $T=\tau=\sqrt\ve$, and $\delta=\epsilon$ to write
 \begin{equation}\label{eq:lower1}
  \ve\log\ppath_\ell(\gball{E}(\cgamma_k, \vestep_-))\geq - I_{\thetasl}(\cgamma_k)-c_\delta \ve^{\efrac 58}.
\end{equation}
 Next, we claim that, for all $\ell$ such that $\Id^+_{\thetaref,\cgamma,\vestep_-}(\ell,k\tau)=1$,
 \begin{equation}\label{eq:JvrsusImeno}
\begin{split}
        \ratef^+_{\ctheta_k,\epsilon}|_{E}(\cgamma_k)
        &\ge (1-C_\epsilon \tau)I_{\thetasl}|_{E}(\cgamma_k) - C_\epsilon \vestep_-^2.
\end{split}
\end{equation}
 The above relation is trivial if the left hand side is infinite. Otherwise, recalling~\eqref{eq:rate2} and~\eqref{eq:zeta-reg-def}, it can be proven along the lines of Sub-Lemma~\ref{l_firstHorribleClaim}. Accordingly, Lemma~\ref{lem:markov} implies
\[
\begin{split}
  \ve\log\ppath_\ell(B(\cgamma, \vestep_-))&\geq -(1+C_\epsilon\sqrt\ve)\sum_{k=0}^{K-1}\ratef^+_{\ctheta_k,\epsilon}|_{E}(\cgamma_k)-C_\epsilon T\ve^{\efrac 18}\\
  &=-(1+C_\epsilon\sqrt\ve)\ratef^+_{\thetaref,\epsilon}(\cgamma)-C_\epsilon T\ve^{\efrac 18} .\qedhere
\end{split}
 \]
 \end{proof}

\subsection{Large and moderate deviations for general sets}\label{subsec:large-long-Q}\ \newline
This subsection contains the second step of our argument that leads to the proof of our
main Large Deviations result.  Concretely, we show how Theorem~\ref{thm:large} follows
from Lemma~\ref{l_finalBounds}. For the upper bound, we use a relatively standard
combinatorial argument which allows to obtain an estimate for the probability of an
arbitrary event by covering it with balls; for the lower bound, we simply bound it from
below with the measure of a ball contained in the event.
  \begin{proof}[{\bf Proof of Theorem~\ref{thm:large}}]
    \ %
  \newcommand{\numI}{Z}
  Let $\ellz$ be a standard pair so that $|\thetaslz - \theta_0|\leq \ve$ and let
  $\epsilon = \Delta_*$. Clearly, it suffices to prove the theorem for such standard
  pairs.

Our goal is to estimate, from above and below, the probability of the event $Q_\ve$.

We start with the lower bound,  let $Q_\ve^-=\{\gamma\in Q_\ve\;:\; \gball{[0,T]}(\gamma, \ve^{\efrac 1{2}})\subset Q_\ve\}$.
    Then, for each $\cgamma\in Q_\ve^-$, inequality \eqref{e_finalBoundLower} implies
    \[
     \ppath_{\ellz}\left(Q_\ve\right)\ge \ppath_{\ellz}\left(B(\cgamma, \ve^{\efrac 1{2}})\right)\ge e^{-\vei[(1+C_\epsilon\ve\eefrac 12) \ratef^+_{\theta_0,\epsilon}(\cgamma)+C_{\epsilon,T} \ve\eefrac{1}{8}]}.
    \]
We can then conclude the argument by taking the sup for $\cgamma\in Q_\ve^-$.

To obtain the upper bound, first recall that  $\supp\ppath_\ellz\subset\Lip_{C,*}([0,T],\bR^d)$ hence, setting $Q=Q_\ve\cap \Lip_{2C,*}([0,T],\bR^d)$ holds
\[
\ppath_{\ellz}(Q_\ve)=\ppath_{\ellz}(Q).
\]
We will construct a class of coverings of $Q$ and use Lemma~\ref{l_finalBounds} to
estimate the probability of each elements of these coverings.  Let us first recap some
notations.  For any set $\tilde Q\subset \gspace$ let
$\varrho(\tilde Q)=\varrho(\theta_0,\tilde Q)=\inf_{\gamma\in \tilde Q}\|\gamma-\bar
\gamma(\cdot,\theta_{0})\|_\infty$, and
    \[
    Q^+:=\bigcup_ {\gamma\in  Q}\gball{}(\gamma,\Rad+(\gamma))\supset \overline Q
    \]
    where
    $ \Rad+(\gamma)
    =C_{\epsilon,T}\min\left\{\ve\eefrac17\|\hgamma\|\nl\infty\eefrac57+\ve^{\efrac
      12}\right\}$.  Note that if $\varrho(Q)\leq C_{\epsilon,T}\sqrt\ve$, then
    $ \gavg_{\ellz}(\cdot)=\gavg(\cdot,\theta_0)\in Q^+$, hence the statement of the
    theorem is trivially true. We can thus assume
    $\varrho(Q)\geq C_{\epsilon,T} \sqrt\ve$.

We want to estimate the measure of $Q$ by covering it with balls of the type
$B(\centgamma,\vesteps)$, for some $\vesteps>0$. To this end we must construct a
$(\vesteps/2)$-net. To do so, subdivide the interval $[0,T]$ in sub-intervals of equal
lengths $\timeI_j = [s_j,s_{j+1})$, where
$s_j = j\efrac{\vesteps}{(1+6C)}=:s_{j-1}+\Delta_s$ and recall that $C$ is an upper bound
on the Lipschitz constant of all paths that are in the support of $\ppath_\ellz$.  Denote
with $\numI = T/\Delta_s$, so that\footnote{ Once again we disregard the possibility that
  $\numI$ is not a natural number.} $s_\numI = T$.  Let
$\aseq = \{a_l\}_{l\in\{0,\cdots,\numI-1\}}$ be a (finite) sequence with values in
$\frac{1}{2\sqrt d}\bZ^d$ and let $\gamma_\aseq$ be the unique (Lipschitz) continuous path
in $\gspace$ that, for each $k$, satisfies (for a.e.\ $s\in[0,T]$) the equation
    \begin{align*}
      \gamma_{\aseq}'(s)
      &=\bar A(\theta^{\gamma_{\aseq}}(s, \theta_0))+a_j &\text{for $s\in \timeI_j$}
    \end{align*}
    with initial condition $\gamma_\aseq(0) = 0$.  Let
    $\pathnet = \{\aseq \st \|a_j\| < 2C \text{ for all $j$}\}$; observe that, by
    construction, $\pathnet$ is a finite set (indeed
    $\#\pathnet < e^{\const T\vesteps\invr}$) and since
    $\supp\ppath_\ellz\subset\Lip_{C,*}([0,T],\bR^d)$ we conclude that if $B(\gamma_\aseq,\vesteps)\cap \supp \ppath_\ellz\neq \emptyset$, then $\aseq\in \pathnet$.  We now claim that
    $\bigcup_{\aseq\in \pathnet} \gamma_\aseq$ is a $\vesteps$-net for the support of
    $\ppath_\ellz$, \ie
    $\bigcup_{\aseq\in \pathnet}B(\gamma_\aseq,\vesteps)\supset\supp\ppath_\ellz$.  In fact, for each $\aseq\in \pathnet$ and $k\in\{0,\cdots,Z\}$,
    $\partial_{a_j}\gamma_\aseq(s_{j+1})=\vesteps\Id+\cO(\vesteps^2)$, by the smooth dependence of a solution from the vector field.
Thus for any path $\gamma\in\Lip_{C,*}([0,T],\bR^d)$, provided $\ve$ is small enough,  there exists $\aseq\in \pathnet$ so that\footnote{ Recall that the lattice $\frac{1}{2\sqrt d}\bZ^d$ is a $r$-net for $\bR^d$ for any $r\geq 1/4$.}
    \begin{align*}
      \|\gamma(s_j)-\gamma_\aseq(s_j)\| < \frac38\Delta_s      \text{ for any $j \in\{0,\cdots, \numI\}$}.
    \end{align*}
    By the Lipschitz property, for any $j\in\{0,\cdots,\numI\}$ and $s\in\timeI_j$,
    \begin{align*}
      \|\gamma(s)-\gamma_\aseq(s)\| < \frac38\Delta_s+3C \Delta_s < \vesteps/2.
    \end{align*}
 This proves our claim and concludes the construction of a $\vesteps$-net of paths.

Next, let us define $Q_k=\{\gamma\in Q\;|\; \|\hgamma\|_{L^\infty}\in [2^k\varrho(Q), 2^{k+1}\varrho(Q))$. By our current assumption $\varrho(Q)\geq C_{\epsilon,T} \sqrt\ve$ and the fact that $\|\gamma\|_{L^\infty}\leq C T$, we have
\[
Q\subset \bigcup_{k=0}^{\const\ln\vei}Q_k.
\]
Let us fix some $k$. Then, by hypothesis, for $\ppath_{\ellz}$-almost all $\gamma\in Q_k$
we have (see~\eqref{eq:last-resort})
$\Clip(\gamma)\leq T^{-\efrac {11}7}\ve^{-\efrac 27}\varrho(Q)^{\efrac{11}7}2^{\efrac{11
    k}7}=:\Clip(k)$, $\vestep(\gamma)\in[\vestep_k 2^{-\efrac 17}, \vestep_k ]$,
$\vestep_k= \sqrt\ve\left(\frac {T^2\ve}{\varrho(Q)^2}\right)^{\efrac 1{14}}2^{-\efrac
  k7}$ and, for each $|s-s'|\leq \frac{\vestep_k}{2 \Clip(k)}=:h_\star$,
\[
\|\gamma(s)-\gamma(s')\|\leq \vestep_k/4
\]
Hence, if $\max\left\{\|\gamma_\aseq(jh_\star)-\gamma(jh_\star)\|,\|\gamma_\aseq((j+1)h_\star)-\gamma((j+1) h_\star)\|\right\}\leq \vestep_k/2$ then we have, for each $s\in [0,h_\star]$
\[
\|(1-sh_\star^{-1})\gamma_\aseq(jh_\star)+sh_\star^{-1}\gamma_\aseq((j+1)h_\star)-\gamma(jh_\star+s)\|\leq 3\vestep_k/4.
\]
Accordingly, we need only $\Const$ paths to describe all possible behaviors in an interval
$[h_\star, (j+1)h_\star]$ with a precision $\vestep_k$. This implies that there exists
$\cA_{Q_k}\subset \cA$ such that
$\bigcup_{\aseq\in\cA_{Q_k}} B(\gamma_\aseq,\vestep_k)\supset Q_k$ and
$\#\cA_{Q_k}\leq e^{\const T h_\star^{-1}}=e^{\const T\vestep_k^{-1} \Clip(k)}$.

Accordingly, Lemma~\ref{l_finalBounds} implies
\begin{equation}\label{eq:herewego}
\begin{split}
  \ppath_{\ellz}\left(Q\right)&\le\sum_{k=0}^{\const\ln\vei}  \ppath_{\ellz}\left(Q_k\right)\leq \sum_{k=0}^{\const\ln\vei}  \sum_{\aseq\in\cA_{Q_k}} \ppath_{\ellz}\left(\gball{}(\gamma_\aseq,\vestep_k)\right)\\
  &\leq\sum_{k=0}^{\const\ln\vei} \hskip-.3cm\#\left(\cA_{Q_k}\right)
  \expo{-\ve^{-1}{\left(1-C_{\epsilon}\frac{T\eefrac
          27\ve\eefrac1{7}}{\varrho(Q_k)\eefrac{2}{7}}\right)}\inf_{\gamma\in Q_k^+}
    \ratef^-_{\theta_0,\epsilon}(\gamma)}.
\end{split}
\end{equation}
Note that, since $\rho(Q_k)\geq C_{\epsilon,T} \sqrt \ve$, we have $\rho(Q_k^+)\geq \frac 12\rho(Q_k)$. Then, by Lemma \ref{l_Jgamma-bound},~\eqref{eq:rate2} and Remark~\ref{rem:quadratic}, we have
\begin{equation}\label{eq:varrho-J}
\begin{split}
\rho(Q_k)^2&\leq C_{\epsilon} T \inf_{\gamma\in Q_k^+} \ratef^-_{\theta_0,\epsilon}(\gamma)\leq C_\epsilon T\inf_{\gamma\in Q_k} \int_0^T\|\hgamma'(s)\|^2 ds\leq C_\epsilon T^2 \Clip(k)^2\\
&\leq C_\epsilon T^{-\efrac 87}\ve^{-\efrac 47}\varrho(Q_k)^{\efrac {22}7} .
\end{split}
\end{equation}
Hence, $\ve\Clip (k)T\vestep_k^{-1}\leq C_{\epsilon}\frac{T\eefrac 27\ve\eefrac1{7}}{\varrho(Q_k)\eefrac{2}{7}}\inf_{\gamma\in Q_k^+} \ratef^-_{\theta_0,\epsilon}(\gamma)$. Accordingly,
\begin{equation}\label{eq:andgo}
 \ppath_{\ellz}\left(Q\right)\leq\sum_{k=0}^{\const\ln\vei}  \expo{-\ve^{-1}{\left(1-C_{\epsilon}\frac{T\eefrac 27\ve\eefrac1{7}}{\varrho(Q_k)\eefrac{2}{7}}\right)}\inf_{\gamma\in Q_k^+} \ratef^-_{\theta_0,\epsilon}(\gamma)}.
\end{equation}
Next, let us define the sequence $k_0=0$, $k_{j+1}$ being the smallest integer $k$ such that $2^{k}\geq C_{\ve} (\varrho(Q)^2/(T^2\ve))^{\efrac 27}2^{\frac {11}{7}k_j}$. By \eqref{eq:varrho-J} it follows $\inf_{\gamma\in Q_{k_{j+1}^+}}\ratef^-_{\theta_0,\epsilon}(\gamma)\geq 2\inf_{\gamma\in Q_{k_j}^+} \ratef^-_{\theta_0,\epsilon}(\gamma)$. One can check by induction that $k_j\leq e^{c_\star j}\ln(\varrho(Q)^2/\ve) $ for some constant $c_\star>0$, depending on $T$. Using again \eqref{eq:varrho-J}, we can finally conclude:
\begin{align*}
 \ppath_{\ellz}\left(Q\right)&\leq\sum_{j=0}^{\infty} e^{c_\star j} \ln(\varrho(Q)^2/\ve) \expo{-\ve^{-1}{\left(1-C_{\epsilon}\frac{T\eefrac 27\ve\eefrac1{7}}{\varrho(Q)\eefrac{2}{7}}\right)}2^j\inf_{\gamma\in Q^+} \ratef^-_{\theta_0,\epsilon}(\gamma)}\\
 &\leq\sum_{j=0}^{\infty}  \expo{-\ve^{-1}{\left(1-C_{\epsilon,T}\frac{\ve\eefrac1{7}}{\varrho(Q)\eefrac{2}{7}}\right)}2^j\inf_{\gamma\in Q^+} \ratef^-_{\theta_0,\epsilon}(\gamma)+\const j}\\
  &\leq\sum_{j=1}^{\infty}  \expo{-j\ve^{-1}{\left(1-C_{\epsilon,T}\frac{\ve\eefrac1{7}}{\varrho(Q)\eefrac{2}{7}}\right)}\inf_{\gamma\in Q^+} \ratef^-_{\theta_0,\epsilon}(\gamma)}\\
  &\leq  \expo{-\ve^{-1}{\left(1-C_{\epsilon,T}\frac{\ve\eefrac1{7}}{\varrho(Q)\eefrac{2}{7}}\right)}\inf_{\gamma\in Q^+} \ratef^-_{\theta_0,\epsilon}(\gamma)}.\qedhere
\end{align*}
\end{proof}

  \subsection{Proof of  Propositions~\ref{p_rateFunctionBasicProperties},~\ref{p_largeDevzQuadraticBound} and Corollaries~\ref{cor:large-dev2},~\ref{cor:large-dev3}}\label{s_proofCorollaries}\ \newline
  We conclude this section by proving the propositions and corollaries that were
  stated in Section~\ref{sec:results} without a proof.
  \begin{proof}[Proof of Proposition~\ref{p_rateFunctionBasicProperties}]
    We start by proving~\eqref{e_largeDeviationPrincipleKifer}.  Fix $R>0$ and
    $\lseparation > 0$; by Lemma~\ref{lem:entropy} for any $C>\| A\|_{L^\infty}$ if
    $\gamma$ is not $C$-Lipschitz, then $\ratef_{\theta_0}(\gamma) = \infty$.  Hence we
    can assume that all elements of $Q$ are $C$-Lipschitz paths; this in particular
    implies that $\Rad+(\gamma) < C_{\lseparation ,T}\ve^{1/8}$ (recall that $\Rad+$ was
    defined in~\eqref{e_defRadPm}).  Now let $Q^{+}_R = \bigcup_{\gamma\in Q}B(\gamma,R)$
    and $Q^{-}_{R} = \{\gamma\in Q\st B(\gamma,R)\subset Q\}$.  For $\ve$ small enough,
    $Q^{-}_{R}\subset Q^-$ and, $Q^{+}\subset Q^{+}_{R}$ (see~\eqref{eq:Qdef} for the
    definition of $Q^-$, $Q^+$) and, by Theorem~\ref{thm:large}, taking first $\liminf$
    and $\limsup$ as $\ve\to 0$ and then the liminf for $R\to 0$:
    \begin{align*}
      -\!\!\inf_{\gamma\in \intr Q}\ratef_{\theta_0,\lseparation}^+(\gamma)\le%
      \liminf_{\ve\to0}\ve\log \ppath_\mu (Q)\le%
      \limsup_{\ve\to 0}\ve\log \ppath_\mu (Q)\le%
      -\inf_{\gamma\in \overline Q}\ratef_{\theta_0,\lseparation}^-(\gamma),
    \end{align*}
    the only non-obvious inequality being the last one. To prove it note that if
    $\rho(\theta_0,Q)=0$, then the inequality is trivially true, we can then assume
    $\rho(\theta_0,Q)>0$, hence, for $\ve$ small enough $Q$ is $\ppath_\mu$-regular (see
    Remark~\ref{rem:shortT}).  Next, let us define
    $\beta=\liminf_{R\to 0}\inf_{\gamma\in
      Q^+_R}\ratef_{\theta_0,\lseparation}^-(\gamma)$. Then for each $\delta>0$ there
    exists $R_\delta<\delta$ and $\gamma_\delta\in Q^+_{R_\delta}$ such that
    $\ratef^-_{\theta_0,\lseparation}(\gamma_\delta)\leq \beta+\delta$. Since the
    $C$-Lipschitz function are compact, there exists a subsequence $\delta_j\to 0$ such
    that $\gamma_{\delta_j}\to \gamma_*\in \bar Q$. The claim follows by the lower
    semicontinuity of $\ratef^-_{\theta_0,\lseparation}$ (see Lemma~\ref{lem:rate-lower}).

Next, we want to take the limit  $\lseparation\to 0$ and prove~\eqref{e_largeDeviationPrincipleKifer}, that is
    \begin{equation}\label{eq:chepallegalattiche}
    \begin{split}
      -\inf_{\gamma\in \intr Q}\ratef_{\theta_0}(\gamma)&\le%
      \liminf_{\ve\to0}\ve\log \ppath_\mu (Q)\\%
      &\le\limsup_{\ve\to 0}\ve\log \ppath_\mu (Q)\le%
      -\inf_{\gamma\in \overline Q}\ratef_{\theta_0}(\gamma).
      \end{split}
     \end{equation}
If    $\eta=\inf_{\gamma\in \intr Q}\ratef_{\theta_0}(\gamma)=\infty$, then the first inequality is trivially true. Otherwise, by Lemma~\ref{lem:rate-lower}, for each $\delta>0$ there exists $\gamma_\delta\in\intr Q \cap \intr\fkD(\ratef_\theta)$ such that $\eta+\delta>\ratef_{\theta_0}(\gamma_\delta)$.
 Accordingly, there exists $\lseparation$ such that
\[
\eta+\delta\geq \ratef_{\theta_0}(\gamma_\delta)=  \ratef^+_{\theta_0, \lseparation}(\gamma_\delta)\geq \inf_{\gamma\in \intr Q}\ratef^+_{\theta_0, \lseparation}(\gamma)
\]
by the arbitrariness of $\delta$ the first inequality of~\eqref{eq:chepallegalattiche} follows.

To prove the last inequality of ~\eqref{eq:chepallegalattiche} let
$\eta=\lim_{\lseparation\to 0}\inf_{\gamma\in \overline
  Q}\ratef^-_{\theta_0,\lseparation}(\gamma)$. If $\eta=\infty$, then
$\overline Q\cap \fkD(\ratef^-_{\theta_0})=\emptyset$ hence the inequality
follows. Otherwise, for each $\delta$ there exists $\Delta_\delta>0$ such that, for all
$\lseparation\leq \Delta_\delta$, there exists
$\gamma_{\lseparation} \in \overline Q\cap \fkD(\ratef^-_{\theta_0})$ such that
$\eta+\delta\geq \ratef^-_{\theta_0, \lseparation}(\gamma_\lseparation)\geq
\ratef^-_{\theta_0, \Delta_\delta}(\gamma_\lseparation)$, where the last inequality
follows form the definition of $\ratef^-_{\theta_0, \lseparation}$. By taking a
subsequence we can assume that $\gamma_ \lseparation$ converges to
$\gamma\in \overline Q\cap \fkD(\ratef^-_{\theta_0})$.  We can then
establish~\eqref{e_largeDeviationPrincipleKifer} by taking first the limit
$\lseparation\to 0$ followed by $\delta\to 0$ and applying Lemma~\ref{lem:rate-lower}
twice.

Item (a) follows from Lemma~\ref{lem:entropy} and Remark~\ref{rem:effective} while item
(b) is a direct consequence of the properties of $\cZ$ detailed in
Lemma~\ref{lem:domainZ}.
\end{proof}
  \begin{proof}[Proof of Proposition~\ref{p_largeDevzQuadraticBound}]
By Lemma~\ref{lem:entropy} , for any
    $\theta\in\bT$, $\cZ(\cdot,\theta)$ (defined in~\eqref{eq:M-0}) is finite only in a
    compact set on which it is bounded.  Then Lemma~\ref{lem:domainZ} implies
    that there exists $c>0$ such that $\cZ(b,\theta)\geq c (b-\bar\omega(\theta))^2$ for all  $\theta\in\bT$. Hence,
    \begin{equation}\label{eq:silly-billy}
     I_{\theta^*_0}(\gamma) =   \int_0^T\cZ(\gamma'(s),\bar\theta(s,\thetas_0)) \deh{}s\geq c
      \int_0^T\left\|\gamma'(t)-\bar A(\bar\theta(s,\thetas_0))\right\|^2 \deh t,
    \end{equation}
    where we assumed that $\gamma$ is Lipschitz (otherwise
    $I_{\theta^*_0}(\gamma)=\infty$ by definition).  Hence, for each $\gamma\in Q_\star=\{\|\hgamma\|_\infty\ge \frac 12 R\}$,
    \[
     R\leq \int_0^T\|\gamma'(t)-\bar\gamma'(t,\theta^*_0)\|\deh t\leq \sqrt {c^{-1}T  I_{\theta^*_0}(\gamma)} ,
    \]
We can now apply Lemma~\ref{lem:upper}. Note that, for $T_{\max}$ small enough and $\bar C$ large enough,
$R_\ve(\gamma)\leq  \|\hgamma\|_\infty/2$.
This implies that $Q_{\ve,+}\subset Q_\star$ and the Lemma follows.
 \end{proof}
  \begin{proof}[Proof of Corollary~\ref{cor:large-dev2}]
    Let us start with the upper bound.  For any $\gamma\in Q$, let
    $\gamma_\ve= \ve^{\beta}\gamma+(1-\ve^{\beta})\gavg$.  Since $Q$ is bounded, we have
    $\|\gamma_\ve-\gavg\|\nc0 < C_Q \ve^\beta$ and in particular (recall the definitions of
    $\Rad+$ given in~\eqref{e_defRadPm} and of $\rho,\vestep, \Clip$ in \eqref{eq:last-resort})
    $\Rad+(\gamma_\ve)\le C_T \ve^{\efrac 17+\efrac {5\beta}7}$, $\rho(\theta_0,Q_\ve)=\ve^\beta\rho(\theta_0,Q)$, $\Clip(\gamma_\ve)=\|\hgamma\|_{L^\infty}^{\efrac{11}7}T^{\efrac {11}7}\ve^{-\efrac 27+\efrac{11\beta}7 }$. Thus $\Clip\leq C$ only if $\beta\geq \ve^{\efrac 2{11}}$, in such a case
\[
\|\gamma_\ve(s)-\gamma_\ve(s')\|\leq \ve^\beta\|\gamma(s)-\gamma(s')\|\leq \Const \ve^\beta|s-s'|\leq C_Q\ve^{-\efrac{4\beta}{7}+\efrac 27}\vestep(\gamma_\ve)
\leq \frac {\vestep(\gamma_\ve)}4
\]
since $\beta<\frac 12$, that is the events $Q_\ve$ are always $\ppath_\mu$-regular.  In
addition, since $\frac 17+\frac 57\beta>\beta$, it follows that, for all $R>0$, for all
$\ve$ small enough we have
$Q_\ve^+\subset
\{\ve^{\beta}\gamma(\cdot)+(1-\ve^{\beta})\gavg(\cdot,\theta_0)\}_{\gamma\in
  Q^+_R}=:Q^+_{\ve,R}$ where $Q^+_R=\bigcup_{\gamma\in \overline Q}B(\gamma, R)$. Also,
for $\ve$ small enough, $Q^+_{\ve,R}\subset \intr\fkD(\ratef_{\theta_0})$. In particular,
for any $\eps > 0$ and sufficiently small $\ve$,
$\ratef_{\theta_0,\eps}^\pm(\gamma_\ve) = \ratef_{\theta_0}(\gamma_\ve)$ for any
$\gamma\in Q^+_{\ve,R}$ (recall the definition of $\ratef_{\theta_0,\eps}^\pm$ given
in~\eqref{eq:rate-good}).  Also by~\eqref{eq:mathscrI} and the smoothness of $\Sigma$,
since $Q$ is Lipschitz bounded,
    \begin{align*}
      \ratef_{\theta_0}(\gamma_\ve)=\frac{\ve^{2\beta}}2\int_0^T\hskip-6pt\langle
      \gamma'(s)-\bar A(\bar\theta(s)), \Sigma(\bar\theta(s))^{-1}\left[\gamma'(s)-\bar
      A(\bar\theta(s))\right]\rangle \deh{}s+o(\ve^{2\beta}).
    \end{align*}
    We then apply Theorem~\ref{thm:large} and the above estimate. Taking the $\limsup$ as
    $\ve\to0$ followed by the limits $R\to 0$ yields the wanted result.  The lower bound follows by similar arguments.
  \end{proof}
  \begin{proof}[Proof of Corollary~\ref{cor:large-dev3}]
    Let $C_*$ large enough and set
    $\gamma_\ve=\ve^{\frac 12}\gamma-(1-\ve^{\frac 12})\bar z$.  For each $\gamma\in Q$
    we have (recall the definition of $\Rad+$ given in~\eqref{e_defRadPm}) $\Rad+(\gamma_\ve)\le \lprecisione\|\hgamma\|_\infty\sqrt\ve$ and that $Q_\ve$ is  $\ppath_\mu$-regular.  Thus, in the
    notation of Theorem~\ref{thm:large}, $(\hat Q^+)_\ve\supset Q^+_\ve$.
    Since~\eqref{eq:mathscrI} implies
    \[
      \ratef_{\theta_0}(\gamma_\ve)=\ve\ratef_{\operatorname{Lin}}(\gamma)+\cO(\ve^{\frac 32})
    \]
    the result follows directly by Theorem~\ref{thm:large}.
  \end{proof}
\section{Local Limit Theorem}\label{sub:lclt}
The results of the previous section allow to study deviations
$\deviationsl_n=\theta_n-\bar\theta(\ve n,\thetasl)$ from the average of order larger than
$\sqrt\ve$, but give no information on smaller fluctuations, except for the fact that with
very high probability the fluctuations are of order $\sqrt \ve$ or smaller.  In fact,
in~\cite{DimaAveraging}, it is proven that the fluctuations from the average, once
renormalized by the multiplicative factor $\ve^{-\efrac 12}$, converge in law to a
diffusion process.  Here we go one (long) step forward and we prove Theorem~\ref{thm:lclt}
which is the equivalent of a Local Central Limit Theorem with error terms for the above
convergence.

\begin{rem}
  As already mentioned before the statement of Theorem~\ref{thm:lclt}, although we will
  restrict our discussion to fluctuations of the variable $\theta$, the same type of
  arguments would yield corresponding results for $z$.
\end{rem}

A standard technique to prove local CLT type results for a dynamical systems leads to the
study of the leading eigenvalue of a suitable transfer operator (see, e.g.,
\cite{guivarch-hardy}).  While this idea works quite well for uniformly hyperbolic
systems, it is much harder to implement for partially hyperbolic systems. Here we will use
the standard pair technology to reduce our problem to a slowly varying uniformly
hyperbolic system. This will be achieved in several steps.

The first step consists in expressing the fluctuation in terms of a more explicit random
variable $\bA$: this is done in Section~\ref{sec:birk}.  Then, in
Section~\ref{subsec:prooflclt}, we first show how Theorem~\ref{thm:lclt} follows rather
easily once one has computed the characteristic function of the random variable $\bAs$,
which is a suitable mollification of $\ve\bA$.  We then discuss which technical estimates
are necessary to compute the Fourier transform defining the characteristic function of
$\bAs$ and we use the standard pair formalism to recast them in a form to which, in the
next sections, it will be possible to apply the transfer operator technique, effectively
reducing the problem to one similar to the skew-product case. The difference being that
the fast dynamics is slowly varying rather than a constant.  Hence, instead of having a
power of a single transfer operator we will have to deal with a product of similar, but
different, operators.

Let $T>0$ be the one appearing in the statement of Theorem~\ref{thm:lclt} and consider
$t\in[\ve\eefrac1{2000}, T]$ to be fixed.  In the following we will find convenient to
work with a definition of ``deviation'' that is independent of the standard pair language.
This definition has been already introduced in~\eqref{e_deviationeDef}, but we report it
here for the reader's convenience.  Recall the notation
$\bar\theta_k=\bar\theta(\ve k,\theta)$; then let
\begin{equation}\label{eq:deviatione-D}
\deviatione_k(x,\theta)=\theta_{k}(x,\theta)-\bar\theta_k(\theta)
\end{equation}
where, as usual, $(x_k,\theta_k)=F_\ve^k(x,\theta)$ and $\bar\theta(t,\theta)$ is the
unique solution of $\dot{\bar\theta}=\bar\omega(\bar\theta)$, with initial condition
$\bar\theta(0)=\theta$.  On the other hand, the deviation $\deviatio^{\!\ve}(t)$, which
appears in the statement of Theorem~\ref{thm:lclt} is related to the initial
measure $\mu$;\footnote{ Recall the definition of the random variable
  $\deviatio^{\!\ve}(t )=\ve^{-\efrac12}\left[\theta_\ve(t)-\bar\theta(t,\thetas_0)\right]$
  where $\theta_\ve(t)$ is defined in \eqref{eq:z-path} by
  $\theta_\ve(t)=\theta_\pint{t\vei}+(t\vei-\pint{t\vei})[\theta_{\pint{t\vei}+1}-\theta_\pint{t\vei}]$, $\mu\in \cP_\ve(\thetas_0)$.}
the first goal of this section is to obtain an explicit relation between the two
definitions.

\begin{rem}\label{eq:rem-complex-sp}
  In the following we will need to iterate complex standard pairs. The basic tool to do so
  will be a generalization of Proposition~\ref{p_invarianceStandardPairs} where the
  potentials that appear are proportional to $\sigma$. This means that we will need
  $c_2\geq \Const |\sigma|$ in order for Proposition~\ref{p_invarianceStandardPairs} to
  apply.  Accordingly, by the condition $c_2\deltacomplex\leq \pi/10$, stated just
  after~\eqref{eq:complex-st}, we will need to consider
  $\deltacomplex\leq \Const |\sigma|^{-1}$.  On the other hand we will see shortly that we
  need worry only about $|\sigma|\leq \ve^{-\efrac 12-2\delta_*}$ for some conveniently
  chosen small constant $\delta_* > 0$ .  Due to this, we are going to consider {\em
    complex} standard pairs with
  $\delta\geq \deltacomplex\geq \deltac=c_* \ve^{\efrac 12+2\delta_*}$ for some
  conveniently chosen small constant $c_*$.  We will call {\em short complex standard
    pairs} the ones for which $\deltacomplex= \deltac$ and {\em long complex standard
    pairs} the ones for which $\deltacomplex= \delta$.
\end{rem}
Due to the above remark it is necessary to write a standard pair $\ell_0$ as a family of
short complex standard pairs. Recall that $\ell_0=(\bG_{\ell_0},\rho_{\ell_0})$,
$\bG_{\ell_0}:[a_{\ell_0},b_{\ell_0}]\to\bT^2$, has length
$|b_{\ell_0}-a_{\ell_0}|\in[\delta/2, \delta]$, where, as in the previous sections,
$\delta$ is some fixed number independent on $\ve$. Hence we must cut
$[a_{\ell_0},b_{\ell_0}]$ in $\delta\deltac^{-1}$ pieces $[\alpha_i,\alpha_{i+1}]$ of
length between $\deltac/2$ and $\deltac$. We can then define the complex standard pairs
$\ellc_{i}=(\bG_{\ell_0,i},\rho_{i})$, where
$ \bG_{\ell_0,i}=\bG_{\ell_0}|_{[\alpha_i,\alpha_{i+1}]}$ and
$\rho_i=Z_i^{-1}\rho_{\ell_0}\Id_{[\alpha_,\alpha_{i+1}]}$,
$Z_i=\int_{\alpha_i}^{\alpha_{i+1}}\rho_{\ell_0}$.\footnote{ The reader should not be
  confused by the fact that the $\ellc_i$ are real: the adjective ``complex''  here refers
  to the fact that they satisfy all the conditions for complex standard pairs, in
  particular the one stated in Remark~\ref{eq:rem-complex-sp} concerning their length.}
Remark that $Z_i\sim \ve^{\efrac 12+2\delta_*}$ and $\sum_i Z_i=1$. Clearly, for each
continuous function $B$,
\begin{equation}\label{eq:put-togehter-st}
\mu_{\ell_0}(B)=\sum_iZ_i\mu_{\ellc_i}(B).
\end{equation}
Let $(x,\theta)$ be distributed according to a measure in $\cP_\ve(\avgtheta{0})$, we can
apply to each standard pair in the family the decomposition~\eqref{eq:put-togehter-st}. We
can thus write
\begin{equation}\label{eq:disintegrate}
  \deviatio^{\!\ve}(t)=\ve\eefrac{-1}2\sum_i \Id_{\ellc_i}\left[\theta_\ve(t)-\bar \theta(t,\avgtheta{\ellc_i})\right]- \Id_{\ellc_i}\left[\bar \theta(t,\avgtheta{0})-\bar \theta(t,\avgtheta{\ellc_i})\right].
\end{equation}
In addition, for any $\alpha>0$, except for a set of exponentially small probability, the
relation between the random variable in~\eqref{eq:deviatione-D} and
$\deviatio^{\!\ve}_{\ellc_i}(t )=\ve\eefrac{-1}2[\theta_\ve(t)-\bar \theta(t,\avgtheta{\ellc_i})]$, under
$\ellc_i$, is given by:
\begin{align}\notag
\deviatio^{\!\ve}_{\ellc_i}(t )&=\ve\eefrac{-1}2[\theta_\ve(t)-\bar \theta(t,\theta)]+\cO(\ve^{1-2\delta_*})\\
&\notag=\veih \left\{\deviatione_\pint{t\vei}+(t\vei-\pint{t\vei})[\deviatione_{\pint{t\vei}+1}-\deviatione_\pint{t\vei}]\right\}+\cO(\ve^{1-2\delta_*})\\
  &=\veih \deviatione_\pint{t\vei}+\veh
  (t\vei-\pint{t\vei})\ho(x_{\pint{t\vei}},\theta_{\pint{t\vei}})+\cO(\ve^{1-2\delta_*}),\label{eq:translate}
\end{align}
where we have argued as in~\eqref{eq:Delta-recursion} and used our large deviation
results.\footnote{ See the arguments around equation~\eqref{eq:largevDelta} for more
  details.}
\begin{rem}\label{r_lazyTimes}
  In the following we will consider only values of $t$ such that $\pint{t\vei}=t\vei$, \ie
  we will assume $t\in\ve\bN\cap[0,T]$.  As the formula above shows, the general case can
  be treated by modifying the last term in the sum defining $\etaRefb_{0,k}$
  in~\eqref{eq:hbolddef} below. We refrain from doing so explicitly to alleviate our
  notation.  Note however that if one wanted to compute the first term of the Edgeworth
  expansion, then one would need to treat explicitly all times and even use a formula
  slightly more precise than~\eqref{eq:translate}, which anyhow also follows from the
  arguments used in~\eqref{eq:Delta-recursion}.
\end{rem}
\subsection{Reduction to a Birkhoff sum}\ \label{sec:birk}\newline As it is often done in
the study of sums of weakly dependent random variables (and already several times in this
paper), we need to divide the time interval $[0,t]$ in blocks.  For technical reasons it
turns out to be convenient to allow such blocks to be of variable length.  We thus
consider a number $R$ of blocks of length identified by the sequence $\{L_k\}_{k=0}^{R-1}$
and set
\begin{align*}
  S_k&=\sum_{j=0}^kL_j, &S_{-1}&=0
\end{align*}
so that $S_{R-1}={t \vei}$.  In our situation, it suffices to consider the case in which
all the blocks are equal except the last one.  More precisely: let us fix\footnote{ The
  choices of $1/32$ and $1/99$ are both arbitrary and largely irrelevant; in fact one could
  work with values of $\delta_*$ arbitrarily small (see Footnote~\ref{f_sigma100}) .}
\begin{align}\label{eq:delta-setting}
  \delta_*&\in(\efrac1{99},\efrac 1{32}),
\end{align}
to be specified later, let $L_*= \ve^{-3\delta_*}$ and define the lengths $L_k$ as
follows:
\begin{equation}\label{eq:L-choice}
\begin{split}
  L_k &= L_* \quad\textrm{ for all }k\in\{0,\cdots, R-2\}.\\
  L_*&\le L_{R-1}\leq 2L_*.
\end{split}
\end{equation}
\begin{rem}
  The estimates in this section are sharper than needed for our purposes, given our choice
  of $L_*$.  Yet, they are instructive as they show, at very little extra cost, how to
  proceed if one wants to obtain a full Edgeworth expansion.
\end{rem}
 \begin{lem}\label{lem:break-smart?}
   For any $\ve>0$, let $t\in\ve\bN\cap[0,T]$ and $\{L_k\}_{k=0}^{R-1}\subset\bN$ as
   above:
  \begin{equation}\label{eq:Deltabloks}
    \begin{split}
      \deviatione_{t\vei}&=\sum_{k=0}^{R-1}\left[\devD_{k}+\cO(\deviatione_{L_k}^3)\right]\circ F_\ve^{S_{k-1}}\\
      \devD_{k}&=\etaExph(t - \ve S_k,\bar\theta_{L_k})\left[\deviatione_{L_{k}}+\frac 12P(t - \ve
        S_k,\bar\theta_{L_k})\deviatione_{L_{k}}^2\right]\\
      \etaExph(s,\theta)&=\expo{\int_0^s\bar\omega'(\bar\theta(\tau,\theta))\deh
        \tau};\quad
      P(s,\theta)=\int_0^s\etaExph(\tau,\theta)\bar\omega''(\bar\theta(\tau,\theta) )\deh
      \tau.
    \end{split}
  \end{equation}
\end{lem}
\begin{proof}
  Note that
  \begin{align*}
    \deviatione_{t\vei}
    &=\theta_{t\vei}-\bar \theta_{t\vei}=\theta_{t\vei-L_0}\circ F_\ve^{L_0}-\bar \theta(t-\ve{L_0},\bar\theta_{L_0})\\
    &= \deviatione_{t\vei-{L_0}}\circ F_\ve^{L_0}-\bar \theta(t-\ve{L_0},\bar\theta_{L_0})+\bar\theta(t-\ve{L_0},\theta_{L_0})\\
    &=\deviatione_{t\vei-{L_0}}\circ F_\ve^{L_0}+\partial_\theta\bar\theta(t-\ve{L_0},\bar\theta_{L_0})\deviatione_{L_0}+\frac 12\partial_\theta^2\bar\theta(t-\ve{L_0},\bar\theta_{L_0})\deviatione_{L_0}^2\\
    &+\cO(\deviatione_{L_0}^3).
  \end{align*}
  Next, note that, by the smooth dependence on initial data of the solutions of ordinary
  differential equations, the functions
  $\eta_1=\partial_\theta\bar\theta, \eta_2=\partial_\theta^2\bar\theta$ solve,
  respectively, the differential equations $\dot\eta_1=\bar\omega'(\bar\theta)\eta_1$,
  $\eta_1(0)=1$ and
  $\dot\eta_2=\bar\omega'(\bar\theta)\eta_2+\bar\omega''(\bar\theta)\eta_1^2$,
  $\eta_2(0)=0$. That is, $\partial_\theta\bar\theta(s,\theta)=\etaExph(s,\theta)$ and
  $\partial_\theta^2\bar\theta(s,\theta)=\etaExph(s,\theta) P(s,\theta)$.  Iterating the
  above formulae yields the lemma.
\end{proof}
Next, we want to write the random variables $\devD_k$, associated to the $k$-th block, in
terms of the (more explicit) random variables defined in~\eqref{e_defEtaRef2}: recall that
$\etaRefb_{k}=\etaRefb_{0,k}+\etaRefb_{1,k}$:
\begin{align}\label{eq:hbolddef}
  \etaRefb_{0,k}&=\sum_{j=0}^{k-1}\etaExp_{j,k}\ho(x_j,\theta_j);
  &\etaRefb_{1,k}&=-\frac{\ve}2\sum_{j=0}^{k-1}\etaExp_{j,k}\bar\omega'(\bar\theta_j)\bar\omega(\bar\theta_{j})
  \\\notag\text{where }\etaExp_{j,k} &=\prod_{l=j+1}^{k-1}\left[1+\ve\bar\omega'(\bar\theta_l)\right].
\end{align}
\begin{lem}\label{lem:h-best}
  There exists $\ve_0$ such that, for all $k\in\{0,\cdots,R-1\}$, $\ve\in[0, \ve_0]$,
  $j\in\{0,\cdots, L_k\}$, $\alpha\in (0,\delta_*]$ and standard pair $\ell$ we have
  \begin{align*}
    \mu_\ell\left(\left\{|\deviatione_{j}| \ge \ve L_k^{\efrac12+\alpha}\right\}\right)&\le e^{-\const L_k^{\alpha}}\\
    \mu_\ell\left(\left\{|\ve\etaRefb_{j}-\deviatione_{j}| \ge \ve^{3} L_k^{2+2\alpha}\right\}\right)&\leq e^{-\const L_k^{\alpha}}\\
    \mu_\ell\left(\left\{\left|(\ve \etaRefb_{0,j})^2-\deviatione_{j}^2\right|
    \ge \ve^{3}L_k^{\efrac 32+\alpha}\right\}\right)&\leq e^{-\const L_k^{\alpha}}.
  \end{align*}
\end{lem}
\begin{proof}
  By Lemma~\ref{l_boundEtaRef} (or, more precisely,~\eqref{e_boundEtaRefb})
  \begin{align}\label{eq:aboutthis}
    \deviatione_j - \ve[\etaRefb_{0,j}+\etaRefb_{1,j}]&
    =\ve\sum_{l=0}^{j-1}\etaExp_{l,j}\left[ \frac{\bar\omega''(\bar\theta_l)}2\deviatione_l^2+
    \cO(\deviatione_l^3+\ve^2)\right].
  \end{align}
Next, let us define $\cB_{\alpha,j}=\{(x,\theta)\in\bT^2\;:\; |\deviationsl_j|\geq \ve j^{\frac 12+\alpha}\}$. By Proposition~\ref{p_largeDevzQuadraticBound}
\begin{equation}\label{eq:largevDelta}
\mu_\ell(\cB_{\alpha,j})\leq  e^{-\const j^{2\alpha}}.
\end{equation}
Hence, for all $j\leq \Const \sqrt L_k$, $|\deviatione_j|\leq  \ve L_k^{\efrac 12}$, while for $j\in[\Const\sqrt L_k,L_k]$, since we have $|\deviationsl_j-\deviatione_j|\leq \Const\ve$, \eqref{eq:largevDelta} implies
\begin{align*}
  \mu_\ell\left(\left\{|\deviatione_{j}|\geq \ve L_k^{\efrac12+\alpha}\right\}\right)\le
  e^{-\const j^{2\alpha}}\le e^{-\const L_k^{\alpha}}
\end{align*}
from which the first assertion of the Lemma follows.  Next, we have
\begin{align*}
    \left|\ve\sum_{l=0}^{\min\{j,\const\sqrt L_k\}-1}\etaExp_{l,j}\left[ \frac{\bar\omega''(\bar\theta_l)}2\deviatione_l^2+ \cO(\deviatione_l^3+\ve^2)\right]\right|\leq \ve^3 L_k^{\efrac 32}.
\end{align*}
This proves the second assertion for $j\leq \Const \sqrt L_k$, while, for
$j\in [\Const \sqrt L_k, L_k]$,
  \begin{equation}\label{eq:poormanH}
    \mu_\ell\left(\left\{|\ve [\etaRefb_{0,j}+\etaRefb_{1,j}]-\deviatione_{j}|\geq \ve^{3}j^{2+2\alpha}\right\}\right)\leq e^{-\const L_k^{\alpha}},
  \end{equation}
  which yields the second assertion in the general case, recalling the constraints
 ~\eqref{eq:L-choice} on $L_k$.  The last assertion follows analogously since
  $\deviatione_{j}^2=(\ve \etaRefb_{0,j})^2+2(\deviatione_{j}-\ve
  \etaRefb_{0,j})\deviatione_{j}-(\deviatione_{j}-\ve \etaRefb_{0,j})^2$, and, for
  $j\leq L_k$,
  \begin{align*}
    \mu_\ell\left(\left\{\sup_{0\le j\le \pint{t\vei}}|(\ve
    \etaRefb_{0,j})^2-\deviatione_{j}^2|\geq \ve^{3}j^{\efrac
    32+\alpha}\right\}\right)\leq e^{-\const L_k^{\alpha}},
  \end{align*}
  where we used the fact that $\Delta_j-H_{0,j} = \cO(H_{1,j}) = \cO(\ve^2j) .$
  \end{proof}
\noindent The above Lemma, which is even sharper than necessary, suggests to define
\begin{align}\label{eq:Mvariable}
  \bM_{k}&=\etaExph(t-\ve
  S_k,\bar\theta_{L_k})\left[\etaRefb_{0,L_k}+\etaRefb_{1,L_k}+\frac \ve2P(t-\ve
  S_k,\bar\theta_{L_k})(\etaRefb_{0,L_k})^2\right].
\end{align}
Then, for $\alpha \leq\delta_*<\efrac1 {32}$, Lemmata~\ref{lem:break-smart?},
\ref{lem:h-best} and equations~\eqref{eq:Deltabloks},~\eqref{eq:largevDelta} yield
\begin{equation}\label{eq:firstDeltared}
\begin{split}
  &\mu_\ellz\left(\{|\deviatione_{t\vei}-\ve \bA|\geq \Const\ve^{2}L_*^{1+2\alpha}\}\right)\leq
  e^{-\const\ve^{-3\alpha\delta_*}}\\
  &\bA=\sum_{k=0}^{R-1}\bM_{k}\circ F_\ve^{S_{k-1}}.
\end{split}
\end{equation}
Since $\ve^{2}L_*^{1+2\alpha}\leq \ve^{\efrac 32+\delta_*}$, $\deviatione_{t\vei}-\ve \bA$
is $o(\ve^{\efrac 32})$ with probability almost one.

Thanks to~\eqref{eq:firstDeltared} we have reduced ourselves to computing the distribution
of the random variable $\ve\bA$. The rest of the paper will therefore mostly deal with the
problem of obtaining a local CLT for the variable $\bA$.

\subsection{Proof of the Local CLT}\ \newline\label{subsec:prooflclt}

In this subsection we will obtain a LCLT for the random variable $\ve\bA$, defined in
\eqref{eq:firstDeltared}, assuming the validity of several propositions that will be
proven later on.  Using this result we will be able to prove the LCLT for
$\deviatio^{\!\ve}(t)$.

Our first problem is that the random variable $\bA$ may have a very rough density (if it
has a density at all): it is then convenient to introduce a regularization
procedure.\footnote{ This is not the only way to handle the problem, it is just the one we
  find more convenient, see Remark~\ref{rem:feller-smooth} for a standard alternative.} To
this end let ${\boldsymbol Z}$ be a bounded, independent, zero average random variable so
that $|{\boldsymbol Z}|\le 1$ with smooth density $\psi\in\cC^\infty$.  We can then
consider the random variable $\bAs=\ve \bA+\ve^{\beta_*} {\boldsymbol Z}$, where
$\beta_*=\frac 32+\delta_*$ and recall that $\delta_*\in(\efrac1{99},\efrac1{32})$.  The random
variable $\bAs$ indeed admits a density, which we denote with $N_{\mu,\bAs}$ (where $\mu$
denotes the distribution of initial conditions).  In fact, denote by $\widehat\psi$ the
Fourier transform of $\psi$:
\begin{align}\notag
  N_{\mu,\bAs}(y)&=\frac 1{2\pi}\int_{\bR}e^{-i\xi y}\bE(e^{i\xi \bAs})\deh\xi\\
  &=\frac 1{2\pi\ve}\int_{\bR}e^{-i\sigma\vei
    y}\mu(e^{i\sigma\bA})\widehat\psi(\ve^{\beta_*-1}\sigma)\deh \sigma.\label{e_mainIntegral}
\end{align}
The above discussion motivates us to prove the following result
\begin{prop}\label{p_smoothlclt}
  For any $T>  0$ there exists $\ve_0$ so that the following holds.  For any real numbers
  $\ve\in(0,\ve_0)$, $t\in[\ve^{\efrac1{1000}},T]$ so that $t\vei=\pint{t\vei}$, any
  $\thetas\in\bT$ and any short complex standard pair $\ellc$ so that
  $\thetas=\Re(\mu_{\ellc} (\theta))$,\footnote{ This generalizes~\eqref{eq:thetastartdef}
    to the case of complex standard pairs.} we have:
  \begin{equation}
    \label{e_smoothlclt}
    N_{\ellc,\bAs}(y) = %
    \frac{e^{-y^2/(2\ve \Var_t^2(\thetas))}}{\Var_t(\thetas)\sqrt{2\pi \ve}} +\cO(\ve^{-7\delta_*}),
  \end{equation}
  where $\Var_t(\cdot)$ is given by~\eqref{e_variancelclt}; in particular it is a
  differentiable function so that $|\Var_t'|\leq \Const$.
\end{prop}

Let us postpone the proof of Proposition~\ref{p_smoothlclt} and see immediately how it
implies our main result.
\begin{proof}[{\bf Proof of Theorem~\ref{thm:lclt}}]
  Let us remind once again the reader that we will give the proof only in the case
  $t\vei = \pint{t\vei}$ (see Remark~\ref{r_lazyTimes}).  By
  equations~\eqref{eq:disintegrate} and~\eqref{eq:translate}, given any $I=[a,b]$ and
  $\shiftPar\in\bR$, we have
\begin{align*}
  \ppath_{\mu}(\deviatio^{\!\ve}(t)\in \veh I + \shiftPar)
  &=\sum_{i}Z_i \ppath_{\mu_{\ellc_i}}(\deviatio^{\!\ve}_{\ellc_i}(t)\in \veh I + \shiftPar+\tau_i)\\
  &\le\sum_{i}Z_i \ppath_{\mu_{\ellc_i}}(\veih \deviatione_{t\vei}\in \veh I^+ + \shiftPar+\tau_i)+\Const e^{-\ve^{-\const}}
\end{align*}
where
$\tau_i=\ve\eefrac{-1}2[\bar\theta(t,\avgtheta{0})-\bar\theta(t,\avgtheta{\ellc_i})]$ and
$I^+=[a-\Const \ve^{1-\delta_*}, b+\Const \ve^{1-\delta_*}]$.  By the same token
\begin{align*}
  \ppath_{\mu}(\deviatio^{\!\ve}(t)\in \veh I + \shiftPar)\geq \sum_{i}Z_i
  \ppath_{\mu_{\ellc_i}}(\veih \deviatione_{t\vei}\in \veh I^- + \shiftPar+\tau_i)
  -\Const e^{-\ve^{-\const}}
\end{align*}
where $I^-=[a+\Const \ve^{1-\delta_*}, b-\Const \ve^{1-\delta_*}]$. From now on we follow
only the upper bound, the lower bound being more of the same.

By \eqref{eq:firstDeltared} and the definition of $\bAs$ (see the beginning of this
subsection) we have
\begin{align*}
  \ppath_{\mu}(\deviatio^{\!\ve}(t)\in \veh I + \shiftPar)\leq\sum_{i}Z_i
  \ppath_{\mu_{\ellc_i}}(\veih \bAs\in \veh I^+ + \shiftPar+\tau_i)+\Const e^{-\ve^{-\const}}.
\end{align*}
We can now use Proposition~\ref{p_smoothlclt} to obtain
\begin{equation}\label{eq:final-lclt-form}
\begin{split}
  \ppath_{\mu}&(\deviatio^{\!\ve}(t)\in \veh I + \shiftPar)\le \sum_i Z_i
  \int_{\shiftPar+\tau_i+\veh
    I^+}\left[\frac{e^{-\eta^2/(2\Var_t^2(\theta^*_{\ellc_i}))}}{\Var_t(\theta^*_{\ellc_i})\sqrt{2\pi
      }} +\Const \ve^{1/2-7\delta_*}\right]\deh\eta\\
  &=\int_{\bT}\int_{\shiftPar+\veh I^+}
  \left[\frac{e^{-(\eta-\ve\eefrac{-1}2[\bar\theta(t,\avgtheta{0})-\bar\theta(t,\theta)])^2/(2
        \Var_t^2(\theta)))}}{\Var_t(\theta)\sqrt{2\pi }} \cN_{\mu}(\deh \theta)
    +\Const\ve^{1/2-7\delta_*}\right]\deh\eta,
\end{split}
\end{equation}
where $\cN_{\mu}$ is the law of $\theta$ under $\mu$.  The obvious analog holds for the
lower bound.

The above formula is valid for any standard family, but if $\mu\in\cP_\ve(\theta^*_0)$,
since by definition $|\bar\theta(t,\avgtheta{0})-\bar\theta(t,\theta)|\le\Const\ve$, we
can obtain the simplified expression:
\begin{align*}
  \veih\ppath_\mu(\deviatio^{\!\ve}(t)\in \veh I + \shiftPar)=\Leb\, I\cdot\left[
  \frac{e^{-\shiftPar^2/2\Var_t^2(\theta^*_0)}}{\Var_t(\theta^*_0)\sqrt{2\pi}}
  +\cO(\ve^{\efrac12-7\delta_*})\right]+\cO(\ve^{\efrac 12-\delta_*}).
\end{align*}
This proves the theorem.
\end{proof}
Our task is then reduced to the proof of Proposition~\ref{p_smoothlclt}.
\begin{proof}[{\bf Proof of Proposition~\ref{p_smoothlclt}}]
  It suffices to compute the integral~\eqref{e_mainIntegral} when $\mu = \mu_{\ellc}$ is a
  short complex standard pair.  To do so, we find convenient to split the integral in five
  different regimes: let us fix $\sigma_0>0$ small enough and $\Czero>0$ large enough to %
  be determined later; also let
  $\largeta=\beta_*-1+\delta_*=\frac 12+2\delta_*$.\footnote{ Informally, $\sigma_0$
    specifies the region in which we can use perturbation theory, while $\Czero$ and
    $\ve^{-\largeta}$ specifies the regions that can be bounded trivially, see equations
    \eqref{eq:JJ0}, \eqref{eq:JJ4}.}  Recall moreover that we have chosen
  $L_* = \ve^{-3\delta_*}$; we consider then the partition $\bR=\bigcup_{k = 0}^4\J k$,
  where
  \begin{align*}
    &\J0=\{ |\sigma|\le \Czero\ve^{2}L_*\},\;\;&\J1=\{\Czero\ve^{2}L_* < |\sigma|\le\ve^{\delta_*}\},\\
    &\J2=\{\ve^{\delta_*} < |\sigma|\le\sigma_0\},&\J3=\{\sigma_0< |\sigma| \le \ve^{-\largeta}\},\hskip.6cm\\
    &\J4=\{\ve^{-\largeta} < |\sigma|\}.
  \end{align*}
  Correspondingly, we can rewrite~\eqref{e_mainIntegral} as
  \begin{align*}
    N_{\mu, \bAs} &= \JJ0+\JJ1+\JJ2+\JJ3+\JJ4,
  \end{align*}
  where each $\JJ{j}$ denotes the contribution of $\J{j}$ to the integral on the right
  hand side of~\eqref{e_mainIntegral}.  Recall that we are allowed to neglect
  contributions that are of order $\ve^{-7\delta_*}$; we will now show that the main
  contribution to~\eqref{e_mainIntegral} is given by $\JJ1$, as the contributions of all
  other terms are, in fact, negligible.  First notice that the contribution of $\JJ0$ can
  be neglected; in fact:
  \begin{equation}\label{eq:JJ0}
    \left|\JJ0\right|
    \le \frac{1}{2\pi\ve}\int_{\sigma\le \Czero\ve^{2}L_*}\left|\widehat\psi(\ve^{\beta_*-1}\sigma)\right|\deh\sigma
     \le \Const\ve L_*\|\psi\|\nl1\leq \Const \ve^{1-3\delta_*}
  \end{equation}
  The contribution of $\JJ4$ can also be neglected, since, for each $r\in\bN$, by Cauchy--Schwarz:
  \begin{equation}\label{eq:JJ4}
    \begin{split}
      \left|\JJ4\right|
      &\le \frac{\ve^{-\beta_*}}{2\pi}\int_{\sigma\ge \ve^{\beta_*-1-\largeta}}\left|\widehat\psi(\sigma)\right|\deh\sigma\\
      &\le \Const\ve^{-\beta_*}\|\psi^{(r)}\|\nl2\left[\int_{\sigma\ge \ve^{-\delta_*}}\sigma^{-2r}\deh\sigma\right]^{\efrac12}\\
      &\le \Const\|\psi^{(r)}\|\nl2\ve^{2r\delta_*-\delta_*-\beta_*}.
    \end{split}
  \end{equation}
  If we take $r$ large enough, depending on the choice of $\delta_*$, we can thus conclude
  that $|\JJ4|\le \Const\|\psi^{(r)}\|\nl2\ve^{100}\le\Const\ve^{-7\delta^*}$.  We are
  then left with the estimate of the contributions of $\JJ1$, $\JJ2$ and $\JJ3$. We will
  (impressionistically) call $\J1$ the \emph{small} (frequencies) regime, $\J2$ the
  \emph{intermediate} regime and $\J3$ the \emph{large} regime.

  The basic tool to compute these integrals is described by Lemma~\ref{lem:one-step-M},
  which will be stated below.  Before giving its statement, however, it is convenient to
  introduce a systematic notation for the many correlation terms that will appear in the
  sequel.  It will turn out that, for the level of precision needed for our current
  investigation, the exact form of such terms is irrelevant.  It will thus suffice to
  consider the following, very rough, bookkeeping strategy.
  \begin{notation}\label{not:bookcor}
    Let $C^*>1$ be some fixed constant sufficiently large.  Given a standard pair $\ell$,
    we will use the symbol $\fkC^{k,p}_{\ell,l,\bar \imath}$ to denote a coefficient which
    depends only on the averaged trajectory $\bar\theta(t,\thetasl)$, indexed by
    $\bar\imath =(i_1,\cdots,i_l)\in\{-1,\cdots, L_k-1\}^l$ (or $\bar\imath=\emptyset$
    if $k = 0$)\/\footnote{ We use the convention that, for any set $A$, $A^0=\{\emptyset\}$.}
    and which satisfies the estimates $|\fkC^{k,p}_{\ell,l,\bar\imath}|\leq (C^*)^l$, and
    $\sum_{\bar\imath}|\fkC^{k,p}_{\ell,l,\bar\imath}|\leq (C^*)^l L_k^p$.

    We will use $A_{j,\bar \imath}$, $i_j\geq 0$, as a placeholder for an arbitrary
    $\cC^2(\bT^2,\bC)$ function possibly explicitly depending on $\ell$
    such that $\|A_{j,\bar\imath}\|_{\cC^1(\bT^2,\bC)}\leq C^*$, and we assume
    conventionally that $A_{j,\bar\imath}=1$ if $i_j=-1$. Finally, we will use the
    notation
    \begin{align*}
      \bbK^{k,p}_{\ell,l}=\sum_{\bar\imath}\fkC^{k,p}_{\ell,l,\bar\imath}
      \prod_{j=1}^l A_{j,\bar\imath}\circ F_\ve^{i_j}.
    \end{align*}
    For obvious reasons we will call such expressions {\em correlation terms}. Note that
    $ \bbK^{k,p}_{\ell,l} \bbK^{k,p'}_{\ell,l'} =\bbK^{k,p+p'}_{\ell,l+l'}$.
    Finally, observe that $ \bbK^{k,p}_{\ell,l}$ can also be written as $\bbK^{k,p}_{\ell,m}$ for any
    $m\geq l$ (just set $\fkC^{k,p}_{\ell,m,\bar\imath}=0$ if $i_{j}\neq -1$ for all $j>l$).
  \end{notation}
  Also let us introduce the potentials (recall that the value of $t$ is fixed)
  \begin{align}\label{eq:potdef}
      \hov^k_{\ell,j}(x,\theta)&=\etaExph(t-\ve S_k,\thetaslk{L_k})\etaExps_{j,L_k}\ho(x,\theta)
  \end{align}
  where $\etaExps_{j,k}$ is defined in~\eqref{e_defEtaRef}, $\etaExph$ is defined in
 ~\eqref{eq:Deltabloks} and, generalizing~\eqref{eq:thetastartdef}:
  \begin{equation}\label{eq:ifnotdefinedyet}
    \thetasl=\Re(\mu_\ell(\theta))\;;\quad\thetaslk{k}=\bar\theta(\ve k,\thetasl).
  \end{equation}
  Let us fix $q\in\bN$ sufficiently large to be specified later; associated with the above
  potentials, choosing a standard pair $\ell$, $k\in\{0,\cdots,R-1\}$,
  $\fkC^{k,0}_{\ell,0,\emptyset}$ and $(\bbK^{k,2s}_{\ell,3s})_{s = 1}^{q-1}$, we define
  an operator\footnote{ To be precise $\cT$ should have a lot of indexes
    ($\{\ell, k, q, \fkC^{k,0}_{\ell,0,\emptyset}, \hov^k_{\ell,j},
    (\bbK^{k,2s}_{\ell,3s})_{s = 1}^{q-1}\}$), we drop all of them (except $k$) for
    readability.} $\cT_k$: the operator $\cT_k$ acts on complex measures over $\bT^2$ as a
  ``weighted $L_k$-push-forward with correlation terms up to $q$ points'', according to
  the following formula
  \begin{align}\label{eq:corr-T-op}
      \cT_k\mu(g)
      &= e^{i\sigma\fkC_{\ell}(\ve)}
      \mu\left( e^{i\sigma
      \sum_{j=0}^{L_k-1}\hov^k_{\ell,j}\circ
      F_\ve^j}\left[1+\sum_{s=1}^{q-1}(i\sigma\ve)^s\bbK^{k,2s}_{\ell,3s}\right]g\circ
      F_\ve^{L_k}\right),
   \end{align}
   where $\fkC_{\ell}(\ve)$ is a constant depending only on $\ell$ and $\ve$.  Observe
   that when $q = 1$ and $\fkC_{\ell}(\ve)=0$, we recover the push-forward operator with
   complex potential~\eqref{eq:potdef} defined in~\eqref{e_preparatory8} .  The
   key fact is that the action of such operators on complex standard families can still be
   described in the standard pair language, as the following lemma shows.
  \begin{lem}\label{lem:one-step-M}
    There exists $\ve_0>0$ such that, for each $k\in\{1,\cdots, R\}$, $\sigma\in\bR$,
    short complex standard pair $\ell$ and $g\in L^\infty(\bT^2,\bC)$ there exist a family
    of short standard pairs $\stdf_{\ell}^{k}$ such that, provided
    $|\sigma|\leq \ve^{-\efrac12-2\delta_*}$ and $L_k\leq \ve^{-\efrac 14+\delta_*}$, we
    have
    \begin{align*}
      &  \mu_\ell\left( e^{i\sigma \bM_{k}} g\circ F_\ve^{L_{k}}\right)=\sum_{\ell'\in\stdf_\ell^{k}}
        \fm_{k,\ell,\ell'}\mu_{\ell'}(g)
        +\cO\left(\ve^q\sigma^qL_{k}^{2q}+\sigma \ve^2\deltacomplex L_{k}^2\right)
             \cdot|\mu_\ell|(|g\circ F_\ve^{L_{k}}|),\\
        &\textrm{where }\sum_{\ell'\in\stdf_\ell^{k}}\fm_{k,\ell,\ell'}\mu_{\ell'}(g)=
          \cT_{k}\mu_\ell(g)\text{ and $\cT_k$ is given by~\eqref{eq:corr-T-op} with
    $\fkC_{\ell}(\ve) = \ve\fkC^{k,1}_{\ell,0,\emptyset}$.}
    \end{align*}
    Moreover, if $|\sigma|\le \sigma_0$, we can take
    $\{\ell\}$ and/or $\stdf_{\ell}^{k}$ to consist of long standard pairs.
    In addition, if we define iteratively the standard families
    $\stdf_\ellz^0=\{\ellz\}=\{\ell\}$ and $\stdf_{\ell_{k}}^{k}$ where, for all
    $\ell_k\in \stdf_{\ell_{k-1}}^{k-1}$, $\stdf_{\ell_{k}}^{k}$ is defined as above,
    then, for each $k\in\{0,\cdots, R-1\}$, if $q\geq 4$ and
    $L_*\le C\ve^{-\efrac14+\delta_*+(\efrac34+\delta_*)/(2q-1)}$ for sufficiently small
    $C$, we have
    \begin{equation}\label{eq:trivial-fm-estimate}
      \sum_{\ell_{1}\in \stdf_{\ell_{0}}^{0}}\cdots\sum_{\ell_{k+1}\in \stdf_{\ell_{k}}^{k}}
      \prod_{j=0}^{k}|\fm_{j,\ell_{j},\ell_{j+1}}|\leq \Const.
    \end{equation}
  \end{lem}
  The proof of the above lemma will be given in the next subsection.  We now show how to
  conclude the proof of Proposition~\ref{p_smoothlclt}: Lemma~\ref{lem:one-step-M}
  and~\eqref{eq:firstDeltared} allow to write the expectation
  $\mu(e^{i\sigma\bA}) = \mu_{\ellc}(e^{i\sigma\bA})$ appearing in~\eqref{e_mainIntegral}
  as (recall $R = \cO(\vei L_*\invr)$):
  \begin{equation}\label{eq:stating-pointmu}
    \mu_{\ellc}(e^{i\sigma\bA}) =
    \sum_{\ell_{1}\in \stdf_{\ellc}^{0}}\cdots\sum_{\ell_{R}\in \stdf_{\ell_{R-1}}^{R-1}}
    \prod_{j=0}^{R-1}\fm_{j,\ell_{j},\ell_{j+1}}+\cO(\ve^{q-1}\sigma^qL_*^{2q-1}+\sigma \ve \deltacomplex L_*).
  \end{equation}
  \begin{rem}\label{eq:long-v-short} Note that the above decomposition depends on the
    choice of $\deltacomplex$ which, in turns, depends on $\sigma$. From now on we will
    talk only of ``complex standard pairs" and it will be understood that the families
    $\stdf_{\ell_{k}}^{k}$ are made of short standard pairs for $\sigma\in \J3$ and long
    standard pairs if $\sigma\in \J1\cup \J2$.
  \end{rem}
  Note that the estimate given by~\eqref{eq:trivial-fm-estimate} is very crude as it
  completely ignores possible cancellations among complex phases. Our next step are the
  following --much sharper-- results which take into consideration such cancellations.
  \begin{prop}[Large $\sigma$ regime]\label{prop:not-so-trivial-0}
    For any $\delta_*\in (0, 1/32)$, 
    $\sigma\in\J3$, let $L_*\leq L_{R-1}\leq 2L_*$. Then, for any complex standard pair
    $\ell_{R-1}\in\stdf^{R-2}_{\ell_{R-2}}$:
    \begin{align*}
      \left|  \sum_{\ell_R\in\stdf^{R-1}_{\ell_{R-1}}}\fm_{R-1,\ell_{R-1},\ell_R}\right|= \cO(\ve^{2-9\delta_*}).
    \end{align*}
  \end{prop}
  The proof of the above proposition will be given in Section~\ref{sec:one-large}.
  \begin{prop}[Intermediate $\sigma$ regime]\label{prop:not-so-trivial-1/2}
    For any $\delta_*\in(\efrac1{99}, \efrac1{32})$ and 
    $\sigma\in\J2$, let $L_*\leq L_{R-1}\leq 2L_*$. Then, for any complex standard pair
    $\ell_{R-1}\in\stdf^{R-2}_{\ell_{R-2}}$:
    \begin{align*}
      \left|  \sum_{\ell_R\in\stdf^{R-1}_{\ell_{R-1}}}\fm_{R-1,\ell_{R-1},\ell_{R}}\right|= \cO(\ve^{2-9\delta_*}).
    \end{align*}
  \end{prop}
  The proof of the above proposition can be found in Section~\ref{sec:one-small}.  As
  mentioned previously, the above propositions imply that the main contribution to the
  integral~\eqref{e_mainIntegral} is given by $\JJ1$.  The next proposition estimates
  precisely this contribution

  \begin{prop}[Small $\sigma$ regime]\label{prop:not-so-trivial-1}
    For $\delta_*\in(\efrac1{99}, \efrac 1{32})$, $\sigma\in\J1$
    , $q\geq 5$ and $L_k=L_*$, $0\le k<R-1$:
    \begin{align*}
      \mu_{\ellc}(e^{i\sigma\bA})&= e^{-\frac2{\ve}\sigma^2\Var_t^2(\thetasl)}
                                   +\cE(\sigma,\ve),
    \end{align*}
    where recall (see~~\eqref{e_variancelclt} and~\eqref{e_definitionBarChi}) that
    \begin{align*}
      \Var_t(\theta) =
      \int_0^{t}e^{2\int_s^t\bar\omega'(\bar\theta(s',\theta))\deh{}s'}\greenkubo^2(\bar\theta(s,\theta))\deh{}s
    \end{align*}
    and $\cE$ is a small remainder term in the sense that
    \begin{align*}
      \frac 1{\ve}\int_{\J1}|\cE(\sigma,\ve)|\deh\sigma=\cO(L_*^2\log\vei).
    \end{align*}

  \end{prop}
  The proof of Proposition~\ref{prop:not-so-trivial-1} can be found in Section~\ref{proof-not-so-trivial-1}.

  Let us now fix $q = 7$ and recall that $L_*=\ve^{-3\delta_*}$.  We can now compute the
  integral~\eqref{e_mainIntegral} by using estimates~\eqref{eq:JJ0},~\eqref{eq:JJ4},
  Lemma~\ref{lem:one-step-M} and~\eqref{eq:stating-pointmu} in the first line below, while
  using Propositions~\ref{prop:not-so-trivial-0}--\ref{prop:not-so-trivial-1}
  and~\eqref{eq:trivial-fm-estimate} in the second line:
  \begin{align*}
    N_{\ellc,\bAs}(y)
    &=\hskip-6pt\sum_{\ell_{1}\in \stdf_{\ell_{0}}^{0}} \cdots
      \sum_{\ell_{R}\in \stdf_{\ell_{R-1}}^{R-1}} \int_{|\sigma|\leq\ve^{-\frac 12-2\delta_*}}\frac {e^{-i\sigma\vei y}}{2\pi\ve} \prod_{j=1}^{R-1}\fm_{j-1,\ell_{j-1},\ell_j}\widehat\psi(\ve^{\beta_*-1}\sigma)+\cO(1)\\
    &=\cO(\ve^{-6\delta_*}\log\vei)+\frac {1}{2\pi \sqrt \ve}\int_{\bR}\widehat\psi(\eta\sqrt\ve)e^{-i\veih \eta y-\frac12\eta^2\Var_t^2(\thetaslz))}d\eta\\
    &=\cO(\ve^{-6\delta_*}\log
      \vei)+\frac{e^{-y^2/(2\ve\Var_t^2(\thetasl))}}{\Var_t(\thetasl)\sqrt{2\pi
      \ve}},
  \end{align*}
  where we have used that $|\widehat\psi(s)-1|\leq\Const s^2$, since ${\boldsymbol Z}$ has
  zero average.
\end{proof}

\begin{rem}\label{rem:feller-smooth}
  In alternative to the above strategy we can choose $\cN$ to be the distribution of a
  Gaussian random variable with density
  $\cN'=\frac1{\sqrt{2\pi\ve}\Var_t }e^{-\efrac{x^2}{2\Var_t^2\ve}}$ and apply \cite[Lemma
  2, Chapter XVI.3]{Feller2} with $T=\ve^{-\beta_*}$
\[
\left|N_{\bA}(x)-\cN(x)\right|\le \frac 1\pi\int_{-\ve^{-\beta_*}}^{\ve^{-\beta_*}}\left|\frac{\widehat N_{\bA}(\xi)-\widehat \cN(\xi)}{\xi}\right|d\xi+\cO(\ve^{\beta_*-\frac 12}),
\]
where $N_{\bA}$ is the distribution of the random variable $\bA$.  The above integral can
be computed, and shown to be small, using Propositions~\ref{prop:not-so-trivial-0},
\ref{prop:not-so-trivial-1/2} and~\ref{prop:not-so-trivial-1} as we have done in the proof
of Proposition~\ref{p_smoothlclt}.  Note however that this would yield weaker results, as
far as we are concerned, since the errors in the distribution function translate badly on
errors for probability of small intervals (which represent our current interest).
 \end{rem}

 \subsection{Standard pairs decomposition}\ \newline\label{subsec:standard-dec} To
 complete our argument we need to provide the proofs of the previously stated
 Propositions. Such proofs turn out to be rather laborious and to them is devoted the rest
 of the paper.

 We start first with a generalization of Proposition~\ref{p_invarianceStandardPairs}.
 \begin{lem}\label{sublem:iteracorr}
   There exists a constant $C_*\in(0,1)$ such that, for each short complex standard pair
   $\ell$, $K\in\bN$, $|\sigma|\leq \ve^{-\efrac 12-2\delta_*}$, imaginary potential
   families $\pot=(i\sigma \hov_j)_{j=0}^{K-1}$, $\hov_j$ defined as in
  ~\eqref{eq:potdef}, finite index set $\cA$ and functions
   $\overline B=(B_a)_{a\in\cA}$, $\|B_a\|\nc1\leq C_*$, and times
    $\{k_a\}_{a\in\cA}\subset\{0,\cdots, K-1\}$, $k_a=k_{a'}\Longrightarrow a=a'$, for any
   $\vartheta\in (0,C_* 3^{-\sharp\cA})$ there exists a short complex standard family
   $\stdf_{K,\pot, \overline B}$ such that, for all $A\in L^\infty(\bT^2,\bC)$:
   \begin{align*}
     \nu_{B,\ell}(A):=\mu_\ell\left(A\circ F_\ve^Ke^{i\sigma\sum_{j=0}^{K-1}\hov_j\circ F_\ve^j}\left[ 1+\vartheta\prod_{a\in\cA} B_a\circ F_\ve^{k_a} \right]\right)=\sum_{\tell\in\stdf_{K,\pot, \overline B}}\fm_\tell\mu_{\tell}(A).
   \end{align*}
   In addition, we have
   \begin{align} \label{eq:fm-not-zero}
       |\fm_\tell|\geq \Const\exp(-\const K)
   \end{align}
   for some uniform $\Const,\const$.

   Finally, if $|\sigma|\leq \sigma_0 $ and $K\geq C\log\vei$, for some $C$ large enough,
   then the above holds also requiring that the family $\stdf_{K,\pot, \overline B}$
   or/and $\ell$ consists of long complex standard pairs.
 \end{lem}
\begin{proof}
  We will use a baby cluster expansion like strategy. Note that, provided $C_*$ is small
  enough, $\pi_a=\log (1+ B_a)$ are allowed potentials for both short and long standard
  pairs.  Then, calling $\sharp\cA$ the cardinality of $\cA$, $\cP(\cA)$ the power set of
  $\cA$ and $S^c = \cA\setminus S$,  we have
\begin{align*}
  \prod_{a\in\cA}B_a\circ F_\ve^{k_i}
  &= \prod_{a\in\cA}(B_a\circ F_\ve^{k_a}+1-1)=\sum_{S\in\cP(\cA)}(-1)^{\sharp S^c}\prod_{a\in S}e^{\pi_a\circ F_\ve^{k_a}}.
\end{align*}
Then, if we set $\bar\pi_{S,k}=0$ if $k\not \in \{k_a\}_{a\in S}$ and
$\bar\pi_{S,k_a}=\pi_a$ otherwise, we can write
\begin{align*}
\nu_{B,\ell}(A)&=\mu_\ell\left(A\circ F_\ve^K e^{i\sigma\sum_{j=0}^{K-1}\hov_j\circ F_\ve^j } \right)\\
&\phantom = +\vartheta\sum_{S\in\cP(\cA)} (-1)^{\sharp S^c}\mu_\ell\left(A\circ F_\ve^K e^{\sum_{j=0}^{K-1}[i\sigma\hov_j+\bar\pi_{S,j}]\circ F_\ve^j } \right).
\end{align*}
We can now use Lemma~\ref{p_invarianceStandardPairs} on each term of the above sums. Note
that the decomposition in complex standard curves does not depend on the details of the
potential but only on $|\sigma|$ and the dynamics.  In particular,  we can write
\[
\nu_{B,\ell}(A)=\sum_{\ell'\in \stdf}\fm^0_{\ell'}\mu_{\ell'}\left(A\right)+\vartheta\sum_{S\in\cP(\cA)}(-1)^{\sharp S^c}\sum_{\ell'\in \stdf_{S}}\fm_{S,\ell'}\mu_{\ell'}\left(A\right)
\]
where $\stdf=\{(\bG_j,\rho^0_{j}),\fm^0\}$ and, for each $S\subset\cA$, $\stdf_{S}=\{(\bG_j,\rho_{S,j}),\fm_S\}$. Note that, if $\ell'_j=(\bG_j,\rho_{S,j})$,
\begin{align*}
  |\fm_{S,\ell'_j}|&=\left|\mu_\ell\left(\Id_{\ell'_j}e^{i\sigma\sum_{j=0}^{K-1}\hov_j\circ F_\ve^j}\prod_{a\in S}(1+ B_a)\circ F_\ve^{k_a} \right)\right|\\
                   &\leq (1+C_*)^{\sharp S}|\mu_\ell|\left(\Id_{\ell'_j} \right),
\end{align*}
see Remark~\ref{rem:fm-bound} for an explanation of the notation $\Id_{\ell}$. Next, notice that, by the usual distortion arguments
\[
|\supp\Id_{\ell'_j}|\left|\frac{d}{dx}\sum_{j=0}^{K-1}\hov_j\circ F_\ve^j(x)\right|\leq \Const \deltacomplex\sum_{j=0}^{K-1}\lambda^{-K+j}\leq \Const\deltacomplex.
\]
Thus
\[
|\mu_\ell|\left(\Id_{\ell'_j} \right)\leq \Const \left|\mu_\ell\left(\Id_{\ell'_j}e^{i\sigma\sum_{j=0}^{K-1}\hov_j\circ F_\ve^j} \right)\right|,
\]
hence
\[
|\fm_{S,j}|\leq \Const (1+C_*)^{\sharp S}|\fm^0_j|.
\]
The above implies
\[
\vartheta\sum_{S\in\cP(\cA)}|\fm_{S,j}|\leq \vartheta\Const |\fm^0_j|\sum_{S\in\cP(\cA)}(1+C_*)^{\sharp S}= \vartheta\Const(2+C_*)^{\sharp \cA}|\fm^0_j|\leq \frac {|\fm^0_j|}2,
\]
provided $C_*$ is small enough.

We can then define the standard family
 $\stdf_{K,\Omega, \overline B}=\{(\bG_j,\rho_j),\fm_{j}\}$ where
 \[
 \fm_j=\fm_j^0 +\vartheta\sum_{S\in\cP(\cA)}(-1)^{\sharp S^c}\fm_{S,j}\;;\quad \rho_j =\fm_j^{-1}\left[\fm_j^0\rho^0_j+\vartheta\sum_{S\in\cP(\cA)}(-1)^{\sharp S^c}\fm_{S,j}\rho_{S,j}\right],
 \]
which concludes the first part of Lemma (see also Remark~\ref{rem:complextoreal}).

If $|\sigma|\leq \sigma_0$, then the above argument works verbatim in the case in which
$\ell$ is a long standard pair. If $\ell$ is a short complex standard pair, then, by
Remark~\ref{rem:prestandard} we can, at each step, use complex standard pairs of length
$\frac 32$ longer than the ones at the previous step, provided the length stays smaller
than $\delta$. Thus, at most after $\Const\log\ve^{-1}$ steps we have families that
consist of long complex standard pairs.
\end{proof}
\begin{proof}[{\bf Proof of Lemma~\ref{lem:one-step-M}}]
  Recall that, by~\eqref{eq:Mvariable} and~\eqref{eq:hbolddef}, we have
  \begin{align*}
    \bM_{k}&=\etaExph(t-\ve S_k,\bar\theta_{L_k})\left\{\sum_{j=0}^{L_k-1}\etaExp_{ j,L_k}\left[\ho(x_j,\theta_j)-\frac{\ve}2\bar\omega'(\bar\theta_{j})\bar\omega(\bar\theta_{j})\right] +\ve\,\bC_k\right\}\\
    \bC_k&=\frac 12P(t-\ve
           S_k,\bar\theta_{L_k})\left[\sum_{j=0}^{L_k-1}\etaExp_{j,L_k}\ho(x_j,\theta_j)\right]^2.
  \end{align*}
  We would like to argue by using Proposition~\ref{p_invarianceStandardPairs}.
  Unfortunately, the above random variables are not of a form suitable to play the role of
  a potential since they contain products of functions computed at different times (that
  is \emph{correlation terms}).  We will solve this problem in three steps. First we will
  express the averaged trajectory $\bar\theta_k$ in terms of one starting from an initial
  condition that depends only on the standard pair, so that the averaged trajectory
  becomes deterministic.  Then we will develop in series the exponential and finally we
  will show how to deal, in general, with the type of objects so obtained (using
  Lemma~\ref{sublem:iteracorr}).

  Arguing as at the end of Lemma~\ref{lem:break-smart?} we have, for any function
  $\vf\in\cC^2$,\footnote{ In this section we use the shorthand notation
    $\cO(\cdot)=\cO_{L^\infty}(\cdot)$.}
  \begin{align*}
    \vf(\bar\theta_k)=\vf(\thetaslk{k})+\vf'(\thetaslk{k})\etaExph(\ve k,\thetasl)(\theta_0-\thetasl)+\cO(\ve^2\deltacomplex^2).
  \end{align*}
  Using~\eqref{eq:ifnotdefinedyet} and Notation~\ref{not:bookcor} we can (see
  Appendix~\ref{app:tedious} for a detailed explanation on how to perform these, and similar,
  computations) rewrite~\eqref{eq:Mvariable} as
  \begin{equation}\label{eq:random-var-def}
    \begin{split}
      \bM_{k}&=\bM^*_{\ell,k}+\ve\bbK^{k,2}_{\ell,3}+\cO(\ve^2\deltacomplex L_k^2)\\
      \bM^*_{\ell,k}&=\etaExph(t-\ve S_k,\thetaslk{L_k})\left\{\sum_{j=0}^{L_k-1}\etaExps_{j,L_k}\left[\ho(x_j,\theta_{j})-\frac{\ve}2\bar\omega'(\thetaslk{j})\bar\omega(\thetaslk{j})\right] \right\}.
    \end{split}
  \end{equation}
  By equations~\eqref{eq:random-var-def},~\eqref{eq:potdef} and the Taylor expansion
  \begin{equation}\label{eq:readytocook}
    \begin{split}
      & e^{i\sigma \bM_k} =e^{i\sigma \left[\bM^*_{\ell,k}+\ve\bbKs^{k,2}_{\ell,3}\right]}+\cO(\sigma \ve^2\deltacomplex L_k^2)\\
      &\phantom{e^{i\sigma \bM_k}}=e^{i\sigma\left[\ve\fkC^{k,1}_{\ell,0,\emptyset}+\sum_{j=0}^{L_k-1}\hov^k_{\ell,j}\circ F^j_\ve\right]}\left[1+\bC_{\ell,k,q}^*\right]+\cE_{\ell,k,q} \\
      &\fkC^{k,1}_{\ell,0,\emptyset}=-\frac{1}2\etaExph(t-\ve S_k,\thetaslk{L_k})\sum_{j=0}^{L_k-1}\etaExps_{j,L_k}\bar\omega'(\thetaslk{j})\bar\omega(\thetaslk{j})\\
      &\bC_{\ell,k,q}^*=\sum_{s=1}^{q-1}(i\sigma\ve)^s\bbK^{k,2s}_{\ell,3s}\;;\quad
      \cE_{\ell,k,q}=\cO(\ve^q\sigma^qL_k^{2q}+\sigma \ve^2\deltacomplex L_k^2),
    \end{split}
  \end{equation}
  where we used Notation~\ref{not:bookcor}.  Set
  \begin{align*}
    \vartheta:=\sum_{s=1}^{q-1}\sum_{\bar\imath}\left|(\sigma\ve)^s\fkC^{k,2s}_{\ell,3s, \bar\imath}\right|
  \end{align*}
  and notice that
  \begin{align*}
    \vartheta \leq C_q\sum_{s=1}^{q-1}(|\sigma|\ve L_*^2)^s\leq C_q |\sigma|\ve  L_*^2 \leq \Const \ve^{\delta_*}.
  \end{align*}
  The above shows that, for $\ve$ small enough,~\eqref{eq:readytocook} is a sum of terms to
  which we can apply Lemma~\ref{sublem:iteracorr}, plus a small remainder; in fact\footnote{
    Note that we have absorbed the sign of $\sigma\fkC^{k,s}_{\ell,3s, \bar\imath}$ into
    some $A_{j,\bar\imath}$, which is always possible since the $A_{j,\bar\imath}$ are names
    for arbitrary functions.}
  \begin{align*}
    \left[1+\bC_{\ell,k,q}^*\right]
    &= \vartheta\invr\sum_{s=1}^{q-1}\sum_{\bar\imath}|(\sigma\ve)^s\fkC^{2k,s}_{\ell,3s, \bar\imath}|\left[1+\vartheta\prod_{j=1}^{3s}A_{j,\bar\imath}\circ F_\ve^{i_j}\right].
  \end{align*}
  We have thus written $e^{i\sigma \bM_k}$ as a weighted sum of terms which satisfy the
  hypotheses of Lemma~\ref{sublem:iteracorr}, also the analogous of Remark~\ref{rem:complextoreal} applies.   Note
  that, again, the decomposition in standard curves can be chosen to be exactly the same for
  all terms.  We can then define the standard family $\stdf^k_\ell$ exactly as it was done
  at the end of the proof of Lemma~\ref{sublem:iteracorr}.  By~\eqref{eq:fm-not-zero} we
  conclude that the total weight of each standard pair differs uniformly from zero, which allows to normalize the densities.
  This proves the first part of the lemma.

  To conclude the proof we need to prove~\eqref{eq:trivial-fm-estimate}; notice that, by the
  first part of the lemma and using the same notation as in
  Remark~\ref{rem:fm-bound},\footnote{ Below we consider $\Id_{\ell_j}$ to be a function
    defined on the standard pair $\ell_{j-1}$.  Also notice that $\Id_{\ell_j}$ can be
    written, if needed, as the restriction to $\ell_{j-1}$ of
    $\vf_{\ell_{j-1},\ell_j}\circ F_\ve^{L_{j-1}}$, for some function
    $\vf_{\ell_{j-1},\ell_j}\in L^\infty(\bT^2,\bR)$.} for any $0\leq r\leq s\leq R$, we
  have\footnote{ We use the convention that $\sum_{j=a}^b c_j=0$ and $\prod_{j=a}^b c_j=1$
    if $b<a$.}
  \[
    \begin{split}
      &\mu_{\ell_{r}}\left(e^{i\sigma\sum_{j=r}^{s} \bM_{j}\circ F_\ve^{S_{j-1}-S_{r-1}}}\Id_{\ell_{s+1}}\circ F_\ve^{S_{s-1}-S_{r-1}}\cdots\Id_{\ell_{r+2}}\circ F_\ve^{S_{r}-S_{r-1}}\Id_{\ell_{r+1}}\right)\\
      &=\mu_{\ell_{r}}\left(\left[e^{i\sigma\sum_{j=r+1}^{s} \bM_{j}\circ F_\ve^{S_{j-1}-S_{r}}}\Id_{\ell_{s+1}}\circ F_\ve^{S_{s-1}-S_{r}}\cdots\Id_{\ell_{r+2}}\right]\circ F_\ve^{L_{r}}\cdot e^{i\sigma \bM_{r}}\Id_{\ell_{r+1}}\right)\\
      &=\fm_{r,\ell_{r},\ell_{r+1}}\mu_{\ell_{r+1}}\left(e^{i\sigma\sum_{j=r+1}^{s} \bM_{j}\circ F_\ve^{S_{j-1}-S_{r}}}\Id_{\ell_{s+1}}\circ F_\ve^{S_{s-1}-S_{r}}\cdots\Id_{\ell_{r+2}}\right)\\
      &\phantom{=}+\cO(\ve^q\sigma^qL_{r}^{2q}+\sigma \ve^2\deltacomplex L_{r}^2)|\mu_{\ell_{r}}|(\Id_{\ell_{s+1}}\circ F_\ve^{S_{s-1}-S_{r-1}}\cdots\Id_{\ell_{r+1}}).
    \end{split}
  \]
  Iterating the above equation yields, for all $r\leq s$,
  \begin{equation}\label{eq:itera-fm}
    \begin{split}
      &\mu_{\ell_{r}}\left(e^{i\sigma\sum_{j=r}^{s} \bM_{j}\circ F_\ve^{S_{j-1}-S_{r-1}}}\Id_{\ell_{s+1}}\circ F_\ve^{S_{s-1}-S_{r-1}}\cdots\Id_{\ell_{r+1}}\right)\\
      &= \prod_{l=r}^s\fm_{l,\ell_{l},\ell_{l+1}}+\cO(\ve^q\sigma^qL_*^{2q}+\sigma \ve^2\deltacomplex L_*^2)\\
      &\phantom{=}\times\sum_{j=r}^{s}\prod_{l=r}^{j-1}|\fm_{l,\ell_{l},\ell_{l+1}}|\,|\mu_{\ell_j}|\left(\Id_{\ell_{s+1}}\circ F_\ve^{S_{s-1}-S_{r-1}}\cdots\Id_{\ell_{j+1}}\right).
    \end{split}
  \end{equation}
  In particular, choosing $s = r$ we conclude that there exists $\Cuno>0$ such that:,
  \begin{equation}\label{eq:induction-step-zero}
    \sum_{\ell_{r+1}\in \stdf_{\ell_{r}}^{r}}|\fm_{r,\ell_{r},\ell_{r+1}}|\leq \sum_{\ell_{r+1}\in \stdf_{\ell_{r}}^{r}}\Const\,|\mu_{\ell_{r}}|\left(\Id_{\ell_{r+1}}\right)\leq \Cuno.
  \end{equation}
  To conclude we prove, by induction on $m=s-r$, that
  \begin{equation}\label{eq:induction-step-fm}
    \sum_{\ell_{r+1}\in \stdf_{\ell_{r}}^{r}} \cdots\sum_{\ell_{s+1}\in \stdf_{\ell_{s}}^{s}} \prod^{s}_{j=r}|\fm_{j,\ell_{j},\ell_{j+1}}|\leq 2\Cuno
  \end{equation}
  Equation~\eqref{eq:induction-step-zero} shows that~\eqref{eq:induction-step-fm} holds for each
  $r$ and $m=0$. Let us suppose it holds for each $r$ and $n\leq m$ for some
  $m\in\{0,\cdots R-2\}$. Let $s=r+m+1$. Then,
  recalling~\eqref{e_trivialTotalVariationBound}, the fact that
  $|\sigma|\deltacomplex\leq \Const$ and the condition on $L_*$, we can
  use~\eqref{eq:itera-fm} to write:
  \begin{align*}
    & \sum_{\ell_{r+1}\in \stdf_{\ell_{r}}^{r}}\cdots\sum_{\ell_{s+1}\in \stdf_{\ell_{s}}^{s}}  \prod^{s}_{j=r}|\fm_{j,\ell_{j},\ell_{j+1}}|\leq
      |\mu_{\ell_{r}}|+ \Const\,|\mu_{\ell_r}|\,(\ve^q\sigma^qL_*^{2q}+\sigma \ve^2\deltacomplex L_*^2)\\
    &+ \Const\sum_{j=r+1}^{s}\sum_{\ell_{r+1}\in \stdf_{\ell_{r}}^{r}} \cdots\sum_{\ell_{j}\in \stdf_{\ell_{j-1}}^{j-1}}\prod_{l=r}^{j-1}|\fm_{l,\ell_{l},\ell_{l+1}}|\,|\mu_{\ell_j}|\,(\ve^q\sigma^qL_*^{2q}+\sigma \ve^2\deltacomplex L_*^2)\\
    &\leq \Cuno+ (m+1) \Cuno^2\Const(\ve^q\sigma^qL_*^{2q}+\sigma \ve^2\deltacomplex L_*^2)\leq 2\Cuno,
  \end{align*}
  provided $\ve$ is small enough and $q\geq 4$.
\end{proof}

The remaining sections of the paper are devoted to the proofs of
Propositions~\ref{prop:not-so-trivial-0},~\ref{prop:not-so-trivial-1/2}
and~\ref{prop:not-so-trivial-1} although we first need a preparatory technical section.

\section{One block estimate: technical preliminaries}\label{sec:one-prelim}
Our next step consists in transforming the sums on the standard pairs associated to each
of the $R$ blocks into an expression involving transfer operators related to a cocycle over the
(slowly varying) averaged dynamics.  This will at last allow us to perform the needed
computations by functional analytic means.

Let us start by defining the slowly varying dynamics. Let $\ell=(\bG,\rho)$ be a complex
standard pair; recall that we introduced the notations
$\thetasl=\Re(\mu_\ell(\theta_0))$, $\thetaslk{k}=\bar\theta(\ve k,\thetasl)$
in~\eqref{eq:ifnotdefinedyet}, where $\bar\theta(t,\theta)$ is the unique solution
of~\eqref{eq:averageeq} with initial condition $\bar\theta(0,\theta) = \theta$.  Recall
also that we defined (in~\eqref{e_deviationsDef}) $\deviationsl_k = \theta_k-\thetaslk{k}$
and that, for real standard pairs, we will regard $x_k$ and $\deviationsl_k$ as functions on $\ell$ (see Remark~\ref{r_randomVariables}).

Let us define the shorthand notations $\favg_k=f(\cdot, \thetaslk{k})$,
$\favg^{(n)}=\favg_{n-1}\circ \cdots\circ \favg_0$;
consider the map $\Favg(x,\theta)=(f(x,\theta),\bar\theta(\ve,\theta))$.  Observe that
$(\favg^{(k)}(x),\thetaslk{k})=\Favg^k(x,\thetasl)$, \ie the first component of $\Favg$
yields our wanted slowly varying dynamics.  Finally, let us define the function
\begin{align}\label{e_definitionLambda}
  \BLambda_{j,k} = \prod_{r=j}^k (\partial_x f)\invr \circ \Favg^r(\cdot,\thetasl)
 \end{align}
Notice that, by definition, $\BLambda_{j,k} < \lambda^{ -(k-j)-1}$ and for any $x\in\bT$:
\begin{align}\label{e_sillyBoundOnLambda}
  \left|\frac{d}{dx} \BLambda_{0,j}(x)\right|
  &\le  \Const \BLambda_{0,j}(x)\sum_{l=0}^{j-1}\BLambda_{0,l}(x)^{-1}\leq \Const.
\end{align}
\subsection{Error in the slowly varying dynamics approximation}\label{subsec:cody}\ \newline
Our first task is to obtain sufficiently good estimates on the difference between
$F_\ve^k$ and $\Favg^k$ when $k$ is not too large.
 \begin{lem}\label{l_shadowAveraged}
   Fix a complex standard pair $\ell=(\bG_\ell, \rho_\ell)$ of length $\deltac$ and $L\in\bN$ so that
   $L\le\Const\ve\eefrac{-1}2$.  There exists a diffeomorphism
   $\diffeoAvg_{\ell,L}:[a,b]\to[\bar a,\bar b]$ such that
   $(x_{L},\theta_{L})=F_\ve^{L}\circ
   \bG_\ell(x)=\left(\favg^{({L})}\circ\diffeoAvg_{\ell,L}(x),\theta_{L}\right)$ with
   \begin{equation}\label{eq:upsilonp}
     \de{\diffeoAvg_{\ell,L}}{x} %
     = (1-G'\stable_{L})\enu_{L}\Lambda_{0,L-1}
   \end{equation}
   where $\enu_{L}$ was defined in~\eqref{e_defineSlopes} and
   $\Lambda_{k,j} = \BLambda_{k,j}\circ\diffeoAvg_{\ell,L}$.
   Moreover $\diffeoAvg_{\ell,L}$ satisfies  the following estimates:
   \begin{align} \label{eq:up-c1-bound}
     \left\|\diffeoAvg_{\ell,L}-\Id\right\|\nc0&\le \Const \ve\min\{1,L^2\deltac\}&
     \frac{d\diffeoAvg_{\ell,L}}{dx} &= 1 + \erR{\ell}^{1,2}
   \end{align}
where the notation $\erR{\ell}^{p,q}$ denotes an arbitrary differentiable function of $x$ that satisfies the bounds
  \begin{align}\label{eq:remainder-measure}
    \|\erR{\ell}^{p,q}\|\nc0&\leq \Const \ve^p L^q &
    \left|\frac{d\,\erR{\ell}^{p,q}}{d x}\right|&\le \Const \Lambda_{0,L-1}\invr.
  \end{align}
   Additionally, for any $k\in\{0,\cdots,{L}\}$, let us introduce the functions
  \begin{align*}
    \bar\xi_{\ell,k}&=x_k\circ\diffeoAvg_{\ell,L}\invr-\bar x_k &
    \bdeviationsl_{k}&= \deviationsl_k\circ\diffeoAvg_{\ell,L}\invr
  \end{align*}
  where we introduced the shorthand notation $\bar x_k=\favg^{(k)}(\cdot)$.  Recall the
  definition of the quantities $\etaExps_{j,k}$ given in~\eqref{e_defEtaExp}; then let
  \begin{subequations}\label{se_definitionWW}
    \begin{align}
      \ooW_k  &= \vei\etaExps_{-1,k}\bdeviationsl_0+ \sum_{j=0}^{k-1}  \etaExps_{j,k}\ho\bcoord l\\
      \ofkW_k &= -\sum_{l=k}^{L-1}\Lambda_{0,l-k}\partial_\theta f\bcoord l\ooW_{l}  
    \end{align}
  \end{subequations}
    Moreover define:
\begin{align*}
      \ooW_{k,2}&= \sum_{j=0}^{k-1}
                  \Big[\partial_x\ho\bcoord j\ofkW_j
                  +\partial_\theta\ho\bcoord j\ooW_{j}
                                                                   -\frac12\bar\omega'(\thetaslk{j})\bar\omega(\thetaslk{j})\Big]\\
      \ofkW_{k,2}&=-\sum_{l=k}^{L-1}\Lambda_{0,l-k}\bigg[\partial_\theta
                   f\bcoord l \ooW_{l,2}+\frac
                  {1}2\partial_\theta^2 f\bcoord l\ooW_{l}^2\\
                &\phantom{=-\sum_{l=k}^{L-1}\Lambda_{0,l-k}\bigg[}+\frac {1}
                  2\partial_\theta\partial_x f\bcoord l\ooW_{l}\ofkW_l+\frac  {1}
                  2\partial_x^2 f\bcoord l\ofkW_l^2\bigg].
    \end{align*}
    Then the following bounds hold true
    \begin{align}\label{se_definitionDD}
      \bdeviationsl_{k}
      &= \ve\ooW_{k}+\ve^2\ooW_{k,2}+\oerR{\ell}^{3,3} &
      \bar\xi_{\ell,k}&=\ve\ofkW_k+\ve^2\ofkW_{k,2}+\oerR{\ell}^{3,3},
    \end{align}
  where the notation $\oerR{\ell}^{p,q}$ is analogous to $\erR{\ell}^{p,q}$ but with $\Lambda$
  replaced by $\BLambda$ in~\eqref{eq:remainder-measure}.
 \end{lem}
 \begin{rem}\label{rem:series-expla}
   Note that the above lemma is essentially a series expansion in which we only keep the
   first few terms.  More precise formulae can be obtained, if needed, at the price of more
   work.
 \end{rem}
 \begin{rem}
   The approximation formulae obtained in the previous lemma are close, in spirit, to the
   ones obtained earlier in Lemma~\ref{l_boundEtaRef}, but differ from them because they
   are written in terms of the averaged dynamics $(\bar x_k,\thetaslk{k})$, rather than the
   real dynamics $(x_k,\theta_k)$.
 \end{rem}
 \begin{rem}
   Observe that the random variables $\bar\xi_{\ell,k}$ and $\bdeviationsl_k$ (defined in
   the previous lemma) do depend on $L$ (through $\diffeoAvg_{\ell,L}$).  In order to make
   the notation precise their symbols should thus have indices $L$.  Since it will not create
   any confusion, we omit some of the indices to ease notation.  Similarly, we will
   suppress the indices in $\diffeoAvg$ as well when no confusion arises,
 \end{rem}
 \begin{proof}[{\bf Proof of Lemma~\ref{l_shadowAveraged}}]
   The lemma follows from a variation on the proof given for Lemma~\ref{lem:shadow}.  Let
   us recall that we denote with $\pi_x:\bT^2\to\bT$ the projection on the $x$-coordinate
   and define, for $\varrho\in[0,1]$,
   \begin{align*}
     \bar \imFun_{\ell,L}(x, z; \varrho)=\pi_x F_{\varrho\ve}^{L}(x,\thetasl+\varrho(G(x)-\thetasl))-
     \pi_x\Favgnosub_{\varrho\ve}^{L}(z,\thetasl).
   \end{align*}
   As in the proof of Lemma~\ref{lem:shadow}, observe that
   $\bar \imFun_{\ell,L}(x,x;0)=0$, and moreover
   $\partial_z\bar \imFun_{\ell,L}=
   -\partial_z(\pi_x\Favgnosub_{\varrho\ve}^{L}(\cdot,\thetasl))<-\lambda^{L}$.  Therefore
   the implicit function theorem implies that for any $\varrho\in[0,1]$ there exists a
   diffeomorphism $\diffeoAvg_{\ell,L}(\cdot;\varrho)$ so that
   $\bar \imFun_{\ell,L}(x,\diffeoAvg_{\ell,L}(x;\varrho);\varrho) = 0$.  Define
   $\diffeoAvg_{\ell,L}(x)=\diffeoAvg_{\ell,L}(x,1)$; then
   $\pi_x F_\ve^L\circ\bG_\ell=x_L=\favg^{(L)} \circ \Upsilon_{\ell,L}$.  The
   expression~\eqref{eq:upsilonp} then immediately follows using the notation and
   discussion of Subsection~\ref{subsec:splitting}.  Let us postpone the derivation
   of~\eqref{eq:up-c1-bound} to the end of the proof and first obtain the
   bounds~\eqref{se_definitionDD}. Using~\eqref{e_boundEtaRefs} yields
   \begin{align}\label{eq:aboutthata}
     \bdeviationsl_{k}\notag
     &= \etaExps_{-1,k}\bdeviationsl_0+\ve\sum_{j=0}^{k-1}\etaExp^*_{\ell,j,k}\left[\ho(x_j,\theta_j)-\frac{\ve}
       2\bar\omega'(\thetaslk{j})\bar\omega(\thetaslk{j})\right] +
       \cO\left(\ve\sum_{j=0}^{k-1}(\bdeviationsl_j)^2+\ve^3k\right)\\
     &=\etaExps_{-1,k}\bdeviationsl_0+ \ve\sum_{j=0}^{k-1}\etaExp^*_{\ell,j,k}\Big[\ho(\bar x_j,\thetaslk{j})+\partial_x\ho(\bar x_j,\thetaslk{j})\bar\xi_j+\partial_\theta\ho(\bar x_j,\thetaslk{j})\bdeviationsl_j\notag\\
     &\phantom{=\etaExps_{-1,k}\bdeviationsl_0+\ve\sum_{j=0}^{k-1}\etaExp^*_{\ell,j,k}\Big[}-\frac{\ve}
       2\bar\omega'(\thetaslk{j})\bar\omega(\thetaslk{j})\Big]+\cO(\ve^3
       k^3)+\sum_{j=0}^{k-1}\cO(\ve\bar\xi_j^2)
   \end{align}
   where we have used~\eqref{eq_trivialBoundDeviations}.  In addition, we can consider
   the Taylor expansion
   \begin{equation}\label{eq:aboutthatb}
     \begin{split}
       \bar\xi_{k+1}&=f(x_{k},\theta_{k})-f(\bar x_k, \thetaslk{k})\\
       &=\partial_x f(\bar x_k,\thetaslk{k})\bar\xi_{k}+%
         \partial_\theta f(\bar x_k,\thetaslk{k})\bdeviationsl_{k}+
\frac 12\partial_\theta^2 f(\bar x_k,\thetaslk{k})(\bdeviationsl_{k})^2\\
       &\phantom{=}+\frac 12\partial_\theta\partial_x f(\bar x_k,\thetaslk{k})\bdeviationsl_{k}\bar\xi_k+\frac 12\partial_x^2 f(\bar x_k,\thetaslk{k})\bar\xi_k^2+
       \cO((\bdeviationsl_k)^3+\bar\xi_k^3).
     \end{split}
   \end{equation}
   From the first line of~\eqref{eq:aboutthatb} we have
   $|\bar\xi_{k+1}|\geq \lambda |\bar\xi_k|-\Const |\bdeviationsl_k|$.  Recalling that, by
   definition, $\bar\xi_{L}=0$, we can conclude that
   \begin{equation}\label{eq:roughboundxi}
     |\bar\xi_k|\le \Const
     \sum_{j=k}^{L-1}\lambda^{k-j}|\bdeviationsl_j|\le \Const \ve (k+1).
   \end{equation}
   Moreover, by the above estimates, we have
   \begin{equation}\label{eq:notsoimpressive}
     \begin{split}
       \bdeviationsl_{k}&=
       \ve\ooW_k+\cO(\ve^2k^2).
     \end{split}
   \end{equation}
   A more precise result can now be obtained by (backward) iteration of~ ~\eqref{eq:aboutthatb}:
   \begin{equation}\label{eq:standardxi}
     \begin{split}
       \bar\xi_k&=-\sum_{j=k}^{L-1}\prod_{l=k}^j(\partial_xf\bcoord l)^{-1} \left[\partial_\theta f\bcoord j \bdeviationsl_j+\cO(\bar\xi_j^2+(\bdeviationsl_j)^2)\right]\\
       &=-\sum_{j=k}^{{L}-1}\Lambda_{k,j}\partial_\theta f\bcoord j
       \bdeviationsl_j+\cO(\ve^2 L^2)=\ve\ofkW_k+\cO(\ve^2 L^2).
     \end{split}
   \end{equation}
   Finally, we can get a sharper estimate for $\bdeviationsl_k$ by
   substituting~\eqref{eq:standardxi} and~\eqref{eq:notsoimpressive}
   in~\eqref{eq:aboutthata}:
   \begin{equation}\label{eq:nowwetalkD}
     \begin{split}
       \bdeviationsl_{k}&=\ve\ooW_{k}+\ve^2\sum_{j=0}^{k-1}\Big[\partial_x\ho(\bar x_j,\thetaslk{j})\ofkW_j+\partial_\theta\ho(\bar x_j,\thetaslk{j})\ooW_{j}\\
       &\phantom{=}-\frac{1}2\bar\omega'(\thetaslk{j})\bar\omega(\thetaslk{j})\Big]+\cO(\ve^3L^3) 
       = \ve\ooW_{k}+\ve^2\ooW_{k,2}+\cO(\ve^3L^3);
     \end{split}
   \end{equation}
   and a sharper estimate for $\bar\xi_k$ by writing~\eqref{eq:aboutthatb} as
   \begin{equation}\label{eq:aboutthatc}
     \begin{split}
       \bar\xi_{k+1}&=\partial_x f(\bar x_k,\thetaslk{k})\bar\xi_{k}+\ve\partial_\theta f(\bar x_k,\thetaslk{k})\ooW_{k} +\ve^2\partial_\theta f(\bar x_k,\thetaslk{k})\ooW_{k,2}\\
       &\phantom{=}+\frac {\ve^2}2 \partial_\theta^2 f(\bar x_k,\thetaslk{k})\ooW_{k}^2\\
       &\phantom{=}+\frac {\ve^2} 2\partial_\theta\partial_x f(\bar
       x_k,\thetaslk{k})\ooW_{k}\ofkW_k+\frac {\ve^2} 2\partial_x^2 f(\bar
       x_k,\thetaslk{k})\ofkW_k^2+ \cO(\ve^3L^3),
     \end{split}
   \end{equation}
   which, iterating backward as before, yields the wanted result.  The bound on the
   derivatives of the error terms, that is needed to write $\cO(\ve^3L^3)$ as
   $\erR{\ell}^{3,3}$, follows by definition of $\bdeviationsl_k$ and
   $\bar\xi_k$,~\eqref{e_sillyBoundOnLambda} and the fact that
   $|\diffeoAvg'_{\ell,L}|\le\Const$, which in turn follows from the second bound
   in~\eqref{eq:up-c1-bound}.

   In order to conclude the proof we now proceed to prove the two bounds
   of~\eqref{eq:up-c1-bound}, which will be obtained by a careful analysis
   of ~\eqref{eq:upsilonp}.  Recall the definition~\eqref{e_defineSlopes} of the quantities
   $\enu_k$ and $u_k$; by the discussion of Subsection~\ref{subsec:splitting}
   (see~\eqref{e_formulaExpansionX}) we have
   \begin{equation}\label{eq:recap-u-cone}
     \begin{split}
       \enu_{L} &=\prod_{k=0}^{L-1}\left[ \partial_{x}f(x_k,\theta_k) +\ve \partial_\theta f(x_k,\theta_k)u_k\right]\\
       u_{k+1}
       &=\frac{\partial_x\omega(x_k,\theta_k)+(1+\ve\partial_\theta\omega(x_k,\theta_k))u_k}{\partial_x
         f(x_k,\theta_k)+\ve\partial_\theta f(x_k,\theta_k)u_k},\quad \ve u_0 = G'(x).
     \end{split}
   \end{equation}
   As already noted we have $|u_k|\leq \spc1$ (one can also see this using
   Proposition~\ref{p_invarianceStandardPairs}, since $\ve u_k$ is the slope of a standard
   curve).  The above immediately implies,
   using~\eqref{eq:upsilonp},~\eqref{eq:roughboundxi}
   and~\eqref{eq_trivialBoundDeviations}:
   \begin{align}\label{e_explicitBoundOnUpsilon'}
     \de{\diffeoAvg_{\ell,L}}{x} %
     &= (1+\cO(\ve))\expo{\sum_{j=0}^{L-1}\log  \partial_xf(x_j,\theta_j)-\log\partial_x f(\bar x_j, \thetaslk{j})+\cO(\ve)}\notag\\
     &=e^{\cO(\ve L^2)}.
   \end{align}
   which yields the $C^0$-bound of the right expression in~\eqref{eq:up-c1-bound}.
   Integrating in $dx$ yields the bound on the left, since by~\eqref{eq:roughboundxi}, we
   know a priori that $|\diffeoAvg_{\ell,L}(x_0)- x_0| = |\xi_0|\le\Const \ve$.

   At last we want to estimate the second derivative
   $\diffeoAvg_{\ell,L}''$; differentiating~\eqref{eq:upsilonp} we obtain
   \begin{align*}
     \frac {d^2\diffeoAvg_{\ell,L}}{dx^2} &= \frac
     d{dx}(1-G'\stable_{L})\enu_L\Lambda_{0,L-1}+ (1-G'\stable_L)\left[\frac d{dx}(\enu_L)\Lambda_{0,L-1}+
                                            \enu_L\frac d{dx}(\Lambda_{0,L-1})\right].
   \end{align*}
   The last term on the right hand side is bounded by $\Const\enu_L$
   using~\eqref{e_sillyBoundOnLambda} and~\eqref{e_explicitBoundOnUpsilon'}.  The second
   term can be estimated by differentiating the first of~\eqref{eq:recap-u-cone}, which gives:
   \begin{align*}
     \left|\frac{d\enu_{L}}{dx}\right| \le \Const\enu_L\sum_{k =
     0}^{L-1}\left[\left|\frac{dx_k}{dx}\right|+\left|\frac{d\theta_k}{dx}\right|+\ve\left|\frac{d u_k}{dx}\right|\right]
   \end{align*}
   To continue, notice that~\eqref{e_defineSlopes} implies $\left|\frac {d
       x_k}{dx}\right|\le\Const\enu_{k}$ and
   $\left|\frac{d \theta_k}{dx}\right|\le\Const\ve\enu_{k}$.
   Moreover, by the second one of~\eqref{eq:recap-u-cone}, we gather
   \begin{align*}
     \left|\frac {d u_k}{d x}\right|\leq \Const \enu_k+ \Const\left|\frac {d u_{k-1}}{d x}\right|\leq \Const \enu_k.
   \end{align*}
   Hence, the second term is also bounded by $\Const\enu_L$.
   To conclude, we need an estimate for the first term on the right hand side.
   \begin{sublem}\label{lem:Y}
     We have
     \begin{align*}
       \left|\de{}x (1-G'\stable_L)\right|\le \Const \ve L e^{\const\ve L}
     \end{align*}
   \end{sublem}
   \begin{proof}
     From~\eqref{e_defineSlopes} it follows, for all $n\in\bN$
     \begin{align*}
       \emu_{n+1} (0,1)  
    &=\deh_p F_{\ve}^{n} \begin{pmatrix}\partial_x f &\partial_\theta f\\ \ve\partial_x\omega&1+\ve\partial_\theta\omega\end{pmatrix}(\stable_{n+1} ,1) \\
    &= (1+\ve(\partial_\theta\omega+\partial_x\omega \stable_{n+1}))%
      \deh_p F_{\ve}^{n} \left(\frac{\partial_\theta f+\partial_x f\stable_{n+1}}{1+\ve(\partial_\theta\omega+\partial_x\omega \stable_{n+1})} ,1\right).
     \end{align*}
     Since $\emu_{n} (0,1) =\deh_p F_{\ve}^{n} (\stable_{n} ,1)$, it follows that
     $\emu_{n+1}(1+\ve(\partial_\theta\omega+\partial_x\omega \stable_{n+1}))^{-1}=\emu_{n}$
     and
     \begin{align*}
       \stable_n=\frac{\partial_\theta f+\partial_x f\stable_{n+1}}{1+\ve(\partial_\theta\omega+\partial_x\omega \stable_{n+1})}.
     \end{align*}
     Inverting the above formula yields
     \begin{equation*}
       \stable_{L-k}(x_{k})=\frac{\stable_{L-k-1}(x_{k+1})(1+\ve\partial_\theta\omega(x_k,\theta_k))-
         \partial_\theta f(x_k,\theta_k)}{\partial_x f(x_k,\theta_k)-\ve\partial_x\omega(x_k,\theta_k)
         \stable_{L-k-1}(x_{k+1})}.
     \end{equation*}
     In order to estimate the derivatives of $s_j$ we proceed by induction. Note that
     $\stable_0(x_L)=0$. Next, suppose
     $|\partial_x\stable_{L-k-1}|\le C (L-k-1)e^{C\ve (L-k-1)} \left|\frac{\partial
         x_{k+1}}{\partial x} \right|$; then
     \[
       |\partial_x\stable_{L-k}|\le \Const \left|\frac{\partial x_k}{\partial x}
       \right|+\left|\partial_x\stable_{L-k-1}\right|e^{\Const \ve}\left|\frac{\partial
           x_k}{\partial x_{k+1}}\right|\le \Const (L-k)e^{\const\ve (L-k)}\left|\frac{\partial
           x_{k}}{\partial x} \right|,
     \]
     provided $\Const$ is large enough.  Since $\|G'\|\nc1=\cO(\ve)$, we conclude the proof by
     using the above formula with $k=0$.
   \end{proof}
   We conclude that
   \begin{align*}
     \left|\frac {d^2\diffeoAvg_{\ell,L}}{dx^2}\right| &\le \Const\enu_L\le\Const\Lambda_{0,L-1},
   \end{align*}
   which gives the needed bound on the derivatives of $\erR{\ell}^{1,2}$ in~\eqref{eq:up-c1-bound} and
   concludes the proof of our lemma.
 \end{proof}

\subsection{Transfer operator representation}\label{subsec:transferop}\ \newline
We are now ready to write the contribution of the standard pairs belonging to one block in
terms of a product of transfer operators.  This is made explicit
by~\eqref{eq:standardotop} in the statement of the next proposition.  Unfortunately, in
the following we will need rather detailed information on the error terms present
in~\eqref{eq:standardotop} which therefore must be painstakingly reported in the statement
of the proposition, making it rather unpleasant. Yet, the reader can skip such details and
come back to them later when they are needed, and recalled.
\begin{notation}\label{not:bookcor1}
  In the sequel we will use notation similar to Notation~\ref{not:bookcor} where, in
  addition, we introduce symbols for correlations terms computed along the averaged
  dynamics which will be denoted with
\begin{align*}
      \bbbK^{k,p}_{\ell,l}=\sum_{\bar\imath}\fkC^{k,p}_{\ell,l,\bar\imath}\prod_{j=1}^l
  A_{j,\bar\imath}\circ \bar F_\ve^{i_j}.
    \end{align*}
\end{notation}
Observe that, according to Notation~\ref{not:bookcor1}, we can write:
\begin{align*}
  \ooW_k  &= \bbbK^{k,1}_{\ell,1} & \ofkW_k  &= \bbbK^{k,1}_{\ell,1} &
  \ooW_{k,2}  &= \bbbK^{k,2}_{\ell,3} & \ofkW_k  &= \bbbK^{k,2}_{\ell,5}
\end{align*}
hence we gather, for any $0\le j\le L$:
\begin{align}\label{e_bbbK-taylor-expansion}
  \bdeviationslk{\ell}{j} &= \ve\bbbK^{k,1}_{\ell,1}+\ve^2\bbbK^{k,2}_{\ell,3} +\oerR{{\ell_k}}^{3,3}& 
  \bar\xi_j &= \ve\bbbK^{k,1}_{\ell,2}+\ve^2\bbbK^{k,2}_{\ell,5} +\oerR{{\ell_k}}^{3,3}
\end{align}
\begin{prop}\label{p_transferOperators}
  For any complex standard pair $\ellz$, let $\{\stdf^k_{\ell_k}\}_{i=1}^{R-1}$ be the
  complex standard families obtained in Lemma~\ref{lem:one-step-M}, of length $\deltac$,
  and assume $|\sigma|\leq \ve^{-\efrac 12-2\delta_*}$.  For any $\Phi\in\cC^2(\bT,\bC)$,
  so that\/\footnote{ The reader should think of $\Phi$ as a function whose real part is
    negative and has very large absolute value} $\Re(\Phi)\leq \Const$,
  $\ve\|\Phi'\|\nc0 L_*\leq \Const$, any $k\in\{0,\cdots,R-1\}$ and $\wp>0$, we have
  \begin{align}\label{eq:standardotop}
    \sum_{\ell_{k+1}\in\stdf^{k}_{\ell_{k}}}\fm_{k,\ell_{k},\ell_{k+1}}e^{\Phi\circ G_{\ell_{k+1}}}\rrho_{\ell_{k+1}}
    &=\cE_{\ell_k}^*\\
    &\hskip-2cm+e^{i\ve\sigma\fkC^{k,1}_{\ell_k,0,\emptyset}}
      \tO_{\ell_k,k, L_k-1}\cdots\tO_{\ell_k,k,0}
    \left[ \Psi_{\ell_k,q}e^{\Phi(\bar\theta_{\ell_k,L_k})}\trhove_{\wp}\right],\notag
  \end{align}
 where $\rrho_{\ell}=\Id_{[a_{\ell}, b_{\ell}]}\rho_{\ell}$ (as introduced in
  Section~\ref{ss_exponentialMoment}) $\bar\theta_{\ell, j}(x)=\bar\theta(\ve
  j,G_{\ell}(x))$, and
  \begin{enumerate}
  \item \label{i_transferOperators} $\tO_{\ell, k,j}$ is the weighted transfer operator
    defined by
    \begin{equation}\label{e_defTransferOp}
      \begin{split}
        &[\tO_{\ell,k,j} g](x)=\sum_{y\in \favg_{j}\invr(x)}
        \frac{e^{\Omega^{k,\Phi}_{\ell, j}(\sigma,y,\thetaslk{j})}}{\favgp_{j}(y)} g(y),\\
        &\Omega^{k,\Phi}_{\ell, j}(\sigma,x,\theta) =%
        i\sigma\hov^k_{\ell,j}(x,\theta)+\ve \Phi'(\thetaslkk{k}{L_k})
        \etaExps_{j,L_k}\ho(x,\theta)
      \end{split}
    \end{equation}
    with $\thetaslk{j}=\bar\theta(\ve j,\avgtheta{\ell})$,
    $\favg_j(\cdot)=f(\cdot,\thetaslk{j})$ and $\hov^k_{\ell,j}$ defined
    in~\eqref{eq:potdef};
  \item \label{i_propertyDelta2} $\Psi_{\ell_k,q}$ is defined by
    \begin{align*}
      \Psi_{\ell_k,q}
      &= \left[1+\sum_{s=1}^{q-1}(i\sigma\ve)^s\bbbK^{k,2s}_{\ell_k,3s}+(i\sigma\ve)^s\ve
        \bbbK^{k,2(s+1)}_{\ell_k,3(s+1)} \right]\\
      &\phantom{ = }\times\expo{ i\sigma\ve\bbbK^{k,2}_{\ell_k,3}+i\sigma\ve^2\bbbK^{k,3}_{\ell_k,6}+\ve \cK_0},
    \end{align*}
    where $\cK_0$ is a $\bbbK^{k,2}_{\ell_k,2}$-type term which satisfies the following
    extra bound:
    \begin{align*}
      \Leb\left[\rrho_{\ell_k}-e^{ \ve
      \cK_0}\,\trhove_{\wp}\right]=\cO(\ve^2L_k^3+ \wp);
    \end{align*}
  \item \label{i_propertyDeltaerr} $\cE_{\ell_k}^*$ satisfies the bounds
    \begin{align*} \|\cE_{\ell_k}^*\|\nl1
      &\leq \Const e^{\Phi^+}\left [\|\Phi'\|\nc1 \ve^2 L_k^2+\ve^2L_k^3+\wp\right]\\
      \|\cE_{\ell_k}^*\|\nBV&\leq \Const e^{\Phi^+}[1+|\sigma|]
    \end{align*}
    where $\Phi^+=\max\,\Re(\Phi)$;
  \item \label{i_estimateRhoComparison} finally $\trhove_{\wp}\in\cC^\infty(\bT^1,\bR)$ is
    a positive function that is close to $\rho_{\ell_k}$ in the sense
    \begin{align}\label{e_estimateRhoComparison} \|\trhove_{\wp} -
      \rrho_{\ell_k}\|\nl1&=\cO(\ve\min\{\deltacomplex\invr,L^2\} +\wp),
    \end{align} and such that, for any $r\in\bN$, $\|\trhove_{\wp}\|\nw{r}1 \le
    \Const\wp^{-r+1}\deltacomplex^{-r}$.
  \end{enumerate}
\end{prop}
\begin{proof}
  Recall (see Remark~\ref{r_randomVariables}) that, for a given $\ell = (\bG,\rho)$ and
  for any $n$, we have $\coord n = F_{\ve}^n(\bG(x))$; in other words, we consider $x_n$
  and $\theta_n$ to be random variables on $\ell$. %
  In particular, we have, for any smooth test function $g:\bT\to\bR$:
  \begin{align*}
    \sum_{\ell_{k+1}\in\stdf^{k}_{\ell_{k}}} \Leb\left( g\cdot
        \fm_{k,\ell_k,\ell_{k+1}}e^{\Phi\circ G_{\ell_{k+1}}}\rrho_{\ell_{k+1}} \right)
      = &\sum_{\ell_{k+1}\in\stdf^{k}_{\ell_{k}}}
      \fm_{k,\ell_k,\ell_{k+1}}\mu_{\ell_{k+1}}\left(g(x_0)e^{\Phi(\theta_0)}\right).
  \end{align*}
  Using~\eqref{eq:corr-T-op} and the definition of $\fm_{k,\ell_k,\ell_{k+1}}$
  (see Lemma~\ref{lem:one-step-M}), we gather:
  \begin{equation}\label{eq:step-one-longp}
    \begin{split}
      & \sum_{\ell_{k+1}\in\stdf^{k}_{\ell_{k}}} \Leb\left( g\cdot
        \fm_{k,\ell_k,\ell_{k+1}}e^{\Phi\circ G_{\ell_{k+1}}}\rrho_{\ell_{k+1}} \right)\\
      = & \int_\bT [g e^{\Phi}]\circ
      F_\ve^{L_k}\circ\bG_{\ell_k}\ e^{i\sigma\left[\ve\fkC^{k,1}_{\ell_k,0,\emptyset}+
          \sum_{j=0}^{L_k-1}\hov^{k}_{\ell_k,j}\circ F_\ve^j\right]}\rrho_{\ell_k} \\
      &+\sum_{s=1}^{q-1}(i\sigma\ve)^s\int_{\bT} [g e^{\Phi}]\circ
      F_\ve^{L_k}\circ\bG_{\ell_k}\cdot e^{i\sigma\left[\ve\fkC^{k,1}_{\ell_k,0,\emptyset}+
          \sum_{j=0}^{L_k-1}\hov^{k}_{\ell_k,j}\circ
          F_\ve^j\right]}\bbK^{k,2s}_{\ell_k,3s}\rrho_{\ell_k}.
    \end{split}
  \end{equation}
  In the following we will find convenient to use $\bar x=\diffeoAvg_{\ell_k, L_k}(x)$,
  rather than $x$, as our fundamental random variable.\footnote{ In the rest of the proof
    we will often suppress the subscripts $ \ell_k, L_k$ in $\diffeoAvg_{\ell_k, L_k}$,
    and related quantities, when this does not create any confusions.}  This can be done
  using Lemma~\ref{l_shadowAveraged}: indeed, the change of variable formula yields that
  the pushforward of the density is given by
  \begin{equation}\label{eq:push-f}
    \diffeoAvg_*\rrho_{\ell_k} ={(\diffeoAvg\invr)}'\cdot\rrho_{\ell_k}\circ\diffeoAvg\invr.
  \end{equation}
  For any smooth function $\vf$ of the random variables $\{(x_i,\theta_i)\}_{i = 0}^{L_k-1}$, under
  $\mu_{\ell_k}$, we can write $\tilde\vf(x)=\vf(\{F_\ve^i\circ \bG_{\ell_k}(x)\})$, where
  $x$ is distributed according to $\rrho_{\ell_{k}}$.  Then our change of variable
  corresponds to looking at the random variable
  $\bar\vf(\bar x)=\vf(\{F_\ve^i\circ \bG_{\ell_k}\circ\diffeoAvg^{-1}_{\ell_k, L_k}(\bar
  x)\})$ under $\diffeoAvg_*\rrho_{\ell_k}$.  In particular,
  \begin{align*}
    \|\bar\vf'\|_\infty\leq \Const \|\vf\|\nc1\BLambda_{0,L_k}^{-1}.
  \end{align*}
  The above considerations would suffice to treat the small terms
  in~\eqref{eq:step-one-longp}, but, unfortunately, are not adequate to treat the main
  term since we only have an exponentially large bound on the derivative of
  $\diffeoAvg_*\rho_{\ell_k}$ (see the last of~\eqref{eq:up-c1-bound}) which would create
  serious problems in our subsequent arguments, unless they can be discarded by some a
  priori estimate.  In order to deal with this problem, we first need to introduce some
  notation.  Let
  \begin{equation}\label{eq:trho-def}
    \trho=\frac{\rrho_{\ell_k}}{[1-G'\stable_L]}\circ\diffeoAvg\invr;
  \end{equation}
  observe that Sub-Lemma~\ref{lem:Y} implies that
  $\|\trho\|_{\BV}\le\Const\deltacomplex\invr$.  We can now state a more
  useful bound for~\eqref{eq:push-f} whose proof is, for convenience, postponed to
  the end to this section.
  \begin{lem}\label{eq:Upsilon-appr}
    The following formula holds true
    \begin{align*}
      \diffeoAvg_*\rrho_{\ell_k}=\rho_*\expo{\ve \cK_0}+\oerR{\ell_k}^{2,3}\rrho_{\ell_k}.
    \end{align*}
  \end{lem}
  Next, we proceed to eliminate the explicit dependence on $x_j$ and $\theta_{j}$: first
  observe that, by definition, for any smooth function $A(x,\theta)$.  Observe that we
  have $\bdeviationslk{\ell_k}{j},\bar\xi_j = \oerR{\ell_k}^{1,1}$; hence we can write
    \begin{align*}
      [A (x_j,\theta_j)]\circ \diffeoAvg\invr
      &= A \bcoordk{k} j+\partial_x  A \bcoordk{k}
        j\bar\xi_j+\partial_\theta A\bcoordk{k} j\bdeviationslk{\ell_k}{j} \\
      &\phantom{=}+\frac12\left[\partial_{\theta\theta}A
        (\bdeviationslk{\ell_k}{j})^2+\partial_{\theta x} A\,\bdeviationslk{\ell_k}{j}\bar\xi_j
        +\partial_{xx} A\,\bar\xi_j^2\right]
        +\oerR{\ell_k}^{3,3}.
    \end{align*}
 In particular, using~\eqref{e_bbbK-taylor-expansion} we gather
 \begin{align*}
   \sum_{j=0}^{L_k-1}\hov^k_{\ell_k,j}(x_j, \theta_j)
   &=\sum_{j=0}^{L_k-1}\hov^k_{\ell_k,j}(\bar x_j, \thetaslkk{k}{j})+%
     \ve\bbbK^{k,2}_{\ell_k,3}+\ve^2\bbbK^{k,3}_{\ell_k,6}+\oerR{\ell_k}^{3,4};\\
   (i\sigma\ve)^s\bbK^{k,2s}_{\ell_k,3s} 
   &= (i\sigma\ve)^s\sum_{\bar\imath}\fkC_{\ell_k,3s,\bar\imath}^{k,2s}\prod_{j=1}^{3s}\left[A_{j,\bar\imath}(\bar x_{i_j},\thetaslkk{k}{i_j})+\ve\bbbK^{k,1}_{\ell_k,3}+\oerR{\ell_k}^{2,2}\right]\\
   &=(i\sigma\ve)^s\bbK^{k,2s}_{\ell_i,3s,*}+(i\sigma\ve)^s\ve\bbbK^{k,2s+1}_{\ell_k,3(s+1)}+\sigma^s\oerR{\ell_k}^{s+2,2s+2}.
    \end{align*}
    The above will suffice to estimate the error terms.  However, to deal with
    $\Phi(\theta_{L_k})$ we will need a more explicit formula.  By
    definition~\eqref{e_deviationeDef} we have
    $\theta_{L_k} = \bar\theta_{\ell_k,L_{k}} + \deviatione_{L_k}$;
    by~\eqref{eq:aboutthis}, and using~\eqref{se_definitionDD} we conclude that
    \begin{align}\label{eq:theta-bartheta}
      \theta_{L_k}
      &= \bar\theta_{\ell_k,L_{k}} + \ve \sum_{j=0}^{L_k-1}  \etaExps_{j,k}\ho(\bar
        x_l\circ\diffeoAvg,\thetaslk{l}) + \erR{\ell}^{2,2}
    \end{align}
    We can now collect all the above relations to write~\eqref{eq:step-one-longp} in terms
    of the slowly varying dynamics
    \begin{equation}
       \begin{split}
       & \sum_{\ell_{k+1}\in\stdf^{k}_{\ell_{k}}}  \int_{\bT}  g \cdot
        \fm_{k,\ell_k,\ell_{k+1}}e^{\Phi\circ G_{\ell_{k+1}}}\rrho_{\ell_{k+1}}= \\
        &= \int_{\bT} \hskip-4pt g(x_{L_k})
        e^{i\sigma\left[\ve\fkC^{k,1}_{\ell_k,0,\emptyset}+
            \sum_{j=0}^{L_k-1}\hov^{k}_{\ell_k,j}\circ
            F_\ve^j\right]+\Phi(\bar\theta_{\ell_k,L_{k}})+\ve
          \Phi'(\thetaslkk{k}{L_{k}})\sum_{j = 0}^{L_k-1}
          \etaExps_{j,k}\ho\bcoord l}\rrho_{\ell_k} \hskip-1.5cm\\
        &\phantom{=}+ \|\Phi'\|\nc1e^{\Phi^+}\int_{\bT}  g(x_{L_k})  \erR{\ell_k}^{2,2}\rrho_{\ell_k}\\
        &\phantom{=}+ \sum_{s=1}^{q-1}i^s\sigma^s\ve^s\int_{\bT} [g e^{\Phi}]\circ F_\ve^{L_k}\circ \bG_{\ell_k}\cdot e^{i\sigma\left[\ve\fkC^{k,1}_{\ell_k,0,\emptyset}+ \sum_{j=0}^{L_k-1}\hov^{k}_{\ell_k,j}\circ F_\ve^j\right]}\bbK^{k,2s}_{\ell_k,3s}\rrho_{\ell_k}\\
        &= \int_{\bT} g\circ \Favg^{L_k}\cdot
        e^{i\sigma\ve\fkC^{k,1}_{\ell_k,0,\emptyset}+
          \sum_{j=0}^{L_k-1}\Omega^{k,\Phi}_{\ell_k,j}\circ
          \Favg^{j}}\Psi_{\ell_k,q}\,e^{\Phi(\bar\theta_{\ell_k,L_{k}})}\rho_*(\bar x)
        \deh
        \bar x\\
        &\phantom{=}+e^{\Phi^+}\int_{\bT} g\circ \Favg^{L_k}\cdot e^{
          \sum_{j=0}^{L_k-1}\Omega^{k,\Phi}_{\ell_k,j}\circ
          \Favg^{j}}\left[\|\Phi'\|\nc1\oerR{\ell_k}^{2,2}+ \oerR{\ell_k}^{2,3}\right]
        \diffeoAvg_*\rho_{\ell_k}\deh \bar x,
      \end{split}
    \end{equation}
    where we have used the fact that, by hypothesis,
    $\ve^2L_k^3\leq\min\{ |\sigma|\ve^3L_k^4, \sigma^2\ve^4 L_k^4\}$ and by definition
    $\erR{\ell}^{p,q}\circ\diffeoAvg\invr = \oerR{\ell}^{p,q}$.  Since the above equation
    holds for all $g$, we have
  \begin{align}\label{eq:step-two-longs}
    \sum_{\ell_{k+1}\in\stdf^{k}_{\ell_{k}}}\fm_{k,\ell_k,\ell_{k+1}}&e^{\Phi\circ G_{\ell_{k+1}}}\rrho_{\ell_{k+1}} =
    e^{i\sigma\ve\fkC^{k,1}_{\ell_k,0,\emptyset}}\tO_{\ell_k,k, L-1}\cdots\tO_{\ell_k,k,0}{\Psi}_{\ell_k,q}\,e^{\Phi(\bar\theta_{\ell_k,L_{k}})}\rho_*\notag\\
    &+e^{\Phi^+} \tO_{\ell_k,k, L-1}\cdots\tO_{\ell_k,k,0}\left[\|\Phi'\|\nc1\oerR{\ell_k}^{2,2}+ \oerR{\ell_k}^{2,3}\right]\diffeoAvg_*\rho_{\ell_k},
  \end{align} where we used the fact that, by definition,
  \begin{align*}
    \int_{\bT}g\circ\Favg^{L_k}\cdot  e^{\sum_{j=0}^{L_k-1}\Omega^{k,\Phi}_{\ell_k,j}\circ \Favg^j}\vf(x)\deh x%
=\int_{\bT} g \tO_{\ell_k,k, L-1}\tO_{\ell_k,k,L-2}\cdots\tO_{\ell_k,k,0}\vf.
  \end{align*}
  In the sequel  we will need to deal with smooth density functions. We can obtain this by
  a mollification procedure; (see, e.g.\ \cite[Lemma B.1]{liverani13}): for each $\wp>0$ there exists a $\trhove_{\wp}$
  such that
  \begin{equation}\label{eq:rho-smooth}
    \|\trhove_{\wp} - \trho\|\nl1\le \Const\wp\quad\textrm{ and }\quad
    \|\trhove_{\wp}\|_{W^{r,1}}\le\Const \wp^{-r+1}\deltacomplex^{-r}.
  \end{equation}
  Note that $\|\tO_{\ell,k,j}\|\nl1\leq e^{\const\ve \|\Phi'\|_\infty}$. Moreover, by iterating~\eqref{eq:cone-est}, we have, for each $\vf\in W^{1,1}$,
  \begin{equation}\label{eq:W11product}
    \left\|\frac{d}{dx}\tO_{\ell_k,k, L_k-1}\cdots\tO_{\ell_k,k,0}\vf\right\|_{L^1}\leq\Const\left[\|\BLambda_{0,L_k}\vf'\|_{L^1}+(1+ |\sigma|)\|\vf\|_{L^1}\right].
  \end{equation}
  We also have
  \begin{equation}\label{eq:rho-smooth2}
    \begin{split}
      &\|\tO_{\ell_k,k, L_k-1}\cdots\tO_{\ell_k,k,0}\Psi_{\ell_k,q}e^{\Phi(\bar\theta_{\ell_k,L_{k}})}[\trhove_{\wp} - \trho]\|\nl1\le \Const\wp\\
      & \|\tO_{\ell_k,k,
        L_k-1}\cdots\tO_{\ell_k,k,0}\Psi_{\ell_k,q}e^{\Phi(\bar\theta_{\ell_k,L_{k}})}\trhove_{\wp}\|_{W^{1,1}}\le\Const
      (1+|\sigma|)
    \end{split}
  \end{equation}
  By the lower semicontinuity of the variation \cite[Section 5.2.1, Theorem
  1]{EvansGariepy}, since $\trhove_{\wp}\to\rho_*$ in $L^1$ as $\wp\to0$, we have
  \begin{align*}
    \|\tO_{\ell_k,k, L_k-1}&\cdots\tO_{\ell_k,k,0}\Psi_{\ell_k,q}e^{\Phi(\bar\theta_{\ell_k,L_{k}})}[\trhove_{\wp} - \trho]\|_{\BV}\\&=\lim_{\wp'\to 0}\|\tO_{\ell_k,k, L_k-1}\cdots\tO_{\ell_k,k,0}\Psi_{\ell_k,q}e^{\Phi(\bar\theta_{\ell_k,L_{k}})}[\trhove_{\wp} - \trhove_{\wp'}]\|_{W^{1,1}} \leq \Const (1+|\sigma|)
  \end{align*}
  By a similar argument, estimating the remainder terms of~\eqref{eq:step-two-longs},
  follows part (c) of the proposition.  Finally, to prove (d), recall that
  $|G'\stable_L|\leq \Const\ve$. Then, by~\eqref{eq:trho-def} and~\eqref{eq:rho-smooth}:
  \begin{align*}
    \|\trhove_{\wp}-\rrho_{\ell_k}\|\nl1
    &\le \|\rrho_{\ell_k}-\rrho_{\ell_k}\circ \diffeoAvg\invr\|\nl1 + \Const(\wp+\ve )
    \le \Const (\ve\min\{L^2,\deltacomplex\invr\}+\wp)
  \end{align*}
  where in the last step we used the first bound in~\eqref{eq:up-c1-bound}.
\end{proof}
We conclude with the missing proof.
{\begin{proof}[{\bf Proof of Lemma~\ref{eq:Upsilon-appr}}]
    \renewcommand{\ooW}{\overline W} \renewcommand{\oW}{W}
    \renewcommand{\ofkW}{\overline{\mathfrak{W}}}\renewcommand{\fkW}{\mathfrak{W}}
    \renewcommand{\etaExps}{\Xi^*} By Lemma~\ref{l_shadowAveraged},
    equations~\eqref{eq:push-f},~\eqref{eq:trho-def},~\eqref{eq:upsilonp} we gather
    \begin{equation}\label{eq:y*rho}
      \begin{split}
        \diffeoAvg_*\rho_{\ell_k}&=\trho\expo{\sum_{j=0}^{L-1}\log  \left[\frac{\partial_xf(x_j,\theta_j)}{\partial_x f(\bar x_j, \thetaslkk{k}{j})}+\ve\frac{ \partial_\theta f(x_j,\theta_j)}{ \partial_xf(x_j,\theta_j)}u_j\right]+\cO(\ve^2 L)}\\
        &=\trho\expo{\ve\sum_{j=0}^{L-1} \frac{\partial_{\theta x}f(\bar x_j,\thetaslkk{k}{j})\ooW_j+\partial_{x x}f(\bar x_j,\thetaslkk{k}{j})\ofkW_j+ \partial_\theta f(\bar x_j,\thetaslkk{k}{j}) u_j}{\partial_x f(\bar x_j, \thetaslkk{k}{j})}}\\
        &\phantom{=\ }+\oerR{\ell_k}^{2,3}\mathring\rho_{\ell_k},
      \end{split}
    \end{equation}
    where the second needed property of $\oerR{\ell_k}^{2,3}$ follows immediately
    from~\eqref{eq:up-c1-bound}{}.  Note that the exponent
    in~\eqref{eq:y*rho} is, at most, of size $\Const \ve L^2\leq \Const \ve^{6\delta_*}$
    (recall~\eqref{eq:L-choice}), and it is of correlation type.  It is then natural to
    expand the exponential in Taylor series and to use Notation~\ref{not:bookcor1}.  We can
    then write
    \begin{align*}
      \sum_{j=0}^{L-1} \frac{\partial_{\theta x}f(\bar x_j,\thetaslkk{k}{j})}{\partial_x f(\bar x_j, \thetaslkk{k}{j})}
      \ooW_j=\bbbK^{k,2}_{\ell_k,2}.
    \end{align*}
    In order to show that also the second term in the exponent can be treated as a
    correlation term, let us set for convenience
    $\alpha(\bar x,\bar\theta)=\frac{\partial_{x x}f(\bar x,\bar\theta)} {\partial_x
      f(\bar x,\bar\theta)}$, then
    \begin{align*}
      \sum_{j=0}^{L-1} \frac{\partial_{x x}f}{\partial_x f}\circ \Favg^j\cdot \ofkW_j%
      &= -\sum_{j,s=0}^{L-1} \alpha\circ \Favg^j\cdot \ho\circ \Favg^s%
        \sum_{l=\max\{j,s+1\}}^{L-1}\etaExps_{s,l}\cdot\Lambda_{0,l-j}\partial_\theta f\circ
        \Favg^{l}\\ %
      &=-\sum_{j\geq s+1}^{L-1}A_{1,j,s}\circ \Favg^s\cdot \alpha\circ \Favg^j-\sum_{j< s+1}^{L-1}\ho\circ \Favg^sA_{2,j,s}\circ \Favg^j
    \end{align*}
    where
    \begin{align*}
      &A_{1,j,s}=\etaExps_{0,l-s}\ho \sum_{l=j}^{L-1}\Lambda_{0,l-j}\partial_\theta f\circ \Favg^{l-s}\\
      &A_{2,j,s}=\alpha\sum_{l=s+1}^{L-1}\etaExps_{s-j,l-j}\Lambda_{0,l-j}\partial_\theta
        f\circ \Favg^{l-j}
    \end{align*}
    (recall that we dropped the subscript $\ell$ from $\etaExps$).  A direct computation
    shows that $\sup_{\bar\theta}\|A_{i,j,s}(\cdot,\bar\theta)\|\nc1\leq \Const$.  In
    order to deal with the third term in the exponential, we need to define the auxiliary
    variables
   \begin{align*}
     \tilde u_{k+1}=\frac{\partial_x\omega(x_k,\theta_k)+\tilde u_k}{\partial_x
     f(x_k,\theta_k)},\quad \ve\tilde u_0=G_\ell'(x).
   \end{align*}
   By~\eqref{eq:recap-u-cone} it is immediate to observe that
   \begin{align*}
     u_{k+1}-\tilde u_{k+1}=\frac{u_k-\tilde u_k}{\partial_x f(x_k,\theta_k)+\ve\partial_\theta f(x_k,\theta_k)u_k}+\cO(\ve)
   \end{align*}
   from which it follows $u_k=\tilde u_k+\cO(\ve)$.  We can thus replace $u_k$ with
   $\tilde u_k$ in the third term in the exponential and computations similar to the
   previous ones yield that also the third term in the exponential can be
   interpreted as a correlation term.  Recalling~\eqref{eq:trho-def}, the above discussion
   implies that we can write
    \begin{equation}\label{eq:measure-corr}
      \diffeoAvg_*\rho_{\ell_k}=\rho_*\expo{\ve \bbbK^{k,2}_{\ell_k,2}}+\oerR{\ell_k}^{2,3}\mathring\rho_{\ell_k}.
    \end{equation}
    Finally, we claim that the term $\bbbK^{k,2}_{\ell_k,2}$ can be written as $\cK_0$.  In
    order to see this, it suffices to integrate the above relation to obtain
    \begin{equation}\label{eq:zero-refine}
      1=\Leb(\mathring \rho_{\ell_k})=  \Leb\left[e^{ \ve \bbbK^{k,2}_{\ell_k,2}}\,\rho_*\right]+\cO(\ve^2L_k^3),
    \end{equation}
    which implies our requirement by taking into account~\eqref{eq:rho-smooth}.
  \end{proof}
}
\section{One block estimate: the large \texorpdfstring{$\sigma$}{σ} regime}\label{sec:one-large}

In the large $\sigma$ regime it suffices to estimate the contribution of the last block.
To this end we first need an estimate on the product of the transfer operators defined
in~\eqref{e_defTransferOp}.  To ease notation, in this section we will omit the indices
$\ell_R,\ell_{R-1}$ and $R-1$, referring to the last block, as no confusion can arise: in
particular $\tO_{j}$ will stand for $\tO_{\ell_{R-1},R-1,j}$ and $L$ will stand for
$L_{R-1}$.  Also, the transfer operators are defined with respect to the purely imaginary
potentials $i\sigma\hov^k_{\ell,j}$, where $\hov^k_{\ell,j}$ is defined
in~\eqref{eq:potdef}, i.e. we have $\Phi\equiv0$ in~\eqref{e_defTransferOp}.

\begin{lem}
  There exists $\Ccinque>0$ and $\tau_1\in (0,1)$ such that, for all
  $n\in [\Ccinque\log\vei,L-1]$ and $j\in[0,L-1-n]$, any $g\in\cC^1(\bT)$ we have
  \begin{equation}\label{eq:sharp-dolgo}
    \|\tO_{j+n}\cdots\tO_{j} g\|\nc1 \le \tau_1^n\|g\|\nc1.
  \end{equation}
\end{lem}
\begin{proof}
  We begin with a preliminary estimate on $\|\tO_{j+n}\cdots\tO_{j} g\|\nc1$; as already
  noticed, the potentials are purely imaginary, thus for any $0\le n\leq L-1$ and
  $0\le i\le L-1$ we have\footnote{ This follows since $|\tO_ig|\le|\tO_{i,\sigma = 0}g|$,
    therefore $\|\tO_i^n g\|\nc0\le\|\tO_i^n 1\|\nc0\|g\|\nc0\le\Const \|g\|\nc0$.}
\begin{align}\label{e_boundInC0}
  \|\tO_i^n\|_{\cC^0\to\cC^0} \leq \Const.
\end{align}
Observe that by the same token
\begin{align}\label{e_iteratedBoundInC0}
  \|\tO_{i+n}\cdots\tO_{i}\|_{\cC^0\to\cC^0} \leq \Const.
\end{align}
Using the Lasota-Yorke inequality~\eqref{eq:cone-est} we gather
\begin{align}
  \notag  \|\tO_{i+n}\cdots\tO_i g\|\nc1
  &\le \Const\lambda^{-n}\|g\|\nc1+\Const\sum_{k=0}^{n}\lambda^{-k}(1+|\sigma|)\|\tO_{i+n-k-1}\cdots\tO_ig\|\nc0\\
  &\le \Const(1+|\sigma|)\|g\|\nc1\label{e_iteratedBoundInC1}
\end{align}
We continue with an estimate of $\left\|\tO_{L-1}\cdots\tO_{0}\right\|_{\cC^1\to\cC^0}$.
Since $\ho$ is not a coboundary, $\hov_j$ is not a coboundary, and the potentials
$(\hov_j)$ satisfy UUNI (see Corollary~\ref{cor:uuni}).  We can thus apply
Theorem~\ref{lem:dolgo-BV}, that implies that there exists $\tau\in (0,1)$ such that, for
any $i\in\{0,\cdots, L-1\}$,
\begin{equation}\label{eq:rough-c1}
  \|\tO_i^n\|_{\cC^1\to\cC^1}\le \begin{cases} \Const (1+|\sigma|)&\textrm{ for }  n <  n_\ve=\pint{ \Ccinque\log\vei}\\
    \tau^n&\textrm{ for } n\ge n_\ve.
  \end{cases}
\end{equation}
Note that we can choose $\Ccinque$ as large as needed.  Also we have the following trivial
estimate for the difference of operators with potentials $\Omega_i$:
\begin{align*}
  \|\tO_{\theta,\Omega_1}-\tO_{\theta,\Omega_2}\|\nc0\leq \Const \|\Omega_1-\Omega_2\|\nc0\|\tO_{\theta,\Omega_1}\|\nc0,
\end{align*}
and, by~\eqref{eq:basic-taylor}, we have, for all $g\in\cC^1$,
\begin{align*}
\begin{split}
  \|\tO_{\theta_1,\Omega}g-\tO_{\theta_2,\Omega}g\|
  &\le \Const |\theta_1-\theta_2|\sup_{\theta\in[\theta_1,\theta_2]}
  \left[\|\tO_{\theta,0}|g'|\|\nc0+\|\tO_{\theta,0}(1+\|\Omega\|\nc1)|g|\|\nc0\right]\\
  &\leq \Const |\theta_1-\theta_2|\left[\|g\|\nc1+(1+\|\Omega\|\nc1)\|g\|\nc0\right].
  \end{split}
\end{align*}
Accordingly, using the explicit formula \eqref{e_defTransferOp} we have, for each $0\le i < k <  L$,
\begin{equation}\label{eq:sharp-Q-diff}
  \|(\tO_k-\tO_i)g\|\nc0\le \Const\ve (k-i)\left[\|g\|\nc1+(1+|\sigma|)\|g\|\nc0\right].
\end{equation}
Observe moreover that we can write
\begin{align*}
  \tO_{i+n}\cdots\tO_i
  &= \tO_i^{n+1} + \sum_{k = 1}^n\tO_{i+n}\cdots\tO_{i+k+1}(\tO_{i+k}-\tO_i)\tO_i^{k}.
\end{align*}
Thus, for $n\in[n_\ve, 3n_\ve]$ and $i\in\{0,\cdots, L-n-1\}$, we can
use~\eqref{e_iteratedBoundInC0}, \eqref{eq:sharp-Q-diff} and \eqref{eq:rough-c1} to write
\begin{align*}
    \|\tO_{i+n}\cdots\tO_i g\|\nc0
    &\le \sum_{k=0}^{n-1}\|\tO_{i+n}\cdots\tO_{i+k+2}(\tO_{i+k+1}-\tO_i)\tO_i^{k+1} g\|\nc0+\tau^{n+1}\|g\|\nc1\\
    &\le \left(\Const(1+|\sigma|)\ve n_\ve^2+\tau^{n_\ve}\right)\|g\|\nc1\leq \Const
    |\sigma| \ve n_\ve^2\|g\|\nc1,
\end{align*}
provided $\Ccinque$ in the definition of $n_\ve$ has been chosen large enough and since
$\sigma\ge\sigma_0$.

Note that, for $|\sigma|\ve n_\ve^3<1$, we can bootstrap the above estimate by writing, for  $n\in[3n_\ve, 4n_\ve]$,
\begin{align*}
   & \|\tO_{i+n}\cdots\tO_i g\|\nc0 \le \sum_{k=0}^{n_\ve}\|\tO_{i+n}\cdots\tO_{i+k+2}(\tO_{i+k+1}-\tO_i)\tO_i^{k+1} g\|\nc0+2\tau^{n_\ve}\|g\|\nc1\\
    &\le \sum_{k=0}^{n_\ve}\Const|\sigma| \ve n_\ve^2\|\tO_{i+n-n_\ve-1}\cdots\tO_{i+k+2}(\tO_{i+k+1}-\tO_i)\tO_i^{k+1} g\|\nc1+2\tau^{n_\ve}\|g\|\nc1\\
    &\le \Const\sum_{k=0}^{n_\ve}(\ve|\sigma|n_\ve^2\left[\lambda^{-n_\ve}|\sigma|+\Const \ve|\sigma|n_\ve^2\right]+\tau^{n_\ve})\|g\|\nc1
    \le \Const \ve^2|\sigma|^2n_\ve^5\|g\|\nc1
\end{align*}
 where we have chosen, again, $\Ccinque$ large enough and, in the last line, we have used
the Lasota--Yorke inequality~\eqref{eq:cone-est}.

Finally, note that, by using the Lasota--Yorke inequality again, it follows, for all $n\in [3n_\ve,4n_\ve]$,
\[
\begin{split}
  \|\tO_{i+n}\cdots\tO_i g\|\nc1 &\le \Const\lambda^{-n}\|g\|\nc1+\Const\sum_{k=0}^{n-1}\lambda^{-k}|\sigma|\|\tO_{i+n-k-1}\cdots\tO_ig\|\nc0\\
&\le \Const \left[\lambda^{-n_\ve}|\sigma|+\ve^2|\sigma|^3n_\ve^6\right]\|g\|\nc1\\
&\le \tau_1^n\|g\|\nc1
\end{split}
\]
for some $\tau_1\in(\tau,1)$, provided, again, $\Ccinque$ has been chosen large enough and
since $\ve^2|\sigma|^3n_\ve^6\leq \ve^{\efrac 1{4}}$ for $\ve$ small enough.
\end{proof}

We are now able to provide the proof of the main result of this section.
\begin{proof}[{\bf Proof of Proposition~{\ref{prop:not-so-trivial-0}}}]
  Observe that, by definition and by Proposition~\ref{p_transferOperators}, with
  $\Phi\equiv 0$, we have\footnote{ Recall that we are suppressing the subscripts
    $\ell_R,\ell_{R-1}, R-1, R$, when this does not create confusion.}
  \begin{align*}
    \sum_{\ell_R\in\stdf^{R-1}_{\ell_{R-1}}}\fm_{R-1,\ell_{R-1},\ell_R} 
 &= \Leb
   \sum_{\ell_R\in\stdf^{R-1}_{\ell_{R-1}}}\fm_{R-1,\ell_{R-1},\ell_R}\mathring\rho_{\ell_R}
    \\ &=\Leb\, \tO_{L-1}\cdots\tO_{0}\left[ \Psi_{q}\trhove_{\wp}\right] +\Leb\,\cE^*.
  \end{align*}
  Recall that $|\sigma|\leq \ve^{-\frac 12-2\delta_*}$ and $\delta_*\in(1/99,1/32)$.  By
  Proposition~\ref{p_transferOperators}\ref{i_propertyDeltaerr}, we have
  \begin{align*}
    |\Leb\,\cE^*|\leq\|\cE^*\|_{L^1}\leq \Const\{\ve^{2-9\delta_*}+\wp\}.
  \end{align*}
  Next, note that for each $r<L/(3n_\ve)-2$,
  by~\eqref{eq:sharp-dolgo},~\eqref{e_iteratedBoundInC0} and~\eqref{e_iteratedBoundInC1}
  we obtain
  \begin{align*}
    \left\|    \tO_{L-1}\cdots\tO_{i_r} A_{r,i_r}\cdots\tO_{i_1} A_{1,i_1}\cdots\tO_{0}
    \trhove_{\ell,\wp}\right\|\nc0\leq \ConstPow{r} \tau_1^{\efrac{L}{r+1}}\wp^{-1}\deltacomplex^{-2}
    \ve^{-\efrac r2 -2r\delta_*}
  \end{align*}
  since at least one string of operators must be longer than $L/(r+1)\geq 3n_\ve$.  This
  allows to estimate the contribution of $\Psi_{\ell_{R-1},q}$ by expanding is series the
  exponential. We thus obtain
  \begin{align*}
    \left|\sum_{\tilde\ell\in\stdf^L_{\ell,\potSequence}
    }\fm_{\tilde\ell}\rho_{\tilde\ell} \Id_{[a_{\tilde\ell}, b_{\tilde\ell}]}\right|
    &= \cO(\ve^{2-9\delta_*}+\wp+\vei \tau_1^{L/(\const q)}\wp\invr\deltacomplex^{-2}\ve^{-\efrac q2 -2q\delta_*}).
  \end{align*}
  Thus, the proposition follows by choosing $\wp=\ve^{2-9\delta_*}$.
\end{proof}
\section{One block estimate: the intermediate \texorpdfstring{$\sigma$}{η} regime}\label{sec:one-small}
The following two cases require a much more accurate description of the one-block
contribution, which can only be obtained for small $\sigma$. It will be achieved thanks to
the technical lemmata contained in this section.

The argument is similar to the one of the previous section, only a different idea is
needed to compute the norms of the relevant operators: provided $\sigma_0$ is small
enough, such norms can be computed via perturbation theory.

\subsection{The Transfer operators product formula}\ \newline\label{subsec:prod-tram} Our
task here is to study the transfer operators defined in~\eqref{e_defTransferOp} and then
their products in the perturbative regime.
\begin{lem}\label{lem:Taylor}
  Let $\sigma_0$ be chosen small enough.  For any $\sigma\in\J1\cup\J2$ and
  $k\in\{0,\cdots,R-1\}$, $j\in\{0,\cdots,L_k-1\}$ and $\Phi$ satisfying the hypotheses of
  Proposition~\ref{p_transferOperators} and, additionally, so that
  $\ve\|\Phi'\|\nc1\sup_{j\leq L_k}\|\etaExps_{j}\ho\|\nc1\leq \sigma_0$ we have:
  \begin{enumerate}
  \item \label{i_spectralGap} $\tO_{\ell_k,k,j}$ is of Perron--Frobenius type, i.e.\ we can
    write $\tO_{\ell_k,k,j}=e^{\chi_{\ell_k,k,j}}\tOp_{\ell_k,k,j}+\tOq_{\ell_k,k,j}$,
    where $e^{\chi_{\ell_k,k,j}}$ is the maximal eigenvalue of $\tO_{\ell_k,k,j}$ (as an
    operator acting on $\cC^1, W^{1,1}$ or $\BV$), $\tOp_{\ell_k,k,j}, \tOq_{\ell_k,k,j}$
    are such that $\tOp_{\ell_k,k,j}^2=\tOp_{\ell_k,k,j}$,
    $\tOp_{\ell_k,k,j}\tOq_{\ell_k,k,j}=\tOq_{\ell_k,k,j}\tOp_{\ell_k,k,j}=0$, the
    operators $\tOp_{\ell_k,k,j}$ are rank one, and there exists $\tau\in(0,1)$ so that
    \begin{equation}\label{e_spectralGap}
      \|\tOq_{\ell_k,k,j}^n\|_{\cC^1,W^{1,1}, \BV}\le \Const \tau^n|e^{n\chi_{\ell_k,k,j}}|\;;\quad \|\tOp_{\ell_k,k,j}\|_{\cC^1,W^{1,1}, \BV}\le \Const.
    \end{equation}
  \item
    \label{i_chiEstimate} 
    If, additionally, $\Phi$ satisfies
    $\ve\|\Phi'\|\nc0\leq \Const \ve\sigma^2 (R-1-k)L_*$, then there exists a twice
    differentiable function $\bchi(\sigma,T,s,\vf, \theta)$, smooth in $\sigma$, with
    derivatives with respect to $\vf,\theta, s$ uniformly bounded by $\Const |\sigma|$,
    such that
    \begin{subequations}\label{e_chiEstimates}
      \begin{align}
        &\chi_{\ell_k,k,j}=\bchi(\sigma,t-\ve S_{k-1},\ve
        j,\theta^*_{\ell_k},\thetaslkk{k}{j})
               +\cO\left(\sigma^2\ve^2 L_*+\sigma^3\ve(R-1-k)L_*\right)\\
        &\bchi(\sigma, T,s,\vf,\theta)=- \frac{\sigma^2}2\etaExph(T- s-\ve,\bar\theta(s+\ve,\vf))^2\bVar(\theta)+ \cO(\sigma^3)\notag\\
         &\phantom{\bchi(\sigma, T,s,\vf,\theta)=}
           \le-\frac{\sigma^2}4\etaExph(T- s-\ve,\bar\theta(s+\ve,\vf))^2\bVar(\theta)\label{e_chiEstimates-bound-from-above}
      \end{align}
    \end{subequations}
    where $\theta^*_{\ell_k},\thetaslkk{k}{j}$ are defined in~\eqref{eq:ifnotdefinedyet},
    $\bVar\in\cC^1(\bT,\bRp)$ is given by the Green--Kubo
    formula~\eqref{e_definitionBarChi} and $\etaExph$ is defined in~\eqref{eq:Deltabloks}.
  \end{enumerate}
\end{lem}
\begin{proof}
  We will use indifferently the notation introduced in
  Proposition~\ref{p_transferOperators} and the one used in
  Appendix~\ref{subsec:transfer}. Such notations are connected by the relation
  $\tO_{\ell_k,k,j}=\tO_{\thetaslkk{k}{j},\Omega^{k,\Phi}_{\ell_k, j}}$, where
  $\Omega^{k,\Phi}_{\ell_k, j}$ is defined in~\eqref{e_defTransferOp}.  In order to apply
  the results of Appendix~\ref{subsec:transfer}, let us consider the transfer operator
  given by $\tO_{\thetaslk{j},\paramL \Omega^{k,\Phi}_{\ell, j}}$, for $\paramL\in[0,1]$.
  Since the operator, for $\paramL=0$, has $1$ as a simple maximal eigenvalue and a
  spectral gap (in any of the above mentioned spaces), it follows that we can choose
  $\sigma_0$ such that, for any $\sigma\in [0,\sigma_0]$, the operator for
  $\paramL\in[0,1]$ has still a simple maximal eigenvalue and a spectral gap (assuming
  $\ve$ to be sufficiently small). Observe that, since the resolvent is continuous in
  $\theta$, $\sigma_0$ can be chosen uniformly in $\theta$ and, consequently, since we
  have a uniform control on all terms appearing in $\hov^k_{\ell_k,j}$, $\sigma_0$ can be
  chosen to be uniform in $k, \ell_k,j$ and $\Phi$ as well.  This proves
  item~\ref{i_spectralGap}.

  We now prove item~\ref{i_chiEstimate}; note that the definition of $\hov^k_{\ell_k,j}$
  in~\eqref{eq:potdef} implies
\begin{equation}\label{eq:hov-Xi}
\begin{split}
  &\hov^k_{\ell_k,j}=\etaExph(t-\ve S_k,\thetaslkk{k}{L_k})\etaExpsk_{j,L_k}\ho\\
  &=\expo{\int_0^{t-\ve S_k}\hskip-.6cm\bar\omega'(\bar\theta(s+\ve L_k,\theta^*_{\ell_k}))\deh s+\sum_{l=j+1}^{L_k-1}\int_{l\ve}^{(l+1)\ve}\hskip-.6cm\bar\omega'(\bar\theta(s,\theta^*_{\ell_k}))\deh s+\cO(\ve^2 L_k)}\ho\\
  &=\expo{\int_{(j+1)\ve}^{t-\ve S_{k-1}}\bar\omega'(\bar\theta(s,\theta^*_{\ell_k}))\deh s+\cO(\ve^2 L_k)}\ho\\
  &=\etaExph(t-\ve [S_{k-1}+j+1],\thetaslkk{k}{j+1})\ho+\cO(\ve^2 L_k)\ho.
\end{split}
\end{equation}
It is then natural to introduce the potentials
\begin{equation}\label{eq:Bomega-def}
  \BOmega(T,s,\vf,x,\theta)=\etaExph(T-s-\ve,\bar\theta(s+\ve,\vf))\ho(x,\theta)
\end{equation}
so that
\begin{equation}\label{eq:hov-compute}
  \hov^k_{\ell_k,j}(x,\theta)=\BOmega(t-\ve S_{k-1},\ve j,\theta^*_{\ell_k},x,\theta)+\cO(\ve^2 L_*)\ho(x,\theta);
\end{equation}
in particular, by definition of $\Omega^{k,\Phi}_{\ell_k, j}$ we gather
\begin{align*}
  i\sigma\Omega^{k,\Phi}_{\ell_k, j}
	& = i\sigma\BOmega(t-\ve S_{k-1},\ve j,\theta^*_{\ell_k},x,\theta)+\cO(\ve^2 L_*+\sigma^2\ve(R-1-k)L_*)\ho(x,\theta).
\end{align*}
Let, $e^{\bchi(\sigma,T,s,\vf, \theta)}$ be the maximal eigenvalue of the operator
associated to the potential $i\sigma{\boldsymbol \Omega}(T,s,\vf,\cdot ,\theta)$ and
dynamics $f(\cdot,\theta)$.  Then, by~\eqref{e_derivative-varrho}, we have
\begin{align*}
  \chi_{\ell_k,k,j}=\bchi(\sigma,t-\ve S_{k-1},\ve
  j,\theta^*_{\ell_k},\thetaslkk{k}{j})+\cO(\sigma\ve^2 L_*+\sigma^2\ve(R-1-k)L_*)\int_0^1m_{\varrho}(\ho h_\varrho)\deh\varrho,
\end{align*}
where $m_\varrho=m_{\thetaslkk{k}{j}, \Omega_\varrho}$, and $h_\varrho=h_{\thetaslkk{k}{j}, \Omega_\varrho}$ with
\begin{align*}
\Omega_\varrho=i\sigma[(1-\varrho)\BOmega(t-\ve S_{k-1},\ve
  j,\theta^*_{\ell_k},x,\theta)+\varrho \Omega^{k,\Phi}_{\ell_k, j} .]
\end{align*}
Then Lemma~\ref{l_derivativess} implies
\begin{align*}
  \chi_{\ell_k,k,j}%
  &=\bchi(\sigma,t-\ve S_{k-1},\ve j,\theta^*_{\ell_k},\thetaslkk{k}{j})\\
  &\phantom = +\cO\left((\sigma\ve^2 L_*+\sigma^2\ve(R-1-k)L_*)
    m_{\thetaslkk{k}{j},0}(\ho h_{\thetaslkk{k}{j},0})\right)\\
  &\phantom = +\cO\left(\sigma^2\ve^2 L_*+\sigma^3\ve(R-1-k)L_*\right).
\end{align*}
The above implies the first equation of~\eqref{e_chiEstimates} since $\ho$ has zero
average by construction.  Next, we use~\eqref{e_derivativeChi},
~\eqref{e_secondDerivativeChia} with $\paramL=0$ and~\eqref{e_thirdDerivativeChi} to
obtain:
  \begin{equation}\label{e_chiSecondDerivative}
  \begin{split}
    e^{\bchi(\sigma,T,s,\vf,\theta)}
    & =1- \frac{\sigma^2}2\mu_{\theta}\left({\BOmega(T,s,\vf,\cdot,\theta)}^2\right)                                          \\
                                     & \phantom{=}-\sigma^2\sum_{m=1}^\infty
    \mu_{\theta}\left(\BOmega(T,s,\vf,\favg_\theta^m(\cdot),\theta)\BOmega(T,s,\vf,\cdot,\theta)\right) + \cO(\sigma^3) \\
                                     & =e^{-\frac{\sigma^2}2\left[\mu_{\theta}\left(\BOmega^2\right)+2\sum_{m=1}^\infty
        \mu_{\theta}\left(\BOmega\circ \favg_\theta^m\cdot \BOmega \right)\right] +
      \cO(\sigma^3)}
  \end{split}
  \end{equation}
  where we have used the the decay of correlations implied by
  item~\ref{i_spectralGap}. The second equation of~\eqref{e_chiEstimates} follows
  immediately, provided $\sigma_0$ has been chosen small enough.
\end{proof}

\begin{rem}\label{rem:motation-lkk}
  As we will use the results below for all blocks, not just the last one, we are
  interested in all the operators $\tO_{\ell_k,k,j}$. Yet, since all our computations are
  uniform in $k$ and $\ell_k$, there is no harm in dropping, again, the subscripts
  $k,\ell_k$ when this does not create confusion. Thus from now to the end of the section,
  to ease notation, $k,\ell_k$ are fixed and implicit. For the same reason we will write
  $L$ rather than $L_k$.  Moreover, to further ease our notation let us set
  $\Omega_j=\Omega^{k,\Phi}_{\ell_k, j}$ and $\bavgtheta{j}=\thetaslkk{k}{j}$. Also,
  $\chi_j=\chi_{\bavgtheta{j},\Omega_j}$, $m_j=m_{\bavgtheta{j},\Omega_j}$ and
  $h_j=h_{\bavgtheta{j},\Omega_j}$.
\end{rem}

Remark that, since $\tOp_j$ is a one dimensional projector, it can be written
as $\tOp_j=h_j\otimes m_j$, where we choose to normalize $h_j$ and $m_j$ according to
Lemma~\ref{l_normalization}.
Also, for future reference, we define
\begin{equation}\label{eq:fkEi-def}
  \overline\fkE_i:=m_{i+1}(h_{i}-h_{i+1}).
\end{equation}
We are now ready to derive a formula for the products of transfer operators in the
perturbative regime.
\begin{lem}\label{lem:product-bound-rough}
  There exists $\ve_0, \Czero>0$ such that, for any $i, n\in \{0,\cdots, L\}$,
  $\ve\le \ve_0$, $|\sigma|\in[ \Czero \ve^2 L_*,|\sigma_0|\,]$ and
  $\ve\|\Phi'\|\nc0\leq \Const \sigma^2$ we have, for any $g\in\BV$,
  \begin{align*}
    \left\|\tOh_{i+n}\cdots\tOh_{i}g-
    \expo{\sum_{j=0}^{n-1}\overline\fkE_{i+j}}
    \overline h_{i,n}\cdot\overline m_{i,n}g\right\|\nBV\le
    &\;\Const \|g\|_{L^1}\ve^2n^2\\
    &\hskip-3cm+\Const\left[e^{-\const n}+n^2\ve^2+\sillyconst\right]\|g\|_{BV},
  \end{align*}
  where
  \begin{align*}
    \overline h_{i,n}
    &=\sum_{k=0}^{n}\tOqh_{i+n}\cdots\tOqh_{i+k+1}h_{i+k}
    & \overline m_{i,n}
    &=\sum_{k=0}^n m_{i+k}\widehat \tOq_{i+k-1}\cdots \widehat\tOq_{i},
  \end{align*}
  with $\widehat\tO_{j}=e^{-\chi_j}\tO_{j}$, $\widehat\tOq_{j}=e^{-\chi_j}\tOq_{j}$.
\end{lem}
\begin{proof}
  Let us define
  \begin{align}\label{eq:Eindef}
    \bX_{i,n} &= \expo{\sum_{j=0}^{n}\chi_{i+j}},
  \end{align}
  and introduce the auxiliary operators\footnote{ In this section we use the standard
    conventions that, given any sequence of operators $\{A_i\}$,
    $A_j A_{j-1}\cdots A_{i+1}A_i=\Id$ if $j<i$.}
  \begin{align}\label{eq:tO-auxiliary}
    \overleftarrow\tO_j
    &=e^{\chi_j}h_{j+1}\otimes m_j+\tOq_j
    & \overleftarrow\tO_{i,n}
    &=\overleftarrow\tO_{i+n}\overleftarrow\tO_{i+n-1}\cdots\overleftarrow\tO_i.
  \end{align}
  Observe that, by construction:
  \begin{align}\label{e_basicChainBound}
    \tO_{i+n}\cdots\tO_{i}-\overleftarrow\tO_{i,n}
    &=\sum_{k=0}^{n}e^{\chi_{i+k}}\tO_{i+n}\cdots\tO_{i+k+1}(h_{i+k}-h_{i+k+1})\otimes
      m_{i+k}\overleftarrow\tO_{i,k-1},
  \end{align}
  and one can check, by induction, that
  \begin{align}
    \overleftarrow\tO_{i,n}
    & =h_{i+n+1}\otimes\left[\sum_{k=0}^n e^{\sum_{j=k}^{n}\chi_{i+j}}m_{i+k}\tOq_{i+k-1}\cdots \tOq_{i}\right]
      +\tOq_{i+n}\cdots \tOq_{i}\nonumber\\
    \label{eq:tO-auxiliary-chain}
    &=\bX_{i,n}\left\{h_{i+n+1}\otimes \overline m_{i,n}+\tOqh_{i+n}\cdots \tOqh_{i}\right\}.
  \end{align}
  In order to continue we need to compare adjacent operators; this can be done using
  perturbation theory.
  \begin{sublem}\label{sublem:perturbation-fancy}
    For any $i\in\{0,\cdots,L\}$ we have:
    \begin{align*}
      |\chi_{i+1}-\chi_i|   & \le \Const \ve |\sigma|.\\
      \|h_{i+1}-h_{i}\|\nc1 & \le \Const \ve.
    \end{align*}
    The same bounds hold for $m_{i+1}-m_i$ as a functional on $W^{2,1}$. Yet, we also have the
    bound, for any $g\in \BV$:
    \begin{align*}
      |m_{i+1}(g)-m_i(g)|\le \Const \ve|\sigma|(\log|\sigma|)^2\|g\|\nBV.
    \end{align*}
    Finally, we have
    \begin{equation}\label{eq:fkE-est-u}
      |\overline\fkE_i|\leq \Const \ve|\sigma|.
    \end{equation}
  \end{sublem}
  \begin{irem}
    The estimate~\eqref{eq:fkE-est-u} reported above suffices for the present level of precision. Yet, if one wanted
    to compute the first term of the Edgeworth expansion, then it would be necessary to
    introduce the function
    \begin{align*}
      \fkE_k(\sigma, s,\vf)=m_{\bar\theta(s+\ve,\vf),\widetilde\Omega_k(\sigma,s+\ve,\vf,\cdot)}\left(h_{\bar\theta(s,\vf),\widetilde\Omega_k(\sigma,s,\vf,\cdot)}-h_{\bar\theta(s+\ve,\vf),\widetilde\Omega_k(\sigma,s+\ve,\vf,\cdot)}\right),
    \end{align*}
    where
    \begin{align*}
      \widetilde\Omega_k(\sigma, s,\vf, x,\theta)
      &= i\sigma\BOmega(t-\ve S_{k-1},s,\vf,x,\theta)\\
      &\phantom = +\ve \Phi'(\bar \theta(L_k\ve,\vf)) \etaExph(L_k\ve-s,\bar\theta(s+\ve,\vf))\ho(x,\theta).
    \end{align*}
    One could then use Appendix~\ref{subsec:pert-theta} to show that
    $\overline\fkE_i=\fkE_k(\sigma, \ve i,\avgtheta{\ell_k})+\cO(\ve^2L_*)$ and
    $\|\fkE_k(\sigma,\cdot,\cdot)\|_{\cC^1}\leq \Const \ve|\sigma|$. So one can keep
    $\fkE_k$ in the definition of $\Phi_\rage$ in Section~\ref{proof-not-so-trivial-1}.
  \end{irem}
  \begin{proof}[{\bf Proof of Sub-lemma~\ref{sublem:perturbation-fancy}}]
    We will have to  vary both the dynamics and the potentials. This makes convenient to
    use, at times, the heavier, but more precise, notation introduced in
    Appendix~\ref{subsec:transfer}.  In this notation
    $\tO_j=\tO_{\ell_k,k,j}=\tO_{\thetaslkk{k}{j},\Omega^{k,\Phi}_{\ell_k, j}}$.  Note
    that $\|\Omega_j\|\nc2\leq \Const(|\sigma|+\ve\|\Phi'\|\nc0)$.
    Also, recall that $m_{\bavgtheta{j},0}=\Leb$ and hence $\Leb(h_{\bavgtheta{j},0})=1$.
    Next, observe that, although $m_i$ is a distribution, it is almost a measure: indeed
    using Lemma~\ref{lem:m-estimate-0} with $n = \Const|\sigma|\invr$ implies
    \begin{equation}\label{eq:m-almost-m1}
      |m_i(g)|\leq \Const \|g\|\nl1+\Const\expo{-\const|\sigma|\invr}\|g\|_{\BV}.
    \end{equation}
    In turn, this implies that $\tOqh_i$ satisfies a Lasota--Yorke inequality as well.
    In order to see this, recall equations~\eqref{eq:cone-est}
    and~\eqref{eq:lasota-yorke-bv} and note that, if $\sigma_0$ is small enough, then
    there exists $\lambda_1\in (1,\lambda)$ such that
    \begin{align}\label{eq:la-yo-Q}
      \|\tOqh_ig\|\nBV&\le \| h_i m_i( g)\|\nBV+\|\tOh_ig\|\nBV\le \lambda_1\invr\|g\|\nBV+\Const\|g\|\nl1.
    \end{align}
    By Lemma~\ref{eq:theta-derivative-all} we have
    \begin{equation}\label{eq:hoihoiwei}
      \begin{split}
        &|\chi_{\bavgtheta{i-1},\Omega_{i}}-\chi_{\bavgtheta{i},\Omega_{i}}|\le \Const |\sigma|\ve\\
        &\|h_{\bavgtheta{i-1},\Omega_{i}}-h_{\bavgtheta{i},\Omega_{i}}\|\nc2\le\Const \ve\\
        &|m_{\bavgtheta{i-1},\Omega_i}(g)-m_{\bavgtheta{i},\Omega_i}(g)|\le \Const  \ve|\sigma|\|g\|\nw21.
      \end{split}
    \end{equation}
    It turns out that the third of the above estimates is not very convenient owing to the
    higher derivative in the right hand side. However, Lemma~\ref{lem:m-theta-BV} implies
    \begin{equation}\label{eq:m-diff-sharp2}
      |m_{\bavgtheta{i-1},\Omega_i}(g)-m_{\bavgtheta{i},\Omega_i}(g)|\leq \Const(\log|\sigma|)^2|\sigma|\ve\|g\|_{\BV}.
    \end{equation}
    Next, equations~\eqref{e_defTransferOp} and~\eqref{eq:potdef} imply
    \begin{equation}\label{eq:pot-est}
      \begin{split}
        &\|\Omega_{i+1}-\Omega_{i}\|\nc1\leq \Const \ve|\sigma| .
      \end{split}
    \end{equation}

    We can then use~\eqref{e_derivative-varrho} and argue as
    in~\eqref{e_estimatesPartialRho} to obtain,\footnote{\label{fot:lower} The
      formula~\eqref{e_estimatesPartialRho-m} holds also with the $\BV$ norm on the left
      hand side due to the lower semicontinuity of the variation \cite[Section 5.2.1,
      Theorem 1]{EvansGariepy}.} for $i\le L$,
    \begin{equation}\label{eq:der-omega-rough}
      \begin{split}
        &|\chi_{\bavgtheta{i-1},\Omega_{i-1}}-\chi_{\bavgtheta{i-1},\Omega_{i}}|\le\Const \ve|\sigma|\\
        &\|h_{\bavgtheta{i-1},\Omega_{i-1}}-h_{\bavgtheta{i-1},\Omega_{i}}\|\nc1\le\Const |\sigma|\ve\\
        &|m_{\bavgtheta{i-1},\Omega_{i-1}}(g)-m_{\bavgtheta{i-1},\Omega_i}(g)|\le \Const|\sigma|  \ve \|g\|_{\BV},
      \end{split}
    \end{equation}
    Collecting the above facts, yields the first three inequalities of the Lemma.

    Next, by the third of equations~\eqref{eq:hoihoiwei}, using that
    $h_{\bavgtheta{i-1},\Omega_{i-1}}\in\cC^2$, see Remark~\ref{rem:high-der},
    and~\eqref{eq:der-omega-rough}, we can write
    \begin{align*}
      m_{i}(h_{i-1}-h_{i})
      &= m_i(h_{i-1})-1 \\
      &= \left[m_{\bavgtheta{i},\Omega_i}-m_{\bavgtheta{i-1},\Omega_{i-1}}\right](h_{\bavgtheta{i-1},\Omega_{i-1}})\\
      &= \left[m_{\bavgtheta{i-1},\Omega_i}-m_{\bavgtheta{i-1},\Omega_{i-1}}\right](h_{\bavgtheta{i-1},\Omega_{i-1}})+\cO(\ve|\sigma|)
        =\cO(\ve |\sigma|),
    \end{align*}
    which concludes the proof of the Sub-lemma.
  \end{proof}
  We also need a bound on products of $\tOqh_i$'s which is rather obvious but a bit lengthy to
  prove.
  \begin{sublem} \label{sublem:Q-product}%
    There exists $n_*\in\bN$ such that, for all $k\in\{0,\cdots,L-n_*\}$, we have
    \begin{align*}
      \|\tOqh_{k+n_*}\cdots \tOqh_k\|_\BV\le e^{-1}.
    \end{align*}
  \end{sublem}
  \begin{proof}
    Note that, by Lemma~\ref{lem:Taylor}, there exist $\Csette>0$
    such that $\sup_{k,j}\|\tOqh_{k}^j\|_\BV\le \Csette$. We are now going to prove, by
    induction, that there exists $C_Q\geq \Csette$ such that for any $N_Q>0$, there
    exists $\ve_0,\sigma_0$ such that, for all $\ve\le \ve_0$ and
    $|\sigma|\le \sigma_0$, we have
    \begin{align}\label{e_claim-tada}
      \sup_{k}\sup_{j\le N_Q}\|\tOqh_{k+j}\cdots \tOqh_k\|_\BV\le C_Q.
    \end{align}
    The claim is trivially true for $N_Q=0$. Suppose it is true for all $j\leq N_Q-1$
    for some $\sigma_0,\ve_0$. Possibly by decreasing $\sigma_0$ assume that
    $\sigma_0^2N_Q\leq 1$ and note that, since we assume
    $\ve\|\Phi'\|\nc0\leq \Const \sigma^2$:
    \begin{align*}
      |\Leb(\tOh_i g)|&\le \|\tOh_i g\|\nl1\le e^{-\chi_i}\Leb(e^{\Re(\Omega^{k,\Phi}_{\ell_k, j})}|g|)\leq e^{\Const \sigma^2} \|g\|\nl1.
    \end{align*}
    Together with Sub-Lemma~\ref{sublem:perturbation-fancy}, the
    above inequality implies
    \begin{align*}
      \|\tOqh_{k+j}
      &\cdots\tOqh_{k}g\|\nl1\\
      &\le \|\tOh_{k+j}\tOqh_{k+j-1}\cdots\tOqh_{k}g\|\nl1 +\Const|(m_{k+j}-m_{k+j-1})(\tOqh_{k+j-1}\cdots\tOqh_{k}g)|\\
      &\le e^{\const \sigma^2}\|\tOqh_{k+j-1}\cdots\tOqh_{k}g\|\nl1+\Const\sillyconst\|\tOqh_{k+j-1}\cdots\tOqh_{k}g\|\nBV.
    \end{align*}
    Iterating the above argument, since $\sigma N_q\le 1$, and by the inductive
    hypothesis:
    \begin{align}\label{eq:zero-to-one-Q}
      \|\tOqh_{k+j}\cdots\tOqh_{k}g\|\nl1&\le\Const \|g\|\nl1+\Const\sillyconst\sum_{l=0}^{j-1}\|\tOqh_{k+l}\cdots\tOqh_{k}g\|_\BV\notag\\
                                         &\le \Const \|g\|\nl1+\Const\sillyconst jC_Q\|g\|_\BV.
    \end{align}
    We can now use~\eqref{eq:la-yo-Q} to write
    \begin{align*}
      \|\tOqh_{k+j}\cdots\tOqh_{k}g\|_\BV
      &\le \lambda_1^{-1}\|\tOqh_{k+j-1}\cdots\tOqh_{k}g\|_\BV+\Const \|\tOqh_{k+j-1}\cdots\tOqh_{k}g\|\nl1\\
      &\le[(\lambda_1^{-1}+\Const \sillyconst j)C_Q+\Const]\|g\|_\BV
    \end{align*}
    from which the claim follows.

    Next, by Sub-Lemma~\ref{sublem:perturbation-fancy},
    \begin{equation}\label{eq:Q-BV-to-L1}
      \begin{split}
        \|(\tOqh_{k+j}-\tOqh_{k})g\|\nl1&\leq\|(\tOh_{k+j}-\tOh_{k})g-(h_{k+j}\otimes m_{k+j}- h_j\otimes m_j)g\|\nl1\\
        &\leq \Const\sillyconst j \|g\|_\BV.
      \end{split}
    \end{equation}

    Using Lemma~\ref{lem:Taylor} once again, there exists $m_*\in\bN$, independent of
    $\ve$ and $\sigma$, such that, for all $k\in\{0,\cdots,L\}$,
    $\|\tOqh_k^{m_*}\|_\BV\le (2 eC_Q)^{-1}$ and $\lambda_1^{m_*}\geq 8 eC_Q^2$. We will
    use the above claim with $N_Q=m_*$. Note that, in particular, this implies that
    $\sigma_0^2m_*\leq 1$.

    Then, using
    equations~\eqref{eq:la-yo-Q},~\eqref{e_claim-tada},~\eqref{eq:zero-to-one-Q}
    and~\eqref{eq:Q-BV-to-L1}:
    \begin{align*}
      &\|(\tOqh_{k+2m_*}\cdots \tOqh_k-\tOqh_{k+2m_*}\cdots \tOqh_{k+m_*+1}\tOqh_{k+m_*}^{m_*})g\|_\BV\\
      &\le (\lambda_1^{-m_*}+\Const \sillyconst m_*C_Q)\|(\tOqh_{k+m_*}\cdots \tOqh_k-\tOqh_{k+m_*}^{m_*})g\|_\BV\\
      &\phantom{\leq}+\Const \|(\tOqh_{k+m_*}\cdots
        \tOqh_k-\tOqh_{k+m_*}^{m_*})g\|\nl1\\
      &\le  2C_Q (\lambda_1^{-m_*}+\Const \sillyconst m_*C_Q)\|g\|_\BV\\
      &\phantom{\leq}+\Const\sum_{j=1}^{m_*}\|\tOqh_{k+m_*}^j(\tOqh_{k+m_*-j}-\tOqh_{k+m_*})\tOqh_{k+m_*-j-1}\cdots\tOqh_{k}g\|\nl1\\
      &\le \left[2\lambda_1^{-m_*}+\Const\sillyconst m_*^2 \right]C_Q^2\|g\|_\BV\leq \frac 1{2 e}\|g\|_\BV
    \end{align*}
    provided $\ve_0, \sigma_0$ have been chosen small enough. Hence
    \begin{align*}
      \|\tOqh_{k+2m_*}\cdots \tOqh_k\|_\BV
      &\leq\|\tOqh_{k+2m_*}\cdots \tOqh_k-\tOqh_{k+2m_*}\cdots
        \tOqh_{k+m_*+1}\tOqh_{k+m_*}^{m_*})\|\nBV\\
      &\phantom{\leq}+\|\tOqh_{k+2m_*}\cdots \tOqh_{k+m_*+1}\tOqh_{k+m_*}^{m_*}\|\nBV \leq
        \frac 1{2 e}+C_Q(2eC_Q)^{-1}\leq e^{-1}.
    \end{align*}
    We have thus proved our claim, with $n_*=2m_*$.
  \end{proof}
  We can now use Sub-Lemmata~\ref{sublem:perturbation-fancy} and~\ref{sublem:Q-product} to
  continue the argument that we left at~\eqref{eq:tO-auxiliary-chain}: we immediately
  obtain
  \begin{equation}\label{eq:op-arrow-est}
    \begin{split}
      &\left\|
      \bX_{i,n}\invr \overleftarrow\tO_{i,n}-h_{i+n+1}\otimes \overline m_{i,n}\right\|_\BV\le \Const e^{-\const n}\\
      &|\overline m_{i,n}(g)-m_i(g)|\le\sum_{k=1}^{i+n}|(m_{i+k}-m_{i+k-1})(\tOqh_{i+k-1}\cdots\tOqh_ig)|\\
      &\phantom{|\overline m_{i,n}(g)-m_i(g)|}\le \Const\sillyconst\|g\|_\BV.
    \end{split}
  \end{equation}
  At this point we can write, using repeatedly~\eqref{e_basicChainBound}:
  \begin{align}\label{eq:new-chain-bound}
    \tO_{i+n}\cdots\tO_{i}-\overleftarrow\tO_{i,n}
    &=\sum_{k=0}^{n}e^{\chi_{i+k}}\tO_{i+n}\cdots\tO_{i+k+1}(h_{i+k}-h_{i+k+1})\otimes
      m_{i+k}\overleftarrow\tO_{i,k-1},\notag\\
    &=\sum_{k=0}^{n}e^{\chi_{i+k}}\overleftarrow\tO_{i+k+1, n-k-1}(h_{i+k}-h_{i+k+1})
      \otimes m_{i+k}\overleftarrow\tO_{i,k-1}\\\notag
    &\quad+\sum_{k=0}^{n}\sum_{j=0}^{n-k-1}e^{\chi_{i+k}+\chi_{i+k+j+1}}\tO_{i+n}
      \cdots\tO_{i+k+j+2}(h_{i+k+j+1}-h_{i+k+j+2})\\\notag
    &\quad\quad\otimes m_{i+k+j+1}\overleftarrow\tO_{i+k+1,j-1}(h_{i+k}-h_{i+k+1})
      \otimes m_{i+k}\overleftarrow\tO_{i,k-1}.
  \end{align}
  Note that, by~\eqref{eq:op-arrow-est} and~\eqref{eq:m-almost-m1} 
  we have
  \begin{align*}
    \|\bX_{i,j}^{-1}\overleftarrow\tO_{i, j}\|\nBV\le\Const.
  \end{align*}
  Then, the first line of~\eqref{eq:new-chain-bound}, together with
  Sub-Lemma~\ref{sublem:perturbation-fancy}, suffices to write
  \begin{align*}
    \left\|\tO_{i+n}\cdots\tO_{i}-\overleftarrow\tO_{i,n}\right\|_\BV\le \Const \ve\sum_{k=0}^{n-1} \|\tO_{i+n}\cdots\tO_{i+k+1}\|_\BV \left|\bX_{i,k-1}\right|.
  \end{align*}

  From the above it follows by induction:
  \begin{equation}\label{eq:rough-Lchin-bound}
    \left\|\tO_{i+n}\cdots\tO_{i}\right\|_\BV\le \Const |\bX_{i,n}|.
  \end{equation}
  Recall that Lemma~\ref{lem:m-estimate}, implies that, for any $j\in \{0,\cdots, n\}$,
  \begin{align*}
    |\overline m_{i,j}(g)|+|m_{i+j}(g)|\leq \Const \|g\|_{L^1}+\Const |\sigma|^{100}\|g\|_{BV}=\Const\|g\|_{BV_\sigma},
  \end{align*}
  where we have introduced the shorthand notation
  $\|g\|_{\BV_\sigma}=\|g\|\nl1+|\sigma|^{100}\|g\|\nBV$.
  Note that~\eqref{eq:tO-auxiliary-chain} and the definition of $\overline m_{i,k}$
  in the statement of Lemma~\ref{lem:product-bound-rough} imply
  \begin{align*}
    m_{i+k}\overleftarrow\tO_{i,k-1}&=\bX_{i,k-1}\left\{ \overline m_{i,k-1}+m_{i+k}\tOqh_{i+k-1}\cdots \tOqh_{i}\right\}\\
    &=\bX_{i,k-1}\overline m_{i,k}.
  \end{align*}
  Given the above, we can now use the full force of~\eqref{eq:new-chain-bound} and
  Sub-Lemma~\ref{sublem:perturbation-fancy}, using~\eqref{eq:tO-auxiliary-chain}:
  \begin{align*}
    \tO_{i+n}\cdots\tO_{i}g
    &= \overleftarrow\tO_{i,n}g+\sum_{k=0}^{n}\bX_{i,k}\overleftarrow\tO_{i+k+1, n-k-1}(h_{i+k}-h_{i+k+1})\cdot \overline m_{i,k}(g)\\
    &\phantom = +\bX_{i,n}\|g\|_{\BV_\sigma}\cO\nc1(\ve^2n^2)\\
    &=\overleftarrow\tO_{i,n}g+\bX_{i,n}(h_{i+n}-h_{i+n+1})\cdot \overline m_{i,n}(g)\\
    &\phantom =+\bX_{i,n}\sum_{k=0}^{n-1}h_{i+n+1}\cdot\overline m_{i+k+1, n-k-1}(h_{i+k}-h_{i+k+1})
      \cdot \overline m_{i,k}(g)\\
    &\phantom =+\bX_{i,n}\sum_{k=0}^{n-1}\tOqh_{i+n}\cdots\tOqh_{i+k+1}h_{i+k}\cdot \overline
      m_{i,k}(g)\\
    &\phantom = +\bX_{i,n}\|g\|_{\BV_\sigma}\cO\nc1(\ve^2n^2).
  \end{align*}
  Finally, by definition of
  $\overline h_{i,n}, \overline m_{i,n}$,~\eqref{eq:op-arrow-est},
  Sub-Lemmata~\ref{sublem:perturbation-fancy},~\ref{sublem:Q-product}, since $L^2\ve<1$
  and recalling \eqref{eq:fkEi-def} we have\footnote{ Here we use repeatedly that
    $\tOqh_{j}h_{j-1}=\tOqh_{j}(h_{j-1}-h_j)$ and the similar relation for $m_j$.}
  \begin{align*}
    \tOh_{i+n}\cdots\tOh_{i}g
    &= h_{i+n+1} \overline m_{i, n}(g)+(h_{i+n}-h_{i+n+1})\cdot \overline m_{i, n}(g)\\
    &\phantom = +\sum_{k=0}^{n-1}\tOqh_{i+n}\cdots\tOqh_{i+k+1}h_{i+k}\cdot  \overline
      m_{i, n }(g)\\
    &\phantom = +\sum_{k=0}^{n-1}m_{i+k+1}(h_{i+k}-h_{i+k+1})\overline h_{i,n} \cdot  \overline m_{i,n}(g)\\
    &\phantom = +\|g\|_{\BV_\sigma}\cO\nc1(\ve^2n^2)+\|g\|\nBV\cO\nc1(e^{-\const n}+\sillyconst)\\
    &=\expo{\sum_{k=0}^{n-1}\overline\fkE_{i+k}}\overline h_{i,n}\cdot \overline m_{i, n}(g)+\|g\|\nl1\cO\nc1(\ve^2n^2)\\
    &\phantom = +\|g\|\nBV\cO\nc1(e^{-\const n}+e^{\const n|\sigma|\ve}n^2\ve^2|\sigma|^2+\sillyconst).
  \end{align*}
\end{proof}
\subsection{ Main result for the intermediate regime}\
\newline\label{subsec:main-intermediate}%
Lemma~\ref{lem:product-bound-rough} is the basic tool to conclude the proof of the Local
Central Limit Theorem.  In this subsection we see how to use the lemma to prove the
results we are interested in for the (easier) intermediate regime.  The case of the small
regime will be dealt with in the next section.
\begin{proof}[{\bf Proof of Proposition~{\ref{prop:not-so-trivial-1/2}}}]
  First of all recall (see Remark~\ref{eq:long-v-short} that in this regime we are
  considering only families of long standard pairs.  Let us apply
  Proposition~\ref{p_transferOperators} with $\Phi\equiv 0$: we have, choosing
  $\wp=\ve^2$ and recalling that $L_* = \cO(\ve^{-3\delta_*})$:
  \begin{align*}
    \sum_{\ell_R\in\stdf^{R-1}_{\ell_{R-1}} }\fm_{R-1,\ell_{R-1},\ell_R}
    &=\Leb \sum_{\ell_R\in\stdf^{R-1}_{\ell_{R-1}} }\fm_{R-1,\ell_{R-1},\ell_R}\mathring\rho_{\ell_R}\\
    &=e^{i\sigma\ve\fkC^{k,1}_{\ell_{R-1},0,\emptyset}} \Leb\,\tO_{\ell_{R-1},R-1,L_{R-1}-1}\cdots\tO_{\ell_{R-1},R-1,0}[\Psi_q \trhove_{\ell_{R-1},\wp}]\\
    &\phantom{=}+\cO(\ve^{2-9\delta_*}).
  \end{align*}
  Next, we analyze each of the terms separately.  Lemma~\ref{lem:product-bound-rough} and
  Sub-Lemma~\ref{sublem:perturbation-fancy} imply
  \begin{align*}
    &\left\|\left[ \tO_{\ell_{R-1},R-1,L_{R-1}-1}\cdots\tO_{\ell_{R-1},R-1,0}-\bX_{\ell_{R-1}}^*\overline h_{0,L_{R-1}-1}\otimes\overline m_{0,L_{R-1}-1}\right]\trhove_{\ell_{R-1},\wp}\right\|_{\BV}\leq \\
    &\phantom{[\tO_{L-1}\tO_{L-2}\cdots\tO_{0}-}\leq\Const \left|\bX_{0,L_{R-1}-1}\right|\ve,
  \end{align*}
  where
  \begin{align*}
    &\bXs_{\ell_{R-1}}=\expo{\sum_{j=0}^{L_{R-1}}\chi_j+\overline\fkE_j}.
  \end{align*}
  Notice that Lemma~\ref{lem:Taylor} and~\eqref{eq:fkE-est-u} allow to write:
  \begin{align}\label{eq:Esize0}
    \notag \log  \bXs_{\ell_{R-1}}
    &= - \frac{\sigma^2}2\sum_{j=0}^{L_{R-1}}\etaExph(t-\ve[S_{R-2}+j+1],\thetaslkk{R-1}{j})^2
      \bVar(\thetaslkk{R-1}{j})\\
    &\phantom = + \cO(\sigma^3+\sigma^2\ve^2 L_{R-1}+\sigma\ve L_{R-1}),
  \end{align}
  and the same estimate holds for $\bX_{0,L_{R-1}}$.  The above implies that for
  $|\sigma|\in[\ve^{\delta_*}, \sigma_0]$, given the choice $L=\ve^{-3\delta_*}$, we have
  $\bXs_{\ell_{R-1}}=\cO(e^{-\const\ve^{-\delta_*}})$, and the same for
  $\bX_{0,L_{R-1}}$. Also, by similar arguments, the correlation terms will give a smaller
  contribution since $9\delta_*<1$.  It follows that
  \begin{align*}
    \left|\sum_{\ell_R\in\stdf^{R-1}_{\ell_{R-1}} }\fm_{R-1,\ell_{R-1},\ell_R}\right|\leq \Const \ve^{2-9\delta_*}.&\qedhere
  \end{align*}
\end{proof}
\section{One block estimate: the small \texorpdfstring{$\sigma$}{σ} regime}
\label{sec:one-smaller}
As already mentioned, in the small $\sigma$ regime the contraction of a single block is
not sufficient for our needs; we thus need to combine together several blocks.  To this
end, in this section, we provide a suitable description of the one block contribution.
Given a complex standard pair $\ell=(\bG_\ell,\rho_\ell)$, recall the notation
$\mathring{\rho}_{\ell}=\rho_\ell\Id_{[a_\ell,b_\ell]}$.

\begin{prop}\label{prop:not-so-trivial-oneBlock}
  Let $\sigma\in\J1$, $\wp=\ve^2$, $k\in\{0,\cdots, R-1\}$ and $\Phi\in\cC^2(\bT,\bC)$ such
  that $\Phi^+=\max\Re(\Phi)\leq\Const$,
  $\|\Phi'\|\nc0\leq \Const\min\{\sigma^2(R-1-k)L_*, \vei L_*^{-1}\}$ and
  $\ve\|\Phi'\|\nc1\sup_{0\le j\le L_k}\|\etaExps_{j}\ho\|\nc1\leq \sigma_0$.  Then\footnote{
    Recall (see Remark~\ref{eq:long-v-short}) that in this regime we are dealing with
    families of \emph{long} standard pairs.} we have the estimates:
  \begin{align}\label{e_crucialOneBlock}
    \sum_{\ell_{k+1}\in\stdf^k_{\ell_k}}e^{\Phi\circ
    G_{\ell_{k+1}}}\fm_{k,\ell_k,\ell_{k+1}}\mathring\rho_{\ell_{k+1}}
    &= 
      \bXss_{\ell_k}\overline h_{0,L_k-1} \overline
      m_{0,L_k-1}\left[e^{\Phi(\bar\theta_{\ell_k,L_k})} \mathring{\rho}_{\ell_k}\right]\\
    &\phantom= +\cE^{**}_{\ell_k,1}+\cE^{**}_{\ell_k,2}\notag
  \end{align}
  where $\overline h_{i,n}$ and $\overline m_{i,n}$ are defined in
  Lemma~\ref{lem:product-bound-rough},
  \begin{equation}\label{eq:chi-est-u}
    \begin{split}
      &\bXss_{\ell_k}= \expo{\vei\int_0^{\ve L_k}\cG_k(\theta^*_{\ell_k},s,\sigma)\deh s} \\
      &\cG_k(\theta,s,\sigma)= \bchi(\sigma,t-\ve S_{k-1},s,\theta,\bar\theta(s,\theta)),
    \end{split}
  \end{equation}
  $\bchi$ being defined just below~\eqref{eq:hov-compute}.  Finally
  \begin{align*}
    &\|\cE^{**}_{\ell_k,1}\|_{L^1}\leq \Const e^{\Phi^+}(\ve^2L_k^3+|\sigma|\ve
      L_k^2+(R-k)|\sigma|^3\ve L_k^2)\;;\\
    &\|\cE^{**}_{\ell_k,2}\|_{L^1}\leq\Const e^{\Phi^+}  \ve L_k^2\;;\quad\quad |\Leb(\cE^{**}_{\ell_k,2})|\leq  e^{\Phi^+}\Const (|\sigma|\ve L_k^3+\ve^2L_k^3)\\
    &\|\cE^{**}_{\ell_k,1}\|_\BV\leq \Const\;;\quad\quad \|\cE^{**}_{\ell_k,2}\|_\BV\leq\Const,
  \end{align*}

\end{prop}
{}\begin{proof}
  First let us apply Proposition~\ref{p_transferOperators} to the left hand side
  of~\eqref{e_crucialOneBlock}, obtaining:
  \begin{align}
    \sum_{\ell_{k+1}\in\stdf_{\ell_k}^k }e^{\Phi\circ G_{\ell_{k+1}}}\fm_{k,\ell_{k},\ell_{k+1}} \rrho_{\ell_{k+1}}
    = &e^{i\sigma\ve\fkC^{k,1}_{\ell_k,0,\emptyset}}
      \tO_{\ell_k,k,L_k-1}\cdots\tO_{\ell_k,k,0}[\Psi_{\ell_k,q} e^{\Phi(\bar\theta_{\ell_k,L_k})}\trhove_{\ell_k,\wp}] \notag\\
      &+\cE^*_{\ell_k}.\label{e_rewriteOverAgain}
  \end{align}
  Then observe that, by definition
  \begin{align*}
    |e^{i\sigma\ve\fkC^{k,1}_{\ell_k,0,\emptyset}} - 1| \le\Const \ve|\sigma| L_k.
  \end{align*}
  Next, we rewrite the first term on the right hand side of~\eqref{e_rewriteOverAgain}
  \begin{equation}\label{eq:step-2-small}
    \begin{split}
      & \tO_{\ell_k,k,L_k-1}\cdots\tO_{\ell_k,k,0}[\Psi_{\ell_k,q}e^{\Phi(\bar\theta_{\ell_k,L_k})} \trhove_{\ell_k,\wp}]\\
      &=\tO_{\ell_k,k,L_k-1}\cdots\tO_{\ell_k,k,0}\bigg\{[\Psi_{\ell_k,q} -e^{\ve\cK_0}]\trhove_{\ell_k,\wp}+[e^{\ve\cK_0}\trhove_{\ell_k,\wp}-\mathring\rho_{\ell_k}]\bigg\}e^{\Phi(\bar\theta_{\ell_k,L_k})}\\
      &\phantom{=}
      +\tO_{\ell_k,k,L_k-1}\cdots\tO_{\ell_k,k,0}\mathring\rho_{\ell_k}e^{\Phi(\bar\theta_{\ell_k,L_k})}.
    \end{split}
  \end{equation}
  We can now apply Lemma~\ref{lem:product-bound-rough} to each term
  separately:
  \begin{equation}\label{eq:term-zero}
    \begin{split}
      &\left\|\left[\tO_{\ell_k,k,L_k-1}\cdots\tO_{\ell_k,k,0}-\bX_{\ell_k}^*e^{-\overline\fkE_{L_k-1}}\overline h_{0,L_k-1}\otimes\overline m_{0,L_k-1}\right]e^{\Phi}\mathring\rho_{\ell_k}\right\|_\BV \\
      &\phantom{\tO_{\ell_k,k,L_k-1}\cdots\tO_{\ell_k,k,0}-\bX_{\ell_k}^*\overline
        h_{0,L_k-1}\otimes\overline m_{0,L_k-1}]}
      \leq \Const e^{\Phi^+}|\bX_{0,L_k-1}| L_k^2\ve^2\\
      & \bXs_{\ell_{k}}=\expo{\sum_{j=0}^{L_{k}}\chi_{\ell_k,k,j}+\overline\fkE_j}\\
      &\phantom{\bXs_{\ell_{k}}} =\expo{\vei\int_0^{\ve
          L_k}\cG_k(\theta^*_{\ell_k},s,\sigma)\deh s}+\cO(|\sigma|\ve L_k+|\sigma|^3\ve(R-k)
      L_k^2),
    \end{split}
  \end{equation}
  where we used the fact that
  $\|e^{\Phi(\bar\theta_{\ell_k,L_k})}\|_\BV\leq \Const\ve \|\Phi'\|\nc0$; in the second
  line, we have used the definition~\eqref{eq:Eindef} and, in the last line, we have used
  Lemma~\ref{lem:Taylor}-\ref{i_chiEstimate} and~\eqref{eq:fkE-est-u}.

  Next, we want to compute the correlation terms. They are sum of terms of the following
  type (possibly expanding in series the exponential), with $s\leq q$ which, recall, has
  been fixed $q = 7$:
  \[
    \begin{split}
      &\|\tO_{\ell_k,k,L_k-1}\cdots\tO_{\ell_k,k,0}\prod_{i=1}^sA\circ \Favg^{i_s}e^{\Phi(\bar\theta_{\ell_k,L_k})}\trhove_{\ell,\wp}\|_\BV=\\
      &=\|\tO_{\ell_k,k,L_k-1}\cdots\tO_{\ell_k,k,i_s+1}A_{s,\bar\imath}\tO_{\ell_k,k,i_s}\cdots A_{1,\bar\imath}\tO_{\ell_k,k,i_1}\cdots\tO_{\ell_k,k,0}e^{\Phi}\trhove_{\ell,\wp}\|_\BV\\
      &\leq (\Const)^{q},
    \end{split}
  \]
  where we have used equations~\eqref{eq:rough-Lchin-bound} and~\eqref{e_chiEstimates}
  (which implies $|\bX_{i,n}|\leq e^{-\const \sigma^2 n}$).  %
  Thus, by Proposition~\ref{p_transferOperators}-\ref{i_propertyDelta2},
  Notation~\ref{not:bookcor} and since, by hypothesis,
  $\|\tO_{\ell_k,k,j}\|_{L^1}\leq e^{\const L_*^{-1}}$,
  \begin{equation}\label{eq:term-uno}
    \begin{split}
      &\|\tO_{\ell_k,k,L_k-1}\cdots\tO_{\ell_k,k,0}[\Psi_{\ell_k,q} -e^{\ve\cK_0}]e^{\Phi}\trhove_{\ell_k,\wp}\|_{L^1}\leq \Const e^{\Phi^+} \ve|\sigma|L_k^2\\
      &\|\tO_{\ell_k,k,L_k-1}\cdots\tO_{\ell_k,k,0}[\Psi_{\ell_k,q} -e^{\ve\cK_0}]e^{\Phi}\trhove_{\ell_k,\wp}\|_\BV\leq \Const.
    \end{split}
  \end{equation}
  Finally, we compute the remaining term on the left hand side of~\eqref{eq:step-2-small}.
  Since Proposition~\ref{p_transferOperators}-\ref{i_estimateRhoComparison} implies
  \begin{equation}\label{eq:K0-rough0}
    \|e^{\ve\cK_0}\trhove_{\ell_k,\wp}-\mathring\rho_{\ell_k}\|_{L^1}\leq \Const \ve L_k^2,
  \end{equation}
  a brute force estimate, as the one above, yields
  \begin{equation}\label{eq:K0-rough}
    \begin{split}
      &\|\tO_{\ell_k,k,L_k-1}\cdots\tO_{\ell_k,k,0}\,e^{\Phi}[e^{\ve\cK_0}\trhove_{\ell_k,\wp}-\mathring\rho_{\ell_k}]\|_{L^1}\leq \Const e^{\Phi^+}\ve L_k^2\\
      &\|\tO_{\ell_k,k,L_k-1}\cdots\tO_{\ell_k,k,0}\,e^{\Phi}[e^{\ve\cK_0}\trhove_{\ell_k,\wp}-\mathring\rho_{\ell_k}]\|_\BV\leq \Const .
    \end{split}
  \end{equation}
  Unfortunately, inequalities~\eqref{eq:K0-rough} yields a mistake is too large for our needs.  We must be a bit more careful and compute the term in more detail.\footnote{ Note that this could be done also for other terms, hence allowing for smaller errors.} It happens that a more precise estimate of the average with respect to Lebesgue will suffice.
  Let us us call $\tO_{\ell,k,j,0}$ the transfer operator $\tO_{\ell,k,j}$ computed for $\sigma=0$. Then, by standard perturbation theory,
  \begin{equation}\label{eq:hate-proofreading}
    \|\tO_{\ell_k,k,L_k-1}\cdots\tO_{\ell_k,k,0}-\tO_{\ell_k,k,L_k-1,0}\cdots\tO_{\ell_k,k,0,0}\|_{L^1}\leq\Const L_k|\sigma|.
  \end{equation}
  Thus, since $\Leb\, \tO_{\ell_k,k,j,0}=\Leb$, by equations~\eqref{eq:K0-rough0},~\eqref{eq:hate-proofreading} and Proposition~\ref{p_transferOperators}-\ref{i_propertyDelta2} we have
  \begin{equation}\label{eq:K0+Lebesgue}
    \begin{split}
      &\Leb\left\{\tO_{\ell_k,k,L_k-1}\cdots\tO_{\ell_k,k,0}\, e^{\Phi(\bar\theta_{\ell_k,L_k})}[e^{\ve\cK_0}\trhove_{\ell_k,\wp}-\mathring\rho_{\ell_k}]\right\}\\
      &=\Leb\,e^{\Phi(\bar\theta_{\ell_k,L_k})}\left\{e^{\ve\cK_0}\trhove_{\ell_k,\wp}-\mathring\rho_{\ell_k}\right\}
      +e^{\Phi^+}\cO(\sigma L_k^3\ve)\\
      &=e^{\Phi(\bar\theta^*_{\ell_k,L_k})}\Leb\,\left\{e^{\ve\cK_0}\trhove_{\ell_k,\wp}-\mathring\rho_{\ell_k}\right\}+e^{\Phi^+}\cO(\sigma L_k^3\ve+\ve^2L_k^2)\\
      &=e^{\Phi^+}\cO(\sigma L_k^3\ve+\ve^2L_k^3).
    \end{split}
  \end{equation}
  The proposition then follows by collecting the previous inequalities, setting
  $\cE^{**}_{\ell_k,2}=\tO_{\ell_k,k,L_k-1}\cdots\tO_{\ell_k,k,0}\,
  e^{\Phi(\bar\theta_{\ell_k,L_k})}[e^{\ve\cK_0}\trhove_{\ell_k,\wp}-\mathring\rho_{\ell_k}]$
  and putting all the other error terms in $\cE^{**}_{\ell_k,1}$.
\end{proof}

\section{Combining many blocks: main result for the small regime}\label{proof-not-so-trivial-1}
This section contains the proof of Proposition~\ref{prop:not-so-trivial-1}, which follows
by iterating the one block estimates obtained in the previous section (i.e.\
Proposition~\ref{prop:not-so-trivial-oneBlock}); we also rely on the results detailed in
Appendix~\ref{subsec:transfer}.  The proof essentially follows from the next technical
lemma. Before stating it let us fix and recall some notation.

Let $t\in\bR_+$ be fixed.  Given $(x,\theta)\in\bT^2$, recall the definitions
$(x_j,\theta_j)=F^j_\ve(x,\theta)$, $\thetaslk{k}=\bar\theta(\ve k,\thetasl)$,
$\thetasl=\Re(\mu_\ell(\theta))$, while
$\bar\theta_{\ell,\ve k}(x)=\bar\theta(\ve k,G_\ell(x))$; finally recall that, as defined
in Section~\ref{sec:birk}, we defined $S_k=\sum_{j=0}^k L_j$ with $S_{-1} = 0$ and
$\ve S_{R-1}=t$.  Moreover, for convenience, let us define, for $0\le \rage\le R-1$:
\begin{equation}\label{eq:hate}
  \cS_{\rage,\ell_\rage}=\sum_{\ell_{\rage+1}\in
    \stdf_{\ell_{\rage}}^{\rage}}\cdots \sum_{\ell_{R}\in \stdf_{\ell_{R-1}}^{R-1}}
  \prod_{j=\rage+1}^{R}\fm_{j-1,\ell_{j-1},\ell_j}.
\end{equation}
Note that, by~\eqref{eq:stating-pointmu} and since, by hypotheses $q\geq 5$:
\begin{equation}\label{eq:hereweare-almost}
\mu_{\ellz}(e^{i\sigma\bA})=\cS_{0,\ell_0}+\cO(\sigma \ve L_*).
\end{equation}
Also we define, for $1\le \rage \le R$,
\begin{equation}\label{eq:Phiragedef}
  \Phi_{\rage}(\theta)=\vei\int_0^{\ve S_{R-1}-\ve S_{\rage-1}}\cG_{\rage}(\theta,s,\sigma) ds
\end{equation}
where $\cG_\rage$ is defined in~\eqref{eq:chi-est-u}.  We will also use the operators
defined in~\eqref{e_defTransferOp}, with the potentials
$\Omega^{\rage,\Phi_{\rage+1}}_{\ell_\rage,j}=i\sigma\varpi^\rage_{\ell_\rage,j}+\ve
\Phi'_{\rage+1}(\thetaslkk{\rage}{L_\rage}) \etaExpsrage_{j,L_\rage}\ho$, where
$\varpi^\rage_{\ell_\rage,j}$ is defined in~\eqref{eq:potdef} (but
see~\eqref{eq:hov-compute} for a more convenient expression).  As in the previous section
we will use indifferently the notations
$\tO_{\ell_\rage,\rage,j}=\tO_{\ell_\rage,j}=\tO_{\thetaslkk{\rage}j,\Omega^{\rage,\Phi_{\rage+1}}_{\ell_\rage,j}}$
for such operators and similarly for all the corresponding related quantities.  To
simplify notations, let\footnote{ This is just a more convenient notation, limited to the
  present context, for the objects $\overline h_{0,L_k-1}$ and $\overline m_{0,L_k-1}$
  defined in Lemma~\ref{lem:product-bound-rough}.}
\begin{equation}\label{eq:obarh-obarm}
\begin{split}
  \overline h_{\ell_\rage}
  &=\sum_{k=0}^{L_\rage-1}\tOqh_{\ell_\rage,L_\rage-1}\cdots\tOqh_{\ell_\rage,k+1}h_{\thetaslkk{\rage}{k},
    \Omega^{\rage,\Phi_{\rage+1}}_{\ell_\rage,k}}\\
  \overline m_{\ell_\rage}%
  &=\sum_{k=0}^{L_\rage-1} m_{\thetaslkk{\rage}{k},\Omega^{\rage,\Phi_{\rage+1}}_{\ell_\rage,k}}
  \widehat \tOq_{\ell_\rage,k-1}\cdots \widehat\tOq_{\ell_\rage,0},\\
  \Gamma(\theta,\sigma)
  &=\Leb \,h_{\theta,i\sigma\ho},
\end{split}
\end{equation}
where $\widehat\tOq_{\ell_\rage,k}$ are the operators introduced in Lemma~\ref{lem:Taylor}
with the normalization  specified in Lemma~\ref{lem:product-bound-rough}.
Remark that, by~\eqref{e_partialhTheta},~\eqref{eq:D(theta,1)-est} and
Lemma~\ref{l_normalization}, $\Gamma\in\cC^2$ and
\begin{equation}\label{eq:partialGamma}
|\partial_\theta \Gamma|+|\partial^2_\theta \Gamma|\leq \Const |\sigma|.
\end{equation}
\begin{lem}\label{sl_inductive}
  There exists $\ve_0>0$ such that, for all $\ve\in(0,\ve_0)$,
  $L\leq \Const \ve^{-3\delta_*}$, $\sigma\in \J1$, $q\geq 3$,
  $\rage\in\{0,\cdots, R-1\}$, $R\leq \Const\vei L_*\invr$, and
  $\ve\sigma^2L(S_{R-1}-S_{\rage-1})\leq \Const$
  \begin{equation}\label{eq:final-palla-hope}
    \cS_{\rage,\ell_\rage}=%
    \overline m_{\ell_\rage}\left(e^{\Phi_\rage\circ G_{\ell_\rage}}%
      \rrho_{\ell_\rage}\right)\Gamma(\thetaslkk{\rage}{S_{R-1}-S_{\rage-1}},\sigma) +\overline\cE_{\rage,\ell_\rage},
  \end{equation}
  where $\overline\cE_{\rage,\ell_\rage}$ is a remainder term satisfying the following
  bound:\footnote{ We use the convention that the inner sums equal $1$ when $k=\rage$.}
  \begin{align*}
    |\overline \cE_{\rage,\ell_\rage}|
    &\leq\Const\sum_{k=\rage}^{R- 1}e^{- (S_{R-1}-S_{k-1}) \sigma^2\bchimin}\sum_{\ell_{\rage+1}\in \stdf_{\ell_{\rage}}^{\rage}} \cdots\sum_{\ell_{k}\in \stdf_{\ell_{k-1}}^{k-1}} \prod_{j=\rage+1}^{k}|\fm_{j,\ell_{j},\ell_{j+1}}|\\
    &\phantom\le\times\left(|\sigma|\ve L_*^3+\ve^2L_*^3+\ve L_*^2|\sigma|^3 (R-\rage)\right),
  \end{align*}
  with $\bchimin=\min_\theta\bVar(\theta)$.
\end{lem}
\begin{proof}
  Let $\Phi^+_{\rage}(\theta) = \max_{\theta}\Re\,\Phi_\rage(\theta)$;
  then~\eqref{e_chiEstimates-bound-from-above} implies that
  \begin{align}\label{eq:phirageboud}
    \Phi^+_{\rage} &\leq -\const\sigma^2\bchimin( S_{R-1}-S_{\rage-1}).
  \end{align}
  Next, we proceed to prove~\eqref{eq:final-palla-hope} by backward induction.  The base
  step $\rage=R-1$ follows from Proposition~\ref{prop:not-so-trivial-oneBlock}, with
  $\Phi=\Phi_{R} = 0$. Indeed,
    \begin{align}\label{eq:vienna-sigh1}
      \cS_{R-1, \ell_{R-1}}%
      &=\Leb\,\sum_{\ell_{R}\in \stdf_{\ell_{R-1}}^{R-1}} \fm_{R-1,\ell_{R-1},\ell_R}\mathring\rho_{\ell_R}\notag\\
      &= \bXss_{\ell_{R-1}}
        \Leb\,\overline h_{\ell_{R-1}} \overline m_{\ell_{R-1}}\,
        \rrho_{\ell_{R-1}}+\Leb[\cE^{**}_{\ell_{R-1},1}+\cE^{**}_{\ell_{R-1},2}]\\
      &= \Leb\,\overline h_{\ell_{R-1}} \overline m_{\ell_{R-1}}
        \left(\bXss_{\ell_{R-1}}
        \rrho_{\ell_{R-1}}\right)+\cO(\ve^2L_{R-1}^3+|\sigma|\ve L_{R-1}^3).\notag
    \end{align}
    Next, using the orthogonality relations between eigenvector and the
    operators $\tOqh$:
  \begin{equation}\label{eq:m-mandh-h}
    \begin{split}
      \Leb\,\overline h_{\ell_{R-1}}&=\Leb\, h_{\thetaslkk{R-1}{L_{R-1}-1},\Omega^{{R-1},0}_{\ell_{R-1},L_{R-1}-1}}\\
      &\phantom{=}+ \sum_{k=0}^{L_{R-1}-2}\left(\Leb\,-m_{\thetaslkk{R-1}{L_{R-1}-1},\Omega^{{R-1},0}_{\ell_{R-1},L_{R-1}-1}}\right)\tOqh_{\ell_{R-1},L_{R-1}-1}\\
      &\phantom{=}\times \cdots\tOqh_{\ell_{R-1},k+1}\left(h_{\thetaslkk{R-1}{k},\Omega^{{R-1},0}_{\ell_{R-1},k}}-h_{\thetaslkk{R-1}{k+1},\Omega^{{R-1},0}_{\ell_{R-1},k+1}}\right),
    \end{split}
  \end{equation}
  Recalling definitions~\eqref{e_defTransferOp} and~\eqref{eq:potdef} we see that
  $\Omega^{{R-1},0}_{\ell_{R-1},L_{R-1}-1}=i\sigma\ho$.  Then, by
  equations~\eqref{eq:obarh-obarm},\eqref{eq:m-mandh-h},
  Sub-Lemmata~\ref{sublem:perturbation-fancy},~\ref{sublem:Q-product}
  and~\eqref{e_estimatePartialm}, we have
  \begin{equation}\label{eq:gamma-leb-h}
    \Leb\,\overline h_{\ell_{R-1}}=\Gamma(\thetaslkk{R-1}{L_{R-1}-1},\sigma)+\cO(|\sigma| \ve).
  \end{equation}
  In addition, by~\eqref{eq:chi-est-u} we have
  \begin{align*}
    \bXss_{\ell_{R-1}}= \expo{\Phi_{R-1}(\theta^*_{\ell_{R-1}})}+\cO(|\sigma|\ve L_{R-1}).
  \end{align*}
  On the other hand, since $\Phi_{R-1}$ is the integral of
  $\bchi(t-S_{R-2},s,\theta^*_{\ell_{R-1}},\bar\theta(s,\theta^*_{\ell_{R-1}}))$,\footnote{
    Which is the logarithm of the maximal eigenvalue associated to the potential
    $i\sigma\BOmega(t-S_{R-2},s,\theta^*_{\ell_{R-1}},\cdot,\bar\theta(s,\theta^*_{\ell_{R-1}}))$
    with respect to the dynamics $f_{\bar\theta(s,\theta^*_{\ell_{R-1}})}(\cdot)$,
    see~\eqref{eq:Bomega-def} and related comments.}  and since $\BOmega$ is zero-average
  with respect to the SRB measure (see~\eqref{eq:Bomega-def}), we can
  use~\eqref{e_estimatesecondpartchi0} to obtain:
  \begin{equation}\label{eq:chider-phi}
    \|\partial_\theta\bchi(\sigma,T,s,\vf,\theta)\|\leq \Const\sigma^2.
  \end{equation}
  On the other hand,~\eqref{eq:pvarrhochi} and Lemma~\ref{l_derivativess} similarly imply
  \begin{equation}\label{eq:chider-theta}
    \|\partial_\vf\bchi(\sigma,T,s,\vf,\theta)\|\leq \Const\sigma^2.
  \end{equation}
  The above equations yield, for any
  $x\in[a_{\ell_{R-1}},b_{\ell_{R-1}}]$,
  \begin{align*}
    |\Phi_{R-1}(\theta^*_{\ell_{R-1}})-\Phi_{R-1}\circ G_{\ell_{R-1}}(x)|\leq\Const L_{R-1}\sigma^2\ve,
  \end{align*}
  hence
  \begin{equation}\label{eq:vienna-sigh2}
    \bXss_{\ell_{R-1}}= \expo{\Phi_{R-1}\circ G_{\ell_{R-1}}}+\cO(\sigma\ve L_{R-1}).
  \end{equation}
  Collecting equations~\eqref{eq:vienna-sigh1},~\eqref{eq:gamma-leb-h}
  and~\eqref{eq:vienna-sigh2} proves the case $\rage=R-1$.

  Next, let us assume that~\eqref{eq:final-palla-hope} holds for
  $\rage+1\leq R-1$, then
  \begin{align*}
      \cS_{\rage,\ell_\rage}
    &=\sum_{\ell_{\rage+1}\in \stdf_{\ell_{\rage}}^{\rage}} \fm_{\rage,\ell_{\rage},\ell_{\rage+1}}\cS_{\rage+1,\ell_{\rage+1}}
      =\sum_{\ell_{\rage+1}\in \stdf_{\ell_{\rage}}^{\rage}} \fm_{\rage,\ell_{\rage},\ell_{\rage+1}}\\
      &\phantom = \times \overline m_{\ell_{\rage+1}}\left(\expo{\Phi_{\rage+1}\circ G_{\ell_{\rage+1}}}\mathring \rho_{\ell_{\rage+1}}\right)\Gamma(\thetaslkk{\rage+1}{S_{R-1}-S_{\rage}},\sigma) \\
      &\phantom = +\sum_{\ell_{\rage+1}\in \stdf_{\ell_{\rage}}^{\rage}} \fm_{\rage,\ell_{\rage},\ell_{\rage+1}}\overline\cE_{\rage+1,\ell_{\rage+1}}.
  \end{align*}
  To continue, it is necessary to remove the dependence of $\overline m_{\ell_{\rage+1}}$
  and $\Gamma$ on $\ell_{\rage+1}$ in such a way that we can apply
  Proposition~\ref{prop:not-so-trivial-oneBlock}. This will be done in two steps: first
  notice that for $(x,\theta)$ in the support of
  $\mu_{\ell_\rage}$,\footnote{ In fact, using large deviations, it is possible to have a
    better estimate with large probability. We will not push this possibility as it is not
    needed for the level of precision we are currently after.}
  \begin{equation}\label{eq:theta-bartheta-baby}
    |\theta_{L_\rage}-\bar\theta(\ve L_\rage,\avgtheta{\ell_{\rage}}) |\leq \Const \ve L_\rage.
  \end{equation}
  Accordingly, using~\eqref{eq:partialGamma},
  \begin{align*}
    |\Gamma(\thetaslkk{\rage+1}{S_{R-1}-S_{\rage}},\sigma)-\Gamma(\thetaslkk{\rage}{S_{R-1}-S_{\rage-1}},\sigma)|\leq
    \Const \ve|\sigma| L_\rage,
  \end{align*}
  thus
  \begin{equation}\label{eq:S-step-one}
    \begin{split}
      &\cS_{\rage,\ell_\rage}=e^{\Phi^+_{\rage+1}}\cO(\ve|\sigma| L_\rage)+\sum_{\ell_{\rage+1}\in \stdf_{\ell_{\rage}}^{\rage}} \fm_{\rage,\ell_{\rage},\ell_{\rage+1}}\overline\cE_{\rage+1,\ell_{\rage+1}}\\
      & +\sum_{\ell_{\rage+1}\in \stdf_{\ell_{\rage}}^{\rage}} \hskip-.2cm\overline
      m_{\ell_{\rage+1}}\left(\expo{\Phi_{\rage+1}\circ
          G_{\ell_{\rage+1}}}\fm_{\rage,\ell_{\rage},\ell_{\rage+1}}\mathring
        \rho_{\ell_{\rage+1}}\right)\Gamma(\thetaslkk{\rage}{S_{R-1}-S_{\rage-1}},\sigma).
    \end{split}
  \end{equation}
  Before continuing we need a bound on $\Phi_\rage$.
  \begin{sublem}\label{sublem:Phi-rage-bound} For any $j\in\{0,\cdots, R-1\}$ we have
    \begin{align*}
      \|\Phi'_j\|\nc0&\le \Const \sigma^2 (S_{R-1}-S_{j-1})\\
      \|\Phi'_j\|\nc1&\le \Const |\sigma| (S_{R-1}-S_{j-1}) .
    \end{align*}
  \end{sublem}
  \begin{proof}
    By equations~\eqref{eq:chider-theta} and~\eqref{eq:chider-phi} it follows
    \begin{align*}
      |\Phi'_j(\theta)|
      &\le \Const \vei\int_0^{\ve S_{R-1}-\ve S_{j-1}}|\partial_\theta\bchi(\sigma, t-s-\ve S_{j-1},\bar\theta(s,\theta),\bar\theta(s,\theta))| ds\\
      &\phantom\le+ \Const \vei\int_0^{\ve S_{R-1}-\ve S_{j-1}}|\partial_\vf\bchi(\sigma, t-s-\ve S_{j-1},\bar\theta(s,\theta),\bar\theta(s,\theta))| ds\\
      &\phantom\le \Const \sigma^2 (S_{R-1}-S_{j-1}).
    \end{align*}
    This proves the first inequality, the second is obtained similarly by using the above
    formulae,~\eqref{e_derivative-varrho} and Lemma~\ref{eq:theta-derivative-all} (in
    particular~\eqref{e_estimatesecondpartchi}).
  \end{proof}
  Note that Lemma~\ref{lem:Taylor}\ref{i_chiEstimate} implies that
  $\Phi_\rage^+\le-\Const\sigma^2(R-\rage) L_*$.  Moreover,
  Sub-Lemma~\ref{sublem:Phi-rage-bound}, together with our hypotheses on $\rage$, implies
  that the hypotheses of Proposition~\ref{p_transferOperators},
  Lemma~\ref{lem:product-bound-rough} and Lemma~\ref{prop:not-so-trivial-oneBlock}, are
  all satisfied for $\Phi=\Phi_\rage$. We can therefore apply all such results to the
  present situation.

  Observe, moreover, that Sub-Lemma~\ref{sublem:Phi-rage-bound} and the definition of
  $\Omega^{\rage+1,\Phi_{\rage+2}}_{\ell_{\rage+1},k}$ imply:
  \begin{align*}
    \|\Omega^{\rage+1,\Phi_{\rage+2}}_{\ell_{\rage+1},k}-\Omega^{\rage+1,\Phi_{\rage+2}}_{\ell_{\rage+1},k-1}\|\nc1
    &\le \Const (\ve|\sigma|+\ve^2\sigma^2L_*(R-\rage))\\
    &\le \Const \ve|\sigma|,
  \end{align*}
  where we used the fact that by definition $(R-\rage) < R = \vei L_*\invr$ and that since
  $\sigma\in\J1$ we have $\sigma^2 < |\sigma|$.  Similarly, using~\eqref{eq:hov-Xi} we  have
  \begin{align*}
    \|\Omega^{\rage+1,\Phi_{\rage+2}}_{\ell_{\rage+1},0}-\Omega^{\rage,\Phi_{\rage+1}}_{\ell_{\rage},L_\rage-1}\|\nc1
    &\leq|\sigma| \left| \etaExph(t-\ve S_\rage,\avgtheta{\ell_{\rage+1}})-\etaExph(t-\ve S_\rage,\thetaslkk{\rage}{L_\rage})\right|\|\ho\|\nc1\\
    &\phantom = +\Const (\ve^2|\sigma|L_*+\ve^2 |\sigma| L_*^2(R-\rage))\\
    &\leq\Const (\ve|\sigma| L_*+\ve^2 |\sigma| L_*^2(R-\rage))\\
    &\leq\Const \ve|\sigma| L_*
  \end{align*}
  We can now take care of $\overline m_{\ell_{\rage+1}}$: observe that, by
  applying~\eqref{e_estimatesPartialRho-m} and recalling footnote~\ref{fot:lower} and Lemma~\ref{lem:m-theta-BV}
  \begin{align}
    \notag
    \overline m_{\ell_{\rage+1}}(\vf)
    &=m_{\thetaslkk{\rage+1}{0},\Omega^{\rage+1,\Phi_{\rage+2}}_{\ell_{\rage+1},0}}(\vf)\\\notag
    &\phantom{=}+\sum_{k=1}^{L_{\rage+1}-1}
      \left(m_{\thetaslkk{\rage+1}{k},\Omega^{\rage+1,\Phi_{\rage+2}}_{\ell_{\rage+1},k}}
            -m_{\thetaslkk{\rage+1}{k-1},\Omega^{\rage+1,\Phi_{\rage+2}}_{\ell_{\rage+1},k-1}}\right)
      \left(\widehat \tOq_{\ell_{\rage+1},k-1}\cdots \widehat\tOq_{\ell_{\rage+1},0}\vf\right)\\\notag
    &= m_{\thetaslkk{\rage+1}{0},\Omega^{\rage+1,\Phi_{\rage+2}}_{\ell_{\rage+1},0}}(\vf)+\cO(\ve|\sigma|(\log\sigma)^2)\|\vf\|_\BV;
  \end{align}
and applying once again~\eqref{e_estimatesPartialRho-m} and Lemma~\ref{lem:m-theta-BV} together
      with~\eqref{eq:theta-bartheta-baby}
\begin{equation}
\begin{split}
       \overline m_{\ell_{\rage+1}}(\vf)=&\; m_{\thetaslkk{\rage}{L_\rage-1},\Omega^{\rage,\Phi_{\rage+1}}_{\ell_\rage,L_\rage-1} }(\vf)
       +\cO\big(\ve|\sigma|(\log\sigma)^2L_* \big)\|\vf\|\nBV\label{eq:change-m}.
\end{split}
\end{equation}
  We can now continue with the estimate we left at~\eqref{eq:S-step-one} and
  obtain
  \[
    \begin{split}
      \cS_{\rage,\ell_\rage} &=e^{\Phi^+_{\rage+1}}\cO(\ve|\sigma|(\log\sigma)^2
      L_\rage) +\sum_{\ell_{\rage+1}\in \stdf_{\ell_{\rage}}^{\rage}}
      \fm_{\rage,\ell_{\rage},\ell_{\rage+1}}\overline\cE_{\rage+1,\ell_{\rage+1}}\\ %
      &\phantom =  + m_{\thetaslkk{\rage}{L_\rage-1},\Omega^{\rage,\Phi_{\rage+1}}_{\ell_\rage,L_\rage-1} }\left(\sum_{\ell_{\rage+1}\in \stdf_{\ell_{\rage}}^{\rage}}e^{\Phi_{\rage+1}\circ G_{\ell_{\rage+1}}}\fm_{\rage,\ell_{\rage},\ell_{\rage+1}}\mathring \rho_{\ell_{\rage+1}}\right)\\
      &\phantom = \times\Gamma(\thetaslkk{\rage}{S_{R-1}-S_{\rage-1}},\sigma).
    \end{split}
  \]
  Finally we can apply Proposition~\ref{prop:not-so-trivial-oneBlock} with $\bar s=S_{R-1}-S_{\rage-1}$:
  \[
    \begin{split}
      \cS_{\rage,\ell_\rage}= &e^{\Phi^+_{\rage+1}}\cO\left(|\sigma| (\log\sigma)^2\ve L_\rage \right)
      +\sum_{\ell_{\rage+1}\in \stdf_{\ell_{\rage}}^{\rage}}
      \fm_{\rage,\ell_{\rage},\ell_{\rage+1}}\overline\cE_{\rage+1,\ell_{\rage+1}}
      \\
      &+e^{\vei\int_0^{\ve L_\rage}\cG_\rage(\avgtheta{\ell_\rage},s,\sigma)\deh s}\Gamma(\thetaslkk{\rage}{S_{R-1}-S_{\rage-1}},\sigma) \\
      &\times m_{\thetaslkk{\rage}{L_\rage-1},\Omega^{\rage,\Phi_{\rage+1}}_{\ell_\rage,L_\rage-1} }(\overline h_{\ell_\rage})\overline m_{\ell_\rage}\left(e^{\Phi_{\rage+1}(\bar\theta_{\ell_\rage,L_\rage})}\mathring \rho_{\ell_\rage}\right)\\
      &+\cO\left(m_{\avgtheta{\ell_\rage},\Omega^{\rage,\Phi_{\rage+1}}_{\ell_\rage,L_\rage-1}
        }(\cE^{**}_{\ell_\rage,1})\right)
      +\cO\left(m_{\avgtheta{\ell_\rage},\Omega^{\rage,\Phi_{\rage+1}}_{\ell_\rage,L_\rage-1}
        }(\cE^{**}_{\ell_\rage,2})\right),
    \end{split}
  \]
  where $\cG_\rage$ is defined in~\eqref{eq:chi-est-u}.
  Observe that, by definition:
  \begin{align*}
    \Phi_{\rage+1}(\bar\theta_{\ell_\rage,L_\rage})
    &= \vei\int_{0}^{\ve(S_{R-1}-S_{\rage})}
      \bchi(\sigma,t- s-\ve
      S_{\rage}),0,\bar\theta(s,\bar\theta_{\ell_\rage,L_\rage}),\bar\theta(s,\bar\theta_{\ell_\rage,L_\rage})ds\\
    &= \vei\int_{\ve L_\rage}^{\ve(S_{R-1}-S_{\rage-1})}
      \bchi(\sigma,t- s'-\ve S_{\rage-1},0,\bar\theta(s',G_{\ell_\rage}),\bar\theta(s',G_{\ell_\rage}))ds'
  \end{align*}
  Hence, using~\ref{lem:Taylor}\ref{i_chiEstimate}, we can write
  \[
    \begin{split}
      &\Phi_{\rage+1}(\bar\theta_{\ell_\rage,L_\rage})+\vei\int_0^{\ve L_\rage}\cG_\rage(\avgtheta{\ell_\rage},s,\sigma)\deh s\\
      &=\vei\int_{\ve L_\rage}^{\ve(S_{R-1}-S_{\rage-1})}\bchi(\sigma, t-s-\ve S_{\rage-1},0,\bar\theta(s,G_{\ell_\rage}),\bar\theta(s,G_{\ell_\rage})) \deh s\\
      &\phantom{=}
      +\vei\int_0^{\ve L_\rage}\bchi(\sigma, t-s-\ve S_{\rage-1},0,\bar\theta(s,G_{\ell_\rage}),\bar\theta(s,G_{\ell_\rage}))\deh s +\cO(\ve |\sigma|L_\rage)\\
      &=\Phi_\rage\circ G_{\ell_{\rage}}+\cO(\ve|\sigma| L_*).
    \end{split}
  \]
  Finally, recalling~\eqref{eq:obarh-obarm},
  \[
    m_{\thetaslkk{\rage}{L_\rage-1},\Omega^{\rage,\Phi_{\rage+1}}_{\ell_\rage,L_\rage-1} }(\overline h_{\ell_\rage})=1
  \]
  from which the lemma readily follows by using Lemma~\ref{lem:m-estimate}
  (observe\footnote{\label{f_sigma100} Since the choice of the power $100$ in
    Lemma~\ref{lem:m-estimate} is arbitrary (see Footnote~\ref{f_sigma100-appendix}), one
    could in principle work with values of $\delta_*$ smaller than $\efrac1{99}$, if
    needed.} that $\sigma^{100} < \sigma\ve$ since $\delta_* > \efrac1{99}$ and
  $|\sigma| < \ve^{\delta_*}$) and the bounds on $\cE^{**}_{\ell_\rage,i}$ provided in
  Proposition~\ref{prop:not-so-trivial-oneBlock}.
\end{proof}
We are now, finally, ready to prove the very last missing piece in our argument.
\begin{proof}[{\bf Proof of Proposition~\ref{prop:not-so-trivial-1}}]
  The basic idea is to apply Lemma~\ref{sl_inductive}.  Unfortunately,
  Lemma~\ref{sl_inductive} holds only under the additional hypothesis
  $\ve\sigma^2L_* (S_{R-1}-S_{\rage-1})\leq \Const$. Note that if
  $|\sigma|\leq \ve^{2\delta_*}$, then
  \[
    \ve\sigma^2L_* (S_{R-1}-S_{\rage-1})\leq\Const \ve^{4\delta_*-3\delta_*}\leq \Const \ve^{\delta_*}.
  \]
  Yet, if $|\sigma|\in [\ve^{2\delta_*},\ve^{\delta_*}]$, we can apply Lemma~\ref{sl_inductive} only for
  \begin{equation}\label{eq:S-bound-fuck}
    (S_{R-1}-S_{\rage-1})\leq\Const \ve^{-1+\delta_*}.
  \end{equation}
  So, choose $\rage_\ve$ such that $S_{R-1}-S_{\rage_\ve-1}=\Const \ve^{-1+\delta_*}$.
  Then, for $|\sigma|\in [\ve^{2\delta_*},\ve^{\delta_*}]$, we can
  rewrite~\eqref{eq:stating-pointmu}, with $q\geq 5$, and~\eqref{eq:hate} as follows
    \begin{align*}
    \mu_\ellz(e^{i\sigma\bA})=\sum_{\ell_{1}\in
       \stdf_{\ell_{0}}^{0}}\cdots\sum_{\ell_{\rage_\ve}\in
      \stdf_{\ell_{\rage_\ve-1}}^{\rage_\ve-1}}
    \prod_{j=1}^{\rage_\ve}\fm_{j-1,\ell_{j-1},\ell_j}\cS_{\rage_\ve,\ell_{\rage_\ve}}+\cO(\ve\sigma
    L_*).
    \end{align*}
  Hence, by~\eqref{eq:trivial-fm-estimate}, we can bound
  \begin{align*}
    \left|\mu_\ellz(e^{i\sigma\bA})\right|\leq
    \sup_{\ell_{\rage_\ve}}\left|\cS_{\rage_\ve,\ell_{\rage_\ve}}\right|+\Const |\sigma|\ve
    L_*.
    \end{align*}
    We can now apply Lemma~\ref{sl_inductive} and~\eqref{e_chiEstimates} to write
    \begin{align*}
      \left|\cS_{\rage_\ve,\ell_{\rage_\ve}}\right|\leq \Const e^{-\const(S_{R-1}-S_{\rage_\ve-1})\sigma^2}+|\overline\cE_{\rage_\ve,\ell_{\rage_\ve}}|
      \leq \Const e^{-\const \ve^{-1+3\delta_*}}+|\overline\cE_{0,\ell_{0}}|.
    \end{align*}
  Collecting the above facts yields
  \begin{align*}
    |\mu_{\ellz}(e^{i\sigma\bA})|&\le
    \overline\cE_{0,\ellz}+\cO\left(\ve \sigma L_*+e^{-\const \ve^{-1+3\delta_*}}\right)\\
    &\le\overline\cE_{0,\ellz}+\cO\left(\ve \sigma L_*\right).
  \end{align*}
  In particular, since for $|\sigma|\in[\ve^{2\delta_*},\ve^{\delta_*}]$:
  \begin{align*}
    \expo{- \frac{\sigma^2}{2\ve}\Var_t^2(\theta_0)} \le \Const e^{-\const
    \ve^{-1+4\delta_*}}\le\Const\ve|\sigma|L_*,
  \end{align*}
  we have:
  \begin{align}\label{eq:not-so-small-sigma}
    \mu_{\ellz}(e^{i\sigma\bA})
&=\expo{- \frac{\sigma^2}{2\ve}\Var_t^2(\theta_0)} +
    \overline\cE_{0,\ellz}+\cO\left(\ve \sigma L_*\right).
  \end{align}
  Next, we consider the case $|\sigma|\leq \ve^{2\delta_*}$, hence Lemma~\ref{sl_inductive}
  holds with $\rage=0$.  Accordingly, since $q\ge5$, we can
  apply~\eqref{eq:hereweare-almost} that implies:
  \begin{align*}
    \mu_{\ellz}(e^{i\sigma\bA})
    &=\overline m_{\ellz}(e^{\Phi_0\circ G_{\ell_0}} \rrho_{\ellz})\Gamma(\thetaslkk{0}{t\vei},\sigma) +\overline\cE_{0,\ellz}+\cO(\sigma \ve L_*).
  \end{align*}
  Note that definition~\eqref{eq:Phiragedef} and
  equations~\eqref{eq:chider-phi},~\eqref{eq:chider-theta}, using~\ref{sublem:Phi-rage-bound}, give
    \begin{align*}
    \|\Phi_0\circ G_{\ell_0}-\Phi_0(\avgtheta{\ellz})\|\nc{1}\leq \Const \sigma^2\leq \Const\ve^{2\delta_*}
  \end{align*}
  and, by~\eqref{e_chiEstimates} and recalling definitions~\eqref{eq:Deltabloks} and~\eqref{e_variancelclt},
  \[
    \begin{split}
      \ve\Phi_0(\avgtheta{\ellz})
      =&- \frac{\sigma^2}{2}\int_0^t \etaExph(t- \ve-s,\bar\theta(s+\ve,\avgtheta{\ellz}))^2\bVar(\bar\theta(s,\avgtheta{\ellz}))ds+ \cO(\sigma^3)\\
      =&- \frac{\sigma^2}{2}\Var_t^2(\theta_0)+ \cO(\sigma^3+\sigma^2\ve).
    \end{split}
  \]
  In addition, by computations similar to~\eqref{eq:change-m}
  and using Lemma~\ref{lem:m-estimate}, we have
  \begin{align*}
    \overline m_{\ellz}\,\mathring\rho_{\ellz}%
    =m_{\thetaslz,\Omega^{0,\Phi_{0}}_{\ell_0,L_0-1}
    }(\mathring\rho_{\ellz})+\cO(\ve\sigma(\log\sigma)^{2}L_*)=\Leb(\mathring\rho_{\ellz})+\cO(\sigma\log|\sigma|^{-1
   }).
  \end{align*}
  Also, recalling the definition~\eqref{eq:obarh-obarm} and using Lemma~\ref{l_normalization} we have
  \[
    \Gamma(\thetaslkk{0}{t\vei},\sigma)=1+\cO(\sigma).
  \]
  Accordingly, recalling~\eqref{eq:not-so-small-sigma}, for all $\sigma\in \J1$ we have
  \begin{align*}
      \mu_{\ellz}(e^{i\sigma\bA})&=\expo{-
        \frac{\sigma^2}{2\ve}\Var_t^2(\theta_0)}\left(1+\cO(\sigma^3\vei+\sigma^2+\sigma\log|\sigma|^{-1})\right)
      +\overline\cE_{0,\ellz}+\cO(\sigma \ve L_*).
  \end{align*}
  Note that
  \begin{align*}
    &\frac 1{2\pi\ve} \int_{\J1}\expo{- \frac{\sigma^2}{2\ve}\Var_t^2(\theta_0)}
    (\sigma^3\vei+\sigma^2+\sigma\log|\sigma|^{-1})\deh \sigma \\
    &=\frac 1{2\pi} \int_{\bR}\expo{- \frac{\eta^2}{2}\Var_t^2(\theta_0)}
    (\eta^3+\ve\eefrac12\eta^2+\eta\log|\eta\veh|^{-1})\deh \eta = \cO( \log\vei).
  \end{align*}
  To conclude the proof of the proposition it then suffices to estimate the integral of
  $\overline\cE_{0,\ellz}$. This is easily done by noting that, for all $p,q\in\bN$,
  \begin{align*}
    \frac 1{2\pi\ve} \int_{\J1}\sum_{k=0}^{R-1}e^{-\const\sigma^2
    (R-k) L}(R-k)^qL^q |\sigma|^p\deh \sigma
    &\le \frac{\Const}{\ve}\sum_{k=1}^{R}\int_{\bR} e^{-\const \eta^2}\frac{|\eta|^p}{(k L)^{\efrac{(p+1)}2-q}} \deh \eta\\
    & \leq \frac{\Const}{ \ve L^{\efrac{(p+1)}2-q}}\sum_{k=1}^{R} k^{-(p+1)/2+q}.
  \end{align*}
  Thus
  \[
    \int_{\J1}\sum_{k=0}^{R-1} \frac {|\sigma|^p(R-k)^qL^q}{2\pi\ve \expo{\const\sigma^2
        (R-k) L}}\deh \sigma\leq \begin{cases}\Const \ve^{-2-q+\efrac {(p+1)}2}L^{-1}&\textrm{ for }p < 1+2q\\
      \Const\vei L^{-1}\log(\vei)&\textrm{ for }p=1+2q\\
      \Const\vei L^{q-\efrac{(p+1)}2}&\textrm{ for }p>1+2q.
    \end{cases}
  \]
  We can now apply the above estimates to compute the integrals of the various
  contributions to $\overline\cE_{0,\ellz}$ obtaining

  \begin{align*}
    &\frac 1{2\pi\ve} \int_{\J1}\sum_{k=0}^{R-1}e^{-\const\sigma^2
      (R-k) L}[\sigma \ve L_*^3] \deh\sigma=\cO(L_*^2\log\ve^{-1})&&&&\text{\tiny $(p,q)=(1,0)$}\\
    &\frac 1{2\pi\ve} \int_{\J1}\sum_{k=0}^{R-1}e^{-\const\sigma^2
      (R-k) L}[\ve^2 L_*^3] \deh\sigma=\cO(\ve\eefrac12L_*^2)&&&&\text{\tiny $(p,q)=(0,0)$}\\
    &\frac 1{2\pi\ve} \int_{\J1}\sum_{k=0}^{R-1}e^{-\const\sigma^2
      (R-k) L}[\sigma^3 \ve L_*^2(R-\rage)] \deh\sigma=\cO(\log\vei)&&&&\text{\tiny $(p,q)=(3,1)$}\\
  \end{align*}
  which prove the proposition.
\end{proof}
\appendix
\section{Spectral theory for transfer operators: a toolbox}\label{subsec:transfer}
In this appendix we collect some known and less known (or possibly unknown) results on
transfer operators that are used in the main part of this paper.  Let us fix $r\ge 3$ and
let $f(\cdot,\theta) = f_\theta\in\cC^r(\bT,\bT)$, with $\theta\in\bT$, be a one parameter
family of orientation preserving expanding maps (i.e., there exists $\lambda>1$ such that
$\inf_{x,\theta}f'_\theta(x)\ge\lambda$).  Let
$\Omega(\cdot,\theta) = \Omega_\theta\in\cC^{r}(\bT,\bC)$ be a family of
\emph{potentials}.  We further assume some regularity\footnote{ The requirements on
  regularity are not optimal, but rather reflect our case of interest.} in $\theta$; more
precisely we require that $f\in\cC^r(\bT^2,\bT)$ and that
$\partial_x\Omega\in\cC^{r-1}(\bT^2,\bC)$.  For any $\paramL\in\bR$ we can then consider
the family of operators $\tO_{\theta,\paramL\Omega}$ defined as:
\begin{equation}\label{eq:defLOp}
  \left[\tO_{\theta,\paramL\Omega}\, g\right](x)=\sum_{y\in f_\theta\invr(x)}
    \frac{e^{\paramL\Omega_\theta(y)}}{f_\theta'(y)}g(y).
\end{equation}
It is well know that the spectrum of such operators depends drastically on the space on
which they act.  We will be interested in $\BV$, $W^{s,1}$ and $\cC^{s}$, for $s\le r-1$.
\begin{irem}
  We will use $\cC^s$ (and similarly for the other spaces) as a shorthand notation for
  $\cC^s(\bT)$ (which in turn is a shorthand notation for $\cC^s(\bT,\bR)$).  When we need
  to consider functions defined on $\bT^2$ we will write explicitly $\cC^s(\bT^2)$.
\end{irem}

\subsection{General facts}\ \newline
Let us start with a useful result for the case of real potentials.
\begin{lem}\label{lem:quasicompact}
  If $\Omega$ is real, then for any $\paramL\in\bR, \theta\in \bT$, the operator
  $ \tO_{\theta,\paramL\Omega}:\cC^1\to\cC^1$ is of Perron--Frobenius type.  That is, it
  has a simple maximal eigenvalue $e^{\chi\thvs}$ with left and right eigenvectors that we
  denote with $m\thvs$ and $h\thvs$ (respectively), normalized so that
  $m\thvs( h\thvs)=1$.

  In addition, $m\thvs$ is a positive measure; $ h\thvs>0$ and
  \begin{align*}
    \|(\log h\thvs)'\|_\infty\leq \Const (|\paramL|\|\Omega_\theta'\|_\infty+1).
  \end{align*}
  Also, the spectral gap is continuous in $\paramL,\theta$ and the leading eigenvalue and
  eigenprojector are analytic in $
  \paramL$ and differentiable in $\theta$.
\end{lem}
\begin{proof}
  The statement could be proven by reducing the system to symbolic dynamics and then using
  results on the induced transfer operator.  Yet, a much more efficient and direct proof
  can be obtained by the Hilbert metric technique used, e.g., in
  \cite[Section~2]{Liverani95a} or \cite{liverani95b}.  Namely, consider the cone
  $\cK_a=\{g\in\cC^1\;:\; |g'(x)|\le a g(x), \forall x\in\bT\}$.  Since
  \begin{equation}\label{eq:cone-est}
    \frac{\deh}{\deh x}\tO\thvs g =%
    \tO\thvs\left(\frac{g'}{f_\theta'}+
      \paramL\frac{g\Omega_\theta'}{f_\theta'}-\frac{gf_\theta''}{{(f_\theta')}^2}\right),
  \end{equation}
  it follows that
  \begin{align*}
    \left|\frac{\deh}{\deh x}\tO\thvs g\right|\le
    \lambda^{-1}\{a+|\paramL|\|\Omega_\theta'\|_\infty+\lambda\Const\}\tO\thvs g.
  \end{align*}
  Hence, for any $\varrho\in(\lambda\invr,1)$, $\tO\thvs\cK_a\subset \cK_{\varrho a}$
  provided
  \begin{equation}\label{eq:a-estimate}
    a\ge {(\varrho\lambda-1)}\invr(|\paramL|\|\Omega_\theta'\|_\infty+\lambda\Const).
  \end{equation}
  A simple computation shows that the diameter $\Delta$, computed in the Hilbert metric
  $\Theta$ determined by $\cK_a$, of the image is bounded by
  $2\frac{1+\varrho}{1-\varrho}e^a$.\footnote{ It suffices to compute the distance of a
    function from the constant function and recall that, for $f,g\in\cK_a$, $\Theta(f,g)$
    is defined as $\log\frac\mu\lambda$ where $\mu$ is the $\inf$ of the $\alpha$ such
    that $\alpha f-g\in\cK_a$ and $\lambda$ is the $\sup$ of the $\beta$ such that
    $f-\beta g\in \cK_a$.}  From this fact and Birkhoff Theorem~\cite[Theorem
  1.1]{Liverani95a} it follows that $\tO\thvs$ contracts the Hilbert metric by a factor
  $\tanh \frac{\Delta}4$.  Also, notice that if $f-g, f+g\in\cK_a$, then
  $\|f\|_{L^\infty}\ge \|g\|_{L^\infty}$.  Accordingly,~\cite[Lemma 1.3]{Liverani95a}
  implies that, if $\|f\|_{L^\infty}=\|g\|_{L^\infty}$, then
  \[
    \|f-g\|_{L^\infty}\le \left(e^{\Theta(f,g)}-1\right)\|f\|_{L^\infty}.
  \]
  Next, let $f,g\in \cK_a$ with $\Leb(f)=\Leb(g)=1$.  Then $e^{-a}\le f, g\le e^a$ hence
  $e^{-a} \tO\thvs^n1\le \tO\thvs^nf, \tO\thvs^n g\le e^a\tO\thvs^n1$.  This means that,
  for any $n\in\bN$, there exists $\alpha_n\in [e^{-a},e^a]$ such that
  \begin{align}\label{eq:hilbert-contraction}
    \|\tO\thvs^n f-\alpha_n\tO\thvs^ng\|_{L^\infty}%
    &\le \Const \Theta(\tO\thvs^nf,\tO\thvs^ng)\|\tO\thvs^n1\|_{L^\infty}\\
    \notag &\le \Const\Delta \left[\tanh \frac{\Delta}4\right]^n \|\tO\thvs^n1\|_{L^\infty}.
  \end{align}
  Let $e^{\chi_{\theta,\paramL\Omega}}$,  be the maximal eigenvalue of $\tO_{ \theta,\paramL
    \Omega}$ when acting on $\cC^1$.
  The above displayed equations, together with~\eqref{eq:cone-est}, imply that $\tO_{\theta,
    \paramL\Omega}$, when acting on $\cC^1$, has a simple maximal eigenvalue and a spectral
  gap of size at least $e^{\chi_{\theta,\paramL\Omega}}\left(1-\tanh \frac{\Delta}4\right)$.

  Accordingly, there exists $h_{\theta,\paramL\Omega}\in \cC^1$ and a distribution $m_{\theta,
    \paramL\Omega}\in (\cC^1)'$ such that
  \[
    \tO_{ \theta,\paramL\Omega}(g)=e^{\chi_{ \theta,\paramL\Omega}}
    h_{\theta,\paramL\Omega}\,m_{\theta,\paramL\Omega} (g)+\tOq_{\theta,\paramL\Omega}(g)
    =: e^{\chi_{ \theta,\paramL\Omega}}
    \cP_{\theta,\paramL\Omega}(g)+\tOq_{\theta,\paramL\Omega}(g),
  \]
  where, for any $n\in\bN$,
  $\|\tOq_{\theta,\paramL\Omega}^n\|_{\cC^1\to \cC^1} \le C_{\theta,\paramL\Omega}
  e^{\chi_{ \theta,\paramL\Omega}n}\tau_{\theta,\paramL\Omega}^n$, with
  $\tau_{\theta,\paramL\Omega}\in(0,1-\tanh\frac\Delta 4]$,
  $\tOq_{\theta,\paramL\Omega}\cP_{\theta,\paramL\Omega}=\cP_{\theta,\paramL\Omega}\tOq_{\theta,\paramL\Omega}=0$
  and $m_{\theta,\paramL\Omega}(h_{\theta,\paramL\Omega})=1$.  Moreover, by standard
  perturbation theory all the above quantities are analytic in $\paramL$.

  We now show that $m_{\theta,\paramL\Omega}$ is not just an element of $(\cC^1)'$, as
  follows automatically from the general theory, but indeed a measure (\ie an element of
  $(\cC^0)'$).  We have seen that
  \[
    e^a\Leb( h_{ \theta,\paramL \Omega} )\geq h_{ \theta,\paramL \Omega} =e^{-n\chi_{
        \theta,\paramL\Omega}}\tO_{ \theta,\paramL\Omega}^nh_{ \theta,\paramL \Omega}\geq
    e^{-a}e^{-n\chi_{ \theta,\paramL\Omega}}\tO_{ \theta,\paramL\Omega}^n1.
  \]
  Thus, for any $n\in\bN$ and $g\in\cC^1$, $g\geq 0$,
  \[
    0\leq e^{-n\chi_{ \theta,\paramL\Omega}}\tO_{ \theta,\paramL\Omega}^ng=h_{
      \theta,\paramL \Omega} m_{ \theta,\paramL \Omega} (g) +C_{\theta,\paramL\Omega}
    \tau_{\theta,\paramL\Omega}^n\|g\|\nc1
  \]
  which shows that $m_{ \theta,\paramL  \Omega}$ is a positive functional and hence a measure.

  Finally, the perturbation theory in \cite[Section 8]{GL06} implies that
  $\chi_{ \theta,\paramL\Omega}$ and
  $\tOp_{\theta,\paramL\Omega} = h_{\theta,\paramL\Omega}\otimes m_{\theta,\paramL\Omega}
  $ are differentiable in $\theta$ (the latter with respect to the $L(\cC^2,\cC^0)$
  topology) and that $C_{\theta,\paramL\Omega}, \tau_{\theta,\paramL\Omega}$ can be chosen
  to be continuous in $\theta$.\footnote{ The Banach spaces $\cB^i$ in \cite[Section
    8]{GL06} here are taken to be $\cC^i$.}  Indeed, a direct computation shows that,
  setting
  \begin{align}\label{eq:derivativeL}
    \derivL\thvs (g)
    & = -\left[\frac{\partial_\theta f_\theta}{f_\theta'}g\right]'+\paramL\left[\partial_\theta \Omega_\theta-\frac{\partial_\theta f_\theta }{f_\theta'}\Omega_\theta'\right]g,
      \intertext{we have}
      \tO_{\theta+s,\paramL\Omega}
    &=\tO_{\theta,\paramL\Omega}+\int_\theta^{\theta+s}
      \tO_{ \varphi,\paramL\Omega}\derivL_{ \varphi,\paramL\Omega}d\varphi=\tO_{\theta,\paramL
      \Omega}(\Id+s\derivL\thvs)+\frac {s^2}2R_{\theta,s,\paramL\Omega,}
      \label{eq:basic-taylor}
  \end{align}
  where
  $R_{\theta,s,\paramL\Omega}=\frac
  2{s^2}\left[\tO_{\theta+s,\paramL\Omega}-\tO_{\theta,\paramL\Omega}(\Id-s\derivL\thvs)\right]$.
  Moreover, for any $0 < k \le r-2$ we have:\footnote{ To get the first inequality, use the
    spectral decomposition together with~\eqref{eq:cone-est} and its obvious analog for
    higher derivatives.}
  \begin{equation}\label{eq:der-t-bound}
    \begin{split}
      &\|\tO\thvs^n\|_{\cC^k}\le C\thvs e^{n\chi\thvs}\\
     & \|\tO_{\theta+s,\paramL\Omega}-\tO_{\theta,\paramL\Omega}\|_{\cC^{k+1}\to\cC^k}\le
      |s|\sup_{\theta'\in[\theta,\theta+s]}\|\tO_{\theta',\paramL\Omega}\|_{\cC^{k+1}}\|\derivL_{\theta',\paramL\Omega}\|_{\cC^{k+1}\to\cC^k}\\
     & \phantom{\|\tO_{\theta+s,\paramL\Omega}-\tO_{\theta,\paramL\Omega}\|_{\cC^{k+1}\to\cC^k}}
      \le C_{\theta,\paramL\Omega}
      |s|(1+\|\Omega_\theta\|\nc{k+1}+\|\partial_\theta\Omega_
      \theta\|\nc{k})\\
      &\|R_{\theta,s,\paramL\Omega}\|_{\cC^{k+2}\to\cC^k}\le C_{\theta,\paramL\Omega} (1+
      \|\Omega_\theta\|\nc{k+2}+\|\partial_\theta\Omega_\theta\|\nc{k+1}+\|\partial_\theta^2\Omega_\theta\|\nc{k}).
    \end{split}
  \end{equation}
  Hence the hypotheses of~\cite[Theorem 8.1]{GL06} are satisfied and the resolvent
  $\Id z-\tO_{\theta,\paramL\Omega}$, viewed as an operator from $\cC^2$ to $\cC^0$, is
  differentiable in $\theta$.  This implies the same for all spectral data, since they can
  be recovered by integrating the resolvent over the complex plane.
\end{proof}

In the case of arbitrary complex potentials it is also possible to obtain information on
the spectrum, as described in the following result.
\begin{lem}\label{lem:c0-c1}%
  For $\paramL\in\bR, \theta\in\bT$, let $e^{\tau_{\theta,\paramL\Omega}}$ be the
  spectral radius of $\tO_{\theta,\paramL\Omega}$ as an element of $L(\cC^0,\cC^0)$.  Then
  the spectral radius $e^{\chi_{\theta,\paramL\Omega}}$ of $\tO_{\theta,\paramL\Omega}$ as
  an element of $L(\cC^1,\cC^1)$ is bounded by $e^{\tau_{\theta,\paramL\Omega}}$.  In
  addition, the essential spectral radius is bounded by
  $\lambda^{-1}e^{\tau_{\theta,\paramL\Omega}}$.  Finally, the spectrum outside the disk
  of radius $\lambda^{-1}e^{\tau_{\theta,\paramL\Omega}}$ is the same when
  $\tO_{\theta,\paramL\Omega}$ acts on all $\cC^k$, $k\in \{1,\cdots, r-1\}$.
\end{lem}
\begin{proof}
  Note that the computation yielding~\eqref{eq:cone-est} also holds for any power
  $f_\theta^n$; this gives:
  \begin{equation}\label{eq:more-der-L}
    \frac{\deh}{\deh x}\tO\thvs^n g =%
    \tO\thvs^n\left(\frac{g'}{(f_\theta^n)'}+
      \paramL\frac{g\Omega_{\theta,n}'}{(f_\theta^n)'}-\frac{g(f_\theta^n)''}{{[(f_\theta^n)']}^2}\right)
  \end{equation}
  where $\Omega_{\theta,n}=\sum_{k=0}^{n-1}\Omega_{\theta}\circ f_\theta^k$.  Then a direct computation yields
  \begin{align}\label{e_iteratedLasotaYorke}
  \|\tO\thvs^n g\|_{\cC^1}\le
  \|\tO\thvs^n\|_{\cC^0\to\cC^0}\left[\lambda^{-n}\|g\|_{\cC^1}+\Const(|\paramL|\|\Omega_{\theta}'\|\nc0+1)\|g\|\nc0\right].
  \end{align}
  We conclude that the spectral radius of $\tO\thvs$ as an element of $L(\cC^1,\cC^1)$ is
  bounded by $e^{\tau_{\theta,\paramL\Omega}}$.  In addition, it follows from the usual
  Hennion's argument~\cite{Hennion93} that the essential spectral radius is bounded by
  $\lambda^{-1}e^{\tau_{\theta,\paramL\Omega}}$.  To conclude note that, by
  differentiating~\eqref{eq:more-der-L} one see that the essential spectral radius on
  $\cC^k$ is bounded by $\lambda^{-k}e^{\tau_{\theta,\paramL\Omega}}$.  On the other hand,
  an eigenvalue in $\cC^k$ is also an eigenvalue in $\cC^1$.  To prove the contrary,
  define the smoothing operator
  $Q_\epsilon g(x)=\int \epsilon^{-1}q(\epsilon^{-1}(x-y))g(y) d y$, where $q$ is a bump
  function: $q\in \cC^\infty_0(\bR,\bRp)$ with $\int q=1$.  Define
  $\tO_\epsilon=Q_\epsilon\tO\thvs$.  By the perturbation theory in~\cite{liverani99} the
  spectrum of $\tO_\epsilon$ and $\tO\thvs$ are close on each $\cC^k$.  On the other hand
  $\tO_\epsilon$ is a compact operator and its spectrum is the same on each $\cC^k$ since
  each eigenvalue belongs to $\cC^\infty$.
\end{proof}
Note that, in general, it could happen that the spectral radius of $\tO\thvs$ on $\cC^1$
is smaller than $\lambda^{-1}e^{\tau\thvs}$.  In this case, the second part of the above
lemma is of limited interest.
\begin{rem}\label{rem:quasicompact-c}
  If $\Omega$ is real, then the spectral radii of $\tO\thvs$ on $\cC^1$ and $\cC^0$ coincide;
  in fact, for each $\tau< \tau\thvs$ and $\chi>\chi\thvs$, there exists $\bar n\in\bN$
  and $g\in \cC^0$ such that, for all $n\ge \bar n,$
  \begin{align*}
    e^{\tau n}\|g\|_{\cC^0}\le \|\tO\thvs^{n} g\|_{\cC^0}\le \|\tO\thvs^{n}
    1\|_{\cC^1}\|g\|\nc0\le e^{\chi n}\|g\|\nc0.
  \end{align*}
  The claim then follows by  Lemma~\ref{lem:c0-c1}.
\end{rem}
It is worth stressing the fact that the functional $m_{\theta,\paramL\Omega}$ is
guaranteed to be a measure only provided that the potential $\Omega_\theta$ is real: this
is essentially due to the fact that, because of cancellations of complex phases, the
spectral radius on $\cC^1$ might be smaller than the spectral radius on $\cC^0$ if the
potential has a non-zero imaginary part.

\begin{rem}\label{rem:high-der}
  Note that, by arguments similar to the one described in this subsection, $\tO\thvs$ is a
  well defined operator also on $\cC^{r-1}$ or $W^{r-1,1}$ and, on such spaces, it has
  essential spectrum bounded by $\lambda^{-r+1}$. In particular $h\thvs\in\cC^{r-1}$ and
  for any $1\le s < r$, we have
  \begin{align}\label{e_bound-h-Cs}
    \|h\thvs\|\nc{s}\le \Const(1+|\paramL|\|\Omega_\theta\|\nc{s})^{s+1}.
  \end{align}
\end{rem}

\subsection{Perturbation Theory with respect to $\paramL$}\label{subsec:pert-paramL}\ \newline
In this and the following subsections we will consider only the case in which there is a
unique maximal eigenvalue $e^{\chi\thvs}$ which is simple.  Hence $m\thvs$ and $h\thvs$
are well defined, except for their normalization, which is not determined by the spectral
projector $\tOp\thvs=h\thvs\otimes m\thvs$ associated to $e^{\chi\thvs}$.  Note that $\chi_{\theta,0}=0$;
moreover, for $\paramL=0$ there exists a natural normalization for $m_{\theta,0}$ and
$h_{\theta,0}$ so that $m_{\theta,0}$ is the Lebesgue measure and $h_{\theta,0}$ is the
density of the invariant SRB probability measure $\mu_\theta$.  There is, however, no
natural normalization for $\paramL\not=0$; we thus proceed to define one that is
particularly suitable to our purposes.
\begin{rem}\label{eq:neighborhood-paramL}%
  Recall that the spectral data is analytic-- in $\paramL$ in a neighborhood of zero.
  Standard perturbation theory implies that such neighborhood contains the $\paramL$ such
  that, for all $\theta\in\bT$ we have $\|\Omega_\theta\|\nc1|\paramL|\le \sigma_1$ for
  some fixed $\sigma_1\in (0,1)$ small enough.  From now on we will assume $\paramL$ in
  this set unless otherwise specified; in this regime we are guaranteed that
  $\tO_{\theta,\paramL \Omega}$ is a Perron--Frobenius operator.
\end{rem}
\begin{subequations}\label{eq:derivative-m}
  Let us differentiate the relation $m\thvs\tO\thvs h\thvs=e^{\chi\thvs}$ with respect to
  $\paramL$ and obtain
  \begin{equation}
    \partial_\paramL
    \chi_{\theta,\paramL\Omega}=\mnu_{\theta,\paramL\Omega}(\Omega_\theta)
    \label{e_derivativeChi}
  \end{equation}
  where $\mnu\thvs(g)=m\thvs(gh\thvs)$; observe that $\mnu\thvs(1)=1$.
  Let us introduce the renormalized operators
  $\tOh\thvs=e^{-\chi\thvs}\tO\thvs$.  Notice that
  $\tOh\thvs=\tOp\thvs+\tOqh\thvs$, where $\tOqh\thvs=e^{-\chi\thvs}\tOq\thvs$.  Then
  by~\eqref{e_derivativeChi} and the definition of $\tO\thvs$ we obtain
  \[
    \partial_\paramL\tOh\thvs(g)=\tOh\thvs(\Omega\thvs g)
  \]
  with $\Omega\thvs=\Omega_\theta-\mnu\thvs(\Omega_{\theta})$.  Thus, differentiating the
  relations $\tOh_{ \theta,\paramL\Omega}h_{\theta,\paramL\Omega}=
  h_{\theta,\paramL\Omega}$ and $m_{\theta,\paramL\Omega}(\tOh_{
    \theta,\paramL\Omega}g)=m_{\theta,\paramL\Omega}\,g$ yields
  \begin{align}
    \partial_\paramL h\thvs&={[\Id -\tOqh\thvs]}\invr \tOh\thvs(\Omega\thvs
                             h\thvs) - C(\theta,\paramL)h\thvs \label{e_derivativeDensityNN}\\
    \partial_\paramL m\thvs(g)&=m\thvs(\Omega\thvs{[\Id - \tOqh\thvs]}\invr \tilde{g}) +
                                C(\theta,\paramL)m\thvs(g)\label{e_derivativeMeasureNN}
  \end{align}
  where $\tilde g=(\Id-\tOp\thvs)g= g-h\thvs m\thvs(g)$ and $C(\theta,\paramL)$ depends
  on the normalization of $h\thvs$ and $m\thvs$.  Using the above expressions, and
  differentiating~\eqref{e_derivativeChi}, it is immediate to obtain
  \begin{equation}\label{e_secondDerivativeChi.prelim}
    \partial^2_\paramL\chi\thvs=m\thvs(\Omega\thvs{[\Id - \tOqh\thvs]}\invr(\Id + \tOh\thvs)
    \Omega\thvs h\thvs),
  \end{equation}
\end{subequations}
which yields
\begin{subequations}\label{e_secondDerivativeChi0}
  \begin{align}
    \partial^2_\paramL
    \chi\thvs&=\mnu\thvs(\Omega\thvs^2)+2\sum_{k=1}^\infty\mnu\thvs(\Omega\thvs\circ
    f_\theta^k \Omega\thvs) =\label{e_secondDerivativeChia}\\
    &=\lim_{n\to\infty}\frac 1n\mnu\thvs\left({\left[\sum_{k=0}^{n-1}\Omega\thvs\circ
          f_\theta^k\right]}^2\right)\label{e_secondDerivativeChi}
  \end{align}
\end{subequations}
where we used the identity $m\thvs(g_1\tOh\thvs^kg_2)=m\thvs(g_1\circ f^kg_2)$, which is
obtained directly by definition of the Transfer operator $\tO\thvs$.  Observe
that~\eqref{e_secondDerivativeChi} shows that $\chi_{\theta,\paramL\Omega}$ is (for real
potentials) a convex function of $\paramL$.

By further differentiation of~\eqref{e_secondDerivativeChi.prelim} it is simple to show
that
\begin{align*}
  \partial^3_\paramL\chi\thvs &= %
  \mnu\thvs(\Omega\thvs^3) +%
  3\sum_{k=1}^{\infty}\mnu\thvs(\Omega\thvs\circ f_\theta^k\Omega^2\thvs+\Omega^2\thvs\circ
  f_\theta^k\Omega\thvs) +\\
  &\pheq+ 6\sum_{k=1}^{\infty}\sum_{j=k+1}^\infty\mnu\thvs(\Omega\thvs\circ
  f_\theta^j\Omega\thvs\circ f_\theta^k\Omega\thvs),
\end{align*}
which implies the useful estimate
\begin{equation} \label{e_thirdDerivativeChi}
  |\partial^3_\paramL\chi\thvs|\le\Const\|\Omega_\theta\|^3\nc1.
\end{equation}
Next, for all $n\in\bN$,
  \begin{equation}\label{eq:inv-measure-gen}
    \mnu\thvs(\phi g\circ f_\theta^n)=m\thvs(\tOh\thvs^n(\phi g\circ f_\theta^n h\thvs))=m\thvs(g \tOh\thvs^n(\phi h\thvs)).
  \end{equation}
  The above and the iteration of~\eqref{eq:cone-est} imply, setting $g=1$ and taking the
  limit for $n\to\infty$, that $\mnu\thvs$ is a measure provided that $\tOh\thvs$ is power
  bounded as an operator on $\cC^0$.  In addition, taking $\phi=1$ we see that, in
  general, it is an invariant distribution for $f_\theta$.
\begin{lem}\label{l_normalization}
  There exists a normalization for $h\thvs$ and $m\thvs$ so that $m_{\theta,0}=\Leb$ and the corresponding
  $C(\theta,\paramL)$ is identically $0$, that is:
  \begin{subequations}\label{e_derivativess}
    \begin{align}
      \partial_\paramL h\thvs&={[\Id -\tOqh\thvs]}\invr \tOh\thvs(\Omega\thvs
      h\thvs) \label{e_derivativeDensity}\\
      \partial_\paramL m\thvs\,g&=m\thvs(\Omega\thvs{[\Id - \tOqh\thvs]}\invr
      [g-h\thvs m\thvs\,g])\label{e_derivativeMeasure},
    \end{align}
  \end{subequations}
  provided $\Omega$ is real or, for arbitrary potentials, if
  $\|\Omega_\theta\|\nc1|\paramL|\le \sigma_1$ (see Remark~\ref{eq:neighborhood-paramL}).
\end{lem}
\begin{proof}
  Let us temporarily fix a normalization which defines $\bar h\thvs$ and $\bar m\thvs$ so
  that $\Leb(\bar h\thvs)=1$ for any $\paramL$.  Note that for real potentials this can
  always be done since $h\thvs>0$ due to Lemma~\ref{lem:quasicompact}.  For arbitrary
  potentials it is possible only if $\Leb (\cP\thvs(\phi))\neq 0$ for some $\phi\in\cC^0$.
  This is the case for small $\paramL$ due to $\Leb(\cP_{\theta,0}(\phi))=\Leb(\phi)$ and
  the continuity of $\cP\thvs$.

  Using~\eqref{e_derivativeDensityNN} and differentiating this normalization condition
  with respect to $\paramL$ we obtain
  \begin{equation}\label{e_Cbar}
    \bar C(\theta,\paramL) = \Leb({[\Id -\tOqh\thvs]}\invr \tOh\thvs(\Omega\thvs
    \bar h\thvs)).
  \end{equation}
  Define $\alpha(\theta,\paramL)=\int_0^\paramL\bar C(\theta,{\paramL_1})\deh{\paramL_1}$ and
  choose a new normalization so that $h\thvs=e^{\alpha(\theta,\paramL)}\bar h\thvs$ (and
  consequently $m\thvs=e^{-\alpha(\theta,\paramL)}\bar m\thvs$).  Then, an immediate
  computation shows that $h\thvs$ and $m\thvs$ satisfy equations~\eqref{e_derivativess}.
\end{proof}
We now fix once and for all the normalization of $h\thvs$ and $m\thvs$ to be the one constructed
in Lemma~\ref{l_normalization} and refer to it as the \emph{standard normalization}.
\begin{lem}\label{l_derivativess}
  For any $g\in W^{1,1}$ and under the assumptions described
  in Remark~\ref{eq:neighborhood-paramL}, we have
  \begin{subequations}\label{e_estimatesPartialhm}
    \begin{align}
      \|\partial_\paramL h_{\theta,\paramL\Omega}\|\nc1&%
      \le\Const\|\Omega_\theta\|\nc1\label{e_estimatePartialh}\\%
      |\partial_\paramL{}m_{\theta,\paramL\Omega}\,g|&%
      \le\Const\|\Omega_\theta\|\nc1\|g\|\nw11\label{e_estimatePartialm}
    \end{align}
  \end{subequations}
  and, moreover,
  \begin{align}\label{e_quadraticEstimates}
    |m_{\theta,0}\,h\thvs-1|&\le \Const\paramL^2\|\Omega_\theta\|^2\nc1&
    |m\thvs(h_{\theta,0})-1|&\le \Const\paramL^2\|\Omega_\theta\|^2\nc1.
  \end{align}
\end{lem}
\begin{proof}
  Since all the quantities are analytic in $\paramL$ and $\paramL$ belongs to a fixed
  compact set, we have uniform bounds on $\|h\thvs\|\nc1$ and $\|m\thvs\|_{(W^{1,1})'}$.
  Thus, by~\eqref{e_derivativeDensity}, taking the $\cC^1$-norm, we obtain:
  \[
  \|\partial_\paramL h\thvs\|\nc1 \le \|{[\Id
    -\tOqh\thvs]}\invr \tOh\thvs(\Omega\thvs h\thvs)\|\nc1\le  \Const  \|\Omega_\theta\|_{\cC^1}.
  \]
  which implies $\|h\thvs\|_{\cC^1}\le\Const(1+|\paramL|\|\Omega_\theta\|_{\cC^1})$.
  Similar computations yield the corresponding result for $m\thvs$.

  Finally, in order to obtain equations~\eqref{e_quadraticEstimates}, observe that
  $m_{\theta,0}\,h\thvs-1=m_{\theta,0}(\int_0^\paramL\deh{\paramL_1}\partial_\paramL
  h_{\theta,\paramL_1\Omega})$; then, since
  $m_{\theta,0}\,\partial_\paramL h_{\theta,0}=0$ we can write
  \begin{align*}
    |m_{\theta,0}\,h\thvs-1|\le\int_0^\paramL\deh{\paramL_1}\int_0^{{\paramL_1}}\deh{\paramL_2}\left\|
    \partial_\paramL\left({[\Id-\tOqh\thvss]}^{-1}
      \tOh\thvss(\Omega\thvss h\thvss) \right)\right\|\nc0
  \end{align*}
  from which follows the first of~\eqref{e_quadraticEstimates}; a similar computation
  yields the second estimate, which concludes the proof.
\end{proof}
We now deal with transfer operators weighted with two different families of potentials,
which we denote by $\Omegaa$ and $\Omegab$.  If $\|\Omegaa-\Omegab\|\nc1$ is small enough,
we can once again use perturbation theory to compare spectral data.
 Until the end of this subsection we assume $\theta$ to be fixed
  and we will drop it from our notation since it will not cause any confusion.  Also, we
  assume that either both $\Omegaa$ and $\Omegab$ are real, or $\|\Omegaa\|\nc1$ and
  $\|\Omegab\|\nc1$ to be sufficiently small (\ie smaller than $\sigma_1$) so that we can
  assume $\paramL = 1$ and still be in the perturbative regime (see
  Remark~\ref{eq:neighborhood-paramL}).  For $\varrho\in[0,1]$, let us define the convex
interpolation $\Omegap\varrho=\Omegaa+\varrho(\Omegab-\Omegaa)$, and let
$\delta\Omega = \partial_\varrho\Omegap\varrho=\Omegab-\Omegaa$; consider the transfer
operators $ \tO_\varrho = \tO_{\theta,\Omegap\varrho}$; similarly let
$h_\varrho=h_{\theta, \Omegap\varrho}$ and $m_\varrho=m_{\theta, \Omegap\varrho}$.  Then,
by arguments analogous to the ones leading to equations~\eqref{eq:derivative-m}, we obtain
\begin{subequations}\label{e_derivative-varrho}
  \begin{align}
    \partial_\varrho \chi_\varrho &=m_\varrho(\delta\Omega h_\varrho)\label{eq:pvarrhochi}\\
    \partial_\varrho h_\varrho &= {[\Id-\tOqh_\varrho]}\invr\tOh_\varrho(\delta\widehat \Omega_\varrho h_\varrho)- C(\varrho)h_\varrho\label{eq:pvarrhoh-nn}\\
    \partial_\varrho m_\varrho g &= m_\varrho(\delta\widehat \Omega_\varrho{[\Id-\tOqh_\varrho]}\invr(g-h_\varrho m_\varrho g)) + C(\varrho)m_\varrho g.
  \end{align}
\end{subequations}
where we defined
$\delta\widehat\Omega_\varrho=\delta\Omega - m_\varrho(\delta\Omega h_\varrho)$ and the
function $C(\varrho)$ depends on the normalization for $h_\varrho$ and $m_\varrho$.
\begin{lem}
  For any $g\in W^{1,1}$ and under the assumptions described in
  Remark~\ref{eq:neighborhood-paramL}, and choosing the standard normalization we have
\begin{subequations}\label{e_estimatesPartialRho}
  \begin{align}\label{e_estimatesPartialRho-h}
    \|\partial_\varrho h_\varrho\|\nc k&\le\Const\|\delta\Omega\|\nc k\\
    \label{e_estimatesPartialRho-m}
   \hskip-.8cm |\partial_\varrho m_\varrho g- m_\varrho(\delta\widehat
      \Omega_\varrho{[\Id-\tOqh_\varrho]}\invr(g-h_\varrho m_\varrho g))|&\le
      \Const\|\Omegaa\|\nc1\|\delta\Omega\|\nc1\|g\|\nw11.\hskip-1cm
  \end{align}
\end{subequations}
\end{lem}
\begin{proof}
  As in the proof of Lemma~\ref{l_normalization}, let us denote by $\bar h_\varrho$ the
  eigenvector normalized so that $\Leb(\bar h_\varrho)=1$, let $\bar{\bar C}(\varrho)$ be
  the corresponding normalization in~\eqref{e_derivative-varrho}.  Then a direct
  computation (differentiating the normalization condition and
  using~\eqref{eq:pvarrhoh-nn}) shows that
  \[
    \bar{\bar
      C}(\varrho)=\Leb\left({[\Id-\tOqh_\varrho]}\invr\tOh_\varrho\delta\widehat\Omega_\varrho
      \bar h_\varrho\right).
  \]
  Moreover, let $\bar C_\varrho(\paramL)$ be defined as in~\eqref{e_Cbar} with the choice
  $\Omega=\Omegap{\varrho}$; then
  \[
    C(\varrho)=\bar{\bar C}(\varrho)-\int_0^1\partial_\varrho\bar
    C_\varrho(\paramL)\deh\paramL.
  \]
  Note that, setting
  $g_{\varrho,\paramL}={[\Id -\tOqh_{\theta,\paramL\Omegap{\varrho}}]}\invr
  \tOh_{\theta,\paramL\Omegap{\varrho}}(\widehat\Omega_{\varrho,\paramL} \bar
  h_{\theta,\paramL\Omegap{\varrho}})$, where   $\widehat\Omega_{\varrho,\paramL}=\Omegap{\varrho}-m_{\theta,\paramL\Omegap{\varrho}}(\Omegap{\varrho}
  h_{\theta,\paramL\Omegap{\varrho}})$,
  \begin{align*}
    \bar C_\varrho(\paramL)&= \Leb(g_\varrho) =m_{\theta,\paramL\Omegap{\varrho}}(g_{\varrho,\paramL})-\int_0^\paramL \partial_{\paramL_1} m_{\theta,\paramL_1\Omegap{\varrho}}(g_{\varrho,\paramL})\deh\paramL_1.
  \end{align*}
  Note that the first term of the rightmost hand side of the equation above is identically
  zero, hence, by~\eqref{e_derivativeMeasure}, we conclude that
  \begin{align*}
    &\bar C_\varrho(\paramL)=-\int_0^\paramL m_{\theta,\paramL_1\Omegap{\varrho}}(\widehat\Omega_{\varrho,\paramL_1}{[\Id - \tOqh_{\theta,\paramL_1\Omegap{\varrho}}]}\invr \widehat g_{\varrho,\paramL})\deh\paramL_1
  \end{align*}
  where
  $\widehat
  g_{\varrho,\paramL}=g_{\varrho,\paramL}-m_{\theta,\paramL_1\Omegap\varrho}(g_{\varrho,\paramL})h_{\theta,\paramL_1\Omegap{\varrho}}
  $; this implies $|\bar C_\varrho(\paramL)|\leq \Const \|\Omegaa\|\nc1^2$ and,
  since~\eqref{e_derivative-varrho} implies that each derivative with respect to $\varrho$
  of the eigenvectors or operators yields an extra factor $\|\delta\Omega\|\nc1$,
  $|\partial_\varrho\bar C_\varrho(\paramL)|\leq \Const
  \|\Omegaa\|\nc1^2\|\delta\Omega\|\nc1$.  By similar arguments we obtain
  $|\bar{\bar C}_\varrho(\paramL)|\leq \Const \|\Omegaa\|\nc1\|\delta\Omega\|\nc1$, which
  then implies equations~\eqref{e_estimatesPartialRho}.
\end{proof}

\subsection{Perturbation Theory with respect to
  \texorpdfstring{$\theta$}{θ}}\label{subsec:pert-theta}\ \newline

Recalling the notation and computations at the end of the proof of
Lemma~\ref{lem:quasicompact} and by argument analogous to the ones leading to
equations~\eqref{eq:derivative-m}, but differentiating with respect to $\theta$, we gather:
\begin{subequations}\label{eq:derivative-theta}
  \begin{align}
    \partial_\theta \chi\thvs&=m\thvs\left(\derivL\thvs h\thvs\right)\label{e_partialThetachi}\\
    \partial_\theta h\thvs&=[\Id -\tOqh\thvs]^{-1} \tOh\thvs \widetilde\derivL\thvs h\thvs - D(\theta,
    \paramL)h_{\theta,\paramL\Omega}\label{e_partialhTheta}\\
    \partial_\theta m_{\theta,\paramL\Omega}\,g&=m_{\theta,\paramL\Omega}( \derivL\thvs[\Id
    -\tOqh_{ \theta,\paramL\Omega}]^{-1} \tilde g ) +
    D(\theta,\paramL)m_{\theta,\paramL\Omega}\,g \label{e_partialhTheta-m}
  \end{align}
\end{subequations}
where $\widetilde\derivL\thvs g=\derivL\thvs g - h\thvs m\thvs(\derivL\thvs g)$, and
recall that $\tilde g= g-h\thvs m\thvs(g)$ and $\derivL\thvs $ is defined
in~\eqref{eq:derivativeL}.  Once again $D(\theta,\paramL)$ is a function which depends on
the normalization for $h_{\theta,\paramL\Omega}$ and $m_{\theta,\paramL\Omega}$.  Note
that we cannot, in general, assume that $D=0$; since $m_{\theta,0}=\Leb$, it is however
true that $D(\theta,0)=0$ for any $\theta$.  Similarly, since $\chi_{\theta,0}=0$, we have
$\partial_\theta\chi_{\theta,0}=0$.
\begin{lem}\label{eq:theta-derivative-all}%
  There exists $\sigma_0\in (0,\sigma_1)$ such that, if
  \begin{align*}
    \|\partial^2_\theta\Omega\|\ncd1{\bT^2}+\|\partial_x\Omega\|\ncd2{\bT^2}+\|\Omega\|\ncd2{\bT^2}
     &\leq \sigma_0
  \end{align*}
  we have, for any $1\le k < r-1$ and using the standard normalization:
  \begin{subequations}\label{e_estimatePartialTheta}
    \begin{align}
      &\left|\partial_\theta\chi_{\theta,\Omega}- \partial_\theta
        \Leb(\Omega_\theta h_{\theta,0})\right|\leq\Const \|\Omega\|^2_{\cC^1(\bT^2)} .\label{e_estimatesecondpartchi0}\\
      &  \|\partial_\theta h_{\theta,\Omega}\|\nc k\le
        \Const(1+\|\Omega_\theta\|\nc{k+1}+\|\partial_\theta\Omega_\theta\|\nc{k})\label{e_estimatePartialThetah}\\
      &  |\partial_\theta m_{\theta,\Omega}\,g| \le \Const\|\Omega\|_{\cC^2(\bT^2)}\|g\|\nw21\label{e_estimatePartialThetam}\\
      & \left|\partial_\theta^2\chi_{\theta,\Omega}\right|\le \Const (\|\partial^2_\theta\Omega\|\ncd1{\bT^2}+\|\partial_x\Omega\|\ncd2{\bT^2}+\|\Omega\|\ncd2{\bT^2}).\label{e_estimatesecondpartchi}
    \end{align}
  \end{subequations}
  Additionally,  $\partial_\paramL h_{\theta,\paramL\Omega}$, $\partial_\paramL m_{\theta,\paramL\Omega}
  $ and $\partial^2_\paramL{}\chi_{\theta,\paramL\Omega}$ are differentiable in $\theta$.
\end{lem}
\begin{proof}
  Plugging~\eqref{e_derivativeMeasure},~\eqref{e_derivativeDensity}
  and~\eqref{eq:derivativeL} in~\eqref{e_partialThetachi}, we have:
  \begin{align}\label{e_partialThetaChi-prelim}
    \partial_\theta\chi_{\theta,\Omega}
    &=\Leb(\derivL\thvz h\thvz)
      +\int_0^1m\thvs\left(\Omega\thvs{[\Id - \tOqh\thvs]}\invr \widetilde\derivL\thvz h\thvz\right)\deh \paramL\notag\\
    & =\Leb\left(\partial_\theta\Omega_\theta\cdot h_{\theta,0}-\frac{\partial_\theta f_\theta}{f'_\theta}\Omega_\theta' h_{\theta,0}\right)-\Leb\left(\Omega_{\theta,0}{[\Id - \tOqh_{\theta,0}]}\invr\left[\frac{\partial_\theta f_\theta}{f'_\theta} h_{\theta,0}\right]'\right)\\
    &\phantom{=}+\cO(\|\Omega\|_{\cC^1(\bT^2)}^2),\notag
  \end{align}
  where the term having a derivative in~\eqref{eq:derivativeL} disappears by integration
  by parts against Lebesgue.  Next, note that ~\eqref{e_partialhTheta} implies
  \[
    \begin{split}
      \partial_\theta \Leb(\Omega_\theta h_{\theta,0})&=\Leb(\partial_\theta\Omega_\theta h_{\theta,0})+\Leb(\Omega_\theta{[\Id - \tOqh_{\theta,0}]}\invr
      \tOqh_{\theta,0} \widetilde\derivL_{\theta,0} h_{\theta,0})\\
     &=\Leb(\partial_\theta\Omega_\theta h_{\theta,0})-\Leb\left(\Omega_\theta{[\Id - \tOqh_{\theta,0}]}\invr \left[\frac{\partial_\theta f_\theta}{f'_\theta} h_{\theta,0}\right]'\right)\\
      &\phantom = +\Leb\left(\Omega_\theta \left[\frac{\partial_\theta f_\theta}{f'_\theta} h_{\theta,0}\right]'\right)+\cO(\|\Omega\|_{\cC^1(\bT^2)}^2).
    \end{split}
  \]
  In addition,
  \[
    \begin{split}
      \Leb\left([\Omega_{\theta,0}-\Omega_\theta]{[\Id - \tOqh_{\theta,0}]}\invr \left[\frac{\partial_\theta f_\theta}{f'_\theta} h_{\theta,0}\right]'\right)
      &=\Leb(\Omega_\theta h_{\theta,0})\\
      &\phantom{=}\times \Leb\left([\Id - \tOqh_{\theta,0}]\invr \left[\frac{\partial_\theta f_\theta}{f'_\theta} h_{\theta,0}\right]'\right)\\
      &=\Leb(\Omega_\theta h_{\theta,0}) \Leb\left( \left[\frac{\partial_\theta f_\theta}{f'_\theta} h_{\theta,0}\right]'\right)=0.
    \end{split}
  \]
  where the last term disappears again by integration by part against Lebesgue.  Combining
  the above expressions with~\eqref{e_partialThetaChi-prelim}
  yields~\eqref{e_estimatesecondpartchi0}.  Next, recall that, by the construction of the
  standard normalization given in Lemma~\ref{l_normalization} we have set
  $h\thvs=e^{\alpha(\theta,\paramL)}\bar h\thvs$, where $\bar h\thvs$ is normalized so
  that $\Leb(\bar h\thvs)=1$ and
  $\alpha(\theta,\paramL)=\int_0^\paramL\bar C(\theta,\paramL_1)\deh{\paramL_1}$, $\bar C$
  being given by~\eqref{e_Cbar}.  Observe that, differentiating the normalization
  condition for $\bar h\thvs$ with respect to $\theta$, we obtain, using
  equations~\eqref{eq:derivative-theta}:
  \begin{align*}
    \bar D(\theta,\paramL) = \Leb ([\Id -\tOqh\thvs]^{-1} \tOh\thvs \widetilde\derivL\thvs \bar h\thvs).
  \end{align*}
  Then, by definition of $h\thvs$ we get:
  \begin{align*}
    D(\theta,\paramL) = \bar D(\theta,\paramL) - \partial_\theta\alpha(\theta,\paramL).
  \end{align*}
  Thus, by the definition of $\alpha(\theta,\paramL)$ and using the hypothesis on
  $\|\Omega\|\nc0$, a direct computation, which is left to the reader, yields
  \begin{equation}\label{eq:alpha-der-theta}
    |\partial_\theta\alpha(\theta,\paramL)|\leq \Const \|\Omega\|_{\cC^{1}(\bT^2)}.
  \end{equation}
  The proof of~\eqref{e_estimatePartialThetah} immediately follows from \eqref{e_partialhTheta} using the
  definition~\eqref{eq:derivativeL}.  %
  In order to prove~\eqref{e_estimatePartialThetam}, one has to examine $\bar D$ in more
  detail.  By~\eqref{eq:derivativeL} we have
  \begin{align*}
    \bar D(\theta,1)
    &= - \Leb \left([\Id-\tOqh\thvz]\invr
      \tOh\thvz \left\{\left[\frac{\partial_\theta f_\theta}{f_\theta'}\bar h\thvz\right]'- h\thvz m\thvz\left( \left[\frac{\partial_\theta f_\theta}{f_\theta'}\bar h\thvz\right]'\right)\right\}\right) \\
    &\phantom =
      +\cO(\|\Omega\|_{\cC^1(\bT^2)}).
  \end{align*}
  Let
  $\widehat g=[\Id-\tOqh\thvz]\invr\tOh\thvz\left\{\left[\frac{\partial_\theta
      f_\theta}{f_\theta'}\bar h\thvz\right]'- h\thvz m\thvz\left(
    \left[\frac{\partial_\theta f_\theta}{f_\theta'}\bar h\thvz\right]'\right)\right\}$ and
  observe that $m\thvz(\widehat g) = 0$ and $\|\widehat g\|\nBV\le\Const$;
  then,~\eqref{e_estimatePartialm} implies
  \begin{align*}
    |\Leb(\widehat g)| \le |m\thvz(\widehat g)| + \left|\int_1^0\partial_\paramL
    m\thvs(\widehat g)\deh \paramL\right|\le \Const\|\Omega_\theta\|\nc1\|\widehat g\|\nw11.
  \end{align*}
  It follows that
  \begin{equation}\label{eq:D(theta,1)-est}
    |D(\theta, 1)|\leq \Const\|\Omega\|_{\cC^1(\bT^2)}.
  \end{equation}
  We thus obtain~\eqref{e_estimatePartialThetam} by~\eqref{eq:D(theta,1)-est} and applying
  to~\eqref{e_partialhTheta-m} similar arguments.  %
 In order to prove~\eqref{e_estimatesecondpartchi}, observe that
  differentiating~\eqref{e_partialThetachi} yields
  \begin{align*}
    \partial_\theta^2\chi\thvs=(\partial_\theta m\thvs)\left(\derivL\thvs
    h\thvs\right)+ m\thvs\left((\partial_\theta\derivL\thvs) h\thvs\right)
    +m\thvs\left(\derivL\thvs(\partial_\theta h\thvs)\right).
  \end{align*}
  Substituting~\eqref{e_partialhTheta-m},~\eqref{eq:derivativeL}
  and~\eqref{e_partialhTheta} in the above expression we get
  \begin{align*}
    \partial_\theta^2\chi\thvz= m\thvz\left(A'+B\right)
  \end{align*}
  where $A,B$ are two functions that (using~\eqref{e_bound-h-Cs} and our assumptions on
  $\Omega$) satisfy
    \begin{align*}
      \|A'\|\nw11 &\le \|A\|\nc2 \le\Const \\
      \|B\|\nc1  & \le\Const (\|\partial^2_\theta\Omega\|\ncd1{\bT^2}+\|\partial_x\Omega\|\ncd2{\bT^2}+\|\Omega\|\ncd2{\bT^2})
    \end{align*}
 To conclude the proof, we use~\eqref{e_derivativeMeasure} as in the proof
  of~\eqref{e_estimatesecondpartchi0} which yields the result since $\Leb(A')=0$ by
  integration by parts.  Finally, the last statement follows from the above considerations
  and the formulae~\eqref{e_derivativess} and~\eqref{e_secondDerivativeChi.prelim}.
\end{proof}
We conclude the subsection with a non-perturbative result
\begin{lem}\label{lem:non-pert-theta}
  Assume that $\Omega$ is real and $\Leb(\Omega h_{0,\theta})=0$, then
  \[
    \| \partial_\theta\chi\thvz \|\leq\Const \min\{\|\Omega\|\nc1+1, \|\Omega\|\nc1^2\}.
  \]
\end{lem}
\begin{proof}
  If $\|\Omega\|\nc0\leq \sigma_0$, the estimate follows
  from~\eqref{e_estimatesecondpartchi0}.  In the non-perturbative regime, \ie for
  potentials of larger norm, recall that $m_{\theta,\Omega}$ is a positive measure (see
  Lemma~\ref{lem:quasicompact}), and therefore~\eqref{e_partialThetachi}
  with~\eqref{eq:derivativeL} imply
  \begin{align*}
    |\partial_\theta \chi_{\theta,\Omega}|\leq \Const m_{\theta,\Omega}(\|\Omega\|\nc1
  h_{\theta,\Omega}+|h'_{\theta,\Omega}|).
  \end{align*}
  The claim then follows by recalling the normalization
  $1=m_{\theta,\Omega}(h_{\theta,\Omega})$ and that, using once again
  Lemma~\ref{lem:quasicompact}:
  \[
    |h'_{\theta,\Omega}(x)|\leq \Const (\|\Omega\|\nc1 +1) h_{\theta,\Omega}(x)\qedhere
  \]
\end{proof}
\subsection{Results for functions of bounded variation}\label{subsec:bv-ly}\ \newline
In certain parts of the paper it is convenient to consider transfer operators acting on
the Sobolev Space $W^{1,1}$ or on the space of function of bounded variations $\BV$.
Since $W^{1,1}\subset \BV$, with the same norm, we will limit our discussion to the
second, more general, case.

For functions of bounded variation, the Lasota--Yorke inequality reads as follows: for any
$\psi\in\cC^1$ and any $n\in\bN$, setting
$\Omega_{\theta,n}=\sum_{k=0}^{n-1}\Omega_\theta\circ f_\theta^k$,
\begin{equation}\label{eq:lasota-yorke-bv0}
  \begin{split}
    &\left|\int \psi'\tO\thvs^n g\right|=\left|\int g e^{\paramL\Omega_{\theta,n}}(\psi')\circ f_\theta^n\right|\\
    &\le \left|\int g\left[\frac{e^{\paramL\Omega_{\theta,n}}\psi\circ
          f_\theta^n}{(f_\theta^n)'}\right]'\right|+\Const\int |\psi|\tO_{\theta,0}^n\left[ e^{\paramL
        \Re(\Omega_{\theta,n})}(1+|\paramL|\|\Omega'_\theta\|\nc0) |g|\right].
  \end{split}
\end{equation}
Thus, setting
\begin{align*}
\tilde \chi_n
  & = \log\|\tO_{\theta,\paramL \Re (\Omega_{\theta,n})}^n\|\nl1
    = \log\|e^{\paramL\Re(\Omega_{\theta,n})}\|\nl1\\
\tau_n
  &=\log\left\|\frac{e^{\paramL\Re(\Omega_{\theta,n})}}{(f_\theta^n)'}\right\|\nl\infty
\end{align*}
we have
\begin{equation}\label{eq:lasota-yorke-bv}
  \|\tO\thvs^n g\|\nBV\le e^{\tau_n}\|g\|\nBV+\Const e^{\tilde\chi_n}(1+|\paramL|\|\Omega_\theta\|\nc1)\|g\|\nl1.
\end{equation}
By the usual Hennion argument \cite{Hennion93}, the spectral radius of $\tO\thvs$ is
bounded by $e^{\tilde\chi_n/n}$ and the essential spectral radius by $e^{\tau_n/n}$.
  Note
that~\eqref{eq:lasota-yorke-bv0} also implies\footnote{ Recall that, in one dimension,
  $\|g\|_{L^\infty}\le \|g\|\nBV$.}
\begin{equation}\label{eq:power-l1-a}
  \|\tO\thvs^n g\|\nBV\le \left[e^{\tau_n}+\Const(1+|\paramL|\|\Omega_\theta\|\nc1)\int
    e^{\paramL\Re(\Omega_{\theta,n})}\right]\|g\|\nBV.
\end{equation}
In addition, calling $\cH_{n}$ the set of inverse branches of $f_\theta^n$, we have, by
standard distortion estimates,
\begin{equation}\label{eq:power-l1-b}
\begin{split}
\int_{\bT^1} e^{\paramL\Re(\Omega_{\theta,n})}&=\int_{\bT^1}
\tO_{\theta,0}^ne^{\paramL\Re(\Omega_{\theta,n})}\\
&=\sum_{h\in\cH_n}\int_{\bT^1}e^{\paramL
  \Re(\Omega_{\theta,n})\circ h(x)}h'(x)\deh{}x\ge\Const e^{\tau_n}.
\end{split}
\end{equation}
Thus our bound on the spectral radius is larger or equal than our bound on the essential
spectral radius.  Nevertheless, these are just estimates: the real values could be much
smaller.
\begin{rem}\label{rem:bv-contraction}%
  For real potentials and $\paramL\ge 0$, more can be said.  If $\chi$ denotes the
  spectral radius of $\tO_{\theta,\paramL\Omega_{\theta}}$ as an operator on $\cC^1$ and
  $e^{\chi}$ its maximal eigenvalue, then\footnote{ Since
    $m\thvs\tO\thvs^n h\thvs=e^{n\chi\thvs}$ and $m\thvs$ is a measure (see
    Lemma~\ref{lem:quasicompact}), then there exists $x_*\in\bT$ such that
    $\tO\thvs^n h\thvs(x_*)=e^{n\chi\thvs}$, then the claim follows
    recalling~\eqref{eq:a-estimate}.}
  \begin{align*}
    e^{-\const|\paramL|} e^{n\chi}\le  \int_{\bT^1} \tO_{\theta,\paramL\Omega_{\theta}}^n1
    =\int_{\bT^1} e^{\paramL\Omega_{\theta,n}}
    =\int_{\bT^1} \tO_{\theta,\paramL\Omega_{\theta}}^n1\le e^{\const|\paramL|}  e^{n\chi}.
  \end{align*}
  This, together with~\eqref{eq:power-l1-a} and ~\eqref{eq:power-l1-b}, implies that the
  spectral radius of $\tO\thvs$ on $\BV$ coincides with the spectral radius on $\cC^1$
  and, moreover, $e^{-\chi} \tO_{\theta,\paramL\Omega_{\theta}}$ is power bounded on
  $\BV$.  This does not, however, imply that $\tO\thvs$ is a Perron--Frobenius operator
  also when acting on $\BV$: in fact, for large $\paramL$, the essential spectral radius
  could a priori coincide with the spectral radius.  Nevertheless, we can find a simple
  condition that prevents this pathological behavior.  Let
  $\Osc\Omega_\theta:=\sup \Omega_\theta-\inf \Omega_{\theta}$; then
  \begin{align*}
    e^{\tau_n}\le \log\left[\lambda^{-n} e^{\paramL\sup \Omega_{\theta,n}}\right]
    & \le \log\left[\lambda^{-n}e^{n \paramL \Osc\Omega_\theta}\left\|
      e^{\paramL \Omega_{\theta,n}}\right\|\nl1\right].
  \end{align*}
  The above implies that if $|\paramL|\Osc\Omega_\theta<\log\lambda$, then the essential spectral radius
  is strictly smaller than the spectral radius.
\end{rem}
\begin{rem}\label{rem:bv-contraction-conp}
  No such general bounds are available for arbitrary complex potentials and we must then
  rely on perturbation theory.  If the potential is purely imaginary,
  then~\eqref{eq:lasota-yorke-bv} implies that the essential spectral radius is smaller
  than $\lambda^{-1}$.  Since the point spectrum is independent on the space on which the
  operators act,\footnote{ This can be proven as in Lemma~\ref{lem:c0-c1}.} it follows
  that the spectrum outside the disk $\{|z|\le \lambda^{-1}\}$ on $\BV$ coincides with the
  spectrum on $\cC^1$.
\end{rem}
We conclude this brief discussion with a number of estimates on the left eigenvector;
observe first that by definition and by the analytic dependence of all objects on
$\paramL$, we have, for $\|\Omega_\theta\|\nc1 < \sigma_1$ that
$|m\thvz g|\le\Const\|g\|\nBV$, and using~\eqref{e_derivativeMeasure} we thus conclude
that, similarly to~\eqref{e_estimatePartialm}:
\begin{align}\label{e_estimatePartialm-BV}
  |\partial_\paramL m\thvs (g)|\le\Const\|\Omega_\theta\|\nc1\|g\|\nBV,
\end{align}
in particular:
\begin{align}\label{e_basicEstimatem-BV-Lebesgue}
  |m\thvs(g) -\Leb(g)| < \Const\|\Omega_\theta\|\nc1\|g\|\nBV.
\end{align}
We now proceed to obtain a refinement of the above estimates.
\begin{lem}\label{lem:m-estimate-0}
  There exists $\sigma_2, \Cquattro>0$ such that, provided $\|\Omega_\theta\|\nc1 < \sigma_2$
  and $\Omega$ is either real or it satisfies the estimate
  $\|\Omega_\theta\|\nc1 \expo{-\Cquattro\|\Omega_\theta\|\nc1\invr}\le
  \|\Omega_\theta\|\nc0$, for any $g\in\BV$ and $n\in\bN$, we have
  \begin{align}\label{e_m-estimate-0}
    |m\thvz\,g|\leq \Const e^{\const n\|\Omega_\theta\|\nc0}\|g\|\nl1+\Const e^{-\const n} \|g\|\nBV.
  \end{align}
\end{lem}
\begin{proof}
  First of all we choose $\sigma_2\le \sigma_1$ so that $\tO\thvz$ is of Perron--Frobenius
  type and let $h\thvz\otimes m\thvz$ be the eigenprojector associated to its maximal
  eigenvalue $e^{\chi\thvz}$.  In particular we have, for any $n > 0$
  \begin{align*}
    \tOh\thvz^n g = h\thvz m\thvz(g)+\tOqh\thvz^n(g)
  \end{align*}
  and therefore there exists $\tau\in (0,1)$ such that:
  \begin{align}\label{eq:m-bound-horror}
    |m\thvz\,g|\notag
    &\leq \left|\frac{\Leb\,\tOh\thvz^n g}{\Leb\, h\thvz}\right|+\Const\tau^n \|g\|_\BV \\
    &\leq \Const \Leb \left|e^{\Omega_{\theta,n}-n\chi\thvz} g\right|+\Const\tau^n \|g\|\nBV.
  \end{align}
  Observe that if $\Omega$ is real, then Lemma~\ref{lem:quasicompact} states that $m\thvz$ is
  a measure.  Therefore~\eqref{e_derivativeChi} implies that
  \begin{align}\label{e_claimChi}
    |\chi\thvz| < \Const\|\Omega_\theta\|\nc0.
  \end{align}
  We claim that the same estimate holds also for complex potentials satisfying the
  assumption given in the statement.  In fact, a priori
  $|\chi\thvz| < \Const\|\Omega_\theta\|\nc1$, thus~\eqref{eq:m-bound-horror} implies
  \begin{align*}
    |m\thvz\,g|\leq \Const e^{n\|\Omega_\theta\|\nc1} \|g\|\nl1 + \Const\tau^n \|g\|\nBV.
  \end{align*}
  We can then choose $n=\|\Omega_\theta\|\nc1^{-1}$ and apply the resulting bound
  to~\eqref{e_derivativeChi}, obtaining the better estimate:
  \begin{equation}\label{eq:chi-c0-bound}
    |\chi\thvz|\leq \Const \|\Omega_\theta\|\nc0 +
    \Const e^{-\const \|\Omega_\theta\|\nc1^{-1}}\|\Omega_\theta\|\nc1\leq \Const \|\Omega_\theta\|\nc0,
  \end{equation}
  where we have used the hypotheses of the lemma choosing $\Cquattro = |\log\tau|$.  We
  can thus use again~\eqref{eq:m-bound-horror} and obtain~\eqref{e_m-estimate-0},
  concluding the proof of the lemma.
\end{proof}
\begin{lem}\label{lem:m-estimate}
  Under the assumptions of Lemma~\ref{lem:m-estimate-0}, for any $g\in\BV$:
  \begin{align}\label{e_estimateMIterate2BV-derivative}
    |\partial_\paramL m\thvs\,g|
    &\le \Const \|\Omega\thvs\|\nc0 \log\|\Omega_\theta\|\nc1\invr\| g\|\nl1+\Const \|\Omega_\theta\|\nc1^{100} \|g\|_\BV
  \end{align}
  in particular:\footnote{\label{f_sigma100-appendix} Our choice of the power $100$ is
    clearly arbitrary: one could substitute it with any sufficiently large number at the
    expense of increasing the constants.}
  \begin{equation} \label{e_estimateMIterate2BV}
    |m\thvz\,g-\Leb\, g|\le \Const
    \|\Omega\thvz\|\nc0\log\|\Omega_\theta\|\nc1\invr \|g\|\nl1
    +\Const \|\Omega_\theta\|\nc1^{100} \|g\|_\BV.
  \end{equation}
\end{lem}
\begin{proof}
  Again we choose $\sigma_2\le \sigma_1$ so that $\tO\thvz$ is of Perron--Frobenius type
  and let $h\thvz\otimes m\thvz$ denote the eigenprojector associated to its maximal
  eigenvalue $e^{\chi\thvz}$.  Combining~\eqref{e_derivativeMeasure}
  and~\eqref{e_m-estimate-0} with the choice $n = \|\Omega_\theta\|\nc0\invr$ we obtain:
  \begin{equation}\label{e_c0-bv-derivative}
  \begin{split}
    |\partial_\paramL m\thvs|&\le\Const \|(\Omega\thvs[\Id-\tOqh\thvs]\invr\tilde{g})\|\nl1 \\
    &\phantom\le+    \Const \|\Omega_\theta\|\nc1\expo{-\const\|\Omega_\theta\|\nc0\invr}\|g\|_\BV
  \end{split}
  \end{equation}
  where recall that $\tilde g= g-h\thvs m\thvs\,g$.  On the other hand, for any $K\in \bN$
  \begin{equation}\label{eq:hoi-hoi-wei}
    \begin{split}
      &\|\Omega\thvs [\Id-\tOqh\thvs]\invr\tilde{g}\|\nl1%
      \le \sum_{k=0}^\infty \|\Omega\thvs\tOh\thvs^k\tilde{g})
      \|\nl1\\
      & \le \Const
      \|\Omega\thvs\|\nc0\left[K\|g\|\nl1  + K |m\thvs\,g| + \tau^{K}\|\tilde g\|\nBV \right]
    \end{split}
  \end{equation}
  since $\tOh\thvs$ is power bounded in $L^1$ and where $\tau\in (0,1)$ is
  determined by the spectral gap in $\BV$.  Note that in the considered range of $\paramL$
  we can assume $\tau$ to be independent on $\paramL$.  We can now choose $K = C\log\|\Omega_\theta\|\nc1\invr$.
  By Lemma~\ref{lem:m-estimate-0}, with the choice
  $n = \|\Omega_\theta\|\nc0\invr$, we can substitute~\eqref{eq:hoi-hoi-wei}
  in~\eqref{e_c0-bv-derivative} to obtain~\eqref{e_estimateMIterate2BV-derivative},
  provided that $C$ has been chosen large enough.
  Integrating~\eqref{e_estimateMIterate2BV-derivative} with respect to $\paramL$ from $0$
  to $1$ yields~\eqref{e_estimateMIterate2BV}.
\end{proof}
\begin{lem}\label{lem:m-theta-BV}
  Under the hypotheses of Lemma~\ref{lem:m-estimate-0}, for any $\theta,\theta'\in\bT$
  sufficiently close, we have
\[
  |m_{\theta,\Omega}(g)-m_{\theta',\Omega}(g)|\leq
  \Const[\log\|\Omega\|\nc0]^2|\theta-\theta'|\|\Omega\|\nc1\|g\|\nBV.
\]
\end{lem}
\begin{proof}
  We choose $\sigma_2$ so that the operators $\tO_{\theta,\Omega}$ are of
  Perron--Frobenius type for all $\theta$ (see Remark~\ref{rem:bv-contraction-conp}).
  Accordingly, there exists $\tau\in (0,1)$ such that\footnote{ By perturbation theory the
    spectral gap varies continuously in $\theta$, hence by compactness there exist an
    uniform spectral gap.} for any $N\in\bN$
  \begin{align*}
    |m_{\theta,\Omega}(g)-m_{\theta',\Omega}(g)|\leq
    \left|\frac{\Leb(\tOh_{\theta,\Omega}^{
          N}g)}{\Leb(h_{\theta,\Omega})}-\frac{\Leb(\tOh_{\theta',\Omega}^{
          N}g)}{\Leb(h_{\theta',\Omega})}\right|+\Const\tau^{N}\|g\|\nBV.
  \end{align*}
  In addition, by~\eqref{e_partialhTheta},~\eqref{eq:D(theta,1)-est}
  and~\eqref{e_estimatePartialm} we have
  \[
    \begin{split}
      |\Leb(h_{\theta,\Omega})-\Leb(h_{\theta',\Omega})|\leq& \int_{\theta}^{\theta'}\deh \vf\left|(\Leb-m_{\vf,\Omega})([\Id-\tOqh_{\vf,\Omega}]^{-1}\tOh_{\vf,\Omega}\widetilde\derivL_{\vf,\Omega}h_{\vf,\Omega})\right|\\
      &+\cO(|\theta-\theta'|\|\Omega\|\nc1)=\cO(|\theta-\theta'|\|\Omega\|\nc1),
    \end{split}
  \]
  where recall
  $\widetilde\derivL_{\theta,\Omega}g =
  \widetilde\derivL_{\theta,\Omega}g-h_{\theta,\Omega}m_{\theta,\Omega}(\derivL_{\theta,\Omega}g)$
  and $\derivL_{\theta,\Omega}$ is defined in~\eqref{eq:derivativeL}.
Also note that
  \[
    \begin{split}
      \frac{\Leb(\tOh_{\theta',\Omega}^{N}g)}{\Leb(h_{\theta,\Omega})}&-\frac{\Leb(\tOh_{\theta',\Omega}^{N}g)}{\Leb(h_{\theta',\Omega})}
=
\frac{\Leb(\tOh_{\theta',\Omega}^{N}g)}{\Leb(h_{\theta',\Omega})}\left[\frac{\Leb(h_{\theta',\Omega})}{\Leb(h_{\theta,\Omega})}-1\right]\\
      =&m_{\theta',\Omega}(g)\cO(|\theta-\theta'|\|\Omega\|\nc1)+\cO(\tau^N\|g\|_\BV)\\
      =&\cO(|\theta-\theta'|\|\Omega\|\nc1)\left(\|g\|_{L^1}+\Const\expo{-\const\|\Omega_\theta\|\nc0^{-1}}\|g\|_\BV\right)+\cO(\tau^N\|g\|_\BV)
    \end{split}
  \]
  where we have used Lemma~\ref{lem:m-estimate-0} with the choice
  $n = \|\Omega_\theta\|\nc0\invr$.  We can then choose $N=C\log\|\Omega\|\nc0^{-1}$ for
  $C$ large enough and continue our estimate to write\footnote{ Remember
    \eqref{eq:basic-taylor}, from which
    $\partial_\theta\tOh_{\theta,\Omega}=\tOh_{\theta,\Omega}\derivL_{
      \theta,\Omega}-\tOh_{\theta,\Omega} \partial_\theta\ln \chi\thvz$, and
    \eqref{e_estimatesecondpartchi0}. Also, in the second line, we use
    \eqref{e_derivativeChi} to exchange $\chi\thvz$ with $1$.}
  \[
    \begin{split}
      |m_{\theta,\Omega}(g)-m_{\theta',\Omega}(g)|&\leq \Const\sum_{k=1}^{N}
      \int_{\theta}^{\theta'}\left|\Leb(\tOh_{\theta,\Omega}^{N-k}
        \tOh_{ \varphi,\Omega}\derivL_{ \varphi,\Omega}\tOh_{\theta,\Omega}^{k-1}g)d\varphi\right|\\
      &\phantom{\leq}+\Const [\log\|\Omega\|\nc0^{-1}] |\theta-\theta'|\|\Omega\|\nc1\|g\|\nl1+\Const\|\Omega\|\nc0^{100}\|g\|_\BV\\
      &\hskip-1.3cm\leq \sum_{k=1}^{N} \int_{\theta}^{\theta'}\left|\Leb\left(\left[e^{\sum_{j=0}^{N-k}\Omega_\theta\circ \favg_{\theta}^j\circ\favg_\vf+\Omega_\vf}-1\right]\derivL_{ \varphi,\Omega}\tOh_{\theta,\Omega}^{k-1}g\right)d\varphi\right|\\
      &\quad+\Const
      [\log\|\Omega\|\nc0^{-1}]|\theta-\theta'|\|\Omega\|\nc1\|g\|\nl1+\|\Omega\|\nc0^{100}\|g\|_\BV,
    \end{split}
  \]
  which, integrating by parts, yields the lemma.k
\end{proof}

\subsection{Generic conditions}\label{subsec:generic}\ \newline
Here we discuss some conditions that prevent non generic behavior of the transfer
operator.  They are arranged by (apparent) increasing strength.  Yet, we will see at the
end of the section that, although in general they are all different, in the particularly
simple case we are considering, they are in fact all equivalent to the weaker condition:
the potential should not be cohomologous to a constant.  As the latter condition holds
generically, all the conditions stated below also hold generically.
\begin{lem}\label{lem:coboundary}
  Let $\theta,\paramL$ be values for which $\tOh_{\theta,\paramL\Omega}$ has a spectral
  gap, $m_{\theta,\paramL\Omega}$ is a measure and $|h_{\theta,\paramL\Omega}|>0$.  If
  $\partial^2_\paramL\chi_{\theta,\paramL\Omega}$ is zero, then $\Omega_{\theta}$ is
  \emph{cohomologous to a constant}, \ie there exists $\beta\in\bR$ and $\phi\in \cC^1$
  such that
  \[
  \Omega_{\theta}=\beta+\phi-\phi\circ f_\theta.
  \]
\end{lem}
\begin{proof}
  Note that if the second derivative is zero for some $\theta$ and $\paramL$ then, by the
  computation implicit in~\eqref{e_secondDerivativeChi}, it follows that the sequence
  $\sum_{k=0}^{n-1}\Omega_{\theta, \paramL\Omega}\circ f_\theta^k$ is uniformly bounded in
  $L^2(\bT,m_{\theta,\paramL\Omega})$ and hence weakly compact.\footnote{ Indeed,
    recalling~\eqref{e_secondDerivativeChia},
    \[
    \begin{split}
      \mnu_{\theta,\paramL\Omega}\left(\sum_{k=0}^{n-1}\Omega_{\theta, \paramL\Omega}\circ
        f_\theta^k\right)^2&=
      n\left\{\mnu_{\theta,\paramL\Omega}(\Omega_{\theta, \paramL\Omega}^2)+2\sum_{j=1}^{n-1}\mnu_{\theta,\paramL\Omega}\left(\Omega_{\theta, \paramL\Omega}\circ f_\theta^j\Omega_{\theta, \paramL\Omega}\right)\right\}\\
      &\quad-2\sum_{j=1}^{n-1}j\mnu_{\theta,\paramL\Omega}\left(\Omega_{\theta, \paramL\Omega}\circ f_\theta^j\Omega_{\theta, \paramL\Omega}\right)\\
      &=-2\sum_{j=n}^\infty\mnu_{\theta,\paramL\Omega}\left(\Omega_{\theta, \paramL\Omega}\circ
        f_\theta^j\Omega_{\theta, \paramL\Omega}\right)-2\sum_{j=1}^{n-1}j\mnu_{\theta,\paramL\Omega}\left(\Omega_{\theta, \paramL\Omega}\circ
        f_\theta^j\Omega_{\theta, \paramL\Omega}\right)
    \end{split}
    \]
    which is bounded by~\eqref{eq:inv-measure-gen} and the spectral gap of
    $\tOh_{\theta,\paramL\Omega}$.}  Let
  $\sum_{k=0}^{n_j-1} \Omega_{\theta,\paramL\Omega}\circ f_\theta^k$ be a weakly
  convergent subsequence and let $\phi\in L^2$ be its limit.  Hence, for any $\vf\in L^2$
  holds
  \[
  \lim_{j\to\infty}\mnu_{\theta,\paramL\Omega}\left( \vf \sum_{k=0}^{n_j-1}\Omega_{\theta,
      \paramL\Omega}\circ f_\theta^k\right)=\mnu_{\theta,\paramL\Omega}(\vf\phi).
  \]
  It follows that, for any $\vf\in \cC^1$,
  \begin{align*}
    \mnu_{\theta,\paramL\Omega}( &\vf[\Omega_{\theta,\paramL\Omega}-\phi+\phi\circ
    f_\theta])= \\
    &=\mnu_{\theta,\paramL\Omega}(\vf \Omega_{\theta,\paramL\Omega}) +
    \lim_{j\to\infty}\sum_{k=0}^{n_j-1}\mnu_{\theta,\paramL\Omega}(\vf [
    \Omega_{\theta,\paramL\Omega}\circ f_\theta - \Omega_{\theta,\paramL\Omega}]\circ f_
    \theta^k)\\
    &=\lim_{j\to\infty}\mnu_{\theta,\paramL\Omega}( \vf \Omega_{\theta,\paramL\Omega}\circ f_
    \theta^{n_j}) =\lim_{j\to\infty} m_{\theta,\paramL\Omega} (
    \Omega_{\theta,\paramL\Omega}\tOh_{\theta,\paramL\Omega}^{n_j}(\vf h_{\theta,\paramL
      \Omega}))\\
    &=\mnu_{\theta,\paramL\Omega}(\vf)\mnu_{\theta,\paramL\Omega}(\Omega_{\theta,\paramL
      \Omega})=0.
  \end{align*}
  Since $\cC^1$ is dense in $L^2$, it follows that
  \begin{equation}
    \label{eq:cob}
    \Omega_{\theta,\paramL\Omega}=\phi-\phi\circ f_\theta\,,\qquad m_{\theta,\paramL\Omega}-
    \textrm{a.s.}
  \end{equation}
  A function with the above property is called a \emph{coboundary}, in
  this case an $L^2$ coboundary.  In fact, more is true: $\phi\in\cC^1$.
  Indeed, recalling that $\mnu_{\theta,\paramL\Omega}( \phi)=0$,
  \[
  \tOh_{\theta,\paramL\Omega}(\Omega_{\theta,\paramL\Omega}\, h_{\theta,\paramL
    \Omega})=\widehat \tO_{\theta,\paramL\Omega}(\phi\, h_{\theta,\paramL\Omega})-\phi\,
  h_{\theta,\paramL\Omega}=-(\Id-\widehat \tO_{\theta,\paramL\Omega})(\phi\, h_{\theta,
    \paramL\Omega}).
  \]
  Note that the above equation has a unique $L^2(\bT,m_{\theta,\paramL\Omega})$ solution.%
  \footnote{ Assume otherwise that the equation
    $\widehat \tO_{\theta,\paramL\Omega}\psi=\psi$ has more than one solution in
    $L^2$.  But for any such solution let $\{\psi_\ve\}\subset \cC^1$ be a sequence that
    converges to $\psi$ in $L^2$, then it converges in $L^1$, moreover
    $m_{\theta,\paramL\Omega}(|\widehat \tO_{\theta,\paramL\Omega}^n(\psi-\psi_\ve)|)\le
    m_{\theta,\paramL\Omega}\,|\psi-\psi_\ve|$.  Thus,
    $\psi=\widehat \tO_{\theta,\paramL
      \Omega}^n\psi_\ve+o(1)=h_{\theta,\paramL\Omega}m_{\theta,\paramL\Omega}\,\psi+o(1)$.
    Hence $\psi=h_{\theta,\paramL\Omega}m_{\theta,\paramL\Omega}\,\psi$.}  Hence
  \[
  \phi=-h_{\theta,\paramL\Omega}^{-1}(\Id-\tOh_{\theta,\paramL\Omega})^{-1}
  \tOh_{\theta,\paramL\Omega}(\Omega_{\theta,\paramL\Omega}\, h_{\theta,\paramL\Omega}) \in
  \cC^1.
  \]
  The proof then follows recalling that
  $\Omega_{\theta,\paramL\Omega} =
  \Omega_\theta-\mnu_{\theta,\paramL\Omega}(\Omega_\theta)$ and setting
  $\beta = \mnu_{\theta,\paramL\Omega}(\Omega_\theta)$.
\end{proof}
\begin{rem} Note that the above lemma applies in particular to the case of real potentials
  (since $m_{\theta,\paramL\Omega}$ is a measure and $h_{\theta,\paramL\Omega}>0$ by
  Lemma~\ref{lem:quasicompact}) and for $\paramL=0$.
\end{rem}
Following~\cite{Guivarch} we introduce
\begin{defin}\label{def:giuvarch}
  A real function $A\in\cC^1$ is called {\em aperiodic}, with respect to the dynamics $f$,
  if there is no $\BV$ function $\beta$ and $\nu_0,\nu_1\in\bR$ such that $A+\beta\circ
  f-\beta$ is constant on each domain of invertibility of $f$, and has range in
  $2\pi\nu_1\bZ+\nu_0$.
\end{defin}
Also in the following we will need the, seemingly stronger, condition.

\begin{defin}\label{def:noi}
  A real function $A\in\cC^1$ is called {\em c-constant}, with respect to the dynamics $f$,
  if there is a $\BV$ function $\beta$ such that $A+\beta\circ
  f-\beta$ is constant on each domain of invertibility of $f$.
\end{defin}

We conclude with the announced proof that all the above properties are equivalent in our
case of interest.
\begin{lem}\label{lem:allthesame}
  If $f\in\cC^2(\bT,\bT)$ and expanding, then any c-constant zero average function
  $A\in \cC^1(\bT,\bR)$ is necessarily a coboundary.
\end{lem}
\begin{proof}
  By definition there exists $\beta\in \BV$ such that $\alpha=A +\beta\circ f-\beta$ where
  $\alpha$ is constant on the invertibility domains of $f$.  If we apply the normalized
  transfer operator $\tOh$ to the previous relation we have
  $\beta=(\Id-\tOh)^{-1}\tOh(\alpha- A)$.  Note that, by hypothesis, $\tOh^k(\alpha- A)$
  is $\cC^1$ except for at most one point (the common image of the boundary points of
  invertibility domains), for each $k>0$.  Thus $\beta$ has at most one (jump)
  discontinuity, which we assume without loss of generality to be at $x=0$.  Since $A$ is
  smooth on $\bT$, calling $\cP$ the partition of invertibility domains, we have:
  \[
    0=\int_0^1A'\deh x= \sum_{p\in\cP}\int_p(\beta'-(\beta\circ f)')\deh
    x=(1-d)(\beta(1^-)-\beta(0^+))
  \]
  where $d$ is the number if invertibility domains of $f$, i.e.\ its topological degree.
  We thus conclude that $\beta$ is in fact continuous on $\bT$ and therefore $\alpha$ has
  to be constant on $\bT$ (hence identically zero).  Consequently, $\beta$ must be smooth
  and correspondingly $A$ is then a $\cC^1$-coboundary.
\end{proof}
\subsection{Non perturbative results}\label{subsec:non-pert}\ \newline
In this section we collect some results that hold when $\paramL$ is large, \ie well
outside the perturbative regime.  Such results hold under the generic conditions discussed
in the previous section.  Even though we have proven that all the conditions are
equivalent, we will state the next lemmata under the conditions that are most natural in
the proof (and for which the lemma might naturally hold in greater generality).
\begin{lem}\label{lem:spectrum-c} If $i\Omega_\theta$ is real, of zero average with
  respect to $\mnu_{\theta,0}$ and is aperiodic, then, for all $\paramL\neq 0$, the
  spectral radius of $\tO_{\theta,\paramL \Omega}$ when acting on both $\cC^1$ and $\BV$
  is strictly less than $1$ and varies continuously with $\paramL$ unless it is smaller
  than $\lambda^{-1}$.
\end{lem}
\begin{proof}
  We start by noticing that, for $\paramL=0$, the maximal eigenvalue is $1$ and all other
  eigenvalues have modulus strictly smaller than $1$.  As pointed out in
  Remark~\ref{rem:bv-contraction-conp} the relevant spectrum on $\cC^1$ and $\BV$ is the
  same.  Hence, for small $\paramL$ we can apply perturbation theory and the first and
  last of~\eqref{eq:derivative-m} imply that
  $\chi_{\theta,\paramL\Omega}=-\frac{\Sigma(\theta)}2\paramL^2+\cO(\paramL^3)$ for some
  $\Sigma(\theta)>0$.  Note that $\Sigma(\theta)$ is continuous in $\theta$,\footnote{
    This follows from~\eqref{e_secondDerivativeChi} and the perturbation theory
    in~\cite{liverani99}.} hence, by Lemma~\ref{lem:coboundary} and
  Lemma~\ref{lem:allthesame}, $\inf_\theta \Sigma(\theta)>0$.  This yields the results for
  small $\paramL$.
  On the other hand, suppose by contradiction that for $\paramL\in\bR\setminus\{0\}$,
  $\tO_{\theta,\paramL\Omega} h_*=e^{i\vartheta}h_*$ for some $h_*\in\cC^1(\bT,\bC)$ and
  $\vartheta\in[0,2\pi)$.  Then $|h_*|\le \tO_{\theta,0}|h_*|$, but
  $\int [\tO_{\theta,0}|h_*|-|h_*|]=0$ implies $|h_*|= \tO_{\theta,0}|h_*|$, so
  $|h_*|=h_0$, the maximal eigenvector of $\tO_{\theta,0}$.  Accordingly,
  $h_*=e^{i\beta}h_0$ where $\beta$ is some real-valued function.  Note that we can choose
  $\beta$ so that it is smooth a part, at most, a jump of $2\pi n$, for some $n\in\bN$, at
  a fixed point of $f_\theta$.  Next, notice that
  \[
  h_0=e^{-i\beta-i\vartheta}\tO_{\theta,\paramL\Omega}h_*=\tO_{\theta,0}\left(e^{i(-i\paramL
      \Omega_{\theta}+\beta-\beta\circ f_\theta-\vartheta)}h_0\right).
  \]
  If we set $\alpha=-i\paramL\Omega_{\theta}+\beta-\beta\circ f_\theta-\vartheta$ and we
  take the real part of the above we get
  \[
  0=\tO_{\theta,0}\left([1-\cos \alpha]h_0\right).
  \]
  Since the function to which the operator is applied is non negative the range of $\alpha$
  must be a subset of $2\pi \bZ$ and can have discontinuities only at the preimages of the
  discontinuity of $\beta$.

  Finally, the continuity of the maximal eigenvalue follows from standard perturbation
  theory \cite{kato} unless the essential spectral radius coincides with the spectral
  radius.
\end{proof}
The above theorem implies that the spectral radius is smaller than $1$ but does not
provide any uniform bound.  Since we will need a uniform bound, more information is
necessary.  This, as already noticed in~\cite{BV05a, Gou09}, can be gained by using
Dolgopyat's technique~\cite{Dima98}.
\begin{lem}\label{lem:spectrum-d} If $i\Omega_\theta$ is real, of zero average with
  respect to $\mnu_{\theta,0}$ and is a non c-constant function with respect to
  $f_\theta$, then for each $\paramL_0>0$ there exists $\tau\in [0,1)$, such that, for any
  $\paramL\not\in [-\paramL_0,\paramL_0]$ and $\theta\in\bT^1$, the spectral radius of
  $\tO\thvs$, when acting on $\cC^1$, is less than $\tau$.
\end{lem}
\begin{proof}
  By Lemma~\ref{lem:uuni} and Theorem~\ref{lem:dolgo-BV} of Appendix~\ref{sec:dolgo},
  there exists $\paramL_1, A>0$ and $\tau\in (0,1)$ such that for all
  $\paramL\not\in[-\paramL_1,\paramL_1]$, $n\ge A\log|\paramL| $ and $\theta\in\bT^1$,
  \begin{equation}\label{eq:dolgo}
    \|\tO\thvs^n\|_{1,\paramL}\le \tau^n
  \end{equation}
  where $\|f\|_{1,\paramL}=|f|_\infty+|\paramL|^{-1}|f'|_\infty$.  Next, by Lemma
 ~\ref{lem:spectrum-c}, the spectral radius of $\tO\thvs$, for
  $|\paramL|\in [\paramL_0,\paramL_1]$ is uniformly smaller than one, hence~\eqref{eq:dolgo}
  is valid also in such a range perhaps modifying $A$ and $\tau$ accordingly.
  Also note that~\eqref{eq:cone-est} implies, for all $n\in\bN$,
  $\|\tO\thvs^n\|_{1,\paramL}\le \Const \|\Omega\|_{\cC^1}$.  In
  particular, by expanding via the Newman series, for any $z\in\bC$, $\tau<|z|\le 1$:
  \begin{equation}\label{eq:dolgo-res}
    \|(\Id z-\tO\thvs)^{-1}\|_{1,\paramL}
    \le \Const\left\{\frac{\|\Omega\|_{\cC^1}A\log|\paramL|}{|z|^{ A\log|\paramL|}}+\frac{(\tau|z|
      ^{-1})^{A\log|\paramL|}}{1-\tau|z|^{-1}}\right\}.
  \end{equation}
  Hence the spectral radius is bounded by $\tau$ while~\eqref{eq:cone-est} implies that the
  essential spectral radius is bounded by $\lambda^{-1}$.
\end{proof}

\section{Dolgopyat's theory}\label{sec:dolgo}
In this appendix we prove a bound for the transfer operator for large $\paramL$.  The
proof is after the work of Dolgopyat on the decay of correlation in Anosov
flows~\cite{Dima98}.  Unfortunately, we need uniform results in $\theta$, so we cannot use
directly the results in~\cite{pollicott99, BV05, AGY06}.  Although the results below can
be obtained by carefully tracing the dependence on the parameters in published proofs,
e.g., in~\cite{AGY06, BV05}, this is a non trivial endeavor.  Therefore we believe the
reader will appreciate the following presentation that collects a variety of results and
benefits from several simplifications allowed by the fact that we treat smooth maps (even
though the arguments can be easily upgraded to cover all the results in the above
mentioned papers).

\subsection{Setting}\label{sec:dolgoSetting}\ \newline%
Let $f\in\cC^r(\bT^2,\bT^1)$ and $\omega\in\cC^{r-1}(\bT^2,\bR)$, $r\ge 2$.  We will
consider the one parameter family of \emph{dynamics} $f_\theta(x) = f(x,\theta)$, of
\emph{potentials} $\Omega_\theta(x)=\omega(x,\theta)$ and the associated \emph{Transfer
  Operators}
\[
\tO_{\theta, i \paramL\Omega_\theta}g(x)=\sum_{y\in f_\theta^{-1}(x)}\frac {e^{i\paramL
    \Omega_\theta(y)}g(y)}{f'_\theta(y)}.
\]
Also, we assume that there exists $\lambda>1$ such that $\inf_{x,\theta}f_\theta'(x)\ge
\lambda$ (uniform \emph{expansivity}).  It is convenient to fix a partition
$\cP_\theta=\{I_i\}$ of $\bT^1$, such that each $I_i$ is a maximal invertibility domain
for $f_\theta$.  We adopt the convention that the leftmost point of the interval $I_1$ is
always zero, which we assume to be a fixed point for every $f_\theta$.\footnote{ Note that this latter assumption does not imply a loss of generality only if the lines of fixed points of $f_\theta$ are homotopic to $\{(0,\theta)\}_{\theta\in\bT}$. If not, one can simply consider a finite open cover of the torus (in the $\theta$ variable), and make the following argument for each element of the covering.}
Note however that all the following is independent of such a choice of the partition.
\begin{rem}%
  Since the maps $f_\theta$ are all topologically conjugate (by structural stability of
  smooth expanding maps), there is a natural isomorphism between $\cP_0$ and $\cP_\theta$,
  $\theta\in\bT^1$.  From now on we will implicitly identify elements of the partitions
  (and their corresponding inverse branches) for different $\theta$ via this isomorphism
  and will therefore drop the subscript $\theta$ when this does not any create confusion.
\end{rem}
At last we require that the $\Omega_\theta$ satisfies a condition (in general, although
not in the present context, see Appendix~\ref{subsec:generic}) stronger than aperiodicity;
namely we assume it is not \emph{c-constant} (see Definition~\ref{def:noi}).

Let $\cH_{n}$ be the collection of the inverse branches of $f_\theta^n$ as defined by the
partition $\cP$.  Note that an element of $\cH_{n}$ can be written as $h_1\circ\cdots\circ
h_n$ where $h_i\in\cH_{1}$, thus $\cH_{n}$ is isomorphic to $\cH_{1}^n$.  It is then
natural to define $\cH_{\infty}=\cH_{1}^\bN$.

\subsection{Uniform uniform non integrability (UUNI)}\label{sec:uuni}\ \newline
The first goal of this section is to prove the following fact.
\begin{lem}\label{lem:uuni}
  In the hypotheses specified in Subsection~\ref{sec:dolgoSetting} there exist $\Ctre>0$
  and $n_0\in\bN$ such that, for each $n\ge n_0$ and $\theta\in\bT^1$,
  \begin{equation}\label{eq:uni}
    \sup_{h_\theta, \kappa_\theta\in\cH_{ n}}\inf_{x\in\bT^1}\left|\de{}{x} (\Omega_{n,\theta}
      \circ h_\theta)(x)-\de{}{x} (\Omega_{n,\theta}\circ \kappa_\theta)(x)\right|\ge \Ctre,
  \end{equation}
  where,
  $\Omega_{n,\theta}:=\sum_{k=0}^{n-1}\Omega_{\theta}\circ f_{\theta}^k$.
\end{lem}
\begin{rem}
  Condition~\eqref{eq:uni}, when referred to a single map, is commonly called
  \emph{uniform non integrability} (UNI for short) and has been originally introduced by
  Chernov in~\cite{Che98}, a remarkable paper which constituted the first breakthrough in
  the quantitative study of decay of correlations for flows.  The difference here is due
  to the fact that we have a family of dynamics, rather that only one, and we require a
  further level of uniformity.  The relation between UNI and not being cohomologous to a
  piecewise constant function was first showed in~\cite[Proposition 7.4]{AGY06}.  The
  above Lemma constitutes a not very surprising extension of the aforementioned
  proposition.
\end{rem}
\begin{proof}[{\bf Proof of Lemma~\ref{lem:uuni}}]
  Suppose the lemma to be false, then given $a\in\bN$ large enough to be chosen later,
  there exist sequences $\{n_j,\theta_j\}$, $\lambda^{-n_j}<\frac {1-\lambda\invr}{2j^a}$,
  such that for each $h,\kappa\in\cH_{\theta_j,n_j}$ there exists $x_{j,h,\kappa}\in\bT^1$
  such that
  \[
  \left|\de{}{x} (\Omega_{n_j,\theta_j}\circ h)(x_{j,h,\kappa})-\de{}{x}
    (\Omega_{n_j,\theta_j}\circ \kappa)(x_{j,h,\kappa})\right|\le \frac 1{j^a}.
  \]
  Start by noting that if $h\in\cH_{n}$, then $h=h_1\circ \cdots\circ h_n$ with $h_i\in\cH_{
    1}$, and
  \[
  \Omega_{n,\theta}\circ h=\sum_{k=0}^{n} \Omega_{\theta}\circ h_{k+1}\circ\cdots\circ h_{n}.
  \]
  Note that all the $\Omega_{n,\theta}\circ h$ belong to
  $\cC^{r-1}(\bT^1\setminus\{0\},\bR)$.  For further use, given $h\in \cH_{\infty}$, let us
  define, for each $n\in\bN$,
  \begin{equation}\label{eq:como}
    \Xi_{\theta,n, h}(x)=\sum_{k=0}^{n} \Omega_{\theta}\circ h_{k}\circ\cdots\circ h_{n}(x)-
    \sum_{k=0}^{n} \Omega_{\theta}\circ h_{k}\circ\cdots\circ h_{n}(x_0),
  \end{equation}
  for some fixed $x_0\in \bT^1$.  We remark that, by usual distortion arguments, for each
  $h\in\cH_{\infty}$, we have $\| \Xi_{\theta,n, h}\|_{\cC^1}\le \Const$.  For each
  $h\in\cH_{n}$ and $k\le n$ let $\bar h_k:=f_\theta^k\circ h$.  Next, for each $j\in \bN$
  and $p_j\le n_j$, let $\ell\in\cH_{\theta_j, p_j}$.  Then, for each $h, \kappa\in
  \cH_{\theta_j,n_j-p_j}$, letting $x_{j,h\circ \ell,\kappa\circ \ell}=z$,
  \[
  \begin{split}
    &\frac 1{j^a}\ge \left|\de{}{x} (\Omega_{n_j,\theta_j}\circ h\circ
      \ell)(z)-\de{}{x}
      (\Omega_{n_j,\theta_j}\circ \kappa\circ\ell)(z)\right|\\
    &=\left|\sum_{k=0}^{n_j-p_j} \Omega_{\theta_j}'\circ \bar h_k\circ\ell(z) \cdot ( \bar
      h_k\circ
      \ell)'(z)-\Omega_{\theta_j}'\circ  \bar \kappa_k\circ\ell(z) \cdot (\bar \kappa_k\circ \ell)'(z)\right|\\
    &\ge\left|\de{}{x} (\Omega_{n_j-p_j,\theta_j}\circ h)(\ell(z))-\de{}{x}
      (\Omega_{n_j-p_j, \theta_j}\circ \kappa)(\ell(z))\right| \Lambda^{-p_j},
  \end{split}
  \]
  where $\Lambda=\sup_{\theta,x}|f'_{\theta}(x)|$.  Accordingly, setting
  $\cP_{\theta,n}=\{h(\bT^1)\}_{h\in\cH_{ n}}$, for each $I\in\cP_{\theta_j,p_j}$ we have
  \[
  \sup_{x\in I}\left|\de{}{x} (\Omega_{n_j-p_j,\theta_j}\circ h)(x)-\de{}{x} (\Omega_{n_j-
      p_j,\theta_j}\circ \kappa)(x)\right|\le \frac {\Lambda^{p_j}}{j^a}+ \Const\lambda^{-p_j}.
  \]
  Thus, setting $\bar n_j=n_j-p_j$, choosing $p_j=\Const \log j$ and provided that $a$ has
  been chosen large enough, we have that for each $h,\kappa\in\cH_{\theta_j,\bar n_j}$,
  \[
  \left\|\de{}{x} (\Omega_{\bar n_j,\theta_j}\circ h)-\de{}{x} (\Omega_{\bar n_j,\theta_j}
  \circ \kappa)\right\|_\infty\le \frac {\Const}j.
  \]
  Next, for $h_{\theta,1},\cdots, h_{\theta,m}\in\cH_{1}$, let
  $C_m=\sup_{\theta}\sup_{\{h_{\theta,i}\}}\|\partial_\theta [h_{\theta,1}\circ\cdots\circ
  h_{\theta,m}]\|_\infty$.  Then
  \[
  \partial_\theta [h_{\theta,1}\circ\cdots\circ h_{\theta,m}]=[\partial_\theta
  h_{\theta,1}]\circ h_{\theta, 2}\cdots\circ h_{\theta,m}+h_{\theta,1}'\circ h_{\theta,
    2}\cdots\circ h_{\theta,m}\cdot \{\partial_\theta [h_{\theta, 2}\cdots\circ
  h_{\theta,m}]\}
  \]
  implies $C_m\le \Const+\lambda^{-1} C_{m-1}$, that is $C_m\le \Const$.  This implies that
  \[
  \|\partial_\theta\partial_x\left[ \Omega_{n,\theta}\circ h_\theta\right]\|_\infty\le \Const.
  \]
  We can then consider a subsequence $\{j_k\}$ such that $\{\theta_{j_k}\}$ converges, let
  $\bar \theta$ be its limit.  Also, without loss of generality, we can assume that $j_k\ge
  {2\Const} k$ and$|\theta_{j_k}-\bar \theta|\le k^{-1}$.  Thus, for $k$ large enough and
  for each $h,\kappa \in\cH_{\bar n_{j_k}\bar \theta}$, we have
  \begin{equation}\label{eq:indep-hk}
    \left\|\de{}{x} (\Omega_{\bar n_{j_k},\bar\theta}\circ h)(x)-\de{}{x} (\Omega_{\bar
      n_{j_k},\bar\theta}\circ \kappa)\right\|_\infty\le \frac 1k.
  \end{equation}
  We are now done with the preliminary considerations and we can conclude the argument.  Let
  $\overline{\Xi}_{k,h}=\Xi_{\bar\theta,n_{j_k},h}$.  Since
  \[
  \begin{split}
    |\Omega_{\theta}\circ h_{k+1}\circ\cdots\circ h_{n}(x)- \Omega_{\theta}\circ h_{k+1}\circ
    \cdots\circ h_{n}(x_0)|&\le \|\de{}{x}\Omega_{\theta}\circ h_{k+1}\circ\cdots\circ h_{n}\|_\infty\\
    &\le \Const\lambda^{-n+k},
  \end{split}
  \]
  it follows that the limit $\overline{\Xi}_h=\lim_{k\to\infty}\overline{\Xi}_{k,h}$
  exists in the uniform topology.  Note that, since the derivative of
  $\overline{\Xi}_{k,h}$ are uniformly bounded, $\overline{\Xi}$ is Lipschitz in
  $\bT^1\setminus\{0\}$.  In addition, since $\overline{\Xi}_h(x_0)=0$,
  equation\eqref{eq:indep-hk} implies, for each $h,\kappa\in\cH_{\bar\theta,\infty}$,
  \begin{align*}
    \left\|\de{}{x}\left[
      \overline{\Xi}_{k,h}-\overline{\Xi}_{k,\kappa}\right]\right\|_\infty\le\frac 2k.
  \end{align*}
  It follows that $\overline{\Xi}_h=\Phi$ is independent of $h$.  Finally, choose
  $h\in\cH_{1}$ and $\bar h\in\cH_{\infty}$ such that $\bar h=h\circ h\circ \cdots$, then,
  if $x\in h^{-1}(\bT^1)$,
  \[
  \Xi_{\theta,n,\bar h}\circ f_{\theta}(x)-\Xi_{\theta,n, \bar h}(x)=\Omega_\theta
  -\Omega_\theta(h^n(x)).
  \]
  Since $f_\theta$ has exactly one fixed point $x_I$ in each $I\in\cP_\theta$,
  $\lim\limits_{n\to\infty}h^n(x)=x_I$, where $I=h(\bT^1)$.  From the above considerations
  it follows
  \[
  \Phi\circ f_{\bar \theta}-\Phi=\Omega_{\bar \theta}+\Psi
  \]
  where $\Psi$ is constant on the elements of $\cP_{\bar\theta}$ and $\Phi\in\BV$.  That is, $\Omega_{\bar \theta}$ is c-constant, contrary to the hypothesis.
\end{proof}
It is now easy to obtain the result we are really interested in.
\begin{cor}\label{cor:uuni}
  In the hypotheses specified in Subsection~\ref{sec:dolgoSetting}, there exists
  $n_1\in\bN$ and $h,\kappa\in\cH_{n_1}$ such that, for each $n\ge n_1$, $\theta\in\bT^1$
  and $\ell\in\cH_{n-n_1}$
  \begin{equation}\label{eq:uuni}
    \inf_{x\in\bT^1}\left|\de{}{x} (\Omega_{n,\theta}\circ\ell \circ h)(x)-\de{}{x} (\Omega_{n,
        \theta}\circ\ell \circ \kappa)(x)\right|\ge \frac{\Ctre}2,
  \end{equation}
\end{cor}
\begin{proof}
  Let $n_1\ge n_0$.  Then, for $h\in\cH_{n_1}$ and $\ell\in\cH_{n-n_1}$
  \[
  \Omega_{n,\theta}\circ\ell \circ h=\Omega_{n_1,\theta}\circ h+\Omega_{n-n_1,\theta}\circ\ell
  \circ h.
  \]
  Thus, by Lemma~\ref{lem:uuni}, we can choose $h,\kappa$ so that
  \begin{align*}
    \inf_{x\in\bT^1}\left|\de{}{x} \left(\Omega_{n,\theta}\circ\ell \circ
        h-\Omega_{n,\theta}\circ\ell \circ \kappa \right)(x)\right|&\ge
    \frac34 \Ctre- \Const (|h'|_\infty+|\kappa'|_\infty) \\
    &\ge\frac34\Ctre-C_\#\lambda^{-n_1}.
  \end{align*}
  The result follows by choosing $n_1$ large enough.
\end{proof}
\subsection{Dolgopyat inequality}\label{subsec:dolgo}\ \newline
In order to investigate the operator $\tO_{\theta, i \paramL\Omega_\theta}$ for large
$\paramL$ it is convenient to use slightly different norms and operators.  The reason is
that on the one hand, the main estimate is better done in a $\paramL$ dependent norm and,
on the other hand, it is convenient to have operators that are contractions.  Let
$\rho:=h_{\theta,0}[\int h_{\theta,0}]^{-1}$ be the invariant density of the operator
$\tO_{\theta, 0}$ and define
\begin{align*}
  \tOn_{\theta, i \paramL\Omega_\theta}(g) &=%
  \rho\invr \tO_{\theta, i\paramL\Omega_\theta}(\rho g) \;,&%
  \|g\|_{1,\paramL}&=%
  \|g\|\nc0+\frac{\|g'\|\nc0}{|\paramL|},
\end{align*}
Then we have\footnote{ The first follows
  trivially from $|\tOn_{\theta, i \paramL\Omega_\theta}g|\le \|g\|_{L^\infty} \tOn_{\theta, 0}1$ and
  $\tO_{\theta, 0}\rho=\rho$.  The second from the standard $\|\tO_{\theta,
    i \paramL\Omega_\theta}(g)\|_{L^1}\le\|g\|_{L^1}$.}
\begin{align}\label{eq:L2-cont}
  \|\tOn_{\theta, i \paramL\Omega_\theta}(g)\|\nc0&%
  \le\|g\|\nc0, &%
  \|\tOn_{\theta, i \paramL\Omega_\theta}(g)\|_{L^1_\rho}&%
  \le\|g\|_{L^1_\rho},
\end{align}
as announced.  Moreover, by~\eqref{eq:cone-est}, it follows that, for
$\paramL\ge\paramL_0$, with $\paramL_0>0$ large enough, and $n\in\bN$,
\begin{equation}\label{eq:L2-cont-1}
  \|\tOn_{\theta, i \paramL\Omega_\theta}^n(g)\|_{1,\paramL}\le \Const\lambda^{-n}\|g\|
  _{1,\paramL}+B_0\|g\|\nc0,
\end{equation}
for a fixed constant $B_0$.  Fix $\bar\lambda\in (\lambda,1)$ and choose $n_2\in\bN$ such
that $\Const \lambda^{-n_2}\le\bar \lambda^{-n_2}$.  Also, for future use, we chose $n_2$
so that $\bar\lambda^{-n_2}\le \frac 12$.  Iterating the above inequalities by steps of
length $n_3\in \bN$, $n_3\ge n_2$, we have
\begin{equation}\label{eq:phase-zero}
  \|\tOn_{\theta, i \paramL\Omega_\theta}^{kn_3}(g)\|_{1,\paramL}\le \bar\lambda^{-k n_3}\|g\|
  _{1,\paramL}+B_0\sum_{j=0}^{k-1}\bar\lambda^{-(k-j-1)n_3}\|\tOn_{\theta, i \paramL
    \Omega_\theta}^{jn_3}(g)\|\nc0.
\end{equation}
\begin{thm}\label{lem:dolgo-BV}
  If condition~\eqref{eq:uuni} is satisfied, then there exists $A, B>0$, and $\gamma<1$
  such that for all $|\paramL|\ge B$ and $n\ge A\log|\paramL|$, we have
  \[
  \sup_{\theta\in\bT^1}\|\tOn_{\theta, i \paramL\Omega_\theta}^n\|_{1,\paramL}\le \gamma^n.
  \]
\end{thm}
\begin{rem}
  In fact, Theorem~\ref{lem:dolgo-BV}, for fixed $\theta$, is a special case
  of~\cite[Theorem 1.1]{BV05}.  To be precise,~\cite[Theorem 1.1]{BV05} is stated for a
  single map and with strictly positive roof functions (a role here played by
  $\Omega_\theta$).  The latter can easily be arranged by multiplying the transfer operator
  by $e^{i 2\|\Omega_\theta\|}$, which does not change the norm.  In addition, a careful
  look at the proof should show that $A,B, \gamma$ depend on the map and potential only
  via $n_1, \Ctre$ of~\eqref{eq:uuni} and $\|f'_\theta\|_\infty, \|
  (f'_\theta)\invr\|_\infty, \|f''_\theta (f'_\theta)\invr\|_\infty$,
  $\|\Omega_\theta\|_{\cC^2}$ which, in the present case, are all uniformly bounded.
  Nevertheless, we think the reader may appreciate the following simpler, self-contained,
  proof rather than being referred to the guts of~\cite{BV05}.
\end{rem}
\begin{proof}[{\bf Proof of Theorem~\ref{lem:dolgo-BV}}]
  For each $g\in\cC^1$ set $g_k= \tOn_{\theta, i \paramL\Omega_\theta}^{kn_3} g$, with
  $n_3\ge n_1$ from Corollary~\ref{cor:uuni} and $n_3\ge n_2$ as in equation
  ~\eqref{eq:phase-zero}.  The basic idea, going back to Dolgopyat~\cite{Dima00}, is to
  construct iteratively functions $u_k\in\cC^1(\bT^1,\bRp)$ such that $|g_k(x)|\le
  u_k(x)$ for all $k\in\bN$ and $x\in\bT^1$ and on which one has good bounds.  More
  precisely:
  \begin{lem}\label{L1-lemma}%
    There exists constants $K, \paramb,B_1, \paramL_0>1$, $\tau>0$, $n_3\ge\max\{ n_1,n_2\}$ and, for all $g\in\cC^1$ with $\|g'\|_{\cC^0}\le \paramb |\paramL|^{-1}\|g\|_{\cC^0}$, functions $\{\Gamma_{g,k}\}_{k\in\bN}\in \cC^1(\bT^1,[4/5,1])$ such that, for all
    $|\paramL|\ge \paramL_0$ and $k\in\bN$,
    \begin{equation}\label{eq:Gamma-cond1}
      \|\Gamma_{g,k}'\|_{L^\infty}\le B_1|\paramL|,
    \end{equation}
    and, setting
    $u_0=\|g\|_\infty+\paramb^{-1}|\paramL|^{-1}\|g'\|_\infty$ and $u_{k+1}=\tOn_{\theta,
      0}^{n_3}(\Gamma_{g,k}u_k)$, we have, for any $x\in\bT^1$,
    \begin{align*}
      \max\{|u'_k(x)|, |g_k'(x)|\}&\le \paramb|\paramL| u_k(x)\;,&%
      |g_k(x)|&\le u_k(x)
    \end{align*}
    and, for any $I=[a_1,a_2]\subset \bT^1$ so that $|a_2-a_1|\ge 4K|\paramL|\invr$:
    \begin{align*}
      \int_I \tOn_{\theta, 0}^{n_3}\Gamma_{g,k}^2&\le e^{-3\tau}|I|.
    \end{align*}
  \end{lem}Let us postpone the proof of Lemma~\ref{L1-lemma} and see how it implies the wanted
  result.  First of all, note that if $\|g_k'\|_{\cC^0}\ge \paramb |\paramL|^{-1}\|g_k\|_{\cC^0}$, then equation~\eqref{eq:phase-zero} implies $\|g_{k+1}\|_{1,\paramL}\le \gamma\|g_k\|_{1,\paramL}$, provided $\paramb$ has been chosen large enough.  We can thus assume $\|g'\|_{\cC^0}\le \paramb |\paramL|^{-1}\|g\|_{\cC^0}$ without loss of generality.

  Next, note that, for any $j_0\in\bN$, by equation~\eqref{eq:cone-est} and choosing $\bar\lambda$ as in equation
 ~\eqref{eq:phase-zero}, we can write
  \[
  \left|\de{}{x}\tOn_{\theta, 0}^{j_0 n_3}(u_k)\right|\le
  \bar\lambda^{-j_0n_3}\tOn_{\theta, 0} ^{j_0 n_3}(|u_k'|)+B\tOn_{\theta, 0}^{j_0 n_3}(u_k)
  \le (\bar\lambda^{-j_0n_3}\paramb|\paramL|+B)\tOn_{\theta, 0}^{j_0 n_3}(u_k).
  \]
  By eventually increasing $\paramL_0$, we can choose $j_0$ so that, for all
  $\paramL\ge \paramL_0$,
  \begin{equation}\label{eq:another-ly-like}
    \left|\de{}{x}\rho\left[\tOn_{\theta, 0}^{j_0 n_3}(u_k)\right]^2\right|\le\frac{\tau |\paramL|}{4K}\rho\left[\tOn_{\theta, 0}
      ^{j_0 n_3}(u_k)\right]^2.
  \end{equation}
  Thus, given any partition $\{p_m\}$ of $\bT^1$ in intervals of size between
  $3K|\paramL|^{-1}$ and $4K|\paramL|^{-1}$ we have
  \begin{align*}
    \int_{\bT^1} u_{k+j_0}^2\rho&\le\int_{\bT^1} \left\{\tOn_{\theta,
      0}^{n_3}[\Gamma_{k+j_0-1} (\tOn_{\theta, 0}^{(j_0-1)
      n_3}u_{k})]\right\}^2\rho\le\int_{\bT^1}\rho\tOn_{\theta,
      0}^{n_3}\Gamma_{k+j_0-1}^2\cdot \tOn_{\theta, 0}^{j_0 n_3}
    u_{k}^2\\
    &\le \sum_m\int_{p_m}\tOn_{\theta, 0}^{n_3}\Gamma_{k+j_0-1}^2\frac
    {e^\tau}{|p_m|}\int_{p_m}\rho\tOn_{\theta, 0}^{j_0 n_3}
    u_{k}^2 \\
    &\le e^{-2\tau}\int_{\bT^1}\rho\tOn_{\theta, 0}^{j_0
      n_3}u_{k}^2=e^{-2\tau}\int_{\bT^1}u_{k}^2 \,\rho,
  \end{align*}
  where in the second inequality of the first line we have used Schwarz inequality with respect to the
  sum implicit in $\tOn_{\theta, 0}^{n_3}$ and $\tOn_{\theta, 0}^{n_3(j_0-1)}$; the second line follows from
 ~\eqref{eq:another-ly-like}; the first inequality of the third line follows from the last
  assertion of Lemma~\ref{L1-lemma}, while the last inequality follows from the well known
  contraction of $\tOn_{\theta, 0}$ in $L^1_\rho$.

  Finally, iterating the above equation, we obtain
  \[
  \|u_{kj_0}\|_{L^2_\rho}\le e^{-k\tau}\|u_0\|_{L_\rho^2}.
  \]
  Accordingly, there exists $A>0$ such that, for all $n\ge \frac A 2\log|\paramL|$ we have $\bar\lambda\le |\paramL|^{-2}$ and
  \[
  \|g_n\|_{L^2_\rho}\le \|u_n\|_{L^2_\rho}\le|\paramL|^{-4}\|u_0\|_{L_\rho^2}.
  \]
  The above equation together with~\eqref{eq:phase-zero} and the fact that, for all $\tilde g\in
  \cC^1$,\footnote{ Indeed, $|\tilde g(x)|^2\le \|\tilde g\|_{L^2_\rho}^2+2\int_{\bT^1}|\tilde g|\,|\tilde g'|$.}
  \[
  \|\tilde g\|_\infty\le
  \|\tilde g\|_{L^2_\rho}^{\frac 12}\left[\|\tilde g\|_{L^2_\rho}+2\|\tilde g'\rho^{-1} \|_\infty\right]^{\frac 12}
  \]
  yields $\|\tOn_{\theta, i \paramL\Omega_\theta}^{n} g\|_{1,\paramL}\le
  |\paramL|^{-1}\|g\|_{1,\paramL}$ for all $n\in[ A\log|\paramL|,
  2A\log |\paramL|]\cap\bN$.  The latter readily implies Theorem~\ref{lem:dolgo-BV}.
\end{proof}
\begin{proof}[{\bf Proof\ of Lemma~\ref{L1-lemma}}]
  Since $u_0=\|g\|_\infty + \paramb\invr|\paramL|\invr\|g'\|_\infty$,
  trivially, $\max\{|u_0'|,|g_0'|\}\le \paramb|\paramL| u_0$ and $|g_0(x)|\le u_0(x)$ for all
  $x\in\bT^1$.  Suppose, by induction, that $\max\{|u_k'|,|g_k'|\}\le \paramb|\paramL| u_k$,
  and $|g_k(x)|\le u_k(x)$ for all $x\in\bT^1$, then~\eqref{eq:cone-est} implies
  \begin{equation}\label{eq:phase-one}
    \begin{split}
      |g_{k+1}'(x)|&\le \bar\lambda^{-n_3} (\tOn_{\theta, 0}^{n_3}|g'_k|)(x)+B|\paramL| (\tOn_{\theta, 0}
      ^{n_3}|g_k|)(x)\\
      &\le \paramb|\paramL|\left[\bar \lambda^{-n_3}+B\paramb^{-1}\right](\tOn_{\theta, 0}^{n_3}u_k)(x)
      \le \paramb|\paramL|\frac 54 \left[\bar \lambda^{-n_3}+B\paramb^{-1}\right]u_{k+1}
    \end{split}
  \end{equation}
  where we have assumed the existence of the wanted $\Gamma_{g,k}$ that remains to be
  constructed.  By choosing $\paramb$ large enough it follows
  \[
  |g_{k+1}'(x)|\le \paramb|\paramL|u_{k+1}.
  \]
  The proof of the analogous inequality for $u_k$ being similar, but it uses~\ref{eq:Gamma-cond1}.

  Next, let $h_*,\kappa_*\in\cH_{n_3}$ be two branches satisfying~\eqref{eq:uuni}, whose existence follows by
  Corollary~\ref{cor:uuni}, and let us define the set $\widehat \cH=\cH_{n_3}\setminus
  \{h_*,\kappa_*\}$.  Then,
  \begin{equation}\label{eq:phase-two}
    \left|g_{k+1}(x)\right|\le \sum_{h\in\widehat\cH}\frac{(u_k \rho)\circ  h(x)}{\rho(x) (f_
      \theta^{n_3})'\circ h(x)}+\left|\sum_{h\in\{h_*,\kappa_*\}} \frac{e^{i\paramL\Omega_{n_3,\theta}
          \circ h(x)}(\rho g_k)\circ h(x)}{\rho(x)\,(f_\theta^{n_3})'\circ h(x)}\right|.
  \end{equation}
  To conclude we need a sharp estimate for the second term in~\eqref{eq:phase-two}, where a
  cancellation may take place.  To this end it is helpful to introduce a partition of unity.
  This can be obtained by a function $\phi\in\cC^2(\bRp,[0,1])$ such that $\phi(x)=1$ for
  $|x|\le \frac 12$ and $\phi(x)=0$ for $|x|\ge 1$ and $1=\sum_{n\in\bN}\phi(x-n)$, for
  all $x\in\bR$.  We then define $L_\paramL= \lfloor K^{-1}|\paramL|\rfloor$,
  $\psi_m(x)=\phi(L_\paramL x-m)$.  Note that, by construction,
  $\sum_{m=0}^{L_\paramL-1}\psi_m=1$ (here we are interpreting the $\psi_m$ as functions on
  $\bT^1$).  Let $I_m=\supp\psi_m$ and let $x_m$ be its middle point.  Note that $\frac
  K{|\paramL|}\le |I_m|\le \frac {2K}{|\paramL|}$.

  To continue, for each $m\in\{1,\cdots,L_\paramL-1\}$, we must consider two different cases.
  First suppose that there exists $\bar h_m\in\{h_*,\kappa_*\}$ such that $g_k(\bar
  h_m(x_m))\le \frac 12u_k(\bar h_m(x_m))$.  Note that, for $z\in h(I_m)$,
  \[
  e^{-2K\lambda^{-n_3}\paramb}u_k(z)\le u_k(h(x_m))\le e^{2K\lambda^{-n_3}\paramb}u_k(z)
  \]
  hence
  \[
  |g_k'(z)|\le \paramb|\paramL| u_k(z)\le \frac 32\paramb|\paramL| u_k(\bar h_m(x_m))
  \]
  provided $K\lambda^{-n_3}\paramb\le \frac 1{8}$.  Which implies, for all $x\in I_m$,
  \begin{equation}\label{eq:phase-three0}
    |g_k\circ \bar h_m(x)|\le \frac 45|u_k\circ \bar h_m(x)|,
  \end{equation}
  provided $K\lambda^{-n_3}\paramb\le \frac 1{90}$.

  Second, suppose that, for each $h\in\{h_*,\kappa_*\}$, $|g_k(h(x_m))|\ge\frac 12
  u_k(h(x_m))$.  Then
  \[
  |g_k'(z)|\le \paramb|\paramL| u_k(z)\le 2\paramb|\paramL| u_k(h(x_m))\le 4\paramb|\paramL|\, |g_k(h(x_m))|.
  \]
  The above implies $\frac 12|g_k(h(x_m))|\le |g_k(z)|\le 2|g_k(h(x_m))|$ provided
  $K\lambda^{-n_3}\paramb\le \frac 18$.

  Thus, setting $A_{h}=\frac{(\rho g_k)\circ h(x)}{\rho(x)\,(f_\theta^{n_3})'\circ h(x)}$, we have
  \begin{align*}
    |A_{h}'(x)|&\le C_\# (h'(x)^2\paramb|\paramL|+h'(x))|u_k(h(x))|\le C_\# (h'(x)^2\paramb|\paramL|+h'(x))|
    g_k(h(x))| \\
    &\le C_\# (h'(x)\paramb|\paramL|+1) |A_{h}(x)|.
  \end{align*}
  Defining $A_{h}=e^{i\theta_{h}}B_{h}$, with $\theta_{h}, B_{h}$ real and $B_{h}\ge
  0$, we have
  \[
  |A_{h}'|\ge \frac 1{\sqrt 2}\left(|\theta_{h}' B_{h}|+|B_{h}'|\right).
  \]
  The above implies that, given $\paramb$, we can chose $n_3$ and $\paramL_0$ large enough so that
  \[
  \begin{split}
    &|\theta_{h}'(x)|\le C_\# (h'(x)\paramb|\paramL|+1)\le \frac {\Ctre|\paramL|}{32\pi},\\
    &|B'_{h}(x)|\le \frac {\Ctre|\paramL|}{32\pi} B_{h}(x).
  \end{split}
  \]
  Hence, setting $\Theta:=\Omega_{n_3,\theta}\circ h_*-\Omega_{n_3,\theta}\circ \kappa_*+|
  \paramL|^{-1}(\theta_{h_*}-\theta_{\kappa_*})$,
  \[
  C_\#\ge \left|\de{}{x}\Theta\right|\ge\frac{\Ctre}4.
  \]
  In turns, this implies that the phase $\Theta$ has at least one full oscillation in $I_m$ provided
  $K\ge\frac {8\pi}{\Ctre}$.  Also, $\inf_{I_m} B_{h}\ge\frac 12
  \sup_{I_m}|B_{h}|$, provided $K\le\frac{10 \pi}{\Ctre}$.  Next, suppose that $\|B_{h_*}\|_\infty\ge
  \|B_{\kappa_*}\|_\infty$, (hence $4|B_{h_*}(x)|\ge |B_{\kappa_*}(x)|$), and set $\bar
  h_m=\kappa_*$, the other case being treated exactly in the same way (interchanging the role of
  $h_*$ and $\kappa_*$, hence setting $\bar h_m=\kappa_*$).  Given the above notation, the last term of
 ~\eqref{eq:phase-two} reads
  \[
  |B_{h_*}-e^{i\paramL\Theta}B_{\kappa_*}|= \left[B_{h_*}^2+B_{\kappa_*}^2-2 B_{h_*}
    B_{\kappa_*}\cos\paramL\Theta\right]^{\frac 12}.
  \]
  It follows that there exists a constant $\Csei>0$ and intervals $J\subset \stackrel {\circ}{I}_m$,
  $\frac{4\pi}{\Ctre|\paramL|}\ge |J|\ge \frac{\Csei}{|\paramL|}$ on which
  $\cos\paramL\Theta\ge 0$.  Then, on each such interval $J$,
  \[
  |B_{h_*}-e^{i\paramL\Theta}B_{\kappa_*}|\le \left[B_{h_*}^2+B_{\kappa_*}^2\right]^{\frac
    12}\le B_{h_*}+\frac 45 B_{\kappa_*}.
  \]
  We can then define $\Xi_m\in\cC^\infty(I_m,[\frac 45,1])$
  such that $\Xi_m(x)=1$ outside the intervals $J$, $\Xi_m(x)=\frac 45$ on the mid third of
  each
  $J$ and $\|\Xi_m\|_{\cC^1}\le C_\#|\paramL|$.  It follows
  \begin{equation}\label{eq:phase-three}
    \begin{split}
      |B_{h_*}-e^{i\paramL\Theta}B_{\kappa_*}|&\le \left|\frac{(g_k\rho)\circ h_*}{(\rho f^{n_3}_
          \theta)'\circ h_*}\right|+\Xi_m \left|\frac{(g_k\rho)\circ {\kappa_*}}{(\rho f^{n_3}_\theta)'\circ
          \kappa_*}\right|\\
      &\le \frac{(u_k\rho)\circ h_*}{(\rho f^{n_3}_
        \theta)'\circ h_*}+\Xi_m \frac{(u_k\rho)\circ {\bar h_m}}{(\rho f^{n_3}_\theta)'\circ
        \bar h_m}.
    \end{split}
  \end{equation}
  We can finally define the function $\Gamma_{g,k}\in\cC^1(\bT^1,[0, 1])$ as
  \[
  \Gamma_{g,k}(x)=\sum\limits_{m=0}^{L_\paramL-1}\psi_m\circ f^{n_3}(x)\Gamma_{k,m}(x).
  \]
  where
  \[
  \Gamma_{k,m}(x)=\begin{cases}
    1&\textrm{ if }x\in h(I_m), h\neq\bar h_m,\\
    \Xi_m\circ f^{n_3}(x) &\textrm{ if } x\in \bar h_m(I_m).
  \end{cases}
  \]
  Note that with the above definition, condition~\eqref{eq:Gamma-cond1} is satisfied.
  Also, by equations~\eqref{eq:phase-three0} and~\eqref{eq:phase-three}, it follows
  $|g_{k+1}|\le u_{k+1}$.

  Finally, we must check the last claim of the Lemma.  Note that it suffices to consider intervals $I$ of size between
  $4K|\paramL|^{-1}$ and $8K|\paramL|^{-1}$.
  \[
  \begin{split}
    \int_I\tOn_{\theta, 0}^{n_3}\Gamma_{g,k}^2&\le
    \sum_m\int_{I}\psi_m \cdot\tOn_{\theta, 0}^{n_3}(\Gamma_{k,m}^2)\\
    &\le \sum_m\left[\sum_{h\in\cH_{n_3}/\{\bar h_m\}}\int_{I}\psi_m\frac{\rho\circ h\cdot
        h'}{\rho}+\int_{I} \psi_m\Xi_m^2 \frac{\rho\circ \bar h_m\cdot \bar
        h_m'}{\rho}\right].
  \end{split}
  \]
  Note that there exists at least one $m_*$ such that $I_{m_*}\subset I$.  Moreover, at least
  $\frac {\Csei}{3K}$ of $I_{m_*}$ (hence at least $\frac {\Csei}{24 K}$ of $I$) is covered by intervals
  $J$ on which $\Xi_{m_*}=4/5$ and $\psi_{m_*}=1$.  Let $J_*$ be the union of such
  intervals.  Since $\frac{\rho(x)}{\rho(y)}\le e^{\Const |x-y|}$, for each $\eta\in (0,1)$,
  \[
  \begin{split}
    \int_{I} \psi_{m_*}\Xi_{m_*} \frac{\rho\circ \bar h_{m_*}\cdot \bar h_{m_*}'}{\rho}\le\int_{I\setminus J_*}\psi_{m_*}\frac{\rho\circ \bar h_{m_*}\cdot \bar h_{m_*}'}{\rho}+\frac {16}{25}\int_{J_*} \frac{\rho\circ \bar h_{m_*}\cdot \bar h_{m_*}'}{\rho}\\
    \le (1-\eta)\int_{I\setminus J_*}\psi_{m_*}\frac{\rho\circ \bar h_{m_*}\cdot \bar
      h_{m_*}'}{\rho}+\left(\eta \frac{|I|}{|J_*|}+\frac {16}{25}\right)\int_{J_*}
    \psi_{m_*}\frac{\rho\circ \bar h_{m_*}\cdot \bar h_{m_*}'}{\rho}.
  \end{split}
  \]
  Thus, choosing $\eta=\frac {9\Csei}{25(\Csei+24K)}$ we have
  \[
  \int_{I} \psi_{m_*}\Xi_{m_*} \frac{\rho\circ \bar h_{m_*}\cdot \bar h_{m_*}'}{\rho}\le
  (1-\eta)\int_{I} \psi_{m_*}\frac{\rho\circ \bar h_{m_*}\cdot \bar h_{m_*}'}{\rho}.
  \]
  Also note that there exists $M>0$ such that, for all $h\in\cH_{n_3}$ and $m\in\{1,\cdots, L_\paramL-1\}$,
  \[
  \int_{I}\psi_m\frac{\rho\circ h h'}{\rho}\le M\int_{I} \psi_{m_*}\frac{\rho\circ \bar h_{m_*}\bar h_{m_*}'}{\rho}.
  \]
  Moreover, note that $I$ can intersect at most $9$ intervals $I_m$.  By an argument similar
  to the above it then follows that there exists $\tau>0$ such that
  \begin{align*}
    \int_I\tOn_{\theta, 0}^{n_3}\Gamma_{g,k}^2&\le e^{-3\tau}\sum_m\sum_{h\in\cH_{n_3}}\int_{I}\psi_m\frac{\rho\circ h\cdot h'}{\rho}\\
    &=e^{-3\tau}\int_{I}\tOn_{\theta, 0}^{n_3}1=e^{-3\tau}|I|.\qedhere
  \end{align*}
\end{proof}

\section{A tedious computation}\label{app:tedious}
Here we perform explicitly the computations that lead to~\eqref{eq:random-var-def}. These
are simple but tedious computations that, in subsequent occasions, will be left to the
reader.  We provide this appendix so that the reader can see precisely how such
computations are done and be able to reproduce them when equally detailed proofs are not
provided.

Let us recall the starting point (see~\eqref{eq:Mvariable}):
\begin{align*}
  \bM_{k} &= \etaExph(t-\ve S_k,\bar\theta_{L_k})\left\{\sum_{j=0}^{L_k-1}\etaExp_{ j,L_k}\left[\ho(x_j,\theta_j)-\frac{\ve}2\bar\omega'(\bar\theta_{j})\bar\omega(\bar\theta_{j})\right]\right\} +\ve\,\bC_k\\
  \bC_k&=\frac  12\etaExph(t-\ve S_k,\bar\theta_{L_k})P(t-\ve S_k,\bar\theta_{L_k})\left[\sum_{j=0}^{L_k-1}\etaExp_{j,L_k}\ho(x_j,\theta_j)\right]^2.
\end{align*}
Recall, as already observed in Section~\ref{sec:results}, that
$\bar\omega\in\cC^{3-\alpha}$ for any $\alpha > 0$.  Let us compute term by term.
\begin{align*}
  &\etaExph(t-\ve S_k,\bar\theta_{L_k}) \sum_{j=0}^{L_k-1}\etaExp_{j,L_k}\ho(x_j,\theta_j) =
    \etaExph(t-\ve S_k,\thetaslk{L_k})  \sum_{j=0}^{L_k-1}\etaExps_{j,L_k}\ho(x_j,\theta_j)\\
  &
    +\partial_\theta\etaExph(t-\ve S_k,\thetaslk{L_k}) \partial_\theta \bar\theta(\ve L_k,\thetasl)\sum_{j=0}^{L_k-1}\etaExps_{j,L_k}\ho(x_j,\theta_j)(\theta_0-\thetasl)\\
  &+\ve\etaExph(t-\ve S_k,\thetaslk{L_k})\sum_{j=0}^{L_k-1}\sum_{l=0}^{L_k-1}\frac{\etaExps_{j,L_k}}{\left(1+\ve\bar\omega'(\thetaslk{l})\right)}\bar\omega''(\thetaslk{l})\partial_\theta \bar\theta(\ve l,\thetasl)\ho(x_j,\theta_j)(\theta_0-\thetasl)\\
  &+\cO(\ve^2\deltacomplex^2 L_k+\ve^{3-\delta_*}\deltacomplex^{2-\delta_*} L_k^2),
\end{align*}
where we have used that $|\theta_0-\thetasl|\leq \Const \ve\deltacomplex$,
by $\partial_\theta\etaExph$ we mean the derivative with respect to the second variable
and we have used the definition~\eqref{eq:hbolddef} of $\etaExp_{i,j}$. Now note that the
term on the second line of the previous equation is of type (recall Notation~\ref{not:bookcor})
\begin{align*}
  \ve\bbK^{k,1}_{\ell,2}=\ve\sum_{i_1,i_2}\fkC^{k,1}_{\ell, 2,( i_1, i_2)} A_{1,i_1}A_{2,i_2}
\end{align*}
where we have $\fkC^{k,1}_{\ell, 2,(i_1, i_2)}=0$, if $i_1\neq 0$, and
$A_{1,0}(x,\theta)=\ve^{-1}\cdot(\theta-\thetasl)$ while $A_{2,i_2}=\ho(x,\theta)$.  Note
that, provided $C^*$ has been chosen large enough, $\|A_{j,i_j}\|\nc1\leq C^*$, as
required. In fact, the terms has the extra property
$\|\bbK^{k,1}_{\ell,2}\|\nc0\leq C^*\deltacomplex L_k$, but we will not use this in the
following.  Similar arguments show that the term on the third line is of type
$\ve^2\bbK^{k,2}_{\ell,3}$: we can thus subsume both terms as a $\ve\bbK^{k,1}_{\ell,3}$
term.

Next,
\begin{align*}
  &-\frac{\ve}2\etaExph(t-\ve S_k,\bar\theta_{L_k})\etaExp_{
    j,L_k}\sum_{j=0}^{L_k-1}\bar\omega'(\bar\theta_{j})\bar\omega(\bar\theta_{j}) \\
    =&   -\frac{\ve}2\etaExph(t-\ve S_k,\thetaslk{L_k})\etaExps_{ j,L_k}\sum_{j=0}^{L_k-1}\bar\omega'(\thetaslk{j})\bar\omega(\thetaslk{j})\\
  &-\frac{\ve}2\partial_\theta\etaExph(t-\ve S_k,\thetaslk{L_k})\partial_\theta\bar\theta(\ve L_k,\thetasl)\etaExps_{ j,L_k}\sum_{j=0}^{L_k-1}\bar\omega'(\thetaslk{j})\bar\omega(\thetaslk{j}) (\theta_0-\thetasl)\\
  &-\frac{\ve}2\etaExph(t-\ve S_k,\thetaslk{L_k})\etaExps_{ j,L_k}\sum_{j=0}^{L_k-1}\left[\bar\omega''(\thetaslk{j})\bar\omega(\thetaslk{j})+\bar\omega'(\thetaslk{j})^2\right]\partial_\theta\bar\theta(\ve j,\thetasl)(\theta_0-\thetasl)\\
  &+\cO(L_k^2\ve^3\deltacomplex+\ve^{3-\delta_*}\deltacomplex^{2-\delta_*}L_k).
\end{align*}
The terms in the second and third line are of type $\ve^2\bbK^{k,1}_{\ell,1}$, which is a bound smaller that the one for the correlation terms already obtained.
Finally, for the last term we have
\[
\begin{split}
&\bC_k=\frac  12\etaExph(t-\ve S_k,\thetaslk{L_k})P(t-\ve S_k,\thetaslk{L_k})\left[\sum_{j=0}^{L_k-1}\etaExps_{j,L_k}\ho(x_j,\theta_j)\right]^2\\
&+\frac  12\partial_\theta\etaExph(t-\ve S_k,\thetaslk{L_k})\partial_\theta\bar\theta(\ve L_k,\thetasl)P(t-\ve S_k,\thetaslk{L_k})\left[\sum_{j=0}^{L_k-1}\etaExps_{j,L_k}\ho(x_j,\theta_j)\right]^2(\theta_0-\thetasl)\\
&+\frac  12\partial_\theta \etaExph(t-\ve S_k,\thetaslk{L_k})P(t-\ve S_k,\thetaslk{L_k})\partial_\theta\bar\theta(\ve L_k,\thetasl)\left[\sum_{j=0}^{L_k-1}\etaExps_{j,L_k}\ho(x_j,\theta_j)\right]^2(\theta_0-\thetasl)\\
&+\cO(L_k^3\ve^2\deltacomplex+\ve^2L_k^2\deltacomplex^2).
\end{split}
\]
Note that the first two lines can be interpreted as a $\bbK^{k,2}_{\ell,3}$ term; also,
any previous correlation term can be interpreted as a term of this type.  Collecting the
above facts, and recalling the constraints on $L_*$ and $\deltacomplex$, we
obtain~\eqref{eq:random-var-def}.

\bibliographystyle{plain}
\bibliography{rw}
\end{document}